\documentclass[a4paper]{amsart}
\usepackage[a4paper, margin=2cm]{geometry}

\usepackage[foot]{amsaddr}
\usepackage[british]{babel}
\usepackage{amsthm}
\usepackage{bbm}
\usepackage{xcolor}
\usepackage{hyperref}

\hypersetup{
  colorlinks   = true, 
  urlcolor     = {blue!50!black}, 
  linkcolor    = {blue!50!black}, 
  citecolor   = {red!50!black} 
}

\numberwithin{equation}{section}
 
\usepackage{amssymb}
\usepackage{mathtools}

\usepackage{verbatim}
\usepackage{amsmath}
\usepackage{amsfonts}
\usepackage{tikz-cd}
\usepackage{enumitem}
\usepackage{subcaption}
\usepackage{mathabx}
\usepackage{graphicx}
\usepackage[capitalize]{cleveref}
\usepackage[T1]{fontenc}

\newtheorem{thmx}{Theorem}
\newtheorem{corx}[thmx]{Corollary}

\newenvironment{subproof}[1][\proofname]{%
\begin{proof}[Proof of claim]%
	}{%
\end{proof}%
}

\newcommand{\sg}[1]{\langle #1 \rangle} 
\newcommand{\emp}{\varnothing}
\newcommand{\subg}[1]{\langle #1 \rangle}

\theoremstyle{definition}
\newtheorem{thm}{Theorem}[section]
\newtheorem{cor}[thm]{Corollary}
\newtheorem{prop}[thm]{Proposition}
\newtheorem{lem}[thm]{Lemma}
\newtheorem*{lem*}{Lemma}
\newtheorem*{cor*}{Corollary}
\newtheorem*{thm*}{Theorem}
\newtheorem*{fact*}{Fact}
\theoremstyle{definition}
\newtheorem{defn}[thm]{Definition}
\newtheorem{eg}[thm]{Example}
\newtheorem*{eg*}{Example}
\newtheorem*{examples*}{Examples}
\newtheorem{notn}[thm]{Notation}

\newtheorem{conj}[thm]{Conjecture}

\theoremstyle{remark}
\newtheorem{rem}[thm]{Remark}
\newtheorem*{claim*}{\textbf{Claim}}
\newtheorem*{term*}{Terminology}

\DeclareMathOperator{\Age}{Age}

\DeclareMathOperator{\Aut}{Aut}

\DeclareMathOperator{\centr}{C}
\DeclareMathOperator{\dom}{dom}
\DeclareMathOperator{\fcl}{fcl}
\DeclareMathOperator{\Fix}{Fix}
\DeclareMathOperator{\Fx}{Fx}
\DeclareMathOperator{\im}{im}
\DeclareMathOperator{\FrLim}{FrLim}
\DeclareMathOperator{\id}{id}
\DeclareMathOperator{\Norm}{N}
\DeclareMathOperator{\NFr}{NFr}
\DeclareMathOperator{\St}{St}
\DeclareMathOperator{\sSt}{sSt}
\DeclareMathOperator{\syl}{syl}
\DeclareMathOperator{\sylc}{sylc}
\DeclareMathOperator{\tp}{tp}

\DeclareMathOperator{\Inc}{I}
\DeclareMathOperator{\Dec}{D}
\DeclareMathOperator{\Fixed}{F}
\DeclareMathOperator{\Undef}{U}
\DeclareMathOperator{\SInc}{\Sigma_I}
\DeclareMathOperator{\SDec}{\Sigma_D}

\DeclareMathOperator{\supp}{supp}
\DeclareMathOperator{\Nx}{N}
\DeclareMathOperator{\qftp}{qftp}

\newcommand{\N}{\mathbb{N}}
\newcommand{\Q}{\mathbb{Q}}

\newcommand{\Z}{\mathbb{Z}}
\newcommand{\mb}[1]{\mathbb{#1}}
\newcommand{\mc}[1]{\mathcal{#1}}

\newcommand{\eps}{\varepsilon}
\newcommand{\sub}{\subseteq}

\newcommand{\tld}[1]{\tilde{#1}}
\newcommand{\fin}{\subseteq_{\text{fin\!}}}
\newcommand{\fg}{\subseteq_{\text{f.g.\!}}}

\newcommand{\la}{\leftarrow}

\newcommand{\ra}{\rightarrow}

\newcommand{\pl}{$(+)$}
\newcommand{\mi}{$(-)$}

\theoremstyle{definition}

\newtheorem{definition}[thm]{Definition}

\newtheorem{corollary}[thm]{Corollary}
\newtheorem{lemma}[thm]{Lemma}
\newtheorem{remark}[thm]{Remark}
\newtheorem*{remark*}{Remark}
\newtheorem{observation}[thm]{Observation}
\newtheorem{fact}[thm]{Fact}

\newtheorem{proposition}[thm]{Proposition}

\newtheorem{question}[thm]{Question}

\newtheorem*{notation}{Notation}

\newcommand{\D}[0]{\mathcal{D}(\Theta)}
\newcommand{\M}[0]{\mathcal{M}}

\newcommand{\Fr}{Fra\"{\i}ss\'{e} }

\def\Ind#1#2{#1\setbox0=\hbox{$#1x$}\kern\wd0\hbox to 0pt{\hss$#1\mid$\hss}
\lower.9\ht0\hbox to 0pt{\hss$#1\smile$\hss}\kern\wd0}
\def\ind{\mathop{\mathpalette\Ind{}}}
\def\Notind#1#2{#1\setbox0=\hbox{$#1x$}\kern\wd0\hbox to 0pt{\mathchardef
\nn="3236\hss$#1\nn$\kern1.4\wd0\hss}\hbox to
0pt{\hss$#1\mid$\hss}\lower.9\ht0
\hbox to 0pt{\hss$#1\smile$\hss}\kern\wd0}

\colorlet{dlnColor}{-green!40!yellow}
\colorlet{RobColor}{pink}

\newcommand{\sym}[0]{\mathtt{S}}
\newcommand{\alt}[0]{\mathtt{A}}
\newcommand{\cyc}[0]{\mathtt{C}}
\newcommand{\dih}[0]{\mathtt{D}}	
\newcommand{\fr}{\mathtt{F}}

\newcommand{\Wr}[0]{\mathcal{W}_r}

\newcommand{\Wvr}[0]{\mathcal{W}_{vr}} 
  
\makeatletter
\def\saveenum{\xdef\@savedenum{\the\c@enumi\relax}}
\def\resetenum{\global\c@enumi\@savedenum}
\makeatother

\newcommand{\Act}{\mathsf{A}}  
\newcommand{\Se}{\mathsf{Se}}
\newcommand{\Ta}{\mathsf{T}}
\newcommand{\STTa}{\mathsf{STT}}
\newcommand{\SHTa}{\mathsf{SHT}}
\newcommand{\TaF}{\mathsf{T}_{\mathcal{F}}(\mathcal{M}, G)}
\newcommand{\SHTaF}{\mathsf{SHT}_{k, \mathcal{F}}(\mathcal{M}, G)}

\newcommand{\Sh}{\mathsf{S}}
\newcommand{\ShHomog}{\mathsf{SH}}
\newcommand{\ShChar}{\widetilde{\mathsf{S}}}    
\newcommand{\ShHChar}{\widetilde{\mathsf{SH}}}
\newcommand{\ShTChar}{\widetilde{\mathsf{ST}}}

\let\phi\varphi

\author{J. de la Nuez Gonz\'{a}lez}
\email{jnuezgonzalez@gmail.com}
\thanks{The first author is supported by the Mid-Career Researcher Program (RS-2023-00278510) through the National Research Foundation funded by the government of Korea and by the KIAS individual grant SP084001.}

\author{Rob Sullivan}
\email{robertsullivan1990+maths@gmail.com}
\thanks{The second author is funded by Project 24-12591M of the Czech Science Foundation (GA\v{C}R), as well as by the Deutsche Forschungsgemeinschaft (DFG, German Research Foundation) under Germany’s Excellence Strategy EXC 2044–390685587, Mathematics M\"{u}nster: Dynamics–Geometry–Structure and CRC 1442 Geometry: Deformations and Rigidity.}



\begin{document}

\title{Two generalisations of sharp \texorpdfstring{$k$}{k}-transitivity}

\begin{abstract}
    An action $U \curvearrowleft G$ of a group $G$ on a set $U$ is \emph{sharply $k$-transitive} if, for any two $k$-tuples $\bar{a}, \bar{b} \in U^k$ of distinct elements, there is a unique $g \in G$ with $\bar{a} \cdot g = \bar{b}$. We consider two generalisations of this notion.

    Firstly, given a subgroup $\Theta$ of the symmetric group $\mathbb{S}_k$, we define a \emph{sharply $\Theta$-transitive} action $U \curvearrowleft G$ of a group $G$ on a set $U$ to be a $k$-set-transitive action where the restricted action on each $k$-set of its setwise-stabiliser is isomorphic to the permutation action $\mathbf{k} \curvearrowleft \Theta$. It is immediate that an action is sharply $\mathbb{S}_k$-transitive if and only if it is sharply $k$-transitive. We characterise for which $\Theta \leq \mathbb{S}_k$ there is a sharply $\Theta$-transitive action on an infinite set, and show that if such an action exists, then the acting group $G$ can be taken to be a finitely generated non-abelian virtually free group. As a consequence, we obtain for $k = 2, 3$ the first examples of non-split finitely-presented groups admitting sharply $k$-transitive actions on an infinite set, answering a question of Andr{\'{e}} and Tent, and we obtain a strengthening of the well-known result of Tits that, for $k \geq 4$, no group admits a sharply $k$-transitive action on an infinite set.

    Secondly, we consider a generalisation of sharp $k$-transitivity to first-order relational structures. Given an action $\mathcal{M} \curvearrowleft G$ of a group $G$ on a relational structure $\mathcal{M}$, we say that the action is \emph{sharply $k$-homogeneous} if, for any two $k$-tuples $\bar{a}, \bar{b}$ of distinct elements of $\mathcal{M}$ where $\bar{a} \mapsto \bar{b}$ is an isomorphism, there exists a unique $g \in G$ with $\bar{a} \cdot g = \bar{b}$. We show that, for $1 \leq k \leq 3$, a wide range of countable ultrahomogeneous structures admit sharply $k$-homogeneous actions by finitely generated non-abelian virtually free groups, giving a substantial answer to a question of Cameron from 1990 in the book \emph{Oligomorphic Permutation Groups}. We also determine, for each reduct $\mathcal{M}$ of $(\mathbb{Q}, <)$, the integers $k$ such that $\mathcal{M}$ admits a sharply $k$-homogeneous action, and we show that the generic $k$-hypertournament admits a sharply $k$-homogeneous action for $3 \leq k \leq 5$.
\end{abstract}

\maketitle

\section{Introduction}

An action $U \curvearrowleft G$ of a group $G$ on a set $U$ with $|U| \geq k$ is said to be \emph{sharply $k$-transitive} if for all $k$-tuples $\bar{a}, \bar{b}$ of distinct elements of $U$, there is a unique $g \in G$ with $\bar{b} = \bar{a} \cdot g$. Sharply $k$-transitive actions on sets have been extensively studied; we give more background below.

In this paper, we generalise the notion of a sharply $k$-transitive action in two ways. Firstly, we define a new notion of a \emph{sharply $\Theta$-transitive} action of a group on a set, where $\Theta$ is a subgroup of the symmetric group $\sym_k$:

\begin{defn} \label{d:theta-transitive}
    Let $k \geq 1$ and let $U \curvearrowleft G$ be a group action on a set $U$ with $|U| \geq k$. For $\Theta \leq \sym_k$, we say that the action $U \curvearrowleft G$ is \emph{sharply $\Theta$-transitive} if it is transitive on $k$-sets and if, for each $k$-set $A \sub U$, the action $A \curvearrowleft G_{\{A\}}$ of the setwise-stabiliser $G_{\{A\}}$ is isomorphic to the permutation action $\mathbf{k} \curvearrowleft \Theta$.
\end{defn}

(An isomorphism of actions $X \curvearrowleft G$, $X' \curvearrowleft G'$ is a bijection $X \to X'$ equivariant for the actions via a group isomorphism $G \to G'$: see Section \ref{s: notation}.) Note that if an action $U \curvearrowleft G$ is sharply $\Theta$-transitive for some $\Theta \leq \sym_k$, then it is \emph{$k$-sharp}: for any two $k$-tuples $\bar{a}, \bar{b}$ of distinct elements of $U$, there is at most one $g \in G$ with $\bar{b} = \bar{a} \cdot g$. Also note that an action is sharply $k$-transitive if and only if it is sharply $\sym_k$-transitive. See \Cref{s: intro robust subgps} for the results we prove regarding sharply $\Theta$-transitive actions: in particular, as an application of our results, we give the first examples of non-split finitely presented groups admitting sharply $2$-transitive/sharply $3$-transitive actions on an infinite set.

The second generalisation of sharp $k$-transitivity considered in this paper is the notion of a \emph{sharply $k$-homogeneous action} of a group on a \emph{relational structure} in the sense of first-order logic (for example, a graph, a digraph or a poset):

\begin{defn} \label{d: sharply k-homog}
    Let $k \geq 1$, and let $\mc{M}$ be a structure in a first-order relational language whose domain is of size $\geq k$. We say that a group action $\mc{M} \curvearrowleft G$ (where $G$ acts by automorphisms of $\mc{M}$) is \emph{sharply $k$-homogeneous} if, for any two $k$-tuples $\bar{a}, \bar{b}$ of distinct elements of $M$ such that $\bar{a}, \bar{b}$ have the same quantifier-free type, there is a unique $g \in G$ with $\bar{b} = \bar{a} \cdot g$. (Equivalently: for any isomorphism $f : \mc{A} \to \mc{B}$ of substructures $\mc{A}, \mc{B} \sub \mc{M}$ where $\mc{A}$, $\mc{B}$ are each of size $k$, there is a unique $g \in G$ whose action on $\mc{M}$ extends $f$.)\footnote{Some authors call an action which is transitive on $k$-sets a ``$k$-homogeneous action"; this conflicting terminology can be found in the permutation group literature in the context of actions of groups on sets (see for example \cite[Section 2.1]{Cam90}). We refer to such actions as \emph{$k$-set-transitive actions}. In this paper, we only use the term \emph{$k$-homogeneous action} for a group action on a \emph{structure}.}
\end{defn}

Unlike sharply $\Theta$-transitive actions on sets, sharply $k$-homogeneous actions on structures have been considered before by a number of authors: see Subsection \ref{ss: background on sh k homog}. Note that an action of a group on a pure set (that is, a set considered as a relational structure in the empty language) is sharply $k$-homogeneous if and only if it is sharply $k$-transitive.

We ask the following general questions:
\begin{question} \label{intro sh Theta tr q}
    For which $k \geq 1$, $\Theta \leq \sym_k$ and groups $G$ does there exist a sharply $\Theta$-transitive action of $G$ on an infinite set?
\end{question}
\begin{question} \label{intro general q}
    For which infinite relational structures $\mc{M}$, integers $k\geq 1$ and groups $G$ is there a sharply $k$-homogeneous action of $G$ on $\mc{M}$?    
\end{question}

Our first restriction in this paper is that we only consider actions on countably infinite sets and structures. Concerning Question \ref{intro general q}, we note that if $G$ has a sharply $k$-homogeneous action on a structure $\mc{M}$, then it is immediate that the action is faithful (so $G$ is isomorphic to a subgroup of $\Aut(\mc{M})$), and furthermore $\mc{M}$ is \emph{$k$-homogeneous}: any isomorphism between substructures of $\mc{M}$ of size $k$ can be extended to an automorphism of $\mc{M}$. As the class of relational structures which are $k$-homogeneous for some $k$ is rather broad, we will restrict our attention further and only consider \emph{ultrahomogeneous} structures: those structures which are $k$-homogeneous for all $k$. We will assume that the reader is well-acquainted with these in the second half of the paper (Sections \ref{s:seed actions on structures}--\ref{s:main theorem}) -- see \cite{Mac11} for a survey and \cite[Chapter 7]{Hod93} for further background. We refer to countably infinite ultrahomogeneous structures as \emph{\Fr structures}. We will also only consider relational \Fr structures with strong amalgamation (for example, the Rado graph or $\Q$ as a linear order). With these restrictions, we have the following (slightly paraphrased) question from \cite[pg 101]{Cam90}, a classic reference on oligomorphic permutation groups:

\begin{question} \label{Cameron qn}
    Let $\mc{M}$ be a relational \Fr structure with strong amalgamation, and let $k \geq 1$. Is there a group $G$ which acts sharply $k$-homogeneously on $\mc{M}$?
\end{question}

This paper began as an attempt to answer \Cref{Cameron qn}. We find that many well-known \Fr structures admit, for various $k$, a sharply $k$-homogeneous action by a finitely generated non-abelian virtually free group (a \emph{virtually free} group is a group with a free subgroup of finite index), as described in the results at the end of this introduction. We focus mostly on the cases $k = 2, 3$. In the course of answering \Cref{Cameron qn}, we noticed that the tools we developed also applied to the case of pure sets, and this led us to Question \ref{intro sh Theta tr q}. 

Before stating the main results, we give some background and motivation.
 
\subsection{Background: sharply \texorpdfstring{$k$}{k}-transitive actions on sets}

A group $G$ is said to be sharply $k$-transitive if it admits a sharply $k$-transitive action on a set. There is a 150-year history of results classifying sharply $k$-transitive groups.

Each sharply $1$-transitive action $U \curvearrowleft G$ is isomorphic as a $G$-action to the right multiplication action of $G$ on itself (pick a point $u_0 \in U$, and send each $u \in U$ to the unique $g \in G$ with $u = u_0 \cdot g$), and the right multiplication action is sharply $1$-transitive. Thus the problem of classifying sharply $k$-transitive groups is only interesting for $k \geq 2$.

There are several well-known classical examples of sharply $k$-transitive groups (in each case below, showing sharp $k$-transitivity is straightforward):
\begin{enumerate}[label=(\roman*)]
    \item \label{i: sf AGL} for a skew-field $K$, the affine group $\mathrm{AGL}(1, K)$ acts sharply $2$-transitively on $K$;
    \item \label{i: f PGL} for a field $K$, the projective group $\mathrm{PGL}(2, K)$ acts sharply $3$-transitively on the projective line over $K$ (a particularly well-known example of this is that the group of M\"{o}bius transformations acts sharply $3$-transitively on the Riemann sphere);
    \item for each $k \geq 4$, the symmetric groups $\sym_k$, $\sym_{k+1}$ and the alternating group $\alt_{k+2}$ are sharply $k$-transitive.
\end{enumerate}

There is a full classification of finite sharply $k$-transitive groups: see \cite{Pas68} for an exposition of this, and \cite[Section 7.6]{DM96} for an introduction to the topic. Zassenhaus classified the finite sharply $k$-transitive groups for $k = 2, 3$ (see \cite{Zas35K, Zas35U}, \cite{Pas68}, \cite{DM96} -- this classification of course includes examples \ref{i: sf AGL}, \ref{i: f PGL} above in the case where $K$ is finite, and the other cases in the classification are closely related to these two examples). Jordan classified the finite sharply $k$-transitive groups for $k \geq 4$ (\cite{Jor72}): for each $k \geq 4$, the symmetric groups $\sym_k$, $\sym_{k+1}$ and the alternating group $\alt_{k+2}$ are sharply $k$-transitive, and there are two additional sporadic cases: the Mathieu group $M_{11}$ is sharply $4$-transitive, and similarly $M_{12}$ is sharply $5$-transitive.

For infinite $G$, much less is known. The only possible cases are where $k = 2, 3$: Tits (\cite{Tit49}, \cite[Chapitre IV, Th\'{e}or\`eme 1]{Tit52}) proved that for $k \geq 4$ there are no infinite sharply $k$-transitive groups. Yoshizawa strengthened this in \cite{Yos79}, showing that if a group acts $4$-transitively on an infinite set, then the stabiliser of any $4$-tuple must be infinite (this generalises a result of Hall from \cite{Hal54} showing that any group with a $4$-transitive action on an infinite set does not have a stabiliser of a $4$-tuple with finite odd order).

A key property of finite sharply $2$-transitive groups is that they are \emph{split}, which we now define. Recall that, given a group action $X \curvearrowleft G$, a subgroup $H \leq G$ is \emph{regular} if the restriction of the action to $H$ is regular. We say that a group $G$ is \emph{split} if it has a regular normal subgroup $N$: in this case, it is straightforward to see from the definitions that $G$ can be written as the semidirect product of $N$ with a point-stabiliser. The importance of this lies in the following: if a sharply $2$-transitive group $G$ (finite or infinite) is split with regular normal subgroup $N$, then $N$ is abelian and one can define a near-field $F$ on $N$ such that $N$ is the additive group of $F$ (here, a \emph{near-field} is a skew-field where we only require right-distributivity: namely $(a + b) \cdot c = (a \cdot c) + (b \cdot c)$). The sharply $2$-transitive action $X \curvearrowleft G$ is then isomorphic to the action $F \curvearrowleft \mathrm{AGL}(1, F)$. See \cite[Section 7.6]{DM96} for more details of this.

To see that each finitely sharply $2$-transitive group $G$ is split, one can show the existence of a regular normal subgroup of $G$ via the structure theorem for finite Frobenius groups (\cite{Fro02} -- see \cite{Pas68} for an exposition), or via a direct elementary argument.\footnote{The following argument was provided to us by Peter Cameron. Let $X \curvearrowleft G$ be sharply $2$-transitive with $|X| = n$. Let $K$ be the set of elements of $G$ which act freely. It is straightforward that $|G| = n(n-1)$, $|K| = n$ and $|G_x| = n-1$ for each $x \in X$. For each $g \in G \setminus \{1\}$ with $\mathrm{order}(g) \mid n$, we have $\mathrm{order}(g) \nmid n-1$ and thus $g \in K$. Let $p$ be prime with $p \mid n$; take $g \in G$ of order $p$, so $g \in K$. By $2$-sharpness, the centraliser $\centr_G(g)$ is contained in $K$, so the conjugacy class $g^G$ has size $ \geq n-1$; thus $K = g^G \cup \{1\}$ and so $n$ is a power of $p$. Let $H$ be a Sylow $p$-subgroup of $G$. Then $|H| = n$ and $H \sub K$, so $H = K$, and as $K$ is a union of conjugacy classes we have $K \trianglelefteq G$. As each point-stabiliser of $K$ is trivial and $|X| = |K|$, the action $X \curvearrowleft K$ is transitive, so $K$ is regular. Also $K$ is a finite $p$-group, so has non-trivial centre via the class equation; all elements of $K \setminus \{1\}$ are conjugate, so $K$ is abelian.} It was a long-standing open question in the area whether all infinite sharply $2$-transitive groups are likewise also split. Many classes of infinite sharply $2$-transitive groups are indeed split (see \cite{Tit52b}, \cite{Tit56}, \cite{Ker74}, \cite{Wah86}, \cite{Ten00}, \cite{Tur04}, \cite{May06}, \cite{GG14}; for the case of finite Morley rank see \cite{BN94}, \cite{ABW18}, \cite{CT20}, \cite{CT23}). However, this question was answered in the negative in a relatively recent paper of Rips, Segev and Tent (\cite{RST17}), where the authors gave the first examples of non-split infinite sharply $2$-transitive groups. This sparked further interest in the area in recent years: see, for example, \cite{RT19}, \cite{TZ16}, \cite{Ten16I}, \cite{Ten16}, \cite{ABW18}, \cite{GG21}, \cite{AAT23}, \cite{AG24}, \cite{Ame25}, \cite{AA26}.

\begin{rem}
    In some of the above literature (for example \cite{AT23}), a sharply $2$-transitive group $G$ is defined to be split if it has a non-trivial \emph{abelian} normal subgroup. This is equivalent to our definition, as follows.
    
    Let $X \curvearrowleft G$ be a sharply $2$-transitive action, and let $N \trianglelefteq G$, $N \neq \mathtt{1}$. Suppose that $N$ is abelian. Consider the partition of $X$ into $N$-orbits: as $N \trianglelefteq G$, this partition is $G$-invariant, so by sharp $2$-transitivity of $X \curvearrowleft G$ the partition consists of a single part, and thus $X \curvearrowleft N$ is transitive. Freeness of $X \curvearrowleft N$ is straightforward to check (one uses sharp $2$-transitivity and the fact that $N$ is abelian), and so $N$ is regular. For the opposite direction: given a regular normal subgroup $N \trianglelefteq G$, one can define a near-field $F$ on $N$ (see \cite[Section 7.6]{DM96}); one may define a near-field without assuming commutativity of the additive group $F^+$, and then commutativity of $F^+$ follows from the other axioms (see \cite{Neu40} or \cite[Exercise 7.6.5]{DM96}), so $N$ is abelian.
\end{rem}

\subsection{Background: structural actions on \Fr structures} \label{ss: background on sh k homog}

One of the questions motivating this paper, \Cref{Cameron qn}, has been previously investigated in the case $k = 1$ in \cite{Che15}, \cite{JK04}, \cite{Cam00}, \cite{Hen71}. Sharply $1$-transitive actions on sets are usually referred to as \emph{regular} actions, and so we will refer to sharply $1$-homogeneous actions as \emph{structurally regular actions}.

In \cite{Che15}, Cherlin shows that any transitive \Fr structure with free amalgamation in a finite binary relational language has a regular action by a countable group $G$, and sufficient conditions are given for $G$ to have such an action (these are satisfied, for example, by the free nilpotent group of class 2 with $\omega$-many generators). This generalises results of Cameron in \cite{Cam00}, where it is shown that the random graph has a regular action by any countable group with a homomorphism onto $\Z$. In \cite{Cam00} other particular examples of \Fr structures with regular actions are given, and in some cases there are results regarding which groups $G$ admit such actions. In \cite{JK04}, it is shown that any countable group without involutions has a regular action on the random tournament, and in \cite{Hen71}, Henson showed that the generic triangle-free graph $\Gamma_3$ has a regular action by $\Z$. (The papers \cite{Cam00, Che15} also give results on regular normal actions, which we will not consider.)

The case $k > 1$ was considered by Macpherson for $(\Q, <)$: \cite[Theorem 1.1]{Mac96} states that a free group of rank $\omega$ admits a sharply $k$-homogeneous action on $(\Q, <)$ for all $k \geq 1$. Macpherson also showed in Proposition 3.2 of the same paper that if $\mc{M}$ is an $\omega$-categorical \Fr structure with an involutary automorphism fixing at least two points, then $\mc{M}$ does not admit a sharply $4$-homogeneous action. As far as the authors are aware, prior to the current paper these were the only results in the literature regarding \Cref{Cameron qn} for $k > 1$.

\subsection{Background: a sharply \texorpdfstring{$2$}{2}-homogeneous action on the Rado graph due to Amelio and Winkel} \label{ss: split sh 2 homog on rado}

Marco Amelio and Jeroen Winkel gave the initial impetus for this project. They had considered sharply $2$-homogeneous actions on the random graph, and provided us with the following example (in informal communication). Let $p$ be a prime with $p \equiv 1 \pmod{4}$. By the L\"{o}wenheim-Skolem theorem, there is a countably infinite field $\mc{F}$ which is an elementary submodel of the ultraproduct of all finite fields of characteristic $p$ (taken up to isomorphism) over a non-principal ultrafilter on $\N$. As each finite field $\mb{F}_q$ of characteristic $p$ satisfies:
\begin{enumerate}[label=(\roman*)]
    \item \label{i: -1 sq} $-1$ is a square,
    \item \label{i: sqs index 2 subgp} the subgroup of non-zero squares has index $2$ in the multiplicative group of the field;
\end{enumerate}
by {\L}o{\'{s}}'s theorem we have that $\mc{F}$ also satisfies \ref{i: -1 sq}, \ref{i: sqs index 2 subgp}. Define a binary relation $E$ on the underlying set $F$ of $\mc{F}$ by: $uEv$ if $u - v$ is a non-zero square in $\mc{F}$. By \ref{i: -1 sq} the relation $E$ is symmetric, so the structure $\Gamma := (F, E)$ is a graph (called the Paley graph of $\mc{F}$). By \cite[Theorem 3]{BT81} (or an easy adaptation of \cite{GS71}), the graph $\Gamma$ is isomorphic to the random graph.

Consider the affine general linear group $\mathrm{AGL}(1, \mc{F}) = \{g : F \to F \mid g(x) = ax + b \text{ for some } a, b \in \mc{F} \text{ with } a \neq 0\}$. Let $\Sigma \leq \mathrm{AGL}(1, \mc{F})$ consist of those $g : x \mapsto ax + b$ in $\mathrm{AGL}(1, \mc{F})$ for which $a$ is a square. Let $\mu : F \curvearrowleft \Sigma$ be the natural action of $\Sigma$ on $F$. By \ref{i: sqs index 2 subgp}, the map $\rho : \mc{F}^\times \to \{\pm 1\}$ sending squares/non-squares to $+1$/$-1$ respectively is a homomorphism, and so $\mu$ is an action by automorphisms of $\Gamma$. It is straightforward to check that $\mu$ is sharply $2$-homogeneous using the fact that $\rho$ is a homomorphism.

We do not know how to generalise this example. Note that \cite{BT81}, \cite{GS71} rely on difficult results for finite fields.

\subsection{Background: stationary weak independence relations} \label{intro: swir background}

Our main results concern \Fr structures with a certain kind of independence between finite substructures known as a \emph{stationary weak independence relation (SWIR)}, a generalisation of the well-known notion of a stationary independence relation (SIR) due to Tent and Ziegler. Let $\mc{M}$ be a \Fr structure. A SWIR on $\mc{M}$ is a ternary relation $\ind$ on finitely generated substructures of $\mc{M}$, where we write $B \ind_A C$ for $A, B, C \fg \mc{M}$ and say that ``$B$ and $C$ are $\ind$-independent over $A$", such that $\ind$ satisfies certain axioms: automorphism-invariance, existence, stationarity and monotonicity (see \Cref{d:swir}). Equivalently, the age of $\mc{M}$ (the class of finitely generated structures embeddable in $\mc{M}$) has a \emph{standard amalgamation operator (SAO)}: a particularly ``nice" notion of amalgamation which, again, satisfies certain desirable axioms. For example, any free amalgamation class has a SAO, as do the classes of finite tournaments, finite linear orders and finite partial orders. Thus, the \Fr limits of these classes have SWIRs and our main theorem applies to them.

A stationary independence relation (SIR) as defined by Tent and Ziegler (\cite{TZ13}) has the same definition as a SWIR but additionally assumes symmetry ($B \ind_A C$ iff $C \ind_A B$). The paper \cite{TZ13} pioneered the use of SIRs to prove simplicity of automorphism groups of symmetric \Fr structures, extending results in \cite{MT11} for free amalgamation classes. Li originally defined SWIRs in the later preprint \cite{Li19} by removing the symmetry assumption in the definition of a SIR, and used these to prove simplicity of the automorphism groups of certain asymmetric \Fr structures (see also \cite{Li20}, \cite{Li21}, \cite{BSWY26}).

We also note that a related, weaker notion of a \emph{canonical independence relation (CIR)} (unfortunately, the terminology in this area is not yet standardised) is defined in \cite{KS19}. Though we phrase our results in terms of SWIRs, we make use of a proof idea from \cite{KS19} found in the proof of Theorem 3.12 in that paper, and we consider the paper as a key general source of inspiration. The authors do not know any example of a strong relational \Fr structure with a CIR but without a SWIR.

\subsection{Main results for actions on sets} \label{s: intro robust subgps}

The below definition of \emph{robustness} of a permutation group enables us to completely characterise for which $\Theta \leq \sym_k$ there exists a sharply $\Theta$-transitive action on an infinite set.

\begin{defn} \label{d:non-free part}
    Let $\mu : U \curvearrowleft G$ be an action of a group on a set, and let $\Omega \leq G$. We define $\NFr_\mu(\Omega)$ to be the union of the non-free orbits of the action given by the restriction of $\mu$ to $\Omega$. For $\Omega \neq \mathtt{1}$, we define $\Fx_\mu(\Omega)$ to be the set of elements of $U$ which are fixed by each non-trivial element of $\Omega$, and we also define $\Fx_\mu(\mathtt{1}) = \emp$.  
\end{defn}

\begin{defn}\label{d:robust permutation group}
    Let $\Theta \leq \sym_k$. We write $\pi: \mathbf{k} \curvearrowleft \Theta$ for the permutation action of $\Theta$ on $k$. We say that $\Theta$ is \emph{robust} if for each $\Omega \leq \Theta$ and each non-empty $\Omega$-invariant subset $X \sub \NFr_\pi(\Omega)$, we have that $|\Omega|$ does not divide $|X|$. 
\end{defn}

Note that robustness is a property of $\Theta$ as a permutation group, and that if $\Theta$ is robust, so is any subgroup of $\Theta$.

The first main result of this paper is the following:

\begin{thmx} \label{t: sharply Theta-transitive}
    Let $k\geq 1$ and $\Theta\leq\sym_{k}$. Then there exists a sharply $\Theta$-transitive action of a group on an infinite set if and only if $\Theta$ is robust. If such an action exists, the acting group can be taken to be a finitely generated non-abelian virtually free group. 
\end{thmx}

In \cref{p: Theta not robust}\ref{i: no sh Theta-tr}, we show the ``only if" direction, and we prove a more detailed version of the ``if" direction in \cref{p:generic action on sets}.

The symmetric group $\sym_k$ is not robust for $k \geq 4$: a transposition fixes at least two points. Thus Theorem \ref{t: sharply Theta-transitive} immediately implies the well-known result of Tits that, for $k \geq 4$, there is no sharply $k$-transitive action of a group on an infinite set. We shed further light on this result: in \Cref{l:examples} we show that $\alt_4 \leq \sym_4$ and $\alt_5 \leq \sym_5$ are robust, and so in particular Theorem \ref{t: sharply Theta-transitive} implies that there \emph{do} exist sharply $\alt_4$-transitive and sharply $\alt_5$-transitive actions of a group on an infinite set. 

The groups $\sym_2$, $\sym_3$ are robust (see \Cref{l:examples}). The finitely generated non-abelian virtually free groups that we take in Theorem \ref{t: sharply Theta-transitive} in the cases $\Theta = \sym_2, \sym_3$ (see \cref{ex:s2 s3}) are finitely presented and non-split, and the sharply $\Theta$-transitive actions they admit are of characteristic $2$ (see the end of Section \ref{s:constructing actions on sets}), so we have the following corollary of Theorem \ref{t: sharply Theta-transitive}:

\begin{corx} \label{c: non split fin pres}
    Let $k = 2, 3$. Then there exists a non-split finitely presented group $G$ with a sharply $k$-transitive action on an infinite set, where the sharply $k$-transitive action has characteristic $2$.

    If $k = 2$, one may take the said group $G$ to be $G = (\sym_2 \times \fr_2) \ast \fr_2$. If $k = 3$, letting $\sigma \in \sym_3$ be an involution and $\tau \in \sym_3$ an element of order $3$, one may take $G = ((\sym_{3}\ast_{\sg{\sigma}}(\sg{\sigma}\times \fr_2))\ast_{\sg{\tau}}(\sg{\tau}\times \fr_2))\ast \fr_2$.
\end{corx}

(A sharply $2$-transitive action is of characteristic $2$ if all involutions act freely, and a sharply $3$-transitive action is of characteristic $2$ if all involutions have exactly one fixed point -- see for example \cite{Ten16I}.)

We prove Corollary \ref{c: non split fin pres} in a stronger form (including a genericity result) as Corollary \ref{c: non-split fin pres from prop}. This corollary answers a question of Andr\'{e} and Tent which asked whether finitely presented non-split sharply $2$-transitive groups exist (see \cite[Question 1.12]{AT23} -- though note that we do not answer the part of the question regarding simplicity).

\subsection{Main results for actions on structures} \label{ss: intro main str results}

We now give a summary of the main results for actions on structures in a manner which avoids too many technicalities. We actually prove a stronger result than Theorem \ref{t:main_intro}: see Theorem \ref{t:main}, where we also give genericity results in spaces of actions.

\begin{term*} \hfill
    \begin{itemize}
        \item Recall that a structure $\mc{A}$ is \emph{rigid} if it has trivial automorphism group, and \emph{transitive} if all its one-element substructures are isomorphic.
        \item Recall that a \Fr class $\mc{K}$ has \emph{strong amalgamation} if, for all pairs of embeddings $B_0 \la A \ra B_1$ in $\mc{K}$, there exists an amalgam $B_0 \ra C \la B_1$ in $\mc{K}$ such that the images of $B_0$, $B_1$ intersect exactly in the image of $A$. (Some authors refer to this as \emph{disjoint amalgamation}.) We also say that a \Fr limit has strong amalgamation if its age does.
        \item See \Cref{d:swir} for the definition of a SWIR, and recall that a SIR is a symmetric SWIR. See \cref{intro: swir background} for a brief introduction to SWIRs.
    \end{itemize}
\end{term*} 

\begin{thmx}\label{t:Theta only if structural}
    Let $\Theta \leq \sym_k$. Suppose that $\Theta$ is not robust. Then for any $\omega$-categorical relational \Fr structure $\mc{M}$ such that there is $\mc{A}_0 \sub \mc{M}$ of size $k$ with $A_0 \curvearrowleft \Aut(\mc{A}_0)$ isomorphic to the standard permutation action $\pi : \mathbf{k} \curvearrowleft \Theta$, there is no sharply $k$-homogeneous action of a group on $\mc{M}$.
\end{thmx}

We prove Theorem \ref{t:Theta only if structural} as Proposition \ref{p: Theta not robust}\ref{i: no sh k-homog}.

We prove the below Theorem \ref{t:main_intro} in a stronger form as Theorem \ref{t:main}.

\begin{thmx}[see Theorem \ref{t:main}] \label{t:main_intro}   
    Let $\M$ be a relational \Fr structure with strong amalgamation.
    \begin{enumerate}[label=(\Roman*)]
        \item Let $k \geq 1$. If $\M$ has a SWIR and substructures of $\M$ of size $k$ are rigid, then $\M$ admits a sharply $k$-homogeneous action by any finitely generated non-abelian free group.  
        \item If $\M$ has a SIR, then $\M$ admits a sharply $2$-homogeneous action by a finitely generated non-abelian virtually free group. If in addition $\M$ is transitive, then $\M$ admits a sharply $3$-homogeneous action by a finitely generated non-abelian virtually free group.
        \item The random tournament admits a sharply $3$-homogeneous action by a finitely generated non-abelian virtually free group.
    \end{enumerate}
    We also have the following for particular \Fr structures:
    \begin{enumerate}[label=(\Roman*), resume]
        \item The following holds for the reducts of $(\mathbb{Q},<)$: 
  	\begin{enumerate}[label=(\roman*), ref=(\Roman{enumi}.\roman*)]
            \item For $k \geq 1$, the dense linear order $(\Q,<)$ admits a sharply $k$-homogeneous action by any finitely generated non-abelian free group. 
            \item For $k \geq 1$, the betweenness structure $(\Q,B^{(3)})$ admits a sharply $k$-homogeneous action by a finitely generated non-abelian virtually free group. 
            \item For $k \geq 1$, the circular order $(\Q, C^{(3)})$ admits a sharply $k$-homogeneous action by a finitely generated non-abelian virtually free group.
            \item For $k = 1, 2$ and odd $k \geq 3$, the separation structure $(\Q, S^{(4)})$ admits a sharply $k$-homogeneous action by a finitely generated non-abelian virtually free group. For even $k \geq 4$, there is no sharply $k$-homogeneous action of a group on $(\Q, S^{(4)})$.
        \end{enumerate}
        \item Let $3 \leq k \leq 5$ and let $\M$ be the random $k$-hypertournament. Then $\M$ admits a sharply $k$-homogeneous action by a finitely generated non-abelian virtually free group. 
    \end{enumerate}
\end{thmx}

For each \Fr structure $\mc{M}$ to which Theorem \ref{t:main_intro} applies, we shall see that there is a robust permutation group $\Theta\leq\sym_{k}$ such that, for each substructure $\mc{A} \sub \mc{M}$ with $|A|=k$, the action $A \curvearrowleft \Aut(\mc{A})$ is isomorphic to the permutation action $\pi : \mathbf{k} \curvearrowleft \Theta$ or the permutation action of a subgroup of $\Theta$ with a free orbit.

\begin{eg} \label{ex: examples for main str thm}
    We note the following examples of \Fr structures with SWIRs and SIRs to which the above theorem can be applied. (See \Cref{ex: structures with SWIRS}.)
    \begin{itemize}
        \item Let $1 \leq k \leq 3$. Let $\M$ have free amalgamation, and in the case $k = 3$ additionally assume that $M$ is transitive. Then $\M$ has a SIR (given by the free amalgam), and so admits a sharply $k$-homogeneous action. Examples of such $\M$ include the random graph, the random $K_n$-free graph, the Henson digraphs and the random $n$-hypergraph. The generic poset also has a SIR, and so admits a sharply $k$-homogeneous action for $1 \leq k \leq 3$.
        \item Let $1 \leq k \leq 3$. The generic $n$-linear order (the free superposition of $n$ linear orders) and the generic order expansion of a \Fr structure with free amalgamation (e.g.\ the generic ordered graph) have a SWIR, and all $k$-substructures are rigid, and thus they admit a sharply $k$-homogeneous action.
    \end{itemize}
\end{eg}

For a description of the reducts of $(\Q, <)$, see \cite[Example 2.3.1]{Mac11}. These were originally classified in \cite{Cam76}. We now define the $k$-hypertournaments appearing in the above theorem.
\begin{defn} \label{d: ht def}
    Let $k \geq 2$. We define a \emph{$k$-hypertournament} to be a structure $\mc{A}$ in a language consisting of a single $k$-ary relation symbol, such that any substructure of $\mc{A}$ of size $k$ has automorphism group $\alt_k$. For instance, in the case $k = 2$ this gives the usual notion of a tournament, and in the case $k = 3$ a $3$-hypertournament is a set together with a cyclic orientation of each triple of vertices. It is clear that finite $k$-hypertournaments form a strong amalgamation class, and we call the \Fr limit the \emph{generic $k$-hypertournament}. In this paper, for brevity we usually just say \emph{$k$-tournament} instead of $k$-hypertournament.
\end{defn}
 
\subsection{Structure of the paper}

We have attempted to structure this paper so that it can be read by two distinct audiences. A reader only interested in sharply $k$-transitive actions on sets can read Sections \ref{s: notation}--\ref{s:constructing actions on sets} for Theorem \ref{t: sharply Theta-transitive} and Corollary \ref{c: non split fin pres} without any knowledge of model theory or ultrahomogeneous structures, and they may also assume $H = \Theta$ when considering seed groups (see Definition \ref{d:seed group} -- the full generality of the definition of $H$ is only used in the latter half of the paper regarding structures). A reader interested in actions on structures but without much background in geometric group theory (in particular Bass-Serre theory) can treat several of the lemmas in Section \ref{s:completions} requiring this background as black boxes.

\cref{s: notation} establishes some notation and conventions. \cref{s:robust groups and seed groups} elaborates on the notion of a robust subgroup. It also defines the more general notions of a seed group $H$ and of a seed action of such a group on an infinite set. Seed actions are the initial actions which we then extend to actions of larger groups $G$ called completions. Completions and their partial actions are the main subject of \cref{s:completions}, which contains some basic results in combinatorial group theory needed further on. \cref{s:extending sharp partial actions on sets} is a core section of the paper and establishes sufficient criteria for extensions of a $k$-sharp partial action of a completion to remain $k$-sharp when the partial action is extended. This will be enough to be able to prove \cref{t: sharply Theta-transitive} as \cref{p:generic action on sets} in the following \cref{s:constructing actions on sets}.
 
In \cref{s:seed actions on structures} we examine the arguments in the previous sections and provide an abstract outline of a set of sufficient conditions for a strong relational \Fr structure $\M$ to admit sharply $k$-homogeneous actions. We phrase this in terms of a notion of a pleasant seed action of a seed group on $\M$. In \cref{s:structures and independence} we recall some basic definitions and results around the notion of a stationary weak independence relation (SWIR). \cref{s:making independent} is rather technical and is centered around the notion, inspired by \cite{KS19}, of a family of partial actions with the escape property. We show how to apply these ideas in \cref{s:ind use section}. At this point it still remains to show that in some cases it is possible to construct actions of seed groups on $\M$ for which certain groups have large centralisers (in some precise sense), which is done in \cref{s:construction good finite actions}.
 
To conclude, in \cref{s:main theorem} we collect all the results from the previous Sections \ref{s:seed actions on structures}--\ref{s:construction good finite actions} into \cref{t:main}, a stronger version of \cref{t:main_intro}, and in \cref{s:questions} we propose some directions for further research. 

\subsection*{Acknowledgements}

The authors would like to thank Marco Amelio and Jeroen Winkel for posing the initial questions that led to this project, for allowing us to include in Subsection \ref{ss: split sh 2 homog on rado} their argument giving a sharply $2$-homogeneous action on the random graph, as well as for numerous helpful conversations throughout.

\section{Notation and conventions} \label{s: notation}

We collect the most frequently used notation here for the convenience of the reader. Most of our notation will be relatively standard.

Let $k \in \N$ and let $U$ be a set. We write $[U]^k$ for the set of subsets of $U$ of size $k$. We refer to the elements of $[U]^k$ as \emph{$k$-sets}. We write $(U)^k$ for the set of $k$-tuples (elements of $U^k$) with all elements distinct. We sometimes write $\bar{a} \sub U$ to mean that $\bar{a} \in U^k$ for some $k \in \N$.

Let $G$ be a group. We write $1$ for the identity of $G$, and write $G^\ast = G \setminus \{1\}$. We write $\mathtt{1}$ for the trivial group. When considering a generating set of a group, we always assume it does not contain the identity.

\subsection{Actions on sets}

We work with right actions in this paper. We denote a right action $\mu$ of a group $G$ on a set $U$ by $\mu : U \curvearrowleft G$, writing $\mu_g$ for the $\mu$-action of $g \in G$. Given a $G$-invariant set $V \sub U$ in the action $\mu$, we write $\mu|_V$ for the restriction of $\mu$ to $V$. Given a subgroup $H \leq G$, we also write $\mu|_H$ for the restriction of $\mu : U \curvearrowleft G$ to $H$. (Whether restriction is to an invariant set or a group will be clear from context.) We write $\mu|_{V, H}$ for the restriction of $\mu$ to the $G$-invariant set $V$ and the subgroup $H$ (so $\mu|_{V, H}$ is an action $V \curvearrowleft H$).

For $k \in \N$, we write $\mathbf{k}$ when we wish to emphasise that we are considering $k$ as the set $\{0, \cdots, k-1\}$. We often use this notation in the context of actions of subgroups of the symmetric group $\sym_k$, the group of bijections of $\mathbf{k} = \{0, \cdots, k-1\}$.

Throughout this paper, given a permutation group $\Theta \leq \sym_k$, we write $\pi : \mathbf{k} \curvearrowleft \Theta$ for the standard permutation action of $\Theta$ on $\mathbf{k}$.

Let $U \curvearrowleft G$ be an action of a group $G$ on a set $U$. We call $U$ a \emph{$G$-set}, and we say \emph{``let $U$ be a $G$-set"}, to mean ``take an action $U \curvearrowleft G$".

Let $U \curvearrowleft G$, $U' \curvearrowleft G'$ be two actions, and let $f : G \to G'$ be a group isomorphism. Let $\alpha : U \to U'$ be a map of sets. We say that $\alpha$ is an \emph{action isomorphism} (via $f$) if it is a bijection and $\alpha(u \cdot g) = \alpha(u) \cdot f(g)$. If an action isomorphism exists, we write $(U \curvearrowleft G) \simeq (U' \curvearrowleft G')$ (with $f$ usually clear from context). In the frequent situation that $G = G'$ and $f = \id_G$, we say that $\alpha$ is a \emph{$G$-isomorphism} and refer to the property $\alpha(u \cdot g) = \alpha(u) \cdot g$ as \emph{$G$-equivariance}.

Given an action $\lambda : U \curvearrowleft G$ and $A \sub U$,  as usual we write $G_{\{A\}}$ for the setwise-stabiliser of $A$ and $G_{(A)}$ for the pointwise-stabiliser of $A$, but we will also often write $\St(A) = G_{(A)}$ and $\sSt(A) = G_{\{A\}}$. We say that $A \sub U$ is $G$-invariant if $G_{\{A\}} = G$.

For $G_0 \sub G$, $G_0 \neq \mathtt{1}$, we write $\Fx_\lambda(G_0)$ for the set of elements of $U$ fixed (pointwise) in the action $\lambda$ by each non-trivial element of $G_0$. We define $\Fx_\lambda(\mathtt{1}) = \emp$. We may write $\Fx(G_0)$ if $\lambda$ is clear from context. For $G$-invariant $A \sub U$, we write $\Fx_\lambda^A(G_0) = \Fx_\lambda(G_0) \cap A$.

\subsection{Actions on structures}

We write first-order structures $\mc{M}$ in calligraphic script and sets $U$ in normal script. (We do not strictly adhere to this convention when working with SWIRs from Section \ref{s:structures and independence} onwards: see Remark \ref{r: SWIR notn conv for substructures}.) Given a structure $\mc{M}$, we write $M$ for its domain. 

For $G$ a group and $\mc{M}$ a first-order structure, a \emph{structural action} $\mu : \mc{M} \curvearrowleft G$ is an action $\mu : M \curvearrowleft G$ where for each $g \in G$ we have $\mu_g \in \Aut(\mc{M})$. When we write $\mc{M} \curvearrowleft G$ in this paper, rather than $M \curvearrowleft G$, we refer to a structural action (sometimes we emphasise this explicitly for presentational clarity). A \emph{structural action isomorphism} of structural actions $\mc{M} \curvearrowleft G$, $\mc{M}' \curvearrowleft G'$ is an action isomorphism of the underlying set actions $M \curvearrowleft G$, $M' \curvearrowleft G'$ which is also an isomorphism $\mc{M} \to \mc{M}'$, and as in the case of sets, if a structural action isomorphism exists then we write $(\mc{M} \curvearrowleft G) \simeq (\mc{M}' \curvearrowleft G')$. We usually just refer to a structural action isomorphism as an action isomorphism, as the calligraphic script notation makes this clear from context; the same applies to $G$-isomorphisms $\mc{M} \to \mc{M}'$, where we shall assume without comment that these are also isomorphisms of structures.

\section{Robust groups and seed groups} \label{s:robust groups and seed groups}

Recall \cref{d:robust permutation group} of a robust subgroup $\Theta \leq \sym_k$, and recall that $\pi$ denotes the natural permutation action $\mathbf{k} \curvearrowleft \sym_k$.

\subsection{Consequences of non-robustness}

The proof of the below Proposition \ref{p: Theta not robust} is inspired by the proof of \cite[Proposition 3.2]{Mac96}.

\begin{prop} \label{p: Theta not robust}
    Let $\Theta \leq \sym_k$. Suppose that $\Theta$ is not robust. Then:
    \begin{enumerate}[label=(\roman*)]
    	\item \label{i: no sh Theta-tr} there is no sharply $\Theta$-transitive action of a group on an infinite set;
    	\item \label{i: no sh k-homog} for any $\omega$-categorical relational \Fr structure $\mc{M}$ such that there is $\mc{A}_0 \in [\mc{M}]^k$ with $(A_0 \curvearrowleft \Aut(\mc{A}_0)) \simeq (\pi : \mathbf{k} \curvearrowleft \Theta)$, there is no sharply $k$-homogeneous action of a group on $\mc{M}$.
    \end{enumerate}
\end{prop}
\begin{proof}
    \ref{i: no sh Theta-tr}: Suppose for a contradiction that there exists a sharply $\Theta$-transitive action $\mu: U \curvearrowleft G$ of a group $G$ on an infinite set $U$. Let $A_0 \in [U]^k$. As $\mu$ is sharply $\Theta$-transitive, we have $(\mu : A_0 \curvearrowleft G_{\{A_0\}}) \simeq (\pi : \mathbf{k} \curvearrowleft \Theta)$, and as $\Theta$ is not robust we have some $\Omega \leq G_{\{A_0\}}$ and some $\Omega$-invariant non-empty $X \sub \NFr_\mu(\Omega) \cap A_0$ such that $|\Omega|$ divides $|X|$. Let $Y = (\NFr_\mu(\Omega) \cap A_0) \setminus X$. Then $Y$ is a proper subset of $\NFr_\mu(\Omega) \cap A_0$ and there is $m > 0$ with $k = |Y| + m|\Omega|$. If $\mu|_\Omega$ were to have infinitely many non-free orbits in $U$, then in each of the non-free orbits there is a point fixed by a non-identity element of $\Omega$, and so by the pigeonhole principle there is a non-identity element of $\Omega$ fixing each of infinitely many points of $U$, which contradicts sharp $\Theta$-transitivity. So there are finitely many non-free orbits of $\mu|_\Omega$ in $U$, and thus infinitely many free orbits, as $U$ is infinite.

    Let $E$ be the union of $Y$ with $m-1$ free $\Omega$-orbits, and let $\Sigma$ be the set of the remaining free $\Omega$-orbits which are not subsets of $E$. For $F \in \Sigma$, let $B_F = E \cup F$. Let $F_0 \in \Sigma$. For each $F \in \Sigma$ let $\alpha_F : B_{F_0} \to B_F$ be a map given by extending $\id_E$ by an $\Omega$-equivariant bijection $F_0 \to F$.

    For each $F \in \Sigma$, there is some $g_F \in G$ with $B_{F_0} = B_F \cdot g_F$, by the transitivity of $\mu$ on $[U]^k$. So for all $F \in \Sigma$ the map $\mu_{g_F} \circ \alpha_F$ is a permutation of $B_{F_0}$, and thus by the pigeonhole principle there exists an infinite subset $\Sigma_0 \sub \Sigma$ such that for all $F, F' \in \Sigma_0$ we have $\mu_{g_{F \vphantom{F'}}} \circ \alpha_F = \mu_{g_{F'}} \circ \alpha_{F'}$. Let $F_1 \in \Sigma_0$. For each $F \in \Sigma_0$, let $h_F = g_{F_1}^{-1} g^{\vphantom{-1}}_{F_{} \vphantom{F_1}}$, so $\mu_{h_F}|_{B_F} = \alpha_{F_1}^{\vphantom{-1}} \alpha_{F_{}}^{-1}$ and thus $\mu_{h_F}|_{B_F}$ is $\Omega$-equivariant and extends $\id_E$. Hence for each $\omega \in \Omega$ and $F \in \Sigma_0$ we have $[h_F, \omega] = 1$ by $k$-sharpness applied on $B_F$, so $h_F \in \centr_G(\Omega)$, and thus $h_F$ fixes $\NFr_\mu(\Omega)$ setwise. So there is an infinite $\Sigma_1 \sub \Sigma_0$ such that the $\mu$-action of $h_F$ on $\NFr_\mu(\Omega)$ is the same for each $F \in \Sigma_1$, and thus for $F, F' \in \Sigma_1$ we have $h_{F'}^{\vphantom{-1}} h_F^{-1}$ fixing $\NFr_\mu(\Omega) \cup E$ pointwise. As $k = |Y| + m|\Omega| = |\NFr_\mu(\Omega) \cap A_0| + r|\Omega|$, where $r$ is the number of free $\Omega$-orbits inside $A_0$, and as $Y \subsetneq \NFr_\mu(\Omega) \cap A_0$, we have $r < m$ and thus 
    $|\NFr_\mu(\Omega) \setminus Y| \geq |\Omega|$. So $|\NFr_\mu(\Omega) \cup E| \geq k$, and thus, by $k$-sharpness, for all $F, F' \in \Sigma_1$ we have $h_F = h_{F'}$, which implies $g_F = g_{F'}$ and thence $B_F = B_{F'}$, contradiction.

    \ref{i: no sh k-homog}: Suppose for a contradiction that there exists a sharply $k$-homogeneous action $\mu: \mc{M} \curvearrowleft G$. As $\mu$ is sharply $k$-homogeneous, we have $(\mu : A_0 \curvearrowleft G_{\{A_0\}}) \simeq (A_0 \curvearrowleft \Aut(\mc{A}_0))$, and the latter action is isomorphic to $\pi : \mathbf{k} \curvearrowleft \Theta$ by assumption. The remainder of the argument is a straightforward adaptation of the proof of \ref{i: no sh Theta-tr} and we leave this to the reader, noting the following changes: we replace each mention of sharp $\Theta$-transitivity with sharp $k$-homogeneity, we replace $\Omega$-equivariant bijections with $\Omega$-isomorphisms, and when defining $\Sigma$, we use the $\omega$-categoricity of $\mc{M}$ to take $\Sigma$ such that the substructures induced on $E \cup F$ for each $F \in \Sigma$ are $\Omega$-isomorphic (by $\omega$-categoricity we have that there are finitely many $\Omega$-isomorphism classes for substructures induced on $E \cup F$, and as there are infinitely many free orbits $F$ we have such $\Sigma$ by the pigeonhole principle).
\end{proof}

This completes the proof of Theorem \ref{t:Theta only if structural} and the ``only if" direction of Theorem \ref{t: sharply Theta-transitive}.

\begin{rem}
    In fact a stronger statement than \ref{i: no sh k-homog} is true. If $\Theta$ is not robust, then for any $\omega$-categorical relational structure $\mc{M}$ such that there is $A_0 \in [M]^k$ with $(A_0 \curvearrowleft \Aut(\mc{M})_{\{A_0\}}) \simeq (\pi : \mathbf{k} \curvearrowleft \Theta)$, there is no action $\mc{M} \curvearrowleft G$ with the following property: for all $\bar{a}, \bar{b} \in (M)^k$ with $\tp(\bar{a}) = \tp(\bar{b})$, there is a unique $g \in G$ with $\bar{a} \cdot g = \bar{b}$. The proof is essentially that of \ref{i: no sh k-homog}, replacing $\Omega$-isomorphisms with $\Omega$-elementary isomorphisms.
\end{rem}

\subsection{More on robust permutation groups}

\begin{defn} \label{d: docile and unruly}
    Let $\Theta \leq \sym_k$ be robust. We say that a subgroup $\Omega \leq \Theta$ is \emph{docile} if its permutation action on $k$ has at least one free orbit. Otherwise we say that $\Omega$ is \emph{unruly}.
\end{defn}

Note that if $\Omega$ is a docile subgroup of a robust permutation group $\Theta \leq \sym_k$, then so is any conjugate of $\Omega$ in $\Theta$ and any subgroup of $\Omega$. Also note that docile and unruly subgroups of a robust permutation group are themselves robust.

\subsubsection{Strong robustness}
We now define stronger notions than robustness, docility and unruliness, which we use in the second half of the paper concerned with structural actions (in particular, in Subsection \ref{subs:larger seeds}).

\begin{defn} \label{d: strongly docile permutation group} 
    We say that $\Omega \leq \sym_{k}$ is \emph{strongly docile} if it satisfies the following for the permutation action $\pi : \mathbf{k} \curvearrowleft \Omega$:
    \begin{enumerate}[label=(\roman*)]
        \item \label{c: docile free orb} $\Omega$ acts freely on each non-trivial orbit;
        \item \label{c: docile fixed pts} $|\Fx_\pi(\Omega)| < |\Omega|$.
    \end{enumerate}
	
    We say that $\Omega \leq \sym_k$ is \emph{strongly unruly} if, for each non-empty $\Omega$-invariant subset $X \sub \mathbf{k}$, we have that $|\Omega|$ does not divide $|X|$.
\end{defn}

\begin{defn} \label{d: strongly robust permutation group}
    We say that $\Theta\leq\sym_{k}$ is \emph{strongly robust} if it satisfies the following: 
    \begin{enumerate}[label=(\roman*)]
        \item \label{theta is transitive} the action $\mathbf{k} \curvearrowleft \Theta$ is transitive;
        \item \label{robust dichotomy} every $\Omega \leq \Theta$ is strongly docile or strongly unruly;
        \item \label{robust cyclic} every cyclic subgroup of $\Theta$ is strongly docile.
    \end{enumerate}
\end{defn}

We now justify the above terminology:

\begin{lem}
    Let $\Theta \leq \sym_k$ be strongly robust. Then:
    \begin{enumerate}[label=(\roman*)]
        \item \label{i: sr implies r} $\Theta$ is robust;
        \item \label{i: sd exactly d} the strongly docile subgroups of $\Theta$ are exactly the docile subgroups of $\Theta$;
        \item \label{i: su exactly u} the strongly unruly subgroups of $\Theta$ are exactly the unruly subgroups of $\Theta$.
    \end{enumerate} 
\end{lem}
\begin{proof}
    \ref{i: sr implies r}: Suppose for a contradiction there are $\Omega \leq \Theta$ and non-empty $\Omega$-invariant $X \sub \NFr_\pi(\Omega)$ with $|\Omega|$ dividing $|X|$. As $\NFr_\pi(\Omega) \neq \emp$ we have $\Omega \neq \mathtt{1}$. By Definition \ref{d: strongly robust permutation group}\ref{robust dichotomy}, we have that $\Omega$ is strongly docile or strongly unruly, and immediately we have by definition that $\Omega$ is not strongly unruly, so $\Omega$ is strongly docile. So by Definition \ref{d: strongly docile permutation group}\ref{c: docile free orb} and the fact that $\Omega \neq \mathtt{1}$ we have $\NFr_\pi(\Omega) = \Fx_\pi(\Omega)$. As $|\Omega|$ divides $|X|$ and $\emp \neq X \sub \NFr_\pi(\Omega) = \Fx_\pi(\Omega)$, we have $\Fx_\pi(\Omega) \geq |\Omega|$, contradicting Definition \ref{d: strongly docile permutation group}\ref{c: docile fixed pts}.
    
    \ref{i: sd exactly d}: Suppose we are given strongly docile $\Omega \leq \Theta$. The trivial group is docile, so assume $\Omega \neq \mathtt{1}$. So $\Omega$ does not fix $\mathbf{k}$ pointwise, and so by Definition \ref{d: strongly docile permutation group}\ref{c: docile free orb} has a free orbit. Now suppose we are given docile $\Omega \leq \Theta$. As $\Omega$ has some free orbit $X$, and $|X| = |\Omega|$, we have that $\Omega$ cannot be strongly unruly. So by Definition \ref{d: strongly robust permutation group}\ref{robust dichotomy} we have that $\Omega$ is strongly docile.  
    
    \ref{i: su exactly u}: this follows immediately from \ref{i: sd exactly d} and Definition \ref{d: strongly robust permutation group}\ref{robust dichotomy}.
\end{proof}

Most of the robust permutation groups encountered in this paper are strongly robust, with the exception of the dihedral groups $\dih_n$ for $n$ odd and composite (these are robust but not strongly robust). See Lemma \ref{l:examples}.

\begin{lem} \label{docile same order}
    Let $\Theta$ be strongly robust, and let $\Omega, \Omega' \leq \Theta$ be strongly docile subgroups of the same order. Then $|\Fx_\pi(\Omega)| = |\Fx_\pi(\Omega')|$ and $\Omega$, $\Omega'$ have the same number of free orbits in the permutation action $\pi$.
\end{lem}
\begin{proof}
    For $\Omega = \Omega' = \mathtt{1}$ the statement is trivial, so suppose $\Omega, \Omega'$ have order $l > 1$. By \Cref{d: strongly docile permutation group}\ref{c: docile free orb}, each non-trivial orbit of $\Omega$ and of $\Omega'$ is free, so \[k = |\Fx_\pi(\Omega)| + ml = |\Fx_\pi(\Omega')| + m'l\] for some $m, m' > 0$. By \Cref{d: strongly docile permutation group}\ref{c: docile fixed pts} we have $|\Fx_\pi(\Omega)|, |\Fx_\pi(\Omega')| < l$, so $|\Fx_\pi(\Omega)| = |\Fx_\pi(\Omega')|$ and $m = m'$.    
\end{proof}

\subsubsection{Examples of robust permutation groups}




\begin{defn}	  
    Henceforth, when given a robust subgroup $\Theta\leq\sym_{k}$, we also assume that we have chosen some set $\mc{D}(\Theta)$ of representatives of the conjugacy classes in $\Theta$ of the docile subgroups of $\Theta$ (where we mean the action by conjugation on the set of subgroups of $\Theta$).
\end{defn}

\begin{lem} \label{l:examples} \hfill 
    \begin{enumerate}[label=(\roman*)]
        \item \label{robust Sk} The symmetric group $\sym_k$ is robust iff $k \leq 3$, in which case it is strongly robust.
        \item \label{robust Ak} The alternating group $\alt_k \leq \sym_k$ is robust iff $k \leq 5$, in which case it is strongly robust.
        \item \label{robust Ck} The group $\cyc_k \leq \sym_k$ is strongly robust for all $k \geq 1$.
        \item \label{robust Dk} Let $k \geq 3$. The dihedral group $\dih_k \leq \sym_k$ is robust if and only if $k$ is odd, and strongly robust if and only if $k$ is an odd prime. (Here, we consider $\dih_k$ as a subgroup of $\sym_k$ by labelling the vertices of a regular polygon anticlockwise as $0, \cdots, k-1$.)
    \end{enumerate}
\end{lem}
\begin{proof}
    \ref{robust Sk}: It is straightforward to check that $\sym_2$, $\sym_3$ are strongly robust. Let $k \geq 4$. Let $\Omega = \sg{(1\,2)} \leq \sym_k$. Taking $X=\{3, 4\}$, we have $X \sub \Fx_\pi(\Omega) = \NFr_\pi(\Omega)$ and $|\Omega| = |X|$, so $\sym_k$ is not robust. 

    \ref{robust Ak}: The case $k \leq 3$ is immediate. We now sketch the details that $\alt_{4}$ and $\alt_{5}$ are strongly robust. We start by listing $\mathcal{D}(\alt_{4})\setminus\{\mathtt{1}\}$: 
    \begin{itemize}
        \item $\sg{(0\,1)(2\,3)}$, with no fixed points and two free orbits,
        \item $\sg{(1\,2\,3)}$, with one fixed point and one free orbit,
        \item $\sg{(0\,1)(2\,3),(0\,2)(1\,3)}$, with no fixed points and one free orbit.
    \end{itemize}
    The only non-trivial subgroup of $\alt_{4}$ not conjugate to one of the groups above is $\alt_4$ itself, which is strongly unruly.
    
    Similarly, as $\mathcal{D}(\alt_{5})\setminus\{\mathtt{1}\}$ we take:
    \begin{itemize}
        \item $\sg{(1\,2)(3\,4)}$, with one fixed point and two free orbits,
        \item $\sg{(1\,2\,3)}$, with two fixed points and one free orbit,
        \item $\sg{(1\,2)(3\,4),(1\,3)(2\,4)}$, with one fixed point and one free orbit,
        \item $\sg{(0\,1\,2\,3\,4)}$, with no fixed points and one free orbit.
    \end{itemize} 
    It is easy to check every non-trivial subgroup of $\alt_{5}$ which is not conjugate to one of the subgroups above must be conjugate to one of the following groups: $\sg{(0\,1\,2\,3\,4),(1\,2)(3\,4)}\cong D_{5}, \sg{(0\,1\,2),(1\,2)(3\,4)}\cong \sym_{3}, \alt_{5}$, all of which have order strictly larger than $5$.
	
    Now suppose $k \geq 6$. Let $\Omega = \sg{(0\,1)(2\,3)}$. Then $|\Fx_\pi(\Omega)| = k - 4 \geq 2$, so taking $X \sub \Fx_\pi(\Omega)$ with $|X| = 2$, we have that $|\Omega|$ divides $|X|$, so $\alt_k$ is not robust.

    \ref{robust Ck}: Clear.

    \ref{robust Dk}: If $k$ is even, then $\dih_{k}$ contains an involution $\sigma$ with exactly two fixed points $a, b$ (namely, a reflection whose axis is through two antipodal points), and taking $X = \{a, b\}$, we have that $|\sg{\sigma}|$ divides $|X|$, so $\dih_k$ is not robust. 
    
    Suppose $k$ is odd. Let $\rho$ be a rotation by $\frac{2\pi}{k}$ and let $\sigma$ be a reflection with $\dih_k = \sg{\rho, \sigma}$. The subgroups $\Omega \leq \dih_k$ and a description of their actions on $\mathbf{k}$ are as follows (see, for example, \cite[Theorem 3.1]{Con09}):
    \begin{enumerate}[label=(\roman*)]
        \item $\Omega = \sg{\rho^i}$ where $i$ divides $k$, with free action $\mathbf{k} \curvearrowleft \Omega$;
        \item $\Omega = \sg{\rho^j \sigma}$ where $j < k$: here $\rho^j\sigma$ is a reflection with a single fixed point, so the action $\mathbf{k} \curvearrowleft \Omega$ consists of free orbits of size $2$ and a fixed point;
        \item \label{i: dih third type} $\Omega = \sg{\rho^i, \rho^j\sigma}$ where $i$ divides $k$, $i < k$ and $j < i$: letting $m = k / i$, we have $\Omega \cong \dih_m$, and the action $\pi : \mathbf{k} \curvearrowleft \Omega$ has exactly one non-free orbit $X$, with $X \curvearrowleft \Omega$ isomorphic to the standard permutation action $\mathbf{m} \curvearrowleft \dih_{m}$. 
    \end{enumerate} 
    
    It is straightforward to see that in the first two cases $\Omega$ is robust and strongly docile. In the third case, as $|X| < |\Omega|$ we have that $\Omega$ is robust. So $\dih_k$ is robust. If $k$ is prime, then the only group of the third type is $\dih_k$ itself, and $\dih_k$ is strongly unruly. So if $k$ is prime then $\dih_k$ is strongly robust.

    Now suppose $k$ is odd and composite. Let $i$ divide $k$ with $i \neq 1, k$. Take $\Omega = \sg{\rho^i, \sigma}$ as in case \ref{i: dih third type} above, with $m = k / i$ and $X$ the non-free orbit of $\mathbf{k} \curvearrowleft \Omega$ of size $m$. As $\NFr_\pi(\Omega) \neq \Fx_\pi(\Omega)$, we have that $\Omega$ is not strongly docile. As $|\Omega|$ divides $k - |X|$ and $\mathbf{k} \setminus X$ is $\Omega$-invariant, we have that $\Omega$ is not strongly unruly. So $\dih_k$ is not strongly robust.
\end{proof}

\subsubsection{Properties of robust permutation groups}

The following lemma, which we will use often, shows how the specifics of Definition \ref{d:robust permutation group} come into play.

\begin{lem} \label{l:invariant tuples}
    Let $\Theta \leq \sym_k$ be robust. Let $\lambda : U \curvearrowleft \Theta$ be an action consisting of infinitely many free orbits together with a $\Theta$-invariant set $A_0$ such that $(\lambda|_{A_0} : A_0 \curvearrowleft \Theta) \simeq (\pi : \mathbf{k} \curvearrowleft \Theta)$. Then:
    \begin{enumerate}[label=(\roman*)]
        \item \label{k-sets iso to std} for any $\Theta$-invariant set $A\in [U]^k$ we have $(\lambda|_A : A \curvearrowleft \Theta) \simeq (\pi : \mathbf{k} \curvearrowleft \Theta)$, and thus in particular $\NFr_\lambda(\Theta) \sub A$;
        \item \label{set stab is docile} for any $A\in[U]^{k}$ with $A \neq A_0$, letting $\Omega=\sSt_{\lambda}(A)$, we have that $\Omega$ is docile and $\NFr_{\lambda}(\Omega) \sub A$.
    \end{enumerate}	
\end{lem}
\begin{proof}
    \ref{k-sets iso to std}: Let $f : A_0 \to \mathbf{k}$ be a $\Theta$-isomorphism for the actions $\lambda|_{A_0} : A_0 \curvearrowleft \Theta$ and $\pi : \mathbf{k} \curvearrowleft \Theta$. Let $X = \NFr_\lambda(\Theta) \cap A$. As $A$ is the union of $X$ together with free orbits, letting $m$ be the number of free orbits contained in $A$, we have $k = |X| + m|\Theta|$. The same reasoning applied to $\pi$ gives $k = |\NFr_\pi(\Theta)| + m'|\Theta|$, where $m'$ is the number of free orbits of $\pi$. As $\NFr_\lambda(\Theta) \sub A_0$ by assumption, we have $X \sub A_0$. So $f(X) \sub \NFr_\pi(\Theta)$. As $|f(X)| + m|\Theta| = |\NFr_\pi(\Theta)| + m'|\Theta|$, we have that $|\Theta|$ divides $|\NFr_\pi(\Theta)| - |f(X)|$. But $\NFr_\pi(\Theta) \setminus f(X)$ is $\Theta$-invariant, so as $\Theta$ is robust we have $f(X) = \NFr_\pi(\Theta)$. Thus $X = \NFr_\lambda(\Theta)$ and $m = m'$, giving the claim.
    
    \ref{set stab is docile}: Suppose for a contradiction that $\Omega$ is not docile. Then $\mathbf{k} = \NFr_\pi(\Omega)$, and as $\NFr_\lambda(\Omega) \sub \NFr_\lambda(\Theta)$ we have $\NFr_\lambda(\Omega) = A_0$. As $\Omega$ is robust, 
    applying \ref{k-sets iso to std} to $\lambda|_\Omega$ we have $\NFr_\lambda(\Omega) \sub A$, and so $A = A_0$, contradiction.
\end{proof}

\subsection{Seed groups and seed actions} \label{ss:seeds}

Let $\Theta \leq \sym_k$ be robust. We now begin to describe how to construct a group $G$ from $\Theta$ which admits a sharply $\Theta$-transitive action on a set or a sharply $k$-homogeneous action on a \Fr structure. This proceeds in two stages. In this section, first we define a \emph{seed group} $H$ from $\Theta$. Then in \cref{s:completions} we construct the group $G$ from $H$, which we call a \emph{completion} of the seed group $H$.

In fact, in the simplest situation of sharply $\Theta$-transitive actions on an infinite set $U$, we may take $H = \Theta$ and construct $G$ directly -- see \cref{d:free completion} below. In most cases of structural actions, we can also take $H = \Theta$. However, in the more involved case of structural actions with $\Theta = \sym_3$ in Subsection \ref{subs:larger seeds}, we work with a more complicated seed group $H$, containing multiple copies of $\Theta$: see \cref{d:standard s3 seed} and \cref{l:constructed H}. This is the reason for introducing seed groups rather than just working with $\Theta$: they enable us to give a unified presentation of our results. For a first approach to the material in this paper, we recommend the reader to focus on the case $H = \Theta$.

\begin{defn} \label{d:seed group} 
    Let $\Theta\leq\sym_{k}$ be robust. Let $\{c_0, \cdots, c_{r-1}\}$ be a set of representatives of the orbits of $\pi : \mathbf{k} \curvearrowleft \Theta$. For $i < r$, we let $\Gamma_i = \St_\pi(c_i)$.
	
    Let $m \geq 1$. For convenience we assume that one of the following two cases holds:
    \begin{enumerate}[label=(\Roman*)]
        \item \label{robust case} $\Theta$ is robust and $m=1$;
        \item \label{very robust case} $\Theta$ is strongly robust and $r=1$.
    \end{enumerate}

    If $\Theta$ is itself docile we additionally assume that $m=1$.
    
    Let $(H,(\iota_{j})_{j < m})$ be a tuple consisting of a countable group $H$ and embeddings $\iota_{j} : \Theta \to H$, $j < m$. For $\Omega \leq \Theta$ and $\theta \in \Theta$ we write $\Omega_j = \iota_j(\Omega)$ and $\theta_j = \iota_j(\theta)$. The embedding $\iota_0$ will be distinguished as a ``reference copy" of $\Theta$, and we will simply identify $\Theta$ with $\Theta_0$.
	
    We call $(H,(\iota_{j})_{j < m})$ a \emph{$(\Theta,m)$-seed group} if it satisfies the following properties:
    \begin{enumerate}[label=(\roman*)]
        \item \label{seed torsion} every finite subgroup of $H$ has a conjugate contained in some $\Theta_j$;
        \item \label{seed nonconjugate} the groups $\Theta_j$ are pairwise non-conjugate;
        \item \label{seed conjugate docile} for all $j < m$, $\Omega\in\mc{D}(\Theta)$, there is $h \in H$ such that $\iota_j|_\Omega$ is the map given by conjugation by $h$,
        \item \label{seed conjugate stab} for all $i < r$, $j < m$, there is $h \in H$ such that $\iota_j|_{\Gamma_i}$ is the map given by conjugation by $h$;  
        \item \label{seed self-normalised} if $h \in H$ satisfies $\Gamma_i^{h} \cap \Theta \neq \{1\}$ for some $i < r$, then $h \in \Theta$;
    \end{enumerate}
    We usually omit mentioning the embeddings $\iota_j$ and just say that $H$ is a $(\Theta, m)$-seed group, with the embeddings clear from context. For $\Omega_j \leq \Theta_j$, we will say that $\Omega_j$ is docile/unruly if $\iota_j^{-1}(\Omega_j) \leq \Theta$ is docile/unruly.
\end{defn}

Notice that in particular $\Theta$ itself is trivially a $(\Theta,1)$-seed group. The above Definition \ref{d:seed group} is designed to unify the following two cases:
\begin{itemize}
    \item $H = \Theta$ for robust $\Theta$: here $m = 1$, but we do not assume that $\pi : \mathbf{k} \curvearrowleft \Theta$ is transitive;
    \item $H$ contains multiple copies of $\Theta$ (see Definition \ref{d:standard s3 seed} for a concrete example), but we make stronger assumptions on $\Theta$: we assume that $\Theta$ is strongly robust and $\pi : \mathbf{k} \curvearrowleft \Theta$ is transitive.
\end{itemize}

\begin{defn} \label{d:seed action on set} 
    Let $H$ be a $(\Theta, m)$-seed group, and let $\lambda: U \curvearrowleft H$ be an action. We call $\lambda$ a \emph{seed action} if:
    \begin{enumerate}[label=(\roman*)]
        \item \label{seedact infinitely manyorbits}$\lambda$ has infinitely many orbits;
        \item \label{seedact orbit description} there is an $H$-invariant set $V$ such that the $\lambda$-action on $V$ is isomorphic to the right-multiplication action $\bigsqcup_{i < r} \Gamma_i\backslash H \curvearrowleft H$, and all remaining orbits are free.
    \end{enumerate}  
    
    We call $V$ the \emph{paradigm set} of $\lambda$.
    
    (As before $\Gamma_i$ is the stabiliser of $c_i$ in $\pi : \mathbf{k} \curvearrowleft \Theta$. The action $\bigsqcup_{i < r} \Gamma_i\backslash H \curvearrowleft H$ is the right-multiplication action on the disjoint union of the right coset spaces.)
\end{defn}

\begin{observation} \label{o: Theta seed action}
    If $H=\Theta$, then a seed action $\lambda : U \curvearrowleft \Theta$ is just an action with infinitely many orbits such that there is a $\Theta$-invariant set $V$ with $(\lambda|_V : V \curvearrowleft \Theta) \simeq (\pi : \mathbf{k} \curvearrowleft \Theta)$ and the action outside $V$ is free.    
\end{observation}

\begin{lem} \label{l: key seed action props}
    Let $\lambda : U \curvearrowleft H$ be a seed action. Then:
    \begin{enumerate}[label=(\alph*)]
        \item \label{seedact Gamma conj} each $h \in H^\ast$ fixing a point in $U$ is conjugate to some $\gamma \in \bigcup_{i < r} \Gamma_i$, and $|\Fx_\lambda(h)| = |\Fx_\pi(\gamma)|$;
        \item \label{seedact distinguished set} for $j\in\mathbf{m}$ there is some $\Theta_j$-invariant $A_j \in [U]^k$ such that $(A_j \curvearrowleft \Theta_j) \simeq (\mathbf{k} \curvearrowleft \Theta)$ via $\iota_{j}$ and the action $U \setminus A_j \curvearrowleft \Theta_j$ is free;
        \item \label{seedact docile} for $\Omega \in \mc{D}(\Theta)$ and $\Omega$-invariant $A \in [U]^k$, we have $(A \curvearrowleft \Omega) \simeq (\mathbf{k} \curvearrowleft \Omega)$ and $\NFr_\lambda(\Omega) \sub A$;
        \item \label{seedact k-sharp} $\lambda$ is $k$-sharp;
	\item \label{seedact different orbits} no two of the sets $A_j$ ($j\in\mathbf{m}$) are translates of each other under the action of $\lambda$;
        \item \label{seedact unruly unique} for each unruly $\Omega_j \leq \Theta_j$, we have that $A_j$ is the only $\Omega_j$-invariant $k$-set in $U$.
    \end{enumerate}

    We call any $A_0, \cdots, A_{m-1}$ as above a family of \emph{reference sets} for the seed action. (Note that if $\Theta$ is docile, this family will not be unique.)
\end{lem}
\begin{proof}
    For notational convenience, we identify $V_i \curvearrowleft H$ and $\Gamma_i \backslash H \curvearrowleft H$ for each $i < r$.

    \ref{seedact Gamma conj}:  as $\lambda$ acts freely on $U \setminus V$, we have $\Fx_\lambda(h) \sub V$. So $h$ is conjugate to some $\gamma \in \Gamma_i$, $i < r$, which has the same number of fixed points in $U$. Let $\Gamma_{i'} h' \in V$ be a fixed point of $\gamma$. Then $\gamma \in \Gamma_{i'}^{h'}$, so by \Cref{d:seed group}\ref{seed self-normalised} we have $h' \in \Theta$. So the fixed points of $\gamma$ lie in $\bigsqcup_{l < r} \Gamma_l \backslash \Theta$, and thus as $(\bigsqcup_{l < r} \Gamma_l \backslash \Theta \curvearrowleft \Theta) \simeq (\pi : \mathbf{k} \curvearrowleft \Theta)$ we have $|\Fx_\lambda(h)| = |\Fx_\pi(\gamma)|$.

    \ref{seedact distinguished set}: First consider case \ref{robust case} of Definition \ref{d:seed group}, where $m = 1$. Let $A_0 = \bigsqcup_{i < r} \Gamma_i \backslash \Theta$. Then $A_0$ is $\Theta$-invariant, and as $(A_0 \curvearrowleft \Theta) \simeq (\pi : \mathbf{k} \curvearrowleft \Theta)$, it suffices to show that $U \setminus A_0 \curvearrowleft \Theta$ is free. Let $\theta \in \Theta^\ast$, and suppose $\theta$ fixes a point of $U$. By \ref{seedact Gamma conj}, there are $i < r$ and $\gamma \in \Gamma_i$ such that $\gamma^h = \theta$ for some $h \in H$ and $|\Fx_\lambda(\theta)| = \Fx_\pi(\gamma)$. By Definition \ref{d:seed group}\ref{seed self-normalised}, we have $h \in \Theta$, so $|\Fx_\pi(\theta)| = |\Fx_\pi(\gamma)|$. As $|\Fx_\pi(\theta)| = |\Fx_\lambda^{A_0}(\theta)|$, we have $\Fx_\lambda(\theta) = \Fx_\lambda^{A_0}(\theta)$. Thus the $\lambda$-action of $\Theta$ on $U \setminus A_0$ is free as required.
    
    Now consider case \ref{very robust case} of \cref{d:seed group}. Let $\Gamma=\Gamma_{0}$. By \Cref{d:seed group}\ref{seed conjugate stab} there is $h \in H$ with $\iota_j|_\Gamma$ equal to conjugation by $h$. The map $\Gamma^h \backslash \Theta_j \to \Gamma \backslash h\Theta_j$, $\Gamma^h \theta_j \mapsto \Gamma h\theta_j$, is a $\Theta_j$-isomorphism, and $\iota_j$ is an isomorphism $(\Gamma \backslash \Theta \curvearrowleft \Theta) \to (\Gamma^h \backslash \Theta_j \curvearrowleft \Theta_j)$, where the equivariance is also via $\iota_j$. So $(\mathbf{k} \curvearrowleft \Theta) \simeq (\Gamma \backslash h\Theta_j \curvearrowleft \Theta_j)$ via $\iota_j$, and thus we may take $A_j = \Gamma \backslash h\Theta_j \sub V$.
    
    We now show that the action $U \setminus A_j \curvearrowleft \Theta_j$ is free. Let $\theta_j \in \Theta_j^\ast$, and assume $\theta_j$ fixes a point of $U$. By \ref{seedact Gamma conj}, there is $\gamma \in \Gamma$ conjugate to $\theta_j$ by some element of $H$, and $|\Fx_\lambda(\theta_j)| = |\Fx_\pi(\gamma)|$. As $(A_j \curvearrowleft \Theta_j) \simeq (\mathbf{k} \curvearrowleft \Theta)$ via $\iota_j$, we have $|\Fx_\lambda^{A_j}(\theta_j)| = |\Fx_\pi(\iota_j^{-1}(\theta_j))|$. As $\langle \gamma \rangle$, $\langle \iota_j^{-1}(\theta_j) \rangle$ are strongly docile (by \cref{d: strongly robust permutation group}\ref{robust cyclic}) and of the same order, Lemma \ref{docile same order} implies $|\Fx_\pi(\gamma)| = |\Fx_\pi(\iota_j^{-1}(\theta_j))|$. So $|\Fx_\lambda^{\vphantom{A_j}}(\theta_j)| = |\Fx_\lambda^{A_j}(\theta_j)|$ and thus $\Fx_\lambda^{\vphantom{A_j}}(\theta_j) \sub A_j$.

    \ref{seedact docile}: this follows immediately from \ref{seedact distinguished set} and \cref{l:invariant tuples}\ref{k-sets iso to std}.
    
    \ref{seedact k-sharp}: let $h \in H$ fix an element of $(U)^k$. Then by \ref{seedact Gamma conj} there is $\gamma \in \Theta$ conjugate to $h$ with $|\Fx_\lambda(h)| = |\Fx_\pi(\gamma)|$, so $|\Fx_\pi(\gamma)| = k$. Thus $\gamma = 1$, so $h = 1$.

    \ref{seedact different orbits}: For each $j < m$, as $\lambda$ is $k$-sharp we have that $\sSt_\lambda(A_j)$ is finite; by \Cref{d:seed group}\ref{seed torsion} we have $|\sSt_\lambda(A_j)| \leq |\Theta|$, so as $\Theta_j \sub \sSt_\lambda(A_j)$, we have $\Theta_j = \sSt_\lambda(A_j)$. Suppose that there exists $h \in H$ and $j, j' < m$ with $A_j \cdot h = A_{j'}$. Then $\Theta_{j'}^{\vphantom{h}} = \Theta_j^h$, so $j = j'$ by \Cref{d:seed group}\ref{seed nonconjugate}.

    \ref{seedact unruly unique}: let $A \in [U]^k$ be $\Omega_j$-invariant. Then $A$ is the union of the $\Omega_j$-invariant sets $A \cap A_j$ and $A \setminus A_j$, and by \ref{seedact distinguished set} the set $A \setminus A_j$ is a union of free $\Omega_j$-orbits. So $|A| = k = |A \cap A_j| + m|\Omega_j|$ for some $m \geq 0$, and thus $|\Omega|$ divides $k - |A \cap A_j| = |A_j \setminus A|$. By \ref{seedact distinguished set} we have $(A_j \curvearrowleft \Omega_j) \simeq (\mathbf{k} \curvearrowleft \Omega)$ via $\iota_j$, witnessed by some bijection $f : A_j \to \mathbf{k}$. The set $f(A_j \setminus A)$ is $\Omega$-invariant, and as $\Omega$ is unruly we have $f(A_j \setminus A) \sub \NFr_\pi(\Omega)$. As $\Theta$ is robust, and as $|\Omega|$ divides $|f(A_j \setminus A)|$, we have $f(A_j \setminus A) = \emp$ and hence $A = A_j$. 
\end{proof}

\subsection{The standard \texorpdfstring{$(\sym_{3},m)$}{(S3, m)}-seed group}

In fact, the only example of $H \neq \Theta$ featuring in Theorem \ref{t:main} will involve $\Theta=\sym_{3}$; we now construct a $(\sym_{3},m)$-seed group for every $m\geq 1$.

\begin{definition} \label{d:standard s3 seed} 
	By the \emph{standard $(\sym_{3},m)$-seed group} we mean the group $H$ given by the following presentation:
	\begin{equation}
	\label{eq:presentation} H=\sg{\sigma,\tau,z_{1},\dots, z_{m-1}\,|\,\sigma^{2}=1, \tau^{3}=1, \tau^{-1}=\tau^{\sigma} = \tau^{\sigma^{z_{1}}} = \cdots = \tau^{\sigma^{z_{m-1}}}},
	\end{equation} 
	as well as the embeddings $\iota_j : \sym_3 \to H$, $j < m$, given by:
    \begin{alignat*}{2}
        &\iota_0 : (1\;\, 2) \mapsto \sigma,\; &&(0\;\, 1\;\, 2) \mapsto \tau;\\
        &\iota_j : (1\;\, 2) \mapsto \sigma^{z_j},\; &&(0\;\, 1\;\, 2) \mapsto \tau \text{ (for }j \geq 1 \text{)}.    
    \end{alignat*}
\end{definition}

\begin{lemma} \label{l:constructed H} 
	The group $H$ constructed above is a $(\sym_{3},m)$-seed group. 
\end{lemma}
\begin{proof}
    Consider the group $H_{0}$ given as the fundamental group of a graph of groups whose underlying graph is a star with $m$-edges. The vertex group  $\sg{\tau}$ sits at its center, is also equal to the group on each of the edges, while the other vertex groups are of the form $\sg{\sigma_{i},\tau}$ for $i\in\mathbf{m}$. It is easy to see that $H$ can be described as an iterated HNN extension of $H_{0}$, that is $H=H_{m-1}$, for a chain $H_{0}\leq H_{1}\leq\dots \leq H_{m-1}$ where $H_{j}=H_{j-1}\Asterisk_{\sigma^{z_{j}}=\sigma_{j}}$ is an HNN extension of $H_{j-1}$.   
	
	To begin with, notice that any finite subgroup of a graph of groups must be conjugate into one of the vertex groups (see \cite[Theorem 2.4 \& 2.9]{LS1979}). An iterated application of these two results implies that property \ref{seed torsion} holds. 
	So $\sg{\sigma_{j},\tau}$ and $\sg{\sigma_{j'},\tau}$ are conjugate in $H$ if and only if they are conjugate in $H_{0}$. However, in a graph of groups no two vertex groups are conjugate unless they are connected through an edge path along which all edge embeddings are isomorphisms. This establishes \ref{seed nonconjugate}.
	
    We now check \ref{seed self-normalised}. Notice that $\Gamma=\sg{\sigma}$. Take $h$ and $\theta\in \Theta^{*}$ such that $\theta=\sigma^{h}$. Our goal is to show that $h\in \Theta$.
    
	Consider one of the HNN extensions above, $H_{j}=H_{j-1}\Asterisk_{\sigma^{z_{j}}=\sigma_{j}}$. In the corresponding Bass-Serre tree the subgroup $\sg{\sigma}$ stabilises a subtree of diameter two whose vertices consist of the vertex $u$ stabilised by $H_{j-1}$, the vertex $u'$ stabilised by $H_{j-1}^{z_{j}^{-1}}$ and all the translates of the latter by elements of $Z_{H_{j-1}}(\sigma)$. 
   
    \begin{claim*}
        We have $Z_{H}(\sigma)=\sg{\sigma}$.
    \end{claim*}
    \begin{subproof}
    Recall that no element of the group can reverse the orientation of any of the edges of the tree. It easily follows that $Z_{H_{j}}(\sigma)=Z_{H_{j-1}}(\sigma)$. By iterating this argument we conclude that $Z_{H}(\sigma)=Z_{H_{0}}(\sigma)$. On the other hand, in the Bass-Serre tree corresponding to the graph of groups decomposition of $H_{0}$ the element $\sigma$ fixes a unique vertex, which must be fixed by each of the elements in its centraliser, so that $Z_{H_{0}}(\sigma)=Z_{\sg{\tau,\sigma}}(\sigma)=\sg{\sigma}$.
    \end{subproof}
    
    It follows that the only vertices stabilised by $\sigma$ are in fact $u$,$u'$ given above. Since the action on the tree does not invert edges, it follows that $h$ must either fix $u$, and thus belong to $H_{j-1}$, or map $u$ to $u'$ and thus be of the form $h_{0}z_{j}$ for some $h_{0}\in H_{j-1}$.
    We then have $\theta=\sigma^{h_{0}z_{j}}\in H_{j-1}^{z_{j}}\cap \Theta=\subg{\sigma_{j}}\cap\Theta=\{1\}$, a contradiction. Thus $h\in H_{j-1}$. Iterating this argument over all HNN extensions in the constructions we easily conclude that if $\sigma^{h}\in \Theta$, then $h\in H_{0}$. 
    
    In the Bass-Serre tree corresponding to the amalgamated product by which $H_{0}$ is constructed the element $\sigma$ fixes exactly one vertex with $\Theta$ as its stabiliser. Therefore, any element $h\in H_{0}$ such that $\sigma^{h}\in \Theta$ must itself fix this vertex and therefore belong to $\Theta$. This concludes the proof of  \ref{seed self-normalised}.

	Finally, $\Gamma=\sg{\sigma}$, and so conditions \ref{seed conjugate docile}, \ref{seed conjugate stab} are immediate (all involutions in $\sym_{3}$ are conjugate). 
\end{proof}

\section{Completions of seed groups and their partial actions} \label{s:completions}

\begin{definition} \label{d:free completion}
    Let $H$ be a $(\Theta, m)$-seed group. For each $\Omega \in \mc{D}(\Theta)$, let $T_\Omega$ be a finite set with at least two elements. Let $F$ be the free group generated by the disjoint union of the sets $T_\Omega$, and for each $\Omega \in \D$ let $F_\Omega$ denote the free subgroup of $F$ generated by $T_\Omega$.

    The \emph{completion} of $H$ is defined as the tuple $(G,(T_{\Omega})_{\Omega\in\D})$, where $G$ is the group given by the quotient of the free product $H \ast F$ by the normal closure of the set of relations $\bigcup_{\Omega \in \D}\{[\omega,t]=1 \mid \omega\in \Omega, t\in T_{\Omega}^{\pm 1}\}$.
    
    Equivalently, an alternate definition of $G$ is as follows. Take an enumeration $\Omega_{1}, \cdots, \Omega_{d}$ of $\D$ and define a sequence $G_{0}=H < G_{1} < \cdots < G_{d}$ by $G_{i+1}=G_{i}\ast_{\Omega_{i}}(\Omega_{i}\times F_{\Omega_{i}})$ for $i < d$. Then take $G = G_d$.
   
    We write $T_{\mathcal{D}}=\bigcup_{\Omega \in \mc{D}(\Theta)}T_{\Omega}$ and $T=H^\ast \cup T_{\mathcal{D}}$. Note that $T$ is a set of generators of $G$.

    (We often just write $(G, T)$ and take the additional information of $(T_\Omega)_{\Omega \in \mc{D}(\Theta)}$ as implicit, when this is clear from context.)
 \end{definition}

\begin{observation}\label{alternative description}
  Another way to describe the completion $G$ is as an iterated $\mathrm{HNN}$-extension of $H$, where $H$ is the unique vertex group and each of the generators in $t\in T_{\Omega}$, $\Omega\in\D$ is the stable letter associated to a unique edge from the unique vertex to itself. The edge group is just $\Omega$, embedded via the inclusion on both sides, so that $t$ commutes with $\Omega$.
\end{observation}

\begin{eg} \label{ex:s2 s3} \hfill
 	\begin{itemize}
 		\item Let $H=\Theta=\sym_{2}=\sg{\sigma}$. Then $\D=\{\mathtt{1},\sym_{2}\}$, and taking $T_{\sym_{2}}=\{s,s'\}$, $T_{\mathtt{1}}=\{u, u'\}$ we get:
 		\[
 		   G = (\sym_{2}\times\sg{s,s'})\ast\sg{u, u'}.
 		\]
 		\item Let $H=\Theta=\sym_{3}$. Then $\D=\{\mathtt{1},\sg{\sigma},\sg{\tau}\}$ for some involution $\sigma$ and some element $\tau$ of order $3$. Taking $T_{\sg{\sigma}}=\{s,s'\}$, $T_{\sg{\tau}}=\{t,t'\}$, $T_{\mathtt{1}}=\{u, u'\}$ we get
 		\begin{equation*}
 			 G\cong((\sym_{3}\ast_{\sg{\sigma}}(\sg{\sigma}\times\sg{s,s'}))\ast_{\sg{\tau}}(\sg{\tau}\times\sg{t,t'}))\ast\sg{u, u'}.
 		\end{equation*}
 		\end{itemize}
\end{eg}

\begin{lem} \label{torsion completion}
    If $G$ is a completion of a seed group $H$, then any finite subgroup of $G$ is the conjugate of some finite subgroup of $H$. 
\end{lem}
\begin{proof} 
    This follows from the following two basic facts from geometric group theory. One is that finite groups have property $\mathrm{(FA)}$, i.e. any action of a finite group on a tree has a global fixed point (see \cite{serre2002trees}, example 6.3.1). The other is that stabilisers of points in the tree dual to a graph of groups are conjugate of vertex groups (see, for instance, the construction in 5.3 of \cite{serre2002trees}).
\end{proof}

 \begin{observation}
    In the case $H=\Theta$  or in the case $\Theta=\sym_{3}$ and $H$ the standard $(\sym_{3},m)$ seed group, the group $G$ is finitely generated and virtually free (that is, has a finite index free subgroup).
    This follows from \cref{alternative description}, together with the characterisation of finitely generated virtually free groups as fundamental groups of finite graph of groups with finite vertex groups, established in \cite{karrass1973finite}.
    
    For the second of the two cases above, note that the standard $(\sym_{3},m)$-seed group is itself the fundamental group of a finite graph of groups with finite vertex groups. This graph of groups can be easily combined with the one in \cref{alternative description} to show that $G$ is also the fundamental group of such a graph of groups and thus virtually free. 
    
    Likewise, if $H$ is hyperbolic, then $G$ is hyperbolic by the main theorem of \cite{BF92}. 
 \end{observation}
  
\subsection{Group words and small cancellation}

We first establish some basic terminology and notation, most of which is standard.

\begin{defn}
    Let $S$ be a set. We write $S^{\pm 1}$ for the \emph{alphabet} consisting of the disjoint union of $S$ with a set $S^{-1} = \{s^{-1} \mid s \in S\}$ of formal inverse symbols, which we call \emph{inverse letters}. We define the obvious bijection $(-)^{-1} : S \cup S^{-1} \to S^{-1} \cup S$. A \emph{word} $w$ in the alphabet $S^{\pm 1}$ is a finite sequence $(s_1, \cdots, s_m)$ (possibly empty) of elements of $S^{\pm 1}$, and we write $w = s_1 \cdots s_m$. We write $\mc{W}(S)$ for the set of words in $S^{\pm 1}$ (so, in our notation, words in $\mc{W}(S)$ may contain inverse letters). We say that a word $w = s_1 \cdots s_m \in \mc{W}(S)$ is \emph{positive} if $s_i \in S$ for $1 \leq i \leq m$ (that is, there are no inverse letters). 
    
    Given $w = s_1 \cdots s_m \in \mc{W}(S)$, we write $w^{-1}$ for the word $s^{-1}_m \cdots s^{-1}_1$. A \emph{subword} of $w$ is a word of the form $w[i:j]=s_{i} \cdots s_{j}$, for two indices $1 \leq i, j \leq m$, following the convention that $w[i:j]$ is the empty word if $i > j$. We also write $w[i] = s_i$. An \emph{initial subword} of $w$ is a word of the form $w[1:j]$, and we write $\mc{I}(w)$ for the set of initial subwords of $w$. We use the notation $\cdot$ to denote the concatenation of words or symbols if there is a need to be explicit and precise in a context which might otherwise be unclear. If the meaning of concatenation is clear from context we will often just use empty space (so $ww'$ means $w \cdot w'$).
    
    Let $G$ be a group, and let $S$ be a generating set of $G$. Each word $w \in \mc{W}(S)$ evaluates to an element of $G$ by multiplying the letters of $w$ from left to right (replacing inverse symbols with inverse group elements). For $g \in G$, let $\mc{W}(S,g)$ denote the set of words in $\mc{W}(S)$ evaluating to $g$. If a word evaluates to $g$, we will also say that it \emph{represents} $g$. We say that two words $w, w'$ in $\mc{W}(S)$ are \emph{equivalent}, written $w \sim w'$, if they represent the same element of $G$. We say that $w \in \mc{W}(S)$ is \emph{reduced} if it has minimum length amongst all words equivalent to it, and that $w$ is \emph{cyclically reduced} if each cyclic permutation of it is reduced. We write $\Wr(S)$ and $\Wr^c(S)$ for the sets of all reduced words and all cyclically reduced words in $\mc{W}(S)$.
\end{defn}

The following is a variant of the notion of small cancellation, adapted to our needs.

\begin{definition} \label{d: small cancellation}
    Let $S$ be a set. Let $L \in \N_+$. We say that $w \in \Wr(S)$ \emph{has cancellation $< L$} if, for any distinct pairs of indices $i \leq j$, $i' \leq j'$ such that $w[i:j] = w^{\varepsilon}[i':j']$ for some $\varepsilon=\pm 1$, we have $|w[i:j]| < L$. (Informally: any subword of $w$ that occurs twice has length $< L$ -- where an occurrence may also be of its inverse.)
\end{definition}
   
The following is a consequence of the existence of small cancellation words in the standard sense in geometric group theory, but we provide a proof for our specific context. 

\begin{defn} \label{d: syllable sequence}
    Let $S$ be a set, and let $w \in \Wr(S)$ be positive. Write $w = s_1^{e_1} \cdots s_{r \vphantom{1}}^{e_r}$, where $s_i \in S$ and $e_i \geq 1$ for $1 \leq i \leq r$, and $s_i \neq s_{i+1}$ for $1 \leq i \leq r-1$ (note that we may have $s_i = s_j$ for $|i - j| > 1$). We call the sequence $(s_1^{e_1}, \cdots, s_r^{e_r})$ the \emph{syllable sequence} of $w$, denoting it by $\syl(w)$, and we call the sequence $(s_1, \cdots, s_r)$ the \emph{syllable compression} of $w$, and denote it by $\sylc(w)$. We write $|\syl(w)|$, $|\sylc(w)|$ for the length of each sequence respectively (so we always have that $|\syl(w)| = |\sylc(w)|$).
\end{defn}

\begin{lemma} \label{l:existence small cancellation}
    Let $S$ be a finite set with $|S| \geq 2$. Let $s, \tld{s}, t \in S$ with $\tld{s} \neq t$. Let $u \in \mc{W}(S)$ be positive. Then for any $N \in \N_+$ there exists $w \in \Wr(S)$ beginning with $s$ and ending with $s^{-1}$ satisfying the following:
    \begin{enumerate}[label=(\roman*)]
        \item \label{length at least N} $|w| \geq N$;
        \item \label{contains u} any subword of $w$ of length $\geq \lfloor\frac{|w|}{N}\rfloor$ contains $u$ or $u^{-1}$ as a subword;
        \item \label{cancellation w/N} $w$ has cancellation $< \lfloor \frac{|w|}{N} \rfloor$;
        \item \label{special form} there are positive words $v, v' \in \mc{W}(S)$ beginning with $s$ and ending with $t$ such that $\sylc(v) = \sylc(v')$ and $w = v \cdot \tld{s} \cdot (v')^{-1}$.
    \end{enumerate}
\end{lemma}
\begin{proof}
    It suffices to deal with the case where $u$ starts and ends with $\tld{s}$, as otherwise we may replace $u$ by $\tld{s}u\tld{s}$ (note that $u$ is a positive word). 

    Let $M = \max\{N, |u|\}$, and let
    \[
    	v=s u t^{2M+1} u t^{2M+2} \cdots u t^{4M},\quad\quad 
            v'=s u t^{4M+1} u t^{4M+2} \cdots u t^{6M}.
    \]
     Let $w=v \cdot \tld{s} \cdot (v')^{-1}$. It is clear that $w$ is reduced and we have \ref{length at least N} and \ref{special form}. It is straightforward to check that $\tld{s}t^i\tld{s}$ (for $2M + 1 \leq i \leq 6M-1$) and $\tld{s}t^{-6M}\tld{s}^{-1}$ each only occur once in $w$ (where we include occurrences as an inverse), and therefore the longest word occurring twice in $w$ is $t^{6M - 2}ut^{6M - 1}$. Let $C = |u| + 12M > |t^{6M - 2}ut^{6M - 1}|$. Then
     \[|w| = 1 + 4M|u| + \frac{8M + 1}{2} \cdot 4M > 4M|u| + 16M^2 > MC \geq NC,\]
     so $\lfloor \frac{|w|}{N} \rfloor \geq C$. So we have \ref{cancellation w/N}, and \ref{contains u} is easily checked for subwords of length $C$ (one must take particular care with subwords of $ut^{4M}\tld{s}t^{-6M}u$).
\end{proof}

\begin{rem}
    In the case of actions on sets considered in \cref{s:constructing actions on sets}, we do not use the full strength of condition \ref{special form} in \cref{l:existence small cancellation}: we only use the first three conditions and the fact that $w$ begins with $s$ and ends with $s^{-1}$. In the case of actions on structures, condition \ref{special form} will be used in its entirety.
\end{rem}

\subsection{Group words in completions of seed groups}

We now specialise to completions of seed groups. Throughout this section, let $(G, T)$ be a completion of a seed group $H$. (Recall that we consider the distinguished set of generators $T=H^\ast \cup T_{\mc{D}}$.)

\begin{defn} \label{d: very reduced}
    We say that $w \in \mc{W}(T)$ is \emph{very reduced} if it is reduced and does not contain any subword of the form $t \omega t'$, where $\omega\in\Omega\in \mc{D}(\Theta)$ and $t, t' \in T_{\Omega}^{\pm 1}$. 
    We write $\Wvr(T)$ for the set of very reduced words in $T$.
\end{defn}

The following lemma can be easily deduced from \cref{alternative description} and the uniqueness of normal forms in graphs of groups (see \cite[Ch.\ IV, Thm. 2.1 or Thm 2.6]{LS1979} and \cite[Sec. 5.2]{serre2002trees}).

\begin{lem} \label{normal forms}
    Every element of $G$ is represented by some very reduced word. Every two equivalent reduced words are connected by a sequence of moves, where each move is one of the following:
	\begin{enumerate}[label=(\Roman*)]
		\item \label{type1} replace $t \omega$ by $\omega t$ for 
		$t\in T_{\Omega}^{\pm 1}$ and $\omega\in\Omega\in\D$,
		\item \label{type2} delete a subword of the form $hh^{-1}$ for $h\in H$,
		\item \label{type3} replace a subword of the form $hh'$ with $h \ast h'$, for $h, h' \in H$ with $hh'\neq 1$ (where $\ast$ denotes group composition),
	\end{enumerate}
	or an inverse of one of the above moves. Given a word $w$ in $T^{\pm1}$ it is possible to obtain an equivalent reduced word by performing moves of type \ref{type1} and their inverses, types \ref{type2} and \ref{type3} as well as moves of the following type:	
    \begin{enumerate}[label=(\Roman*)] \setcounter{enumi}{3}	
        \item \label{type4} removing a subword of the form $tt^{-1}$, where $t\in T_{\mathcal{D}}^{\pm1}$. 
    \end{enumerate}
\end{lem}

\begin{lem} \label{l:cyc red v red}
        Let $w \in \Wr(T)$ be cyclically reduced. Then there exists an equivalent word $v \in \Wvr(T)$ which is cyclically reduced and very reduced.
\end{lem}

In the below lemma, recall the notation $\mc{I}(w)$ for the set of initial subwords of a word $w$.

\begin{lem} \label{l:conjugation}
    Let $g \in G$. Let $w \in \Wr(T)$ represent $g$. Then there exists a family $(w_{h, 0} w_{h, 1})_{h \in H}$ of decompositions of $w$ and a function $\eta : \bigcup_{h \in H} \mc{I}(w_{h, 0}) \to H$ such that:
    \begin{enumerate}[label=(\roman*)]
            \item \label{l:conjugation:initial} for $h \in H$ and $u \in \mc{I}(w_{h, 0})$, the word $u^{-1} h u$ represents $h^{\eta(u)} \in H$;
            \item \label{l:conjugation:reduced} for each $h \in H$, letting $h'=h^{\eta(w_{h, 0})}$, we have that the word $w_{h, 1}^{-1} h' w_{h, 1}^{\vphantom{-1}}$ is reduced (and so is a reduced representative of $h^g$).
	\end{enumerate}
\end{lem}
\begin{proof}
    We use induction on $|w|$. For $w = \varnothing$, take $w_{h, 0} = w_{h, 1} = \varnothing$ for all $h \in H$ and take $\eta(\varnothing) = 1$. We now do the induction step.

    Decompose $w$ as $w = av$, where $a = w[0]$, and apply the induction hypothesis to $v$ to obtain a family $(v_{h, 0}, v_{h, 1})_{h \in H}$ of decompositions of $v$ and a function $\eta_v$. In the case $a \in H^\ast$, for each $h \in H$ take $(w_{h, 0}, w_{h, 1}) = (av_{h^a, 0}, v_{h^a, 1})$, and for each $u = av' \in \bigcup_{h \in H} \mc{I}(w_{h, 0})$ take $\eta(u) = a \eta_v(v')$ (and also $\eta(\varnothing) = 1$). In the case $a \in T_\Omega^{\pm 1}$, for $h \in \Omega$ take $(w_{h, 0}, w_{h, 1}) = (av_{h, 0}, v_{h, 1})$ and for $h \notin \Omega$ take $(w_{h, 0}, w_{h, 1}) = (\varnothing, w)$, and for each $u = av' \in \bigcup_{h \in H} \mc{I}(w_{h, 0})$ take $\eta(u) = \eta_v(v')$ (and also $\eta(\varnothing) = 1$). It is then straightforward to verify that $(w_{h, 0}, w_{h, 1})_{h \in H}$ and $\eta$ satisfy the required conditions. 
\end{proof}

\begin{cor} \label{c: conj of subset of H in H}
    Let $A \sub H$, and suppose that $A^g \sub H$ for some $g \in G$. Then there is $\tld{h} \in H$ such that $a^{\tld{h}} = a^g$ for all $a \in A$.
\end{cor}
\begin{proof}
    Let $w \in \Wr(T)$ represent $g$, and let $(w_{h, 0}w_{h, 1})_{h \in H}$ be a family of decompositions of $w$ with associated function $\eta$ given by Lemma \ref{l:conjugation}. For each $a \in A$, letting $a' = a^{\eta(w_{a, 0})}$, the word $w_{a, 1}^{-1}a'w_{a, 1}^{\vphantom{-1}}$ specified in Lemma \ref{l:conjugation}\ref{l:conjugation:reduced} is a reduced representative of $a^g$; but $a^g \in H$, so $w_{a, 1} = \emp$, and hence $w_{a, 0} = w$ and $a' = a^{\eta(w)} = a^g$. Taking $\tld{h} = \eta(w) \in H$, we are done.
\end{proof}

\subsection{Partial actions on sets} \label{ssec:partial actions on sets}
   
\subsubsection{General partial actions on sets}

    \begin{defn}
        Let $U$ be an infinite countable set. A \emph{partial bijection} of $U$ is a bijection $\rho : A \to B$ where $A$, $B$ are subsets of $U$. Note that we do not make any assumptions on the domain $A$ and image $B$ -- they may be infinite and may equal $U$ itself. We write $\supp(\rho)=A \cup B$. If $\supp(\rho)$ is finite, we say that $\rho$ is finite. 
        
        We write $\Fix(\rho)$ for the collection of all $a\in U$ such that $(a,a)\in\rho$.
    \end{defn}

   \begin{defn} \label{d:general partial action}
        Let $U$ be an infinite countable set. Let $G$ be a group with generating set $T$. Let $(\phi_{t})_{t\in T}$ be a family of partial bijections of $U$ (which may have different domains and codomains as subsets of $U$). If $\{t, t^{-1}\} \sub T$ we additionally assume $\phi_{t^{-1}} = \phi_t^{-1}$. 
        
        We define $\phi_w$ for $w \in \mc{W}(T)$ as follows:
   	\begin{itemize}
            \item $\phi_\varnothing = \id_U$;
   		\item for $t\in T$ we define $\phi_{t^{-1}}=\phi_{t}^{-1}$;
   		\item for $w=t_0 \cdots t_{m-1} \in \mc{W}(T)$ we define $\phi_{w} = \phi_{t_{m-1}} \circ \cdots \circ \phi_{t_0}$. (Here we take the composition of the functions as relations -- so $\phi_w$ may have a smaller domain than $\phi_{t_0}$.)
   	\end{itemize}

        Note that each $\phi_w$ is a partial bijection of $U$.

        Suppose that for all $t \in T^{\pm 1}$ and $w \in \mc{W}(T, t)$ we have $\phi_w \sub \phi_t$. For $g \in G$, let $\phi_g=\bigcup_{w\in\mc{W}(T,g)}\phi_{w}$, and let $\phi$ denote the family $(\phi_g)_{g \in G}$. If each $\phi_{g}$ is a partial bijection of $U$, then we say that $\phi$ is a \emph{partial action of $(G, T)$ on the set $U$}. Note that if $\phi$ is a partial action, then $\phi_{g'} \circ \phi_g \sub \phi_{g \ast g'}$ for all $g, g' \in G$ (where $\ast$ denotes group composition).

        Let $\phi$ be a partial action of $(G, T)$ on $U$. For $\bar{a}, \bar{b} \in (U)^k$, we say that $\bar{a}, \bar{b}$ are \emph{in the same $\phi$-orbit} if there is $g \in G$ with $\phi_g(\bar{a}) = \bar{b}$. (Note that this implies in particular that $\bar{a} \sub \dom(\phi_g)$.) We likewise say that $A, B \in [U]^k$ are in the same $\phi$-orbit if there is $g\in G$ with $\phi_{g}(A)=B$. We also define pointwise stabilisers $\St_\phi(A)$, setwise stabilisers $\sSt_\phi(A)$ and sets of fixed points $\Fix_\phi(G_0)$ analogously as for group actions. We say that $\phi$ is \emph{finite} if each $\phi_t$, $t \in T$, is finite.
    
        Given two partial actions $\phi$, $\psi$ of $(G,T)$ on $U$, we say that $\psi$ \emph{extends} $\phi$, written $\phi\leq\psi$, if $\phi_{t}\sub\psi_{t}$ for all $t\in T$. We say that $\psi$ is a \emph{finite extension} of $\phi$ if in addition $\psi_{t}\setminus\phi_{t}$ is finite for all $t\in T$.    
    \end{defn}
    
    \begin{defn}
        Let $\phi$ be a partial action of $(G, T)$ on $U$. Let $k \geq 1$. We say that $\phi$ is \emph{$k$-sharp} if, for each $A \in [U]^k$ and each $g \in G^\ast$, the map $\phi_g$ does not contain $\id_A$.
    \end{defn}
  
    \begin{definition} \label{d:orbits and fixing} 
        Let $\phi$ be a partial action of $(G, T)$ on $U$. Let $a \in U$, and let $w = w_1 \cdots w_m \in \mc{W}(T)$ with $a \in \dom(\phi_w)$.
        
        We define the \emph{$w$-arc in $\phi$ starting at $a$}, denoted $\mc{P}^{\phi}_{w}(a)$, to be the sequence $a, \phi_{w_1}(a), \phi_{w_1 w_2}(a), \cdots, \phi_w(a)$. We sometimes treat this sequence notationally as a set. For $\bar{a} \in (U)^k$, we define the $w$-arc in $\phi$ starting at $\bar{a}$ similarly (acting component-wise on each element of the $k$-tuple).
   \end{definition}
      
   \subsubsection{Partial actions of completions} \label{partial actions of completions}
   
    We now assume that we are given robust $\Theta\leq\sym_{k}$, a $(\Theta, m)$-seed action $\lambda : U \curvearrowleft H$ and a completion $(G, T)$ of $H$. We extend $\lambda$ trivially to a partial action of $(G, T)$ via $\lambda_t = \varnothing$ for all $t \in T_{\mc{D}}$.
    
    \begin{defn} \label{d:support}
        For any finite extension $\phi=(\phi_{t})_{t\in T}$ of $\lambda$ we define the support of $\phi$ to be the (necessarily finite) set $\supp(\phi)=\bigcup_{t\in T_{\mathcal{D}}}\supp(\phi_{t})$.
    \end{defn}
   
    \begin{defn} \label{d:omega-tuples}
        Let $\Omega\in \mc{D}(\Theta)$. We say that $A\in [U]^{k}$ is an \emph{$\Omega$-set} (with respect to $\lambda$) if it is $\Omega$-invariant. (Note that we only use the terminology $\Omega$-set for sets of size $k$.) Note that by \cref{l: key seed action props}\ref{seedact docile}, if $A$ is an $\Omega$-set then $(\lambda: A \curvearrowleft \Omega) \simeq (\pi: \mathbf{k} \curvearrowleft \Omega)$. For an $\Omega$-set $A$ (or more generally a finite union of $\Omega$-sets), we write $A_\Omega = A \setminus \NFr_\lambda(\Omega)$, and call $A_\Omega$ the \emph{$\Omega$-free part} of $A$.

        Let $\phi$ be a finite extension of $\lambda$, and let $A$ be an $\Omega$-set. If $\Omega=\sSt_{\phi}(A)$, then we say that $A$ is a \emph{strict} $\Omega$-set with respect to $\phi$.
        
        We refer to any ordering of a (strict) $\Omega$-set as a \emph{(strict) $\Omega$-tuple}. Let $\bar{a}$ be an $\Omega$-tuple obtained by ordering an $\Omega$-set $A$. We let $\bar{a}_\Omega$ be the tuple obtained by taking, in order, the elements of $\bar{a}$ lying in $A_\Omega$, and we likewise call $\bar{a}_\Omega$ the $\Omega$-free part of $\bar{a}$.
        
        For $j < m$ we say that $A\in[U]^{k}$ is a \emph{$\Theta_{j}$-set} if it is $\Theta_{j}$-invariant and $(\lambda : A \curvearrowleft \Theta_j) \simeq (\pi : \mathbf{k} \curvearrowleft \Theta)$ via $\iota_j^{-1}$.
    \end{defn}
  
    \begin{definition} \label{d:taut partial actions}
        We say that an extension $\phi$ of $\lambda$ is \emph{taut} if:
	\begin{enumerate}[label=(\roman*)]
		\item \label{c-ksharp} $\phi$ is $k$-sharp;
            \item \label{c-fullness} for $\Omega\in\mathcal{D}(\Theta)$ and $t\in T_{\Omega}$, the sets $\dom(\phi_{t})$, $\im(\phi_t)$ are $\Omega$-invariant; 
	    \item \label{c-centraliser acts trivially} for $\Omega\in\D$ and $t\in T_{\Omega}$ we have 
		$\NFr_{\lambda}(\Omega)\subseteq \Fix(\phi_{t})$, 
		\item \label{c-no invariant sets} for $\Omega\in\D$ and $A \fin U$, if $A$ is $\phi_t$-invariant for all $t\in T_{\Omega}$ then $A \sub \NFr_\lambda(\Omega)$.
	\end{enumerate}
   \end{definition}
  
    \begin{lem} \label{l: taut transitive first three conds}
        Suppose $\pi : \mathbf{k} \setminus \NFr_\pi(\Omega) \curvearrowleft \Omega$ is transitive for each $\Omega \in \D$. Let $\phi$ be an extension of $\lambda$ satisfying the first three conditions of \Cref{d:taut partial actions}. Then the fourth condition holds, and so $\phi$ is a taut extension of $\lambda$.
    \end{lem}
    \begin{proof}
        Let $\Omega \in \D$ and $A \fin U$ be such that $A$ is $\phi_t$-invariant for all $t \in T_\Omega$. Assume for a contradiction that $A$ is not a subset of $\NFr_\lambda(\Omega)$, and let $B = A \cup \NFr_\lambda(\Omega)$. Then $\NFr_\lambda(\Omega) \subsetneq B$ and by \Cref{d:taut partial actions}\ref{c-centraliser acts trivially} we have that $B$ is $\phi_t$-invariant for all $t \in T_\Omega$. Note that by \Cref{l: key seed action props}\ref{seedact distinguished set} we have that $|\NFr_\lambda(\Omega)| = |\NFr_\pi(\Omega)|$, so $B$ is finite, and by definition $\lambda|_\Omega : U \setminus \NFr_\lambda(\Omega) \curvearrowleft \Omega$ is free. So $B$ contains at least one free $\lambda$-orbit of $\Omega$, and as $\pi : \mathbf{k} \setminus \NFr_\pi(\Omega) \curvearrowleft \Omega$ is transitive we therefore have $|B| \geq k$. As $B$ is $\phi_t$-invariant for all $t \in T_\Omega$ we therefore have an action $\phi : B \curvearrowleft F_\Omega$, and so the pointwise-stabiliser of $B$ in this action has finite index in $F_\Omega$ and is therefore infinite, contradicting the $k$-sharpness of $\phi$.  
    \end{proof}

    \begin{lem} \label{normal forms action}
        Let $(\phi_t)_{t \in T}$ be a family of partial bijections of $U$ with $\phi_h = \lambda_h$ for $h \in H^\ast$, and define $\phi_w$ for $w \in \mc{W}(T)$ and $\phi_g$ for $g \in G$ as in \cref{d:general partial action}. Suppose in addition that:
        \begin{enumerate}[label=(\alph*)]
            \item \label{i: match with commuting Omega T_Omega} $\phi_\omega \circ \phi_t = \phi_t \circ \phi_\omega$ for $\omega \in \Omega \in \mc{D}(\Theta)$, $t \in T_\Omega^{\pm 1}$;
            \item \label{i: dom im T_Omega Omega-inv} $\dom(\phi_t)$, $\im(\phi_t)$ are $\Omega$-invariant for $\Omega \in \mc{D}(\Theta)$, $t \in T_\Omega^{\pm 1}$.
        \end{enumerate}
        Then:
        \begin{enumerate}[label=(\roman*)]
            \item \label{i: contained in reduced dom} for all equivalent words $w, v \in \mc{W}(T)$, if $v$ is reduced then $\dom\phi_{w} \sub \dom\phi_{v}$;
            \item \label{i: reduced dom same as g} for all reduced words $v \in \Wr(T)$, if $v$ represents $g \in G$ then $\dom\phi_v=\dom\phi_g$;
            \item \label{i: get a partial action} $\phi = (\phi_g)_{g \in G}$ is a partial action.
        \end{enumerate}
    \end{lem}

    \begin{proof}
        \ref{i: reduced dom same as g} and \ref{i: get a partial action} follow immediately from \ref{i: contained in reduced dom}. We now show \ref{i: contained in reduced dom}. Firstly, from \cref{normal forms} it follows that for any two equivalent reduced words $u, u' \in \mc{W}(T)$ we have $\dom \phi_u = \dom \phi_{u'}$, as moves of type \ref{type1}, \ref{type2}, \ref{type3} and their inverses do not change the domain (where for moves of \ref{type1} we use the additional assumptions \ref{i: match with commuting Omega T_Omega} and \ref{i: dom im T_Omega Omega-inv} in the statement of the current lemma).
        
        \cref{normal forms} also states that for any word $w$ there is an equivalent reduced word $u$ obtained by applying moves of type \ref{type1} and their inverses (which do not change the domain), moves of type \ref{type2} and \ref{type3} (which again do not change the domain) and moves of type \ref{type4} (which possibly enlarge the domain), and so the result follows.
    \end{proof}

    Note that, in particular, any taut extension $\phi \geq \lambda$ satisfies conditions \ref{i: match with commuting Omega T_Omega}, \ref{i: dom im T_Omega Omega-inv} of the above Lemma \ref{normal forms action} (one sees this via Definition \ref{d:taut partial actions}\ref{c-fullness}), and we will often use Lemma \ref{normal forms action}\ref{i: contained in reduced dom}, \ref{i: reduced dom same as g} in this case, even though we already have that $\phi$ is a partial action.

    \begin{lem} \label{fin taut exts exist}
        Define a family $(\lambda'_{t \vphantom{t \in T}})^{}_{t \in T}$ of partial bijections of $U$ as follows: $\lambda'_h = \lambda^{}_h$ for $h \in H^\ast$, and $\lambda'_t = \id^{}_{\NFr_\lambda(\Omega)}$ for $t \in T_\Omega$, $\Omega \in \mc{D}(\Theta)$. Then $\lambda' = (\lambda'_{g \vphantom{g \in G}})^{}_{g \in G}$ is a taut partial action. (Here for each $g \in G$ we define $\lambda'_g$ as in \cref{d:general partial action}.)
    \end{lem}
    \begin{proof}
        It follows immediately from \cref{normal forms action} that $\lambda'$ is a partial action. Conditions \ref{c-fullness}, \ref{c-centraliser acts trivially}, \ref{c-no invariant sets} in the definition of tautness (\cref{d:taut partial actions}) are immediate. We now show that $\lambda'$ is $k$-sharp. Suppose for a contradiction that there exist $g \in G^\ast$, $\bar{a} \in (U)^k$ with $\lambda'_g(\bar{a}) = \bar{a}$. Let $w \in \Wr(T)$ represent $g$. Then as $\lambda$ is $k$-sharp, there exists some $1 \leq i \leq |w|$ and some $\Omega \in \mc{D}(\Theta)$ with $w[i] \in T_\Omega^{\pm 1}$. As $\lambda'_{w[i]} = \id_{\NFr_\lambda(\Omega)}$, we have that $\lambda'_{w[1 : i-1]}(\bar{a}) \sub \NFr_\lambda(\Omega)$, but $|\NFr_\lambda(\Omega)| < k$ as $\Omega$ is docile, contradiction.
    \end{proof}

    \begin{lem} \label{l: cyc red fixed tuple}
        Let $\phi$ be an extension of $\lambda$ such that for all $\Omega\in\mathcal{D}(\Theta)$ and $t\in T_{\Omega}$, the sets $\dom(\phi_{t})$, $\im(\phi_t)$ are $\Omega$-invariant. Let $g \in G^\ast$, and suppose there exists $\bar{a} \in (U)^k$, $\bar{a} \sub \dom \phi_g$, with $\phi_g(\bar{a}) = \bar{a}$. Then there exists a cyclically reduced word $v \in \Wr(T)$ representing a conjugate of $g$ and $\bar{b} \in (U)^k$, $\bar{b} \sub \dom \phi_v$, with $\phi_v(\bar{b}) = \bar{b}$.
    \end{lem}
    \begin{proof}
        Let $w$ be a reduced word of minimal length such that $w$ represents some conjugate of $g$ and such that there exists $\bar{b} \in (U)^k$, $\bar{b} \sub \dom(\phi_w)$, with $\phi_w(\bar{b}) = \bar{b}$ (note that the set of reduced words satisfying these conditions is non-empty by assumption). Write $w = w_1 \cdots w_m$ and $\mc{P}_w^\phi(\bar{b}) = (\bar{b}_0, \cdots, \bar{b}_m = \bar{b}_0)$. We claim that $w$ is cyclically reduced. Suppose for a contradiction that there exists $i$ with $1 \leq i \leq m-1$ such that the cyclic shift $w' = w_{1 + i \bmod{m}} \cdots w_{m + i \bmod{m}}$ is not reduced. We have that $w'$ represents a conjugate of $g$ and $\phi_{w'}(\bar{b}_i) = \bar{b}_i$. Let $v$ be a reduced word equivalent to $w'$. Then by \cref{normal forms action} we have $\phi_v(\bar{b}_i) = \bar{b}_i$, but $|v| < |w'| = |w|$, contradicting the minimality of $|w|$.
    \end{proof}
        
  \begin{lemma} \label{l:stabilisers of k sets} 
    Let $\phi \geq \lambda$ be a taut finite extension. Let $A\in[U]^{k}$. Then at least one of the following holds:
   \begin{itemize}	
      \item there is $\Omega \in \mc{D}(\Theta)$ such that there exists a $\phi$-strict $\Omega$-set $B$ in the $\phi$-orbit of $A$; 
      \item $\Theta$ is unruly and for some $j\in\mathbf{m}$ there exists a $\phi$-strict $\Theta_{j}$-set $B$ in the $\phi$-orbit of $A$.
    \end{itemize}
  \end{lemma}
    \begin{proof}
        As $\phi$ is $k$-sharp, the action $A \curvearrowleft \sSt_\phi(A)$ is faithful. Thus $\sSt_\phi(A)$ is finite and so by \cref{torsion completion} it is of the form $\Delta^{g}$ for some $g\in G$ and some finite $\Delta \leq H$. Let $w$ be a reduced representative of $g$, and let $(w = w_{h, 0} w_{h, 1})_{h \in \Delta}$ be a family of decompositions given by \cref{l:conjugation} (where we restrict to $\Delta$). Let $\delta \in \Delta$ be such that $w_{\delta, 0}$ is of minimal length amongst all decompositions indexed by $\Delta$, and let $(w_0, w_1) = (w_{\delta, 0}, w_{\delta, 1})$. Then by
        \cref{l:conjugation}\ref{l:conjugation:initial}, there exists $\eta = \eta(w_0) \in H$ such that $w_0^{-1} h w_0^{}$ represents $h^\eta$ for all $h \in \Delta$. Let $\delta' = \delta^\eta \in H$. By \cref{l:conjugation}\ref{l:conjugation:reduced} the word $w_1^{-1} \delta' w_1^{}$ is a reduced representative of $\delta^g$, and so by \cref{normal forms action} we have $\dom \phi_{\delta^g} = \dom \phi_{w_1^{-1} \delta' w_1^{}}$. Thus $A \sub \dom \phi_{w_1^{-1}}$. Let $A' = \phi_{w_1^{-1}}(A)$. Then $\sSt_\phi(A') = \Delta^\eta$. (For example, to see one direction: let $g' \in G$ be the group element represented by $w_1$. Let $h \in \Delta$. Note that $\dom \phi_{h^\eta} = U$. We have $\phi_{h^\eta}(A') = \phi_{g' \ast h^g \ast g'^{-1}}(A')$ (where $\ast$ denotes group composition), and as $\phi_{g'}(A') = A$ and $h^g \in \sSt_\phi(A)$ we have $h^\eta \in \sSt_\phi(A')$.) As $\Delta^\eta \leq_{\text{fin}} H$, by \cref{d:seed group}\ref{seed torsion} there is $h \in H$ and $j < m$ such that $\Delta^{\eta h} \leq \Theta_j$. 
        
        If $\iota_j^{-1}(\Delta^{\eta h}) \leq \Theta$ is docile: by \cref{d:seed group}\ref{seed conjugate docile} there is $h' \in H$ such that $\Delta^{\eta hh'} \in \mc{D}(\Theta)$. Let $\Omega = \Delta^{\eta hh'}$ and $B = \phi_{hh'}(A')$. Then $\Omega = \sSt_\phi(B)$, and $B$ is an $\Omega$-set, so as required $B$ is a strict $\Omega$-set with respect to $\phi$.

        Otherwise $\iota_j^{-1}(\Delta^{\eta h}) \leq \Theta$ is unruly, and therefore so is $\Theta$. Let $B = \phi_h(A')$ and let $\Omega_j = \Delta^{\eta h} \leq \Theta_j$. Then $\Omega_j = \sSt_\phi(B)$. By \cref{l: key seed action props}\ref{seedact unruly unique} we have that $B$ is the unique $\Theta_j$-invariant $k$-set in $U$, and so $\Omega_j = \Theta_j$ and $B$ is a strict $\Theta_j$-set.
  \end{proof}
  
  We wish to emphasize that the previous proof uses in a fundamental way the fact that $G$ is very close to being a free group. 

\begin{lem} \label{l: action of fin subgp of G}
    Let $\phi \geq \lambda$ be an extension of $\lambda$ which is an action $\phi : U \curvearrowleft G$. Let $\Delta \leq_{\textrm{fin}} G$. Then there is a $\Delta$-invariant set $A \in [U]^k$ and $\Omega \leq \Theta$ such that $(A \curvearrowleft \Delta) \simeq (\pi : \mathbf{k} \curvearrowleft \Omega)$ and the action $U \setminus A \curvearrowleft \Delta$ is free.
\end{lem}
\begin{proof}
    By Lemma \ref{torsion completion} and Definition \ref{d:seed group}\ref{seed torsion}, the subgroup $\Delta$ is conjugate to a finite subgroup of some $\Theta_j \leq H$. The result then follows by Lemma \ref{l: key seed action props}\ref{seedact distinguished set}.
\end{proof}

\section{Extending \texorpdfstring{$k$}{k}-sharp partial actions on sets}
    \label{s:extending sharp partial actions on sets}

    Let $\lambda$ be a seed action (extended trivially to a partial action) on $U$ as in the previous section. The main goal of this section is to describe ways in which, under certain conditions, it is possible to extend taut finite extensions of $\lambda$ (see \cref{d:taut partial actions}) in various ways so that $k$-sharpness is preserved. The most flexible way in which this can be done will be described in \cref{p: small cancellation loops} below, which is the main result of the section and key to the entire paper. 
   
    Recall the notions of $\Omega$-sets $A$ and $\Omega$-tuples $\bar{a}$ from \cref{d:omega-tuples}, as well as of their moved parts $A_\Omega$, $\bar{a}_\Omega$.
    Recall the definition of $\supp(\phi)$ from \cref{d:support}. Also recall the definition of a paradigm set of a seed action $\lambda$ from Definition \ref{d:seed action on set}. The simplest type of extension we will consider is the following.
   
    \begin{defn} \label{d:free extension}
        Let $\phi$ be a taut finite extension of $\lambda$. Let $\Omega \in \mc{D}(\Theta)$ and let $A \fin U$ be a finite union of $\Omega$-sets. Let $t \in T_{\Omega}^{\pm 1}$. An \emph{atomic $\Omega$-free extension of $\phi$ to $A$ by $t$} is a finite extension $\phi' \geq \phi$ obtained by extending $\phi_t$ to $\phi'_t=\phi_t \cup \xi$, where $\xi : A \to A'$ is an $\Omega$-isomorphism, such that:
        \begin{enumerate}[label=(\roman*)]
    	\item \label{free-cond-different orbit} the $\lambda$-orbit of $A'_\Omega \setminus \im(\phi_t)$ is disjoint from $\supp(\phi) \cup A$ and the paradigm set of $\lambda$;
    	\item \label{free-cond-different orbits} two elements in $A'_\Omega \setminus \im(\phi_t)$ are in the same $H$-orbit iff they are in the same $\Omega$-orbit.
        \end{enumerate}
    We call $A$ the \emph{extending set} of $\phi'$.

    An \emph{$\Omega$-free extension of $\phi$} is a finite extension $\psi \geq \phi$ such that there exists a chain of atomic $\Omega$-free extensions $\phi = \phi^0 \leq \cdots \leq \phi^r = \psi$ (where we make no assumptions on the extending set at each step).
    
    Let $v=t_{1}\cdots t_{r}\in\Wr(T_{\Omega})$. A \emph{coherent $\Omega$-free extension of $\phi$ to $A$ by $v$} is a finite extension $\psi \geq \phi$ for which there is a chain $\phi=\phi^{0}\leq\phi^{1} \leq \cdots \leq \phi^{r}=\psi$, where $\phi^{l}$ is an atomic $\Omega$-free extension of $\phi^{l-1}$ to $\phi^{l-1}_{t_{1}\cdots t_{l-1}}(A)$ by $t_{l}$ for all $1\leq l\leq r$.
   \end{defn}

    \begin{lem} \label{free extensions exist}
        Let $\phi \geq \lambda$ be a taut finite extension. Let $\Omega \in \mc{D}(\Theta)$ and let $A \fin U$ be a finite union of $\Omega$-sets. Let $B \fin U$. Let $t \in T_\Omega^{\pm 1}$. Then there exists an atomic $\Omega$-free extension $\phi'$ of $\phi$ to $A$ by $t$ such that the $\lambda$-orbit of $\phi'_t(A_\Omega) \setminus \im(\phi_t)$ is disjoint from $B$.
    \end{lem}
    \begin{proof}
        We define $A'$ and $\xi : A \to A'$ satisfying the conditions of \cref{d:free extension}. As $\phi$ is taut, by \cref{d:taut partial actions}\ref{c-fullness} we have that $\dom \phi_t$ is $\Omega$-invariant, and thus $A_\Omega \setminus \dom \phi_t$ is $\Omega$-invariant. As $\lambda$ has infinitely many free $H$-orbits (\cref{d:seed action on set}) and each $\Omega$-orbit in $A_\Omega \setminus \dom \phi_t$ is free (as $\NFr_\lambda(\Omega) \sub \dom \phi_t$ by \cref{d:taut partial actions}\ref{c-centraliser acts trivially}), we may define $\xi$ so that for each $\Omega$-orbit $C$ of $A_\Omega \setminus \dom \phi_t$ we have that $\xi(C)$ is a free $\Omega$-orbit disjoint from $\supp(\phi) \cup A \cup B$ and the paradigm set, and also so that for distinct $\Omega$-orbits $C_0$, $C_1$ we have that $\xi(C_0)$, $\xi(C_1)$ lie in distinct $H$-orbits.
    \end{proof}

\begin{lemma} \label{l:partial k-sharp extend}
    Let $\phi$ be a taut finite extension of $\lambda$. Let $\Omega \in \mc{D}(\Theta)$ and let $A \fin U$ be a finite union of $\Omega$-sets. Let $t \in T_{\Omega}^{\pm 1}$. Let $\phi'$ be an atomic $\Omega$-free extension of $\phi$ by $t$ to $A$. Let $A' = \phi'_t(A)$, and let $X$ be the union of the $\lambda$-orbits of points in $A'_\Omega \setminus \im(\phi_t)$. Then:
	\begin{enumerate}[label=(\roman*)]
		\item \label{no more permutations} for all $B, C \in [U \setminus X]^{k}$ and $g \in G$, we have $\phi_g(B) = C$ if and only if $\phi'_g(B) = C$; 
		\item \label{keeps being strict} for each $\Omega$-set $B \in [A]^k$, if $B$ is $\Omega$-strict for $\phi$, then $B$, $\phi'_t(B)$ are also $\Omega$-strict for $\phi'$;
		\item \label{preservation tautness}$\phi'$ is taut.
	\end{enumerate}
\end{lemma}
\begin{proof}
    We first claim that:
    \begin{enumerate}
        \item[($\ast$)] any $\phi'$-arc of a reduced word starting and ending in $U \setminus X$ must lie entirely in $U \setminus X$.
    \end{enumerate}
    We now prove the above claim. Let $w = w_1 \cdots w_m \in \Wr(T)$. Let $\mc{P}_w^{\phi'}(b) = (b_0, \cdots, b_m)$ be an arc with $b_0, b_m \in U \setminus X$. Suppose for a contradiction that there exists an element of the arc lying in $X$. Then there exists $i \in \mathbf{m}$ such that $b_i \in U \setminus X$ and $b_{i + 1} \in X$. As $X$ is $H$-invariant we have $w_{i+1} \notin H$, and as $X$ is disjoint from $\supp(\phi) \cup A$ we have $w_{i+1} = t$ and $b_{i + 1} \in A'_\Omega \setminus \im(\phi_t)$. If $b_{i + 2} \notin X$ then $w_{i+2} = t^{-1}$, contradicting that $w$ is reduced. So $b_{i+2} \in X$ and thus $w_{i+2} \in H$. As $w$ is reduced we have $w_{i+3} \notin H$, and hence $w_{i+3} = t^{-1}$. So $b_{i+2} \in A'_\Omega \setminus \im(\phi_t)$. By \cref{d:free extension}\ref{free-cond-different orbits} and the fact that the $H$-orbit of $b_{i+2}$ is free, we have $w_{i + 2} \in \Omega$. Let $\omega = w_{i + 2}$. Then $w_{i+1} w_{i+2} w_{i+3} = t \omega t^{-1}$, and as $t$ and $\omega$ commute the word $w$ is not reduced -- contradiction. This concludes the proof of claim ($\ast$).

    \ref{no more permutations}: It is immediate that $\phi_g(B) = C$ implies $\phi'_g(B) = C$ as $\phi' \geq \phi$. Suppose $\phi'_g(B) = C$. Let $w \in \Wr(T)$ be a reduced representative of $g$. By \cref{normal forms action} applied to $\phi'$ and $\phi$ we have $\dom \phi'_w = \dom \phi'_g$ and $\dom \phi_w = \dom \phi_g$. So $B \sub \dom \phi'_w$. By claim ($\ast$), for each $b \in B$ we have $\mc{P}^{\phi'}_w(b) \sub U \setminus X$, and so $B \sub \dom \phi_w$ and $\mc{P}^{\phi'}_w(b) = \mc{P}^\phi_w(b)$ for all $b \in B$ (where we also use the fact that $\phi$ is taut and so satisfies \Cref{d:taut partial actions}\ref{c-centraliser acts trivially}). In particular $\phi'_w(b) = \phi^{}_w(b)$ for each $b \in B$. Thus $\phi_g(B) = C$. 

    \ref{keeps being strict}: Suppose $B$ is $\phi$-strict. Then by taking $C:=B$ in \ref{no more permutations} we have that $B$ is $\phi'$-strict. As $(\sSt_{\phi'}(\phi'_t(B)))^{t^{-1}} = \sSt_{\phi'}(B)$ we have that $\phi'_t(B)$ is $\phi'$-strict.

    \ref{preservation tautness}: Condition \ref{c-centraliser acts trivially} in \cref{d:taut partial actions} is trivial, and condition \ref{c-fullness} follows as $A$, $A'$ are finite unions of $\Omega$-sets. Condition \ref{c-no invariant sets} follows as $X$ is disjoint from $\supp(\phi) \cup A$. It remains to verify that $\phi'$ is $k$-sharp. Let $B\in[U]^{k}$ and suppose for a contradiction that there exists $g \in G^\ast$ with $g\in \St_{\phi'}(B)$. By \cref{l: cyc red fixed tuple}, we may assume that $g$ has a cyclically reduced representative $v = v_1 \cdots v_m \in \Wr(T)$.

    Consider $\mc{P}^{\phi'}_v(b) = (b_0, \cdots, b_m = b_0)$. If any $b_i$ lies in $U \setminus X$, then the entire arc $\mc{P}^{\phi'}_v(b)$ is contained in $U \setminus X$ -- we can see this by applying claim ($\ast$) to the reduced word given by the cyclic shift $v_{i+1 \bmod{m}} \cdots v_{i + m \bmod{m}}$ and $(b_{i \bmod{m}}, b_{i+1 \bmod{m}}, \cdots, b_{i + m \bmod{m}})$. But then $\mc{P}^{\phi'}_v(b) = \mc{P}^{\phi}_v(b)$, contradicting $k$-sharpness of $\phi$. So $\mc{P}^{\phi'}_v(b) \sub X$. As $b_0, b_1 \in X$ and $X$ is disjoint from $\supp(\phi) \cup A$, we have $v_1 \in H$, and as $X$ is disjoint from the paradigm set of $\lambda$ we have $b_1 \neq b_0$. But then $m \geq 2$, and as $b_1, b_2 \in X$ we similarly have $v_2 \in H$, contradicting that $v$ is reduced. Thus $\phi'$ is $k$-sharp.
\end{proof}

    The following corollary follows from the previous lemma by a standard back-and-forth argument. 
    \begin{corollary} \label{c: extension ksharp}
        Any taut finite extension $\phi$ of $\lambda$ extends to a global taut action of $G$ on $U$. 
    \end{corollary}

\begin{lemma} \label{l:disjointness through free extensions}
        Let $\phi \geq \lambda$ be a taut finite extension. Let $\Omega \in \mc{D}(\Theta)$. Then:
        \begin{enumerate}[label=(\roman*)]
            \item \label{positive word} there is a non-empty positive word $v \in \Wr(T_\Omega)$ with $\dom(\phi_{v})=\NFr_{\lambda}(\Omega)$;
            \item \label{positive word avoiding t} there is a non-empty positive word $u \in \Wr(T_\Omega)$ and $t \in T_\Omega$ satisfying the following:
            \begin{itemize}
                \item $\dom(\phi_{u})=\NFr_{\lambda}(\Omega)$,
                \item for each $A \fin U$ which is a finite union of $\Omega$-sets and each $B \fin U$, there is a coherent $\Omega$-free extension $\psi$ of $\phi$ to $A$ by $u$ such that $\psi_{u}(A)\cap (\dom(\psi_t) \cup B)=\NFr_\lambda(\Omega)$.
            \end{itemize}
        \end{enumerate}
    \end{lemma}
    \begin{proof}
        \ref{positive word}: Let $v \in \Wr(T_\Omega)$ be non-empty and positive with $|\dom(\phi_v)|$ minimal. We show that $\dom \phi_v = \NFr_\lambda(\Omega)$. As $\phi$ is taut, by \cref{d:taut partial actions}\ref{c-centraliser acts trivially} we have $\NFr_\lambda(\Omega) \sub \dom(\phi_v)$. For each positive word $v' \in \Wr(T_\Omega)$, as $\dom \phi_{vv'} \sub \dom \phi_v$, by minimality we have $\dom \phi_{vv'} = \dom \phi_v$ and thus $\im \phi_v \sub \dom \phi_{v'}$. Let $C = \bigcup \{\phi_{v'}(\im \phi_v) \mid v' \in \Wr(T_\Omega) \text{ positive}\}$. Then $C$ is finite and $\phi_t$-invariant for all $t \in T_\Omega$, so by \cref{d:taut partial actions}\ref{c-no invariant sets} we have $C = \NFr_\lambda(\Omega)$, and thus $\dom \phi_v = \NFr_\lambda(\Omega)$ (again by \cref{d:taut partial actions}\ref{c-centraliser acts trivially}). 

        \ref{positive word avoiding t}: Take $v$ from \ref{positive word}. Write $v = t_1 \cdots t_m$ and $v_l = v[1:l]$ for $1 \leq l \leq m$. Let $i \in \N$ be greatest such that $t_m^i$ is a subword of $v$, and let $v' = t_m^i$. Let $u = vv'$ and let $t \in T_\Omega$ with $t \neq t_m$. By applying \cref{free extensions exist}, we inductively construct a sequence $\phi = \phi^0 < \cdots < \phi^m$, where for $1 \leq l \leq m$ we have that $\phi^l$ is an atomic $\Omega$-free extension of $\phi^{l-1}$ to $\phi^{l-1}_{v_{l-1}}(A)$ by $t_l$ such that the $\lambda$-orbit of $\phi^l_{v_l}(A) \setminus \im \phi^{l-1}_{t_l}$ is disjoint from $B \cup \supp(\phi^{l-1}) \cup \phi^{l-1}_{v_{l-1}}(A)$ and the paradigm set. Let $A_l = \phi^l_{v_l}(A)$ and $\tld{A}_l = A_l \setminus \im \phi^{l-1}_{t_l}$ for $0 \leq l \leq m$. At each stage of the construction, the set $\tld{A}_l$ of new points added to the image is disjoint from $\supp(\phi)$. So, for $a \in A$, letting $(a_0, \cdots, a_m) = \mc{P}^{\phi^m}_v(a)$, if $a_m \in \supp(\phi)$ then $a_m = \phi_v(a)$. Thus $A_m \cap \supp(\phi) = \NFr_\lambda(\Omega)$.

        Let $\psi^0 = \phi^m$. Apply \cref{free extensions exist} to inductively construct a sequence $\psi^0 < \cdots < \psi^i$, where for $1 \leq l \leq i$, writing $A_{m + l} = \psi^l_{t_m^l}(A_m)$, we have that $\psi^l$ is an atomic $\Omega$-free extension of $\psi^{l-1}$ to $A_{m + l - 1}$ by $t_m$ such that the $\lambda$-orbit of $A_{m+l} \setminus \im \psi^{l-1}_{t_m}$ is disjoint from $B \cup \supp \psi^{l-1} \cup A_{m + l - 1}$ and the paradigm set. Let $\psi = \psi^i$. Let $a \in A_\Omega$. Then $\psi_v(a) \notin \supp(\phi)$, and as $t \neq t_m$ and $v'$ does not occur as a subword of $v$, we have $\psi_{vv'}(a) \notin \supp(\phi^m)$. Thus $\psi_{vv'}(a) \notin \dom \psi_t$, and so the word $u$ and the generator $t$ are as required.
    \end{proof}
    
     \begin{definition} \label{d:arc extension}
        Let $\phi \geq \lambda$ be a taut finite extension. Let $\Omega\in\D$ and let $\bar{a}, \bar{a}' \in (U)^k$ be $\Omega$-isomorphic $\Omega$-tuples. Let $A, A'$ be the underlying sets of $\bar{a}, \bar{a}'$.
    
        Let $w=t_1 \cdots t_m \in \mc{W}(T_\Omega)$, $w \neq \varnothing$, be such that $A \cap \dom(\phi_{t_1}) = A'\cap\im(\phi_{t_m}) = A \cap A' = \NFr_\lambda(\Omega)$.

        An \emph{$\Omega$-extension of $\phi$ by $w$-arcs from $\bar{a}$ to $\bar{a}'$} is a finite extension $\psi \geq \phi$ such that there exists a sequence $\bar{a}^{}_\Omega = \bar{c}_0, \bar{c}_1, \cdots, \bar{c}_m = \bar{a}'_\Omega$ of pairwise-disjoint $\Omega$-isomorphic $\Omega$-tuples satisfying the following properties (where we write $c_{l, j}$ for the $j$th coordinate of $\bar{c}_l$):
        \begin{enumerate}[label=(\roman*)]
            \item \label{arcs different orbits} for all $l, l' \in \mathbf{m+1}$ and for all $j, j'$, if $c_{l,j}$, $c_{l',j'}$ are in the same $H$-orbit then they belong to the same $\Omega$-orbit (in particular, $l=l'$);
    	\item \label{arcs orbit outside support} for all $l$ with $1\leq l\leq m-1$ and for all $j$, the $H$-orbit of $c_{l,j}$ is disjoint from $\supp(\phi)$ and the paradigm set;
            \item the extension $\psi$ of $\phi$ is constructed by adding $(c_{l,j},c_{l+1,j})$ to $\phi_{t_{l}}$ for all $0\leq l < m$ and for all $j$.
        \end{enumerate}

        Note that \ref{arcs different orbits} implies that for all $l$ with $1\leq l\leq m-1$ and for all $j$, the $H$-orbit of $c_{l,j}$ is disjoint from $A \cup A'$.

        For each $j$, we will refer to $(c_{l,j})_{0\leq l\leq m}$ as an \emph{$w$-arc}. We will also refer to any sequence of the form $(c_{l,j})_{l_{0}\leq l\leq l_{1}}$ as a \emph{subarc} of one of the $w$-arcs. For $1 \leq l \leq m-1$, we say that $c_{l,j}$ is a point in the \emph{interior} of the corresponding arc.
    \end{definition}  

    \begin{lem} \label{extensions by w-arcs exist}
        Let $\phi \geq \lambda$ be a taut finite extension. Let $\Omega\in\D$ and let $\bar{a}, \bar{a}' \in (U)^k$ be $\Omega$-isomorphic $\Omega$-tuples. Let $A, A'$ be the underlying sets of $\bar{a}, \bar{a}'$. Let $w=t_1 \cdots t_m \in \mc{W}(T_\Omega)$, $w \neq \varnothing$, with $A \cap \dom(\phi_{t_1}) = A'\cap\im(\phi_{t_m}) = A \cap A' = \NFr_\lambda(\Omega)$. Then:
        \begin{enumerate}[label=(\roman*)]
            \item \label{w-arcs exts exist} there exists an $\Omega$-extension of $\phi$ by $w$-arcs from $\bar{a}$ to $\bar{a}'$;
            \item \label{w-arcs exts easy taut conds} any $\Omega$-extension of $\phi$ by $w$-arcs from $\bar{a}$ to $\bar{a}'$ satisfies conditions \ref{c-fullness}, \ref{c-centraliser acts trivially}, \ref{c-no invariant sets} in \cref{d:taut partial actions}, the definition of a taut extension.
        \end{enumerate}
    \end{lem}
    \begin{proof}
        \ref{w-arcs exts exist}: let $\bar{c}_0 = \bar{a}_\Omega$, $\bar{c}_m = \bar{a}'_\Omega$. We use \cref{free extensions exist} to inductively construct a sequence $\phi = \phi^0 < \cdots < \phi^{m-1}$, where for $1 \leq l \leq m-1$ we have that $\phi^l$ is an atomic $\Omega$-free extension of $\phi^{l-1}$ to $\phi^{l-1}_{w[1:l-1]}(A)$ by $t_l$ and the $\lambda$-orbit of $\phi^l_{w[1:l]}(A_\Omega)$ is disjoint from $A'_\Omega$ (note that $A \cap \dom(\phi_{t_1}) = \NFr_\lambda(\Omega)$). We let $\bar{c}_l = \phi^l_{w[1:l]}(\bar{a}_\Omega)$ for $1 \leq l \leq m-1$. Let $\xi$ be the bijection $\bar{c}_{m-1} \mapsto \bar{c}_m$. We then define $\phi^m \geq \phi^{m-1}$ via $\phi^m_{t_m} = \phi^{m-1}_{t_m} \cup \xi$ (note that $\phi^m$ may not be an atomic $\Omega$-free extension of $\phi^{m-1}$, as $\bar{c}_m$ and $\supp(\phi)$ may not be disjoint), and $\phi^m$ is the extension required.
        
        \ref{w-arcs exts easy taut conds}: immediate.
    \end{proof}

    To produce an $\Omega$-extension by $w$-arcs which remains $k$-sharp requires significant extra work -- this will occupy us for the rest of this section.
    
    The below \cref{p: small cancellation loops} is fundamental to the entire paper.

    \begin{proposition} \label{p: small cancellation loops}
        Let $\phi \geq \lambda$ be a taut finite extension. Let $\Omega \in \mc{D}(\Theta)$, and let $\bar{a}, \bar{a}' \in (U)^k$ be $\Omega$-isomorphic $\Omega$-tuples (where the $\Omega$-isomorphism extends $\id_{\NFr_\lambda(\Omega)}$) such that:
        \begin{enumerate}[label=(\roman*)]
    	\item\label{distinct set orbits} the corresponding underlying $k$-sets $A,A'$ are not in the same $\phi$-orbit; 
    	\item \label{perms of a} $\bar{a}$ is $\Omega$-strict with respect to $\phi$.
        \end{enumerate}    
        Let $w \in \Wr(T_\Omega)$, and suppose there exist $L_{tv}, L_{sc} \in \N_+$ such that: 
        \begin{enumerate}[label=(\alph*)]
   		\item \label{sc a} for each subword $w' = w[i:j]$ of length $|w'| \geq L_{tv}$ we have $\dom\phi_{w'}=\NFr_{\lambda}(\Omega)$;
            \item \label{sc b} $|w| \geq 2 \times (2L_{sc} + L_{tv} - 1)$;
   		\item \label{sc c} $w$ has cancellation $< L_{sc}$.
        \end{enumerate} 
        
        Let $\psi$ be an $\Omega$-extension by $w$-arcs between $\bar{a}$ and $\bar{a}'$. Then $\psi$ is taut (in particular, $k$-sharp).
    \end{proposition} 

    The rest of this section is devoted to proving the above proposition, and proceeds via a series of lemmas. In all subsequent lemmas in this section, we assume the conditions of \cref{p: small cancellation loops}.

    \begin{notation}
        We let $O$ be the union of the $H$-orbits of the points in $\supp(\phi) \cup A \cup A'$ together with the paradigm set.
    \end{notation}
     
    \begin{lemma} \label{l:auxiliary}
        Let $w= w_{0}\cdot w_{1}\cdot w_{2}$ be a decomposition of $w$ into subwords such that $|w_{1}| \geq 2L_{sc}+L_{tv} - 1$. Then:
        \begin{itemize}
            \item $\dom \psi_{w_1} \sub \psi_{w_0}(A)$;
            \item $\dom \psi_{w_1^{-1}} \sub \psi_{w_2^{-1}}(A')$.
        \end{itemize}

        In particular, $\dom \psi_w = A$ and $\im \psi_w = A'$.
    \end{lemma}
    \begin{proof}
        We only prove the first statement, as the same argument will apply symmetrically for the second. We have $\NFr_\lambda(\Omega) \sub \psi_{w_0}(A) \cap \dom \psi_{w_1}$, so it suffices to consider $b \in \dom \psi_{w_1} \setminus \NFr_\lambda(\Omega)$. Suppose for a contradiction that $b \notin \psi_{w_0}(A)$. Let $(b_0, \cdots, b_{|w_1|}) = \mc{P}_{w_1}^\psi(b)$.

        If $\{b_m \mid m \leq L_{sc}\} \cap O = \varnothing$, then as $b_0 \in \dom \psi_{w_1} \setminus \psi_{w_0}(A)$ and as $\psi$ is an extension of $\phi$ by $w$-arcs, we have that $b_0 = c_{l, j}$ for some $j$ and some $l$ with $1 \leq l \leq m$ and $l \neq |w_0|$. Therefore, as $w_1$ is reduced, we have $(b_m \mid m \leq L_{sc}) = (c_{l + \varepsilon m, j} \mid m \leq L_{sc})$ for some $\varepsilon = \pm 1$ (where $c_{l + \varepsilon m, j}$ always lies in the interior of the arc), and so there are two occurrences of the word $w_1[1:L_{sc}]$ in $w$, contradicting that $w$ has cancellation $< L_{sc}$. Thus there exists minimal $m_0$ with $0 \leq m_0 \leq L_{sc}$ such that $b_{m_0} \in O$. 
        
        Let $m_1$ with $m_0 \leq m_1 \leq |w_1|$ be maximal such that $\{b_{m_0}, \cdots, b_{m_1}\} \sub O$. Then $b_{m_1} = \psi_{w_1[m_0 + 1 : m_1]}(b_{m_0}) = \phi_{w_1[m_0 + 1 : m_1]}(b_{m_0})$, so as $b_{m_0} \notin \NFr_\lambda(\Omega)$, by \ref{sc a} we have $|w_1[m_0 + 1 : m_1]| = m_1 - m_0 < L_{tv}$. So $m_1 < L_{tv} + L_{sc}$.

        Now consider the subword $w_1[m_1 + 1 : |w_1|]$. As $b_{m_1} \in O$ and $b_{m_1 + 1} \notin O$, we have that $w_1[m_1 + 1 : |w_1|]$ is an initial subword of $w^{\pm 1}$, and as $|w_{1}| \geq 2L_{sc}+L_{tv} - 1$ and $m_1 < L_{tv} + L_{sc}$ we have $|w_1[m_1 + 1 : |w_1|]| \geq L_{sc}$, contradicting that $w$ has cancellation $< L_{sc}$. (Note that in the particular case $w_0 = \varnothing$ and $m_0 = m_1 = 0$ we have $b_0 \notin A$ by assumption, and therefore $w$ still contains two distinct occurrences of $w_1[m_1 + 1 : |w_1|]$, giving the desired contradiction.)
    \end{proof}

    Recall the notation $\Wvr(T)$ for the set of very reduced words in $T$ (see \cref{d: very reduced}).
   
    \begin{defn}
        Let $u \in \Wvr(T)$. Let $b \in \dom \psi_u$ and let $(b_0, \cdots, b_{|u|}) = \mc{P}^\psi_u(b)$. For $0 \leq l < |u|$, we say that $b_l$ \emph{begins a $(+)$-traverse} if $b_l \in A_\Omega$ and $b_{l+1}$ is an interior point of a $w$-arc. We say that $b_l$ \emph{begins a $(-)$-traverse} if $b_l \in A'_\Omega$ and $b_{l+1}$ is an interior point of a $w$-arc. (Note that as $A^{}_\Omega \cap A'_\Omega = \varnothing$, if $b_l$ begins a $(+)$-traverse then it does not begin a $(-)$-traverse, and vice versa.) We say that $b_l$ \emph{begins a traverse} if it begins a $(+)$-traverse or a $(-)$-traverse.

        For $\bar{b} \in (O)^k$ with $\bar{b} \sub \dom \psi_u$, writing $(\bar{b}_0, \cdots, \bar{b}_u) = \mc{P}^\psi_u(\bar{b})$, we say that $\bar{b}_l$ \emph{begins a $(+)$-traverse} if each coordinate of $\bar{b}_l$ outside $\NFr_\lambda(\Omega)$ begins a $(+)$-traverse, and we define beginning a $(-)$-traverse and beginning a traverse analogously.
    \end{defn}

    \begin{lem} \label{l: whole arc}
        Let $u \in \Wvr(T)$, $u \neq \varnothing$. Let $b \in \dom \psi_u$ with $\{b, \psi_u(b)\} \sub O$. Then:
        \begin{itemize}
            \item if $b$ begins a $(+)$-traverse, then $|u| \geq |w|$ and $u[1 : |w|] = w$, and thus $\mc{P}^\psi_{u[1 : |w|]}(b)$ is a $w$-arc;
            \item if $b$ begins a $(-)$-traverse, then $|u| \geq |w|$ and $u[1 : |w|] = w^{-1}$, and thus $\mc{P}^\psi_{u[1 : |w|]}(b)$ is the inverse of a $w$-arc.
        \end{itemize}
    \end{lem}
    \begin{proof}
        We prove the first statement -- the proof of the second is similar. We show by induction on $l$ that for $1 \leq l \leq |w|$ we have $|u| \geq l$ and $u[1 : l] = w[1 : l]$. Let $(b_0, \cdots, b_{|u|}) = \mc{P}^\psi_u(b)$.
        
        We first check the base case $l = 1$. As $b_0 \in A_\Omega$, we have $b_0 = c_{0, j}$ for some $j$. As $b_1$ is an interior point of a $w$-arc, we have $b_1 = c_{l', j'}$ for some $1 \leq l' \leq |w| - 1$ and some $j'$. So $\psi_{u[1]}(c_{0, j}) = c_{l', j'}$, and by \cref{d:arc extension}\ref{arcs different orbits} we have $u[1] \notin H$. By \cref{d:arc extension}\ref{arcs orbit outside support} we have $c_{l', j'} \notin \supp \phi$, and so $u[1] = w[1]$ and $b_1 = c_{1, j}$.

        For the induction step, suppose $1 \leq l < |w|$ and that $|u| \geq l$ and $u[1 : l] = w[1 : l]$. So $(b_0, \cdots, b_l) = (c_{0, j}, \cdots, c_{l, j})$. As $b_l = c_{l, j} \notin O$ and $\psi_u(b) \in O$ we have $|u| \geq l + 1$ and $u[l+1] \in H^\ast \cup T_\Omega^{\pm 1}$. Suppose for a contradiction that $u[l + 1] \in H^\ast$. Then by \cref{d:arc extension}\ref{arcs different orbits} and \ref{arcs orbit outside support} we have $b_{l+1} \notin O$, so as $\psi_u(b) \in O$ we have $|u| \geq l + 2$ and $u[l + 2] \in H^\ast \cup T_\Omega^{\pm 1}$. As $u$ is reduced we necessarily have $u[l + 2] \in T_\Omega^{\pm 1}$. Thus $b_{l+1}$ is an interior point of a $w$-arc, and so by \cref{d:arc extension}\ref{arcs orbit outside support} we have $u[l+1] \in \Omega$. But then the subword $u[l : l+2]$ contradicts that $u$ is very reduced, as it is of the form $t \omega t'$ for some $t, t' \in T_\Omega^{\pm 1}$ and $\omega \in \Omega$. So $u[l + 1] \in T_\Omega^{\pm 1}$. As $\psi_{u[l+1]}(c_{l, j}) = b_{l+1}$ we therefore have $u[l+1] \in \{w[l+1], w[l]^{-1}\}$, and as $u[1:l] = w[1:l]$ and $u$ is reduced we have $u[l+1] = w[l+1]$, completing the induction step.
    \end{proof}
    
    \begin{lem} \label{l:traverse decomp}
        Let $u\in\Wvr(T)$, $u \neq \varnothing$, and let $\bar{b} \in (O)^k$, $\bar{b} \sub \dom \psi_u$, with $\psi_{u}(\bar{b}) \in (O)^k$. Let $(\bar{b}_0, \cdots, \bar{b}_{|u|}) = \mc{P}^\psi_u(\bar{b})$. Write $b_{i, j}$ for the $j$th coordinate of $\bar{b}_i$ (where $j \in \mathbf{k}$). Let $\rho : A \to A'$ be the bijection $\bar{a} \mapsto \bar{a}'$. Then:
        \begin{enumerate}[label=(\roman*)]
            \item \label{tv sync} if $b_{l, j}$ begins a $(+)$-traverse, then:
            \begin{itemize}
                \item $\bar{b}_l$ begins a $(+)$-traverse,
                \item the underlying sets of $\bar{b}_l$ and $\bar{b}_{l + |w|}$ are $A$ and $A'$ respectively, and $\rho(\bar{b}_l) = \bar{b}_{l + |w|}$;
            \end{itemize}
            \item \label{tv there and back} if $\bar{b}_l$ begins a $(+)$-traverse and there exists a next index $l'$ for which $\bar{b}_{l'}$ begins a traverse (that is, $l' > l$ is minimal with this property), then the following hold:
            \begin{itemize}
                \item $l' > l + |w|$,
                \item $\bar{b}_{l'} = \phi_{u[l + |w| + 1 : l']}(\bar{b}_{l + |w|})$,
                \item $\bar{b}_{l'}$ begins a $(-)$-traverse;
            \end{itemize}
            \item \label{tv minus} analogous statements to \ref{tv sync}, \ref{tv there and back} hold for $(-)$-traverses by switching signs, switching $A$ and $A'$ and switching $\rho$ and $\rho^{-1}$;
            \item \label{tv endpoints} if $\bar{b}_l$ is the first element of $\mc{P}^\psi_u(\bar{b})$ to begin a traverse, then $\bar{b}_l = \phi_{u[1:l]}(\bar{b}_0)$, and if $\bar{b}_l$ is the last element of $\mc{P}^\psi_u(\bar{b})$ to begin a traverse, then $\bar{b}_{|u|} = \phi_{u[l+1:|u|]}(\bar{b}_l)$.
        \end{enumerate}
        
        Let $(\bar{b}_{l_0}, \cdots, \bar{b}_{l_{n-1}})$ be the subsequence of $\mc{P}^\psi_u(\bar{b})$ consisting of the $\bar{b}_l$ which begin a traverse. Then, as a consequence of the above, this subsequence induces a decomposition of $u$ into subwords \[u = v_0 \cdot w^{\varepsilon_0} \cdots v_{n-1} \cdot w^{\varepsilon_{n-1}} \cdot v_n,\] with $\varepsilon_m = \pm 1$ for $0 \leq m < n$, where $v_0 = u[1 : l_0]$, $v_m = u[l_{m-1} + |w| + 1 : l_m]$ for $1 \leq m \leq n-1$ and $v_n = u[l_{n-1} + |w| + 1 : |u|]$, and this decomposition has the following properties:
        \begin{itemize}
            \item the only subwords of $u$ equal to $w^{\pm 1}$ are precisely the $w^{\varepsilon_m}$, $0 \leq m \leq n-1$, in the decomposition;
            \item $|v_m| \geq 1$ for $1 \leq m \leq n-1$;
            \item $\varepsilon_{m + 1} = - \varepsilon_m$ for $0 \leq m < n-1$.
        \end{itemize}
    \end{lem}
    \begin{proof}
        \ref{tv sync}: suppose that $b_{l, j}$ begins a $(+)$-traverse. By \cref{l: whole arc} we have $u[l+1:l+|w|] = w$. We have $\dom \psi_w = A$ by \cref{l:auxiliary}, so as $\bar{b}_l \sub \dom \psi_{u[l+1:l+|w|]}$ it follows that $\bar{b}_l \sub A$ and thus $\bar{b}_l$ begins a $(+)$-traverse. The remaining parts of the claim are immediate. We have the analogous claim for $(-)$-traverses by a similar argument.
        
        \ref{tv there and back}: take $j$ such that $b_{l, j} \notin \NFr_\lambda(\Omega)$. So $b_{l, j}$ begins a $(+)$-traverse and $u[l+1:l+|w|] = w$. Therefore $b_{l+i, j}$ is an interior point of a $w$-arc for $1 \leq i < |w|$ and so does not lie in $\dom \psi_w = A$ or $\dom \psi_{w^{-1}} = A'$. So $l' \geq l + |w|$. 
        
        Suppose for a contradiction that $l' = l + |w|$. Then $u[l'+1:l'+|w|] \in \{w, w^{-1}\}$. If $u[l'+1:l'+|w|] = w^{-1}$, as $u[l+1:l+|w|] = w$ this contradicts that $u$ is reduced. If $u[l'+1:l'+|w|] = w$, then as the underlying set of $\bar{b}_{l'}$ is $A'$ and $\dom \psi_w = A$, this implies $A = A'$, contradicting that $A \cap A' = \NFr_\lambda(\Omega)$. So $l' > l + |w|$.
        
        The tuple $\bar{b}_{l + |w|}$ has underlying set $A' \sub O$. We have that $\mc{P}^\psi_{u[l+|w|+1:l']}(\bar{b}_{l + |w|})$ does not intersect any $H$-orbit of an interior point of a $w$-arc, using the minimality of $l'$ and the fact that the $H$-orbits of interior points of $w$-arcs are disjoint from $O$. So $\psi_{u[l+|w|+1:l']}(\bar{b}_{l + |w|}) = \phi_{u[l+|w|+1:l']}(\bar{b}_{l + |w|})$. Suppose for a contradiction that $\bar{b}_{l'}$ begins a $(+)$-traverse. Then $\bar{b}_{l'}$ has underlying set $A$ and thus $\phi_{u[l+|w|+1:l']}(A') = A$. So $A$, $A'$ lie in the same $\phi$-orbit, contradicting \cref{p: small cancellation loops}\ref{distinct set orbits}. An analogous argument gives the corresponding statement for $(-)$-traverses, and \ref{tv endpoints} follows via a similar argument to that of this paragraph.

        The statements in the lemma regarding the decomposition of $u$ follow immediately from \ref{tv sync}, \ref{tv there and back}, \ref{tv minus}.
    \end{proof}

    \begin{lemma} \label{l:crossing}
        Let $u\in\Wvr(T)$. Let $b\in \dom(\psi_{u})$, and write $(b_0, \cdots, b_{|u|})=\mathcal{P}^{\psi}_{u}(b)$. and suppose that there is $l$, $0 < l < |u|$, such that $b_l$ belongs to some $w$-arc. 
        
         Let $[l_0, l_1] \sub [0, |u|]$ be an interval containing $l$ of maximal length such that $(b_{l_0}, \cdots, b_{l_1})$ is a subarc of some $w$-arc or its inverse. Then:
        \begin{itemize}
            \item if $l_0 \geq 2$, then $b_{l_0}$ is an endpoint of this $w$-arc;
            \item if $l_1 \leq |u|-2$, then $b_{l_1}$ is an endpoint of this $w$-arc.
   	\end{itemize}
   \end{lemma}
    \begin{proof} 
        We prove the first statement -- the proof of the second statement is similar. The proof resembles that given for the induction step in the proof of \cref{l: whole arc}. Suppose for a contradiction that $b_{l_0}$ is not an endpoint of the $w$-arc containing $(b_{l_0}, \cdots, b_{l_1})$. Then $b_{l_0} \notin O$, and so $u[l_0] \in H^\ast \cup T_\Omega^{\pm 1}$. As $[l_0, l_1]$ is of maximal length, we have that $b_{l_0 - 1}$ does not lie in the $w$-arc containing $b_{l_0}$, so $u[l_0] \in H^\ast$ and $b_{l_0 - 1} \notin O$. Thus $u[l_0 - 1] \in H^\ast \cup T_\Omega^{\pm 1}$, and as $u$ is reduced we have $u[l_0 - 1] \in T_\Omega^{\pm 1}$. As $b_{l_0} \notin O$ and $u[l_0] \in H^\ast$ we have $u[l_0 + 1] \in T_\Omega^{\pm 1}$. As $b_{l_0 - 1}$, $b_{l_0}$ both lie in $w$-arcs and $u[l_0] \in H^\ast$, by \cref{d:arc extension}\ref{arcs different orbits} we have $u[l_0] \in \Omega$. But then the subword $u[l_0 - 1 : l_0 + 1]$ is of the form $t \omega t'$ for some $t, t' \in T_\Omega^{\pm 1}$ and $\omega \in \Omega$, contradicting that $u$ is very reduced.
    \end{proof}

    \begin{lemma} \label{l:crossing 3}
        Let $u\in\Wvr(T)$ with $|u| \geq 2 \times (2L_{sc} + L_{tv} - 1)$ and let $B\in [U]^{k}$ with $B \sub \dom(\psi_u)$. Then there is $v \in \mc{W}(T)$ such that $B \sub \dom(\psi_{v})$ and $\psi_{v}(B) \sub O$.  
    \end{lemma}
    \begin{proof}
        Let $C = 2L_{sc} + L_{tv} - 1$. If $\psi_{u[1:C+1]}(B) \sub O$ then we may simply take $v = u[1:C+1]$, so suppose there exists $b \in B$ such that $\psi_{u[1:C+1]}(b) \notin O$. We first show that $u$ has a subword $u_1$ of length $C$ which is a subword of $w^{\pm 1}$. Let $(b_0, \cdots, b_{|u|}) = \mc{P}^\psi_u(b)$. Let $[l_0, l_1] \sub [0, |u|]$ be the interval containing $C+1$ of maximal length such that $(b_{l_0}, \cdots, b_{l_1})$ is a subarc of some $w$-arc or its inverse. If $l_0 \leq 1$, then as $(b_1, \cdots, b_{C+1})$ is a subarc, we have that the subword $u_1 = u[2:C+1]$ of $u$ of length $C$ is also a subword of $w^{\pm 1}$. If $l_0 \geq 2$, by \cref{l:crossing} we have that $l_0$ is an endpoint of the $w$-arc containing $b_{C+1}$ in its interior. As $l_0 \leq C$ and $|u| \geq 2C$, we have that $(b_{l_0}, \cdots, b_{l_0 + C})$ is defined and thus is a subarc of the $w$-arc containing $b_{C+1}$ (where we recall also that $|w| > C$). So again we obtain a subword $u_1 = u[l_0 + 1 : l_0 + C]$ of $u$ of length $C$ which is a subword of $w^{\pm 1}$. Taking $u_1$ from either case above, we decompose $u$ as $u = u_0 u_1 u_2$.

        Assume that $u_1$ (obtained in either case above) is a subword of $w$ -- the case of $w^{-1}$ is similar. Decompose $w$ as $w_0 u_1 w_2$. By \cref{l:auxiliary} we have $\dom \psi_{u_1} \sub \psi_{w_0}(A)$. So $\psi_{u_0}(B) \sub \psi_{w_0}(A)$, and thus $\psi_{u_0 w_0^{-1}}(B) \sub A \sub O$ as required.
    \end{proof}
    
    \begin{lem} \label{l:cr vr word fixing k-tuple}
        Let $u \in \Wvr(T)$ be a very reduced word which is also cyclically reduced. Suppose that there exists $b \in \dom \psi_u$ with $b \notin O$ and $\psi_u(b) = b$. Then $|u| \geq |w|$. 
    \end{lem}
    \begin{proof}
        First suppose for a contradiction that $b$ is not in the interior of some $w$-arc. Then $b \notin \supp \psi$, so $u[1], u[|u|] \in H$. As $b$ does not lie in the non-free $H$-orbit, we thus have $\psi_{u[1]}(b) \neq b$, so $|u| \geq 2$, and thus $u$ is not cyclically reduced -- contradiction. So $b = c_{l, j}$ for some $1 \leq l \leq m-1$ and some $j$. Let $v_0 = w[1:l]$ and $v_1 = w[l+1 : |w|]$. Then as $\psi_u(b) = b$ and as $u$ is very reduced, the word $u$ must decompose as one of the following (where $u'$ denotes the remaining part of $u$ in the decomposition):
        \begin{itemize}
            \item $u = v_1 u' v_0$: clearly here $|u| \geq |w|$;
            \item $u = \omega v_1 u' v_0$ or $u = v_1 u' v_0 \omega$ for some $\omega \in \Omega$: likewise $|u| \geq |w|$;
            \item $u = \omega v_1 u' v_0 \omega'$ for some $\omega, \omega' \in H$: this case cannot occur as $u$ is cyclically reduced;
            \item $u = v_0^{-1}u'v_1^{-1}$ and analogous cases to the above: all have $|u| \geq |w|$;
            \item $u = v_1 u' v_1^{-1}$, $u = \omega v_1 u' v_1^{-1}$, $u = v_1 u' v_1^{-1} \omega$, $u = \omega v_1 u' v_1^{-1} \omega'$: these cannot occur as $u$ is cyclically reduced, and similarly for the cases only involving $v_0$ instead of $v_1$. \qedhere
        \end{itemize}
    \end{proof}

    We now prove \cref{p: small cancellation loops}. The only statement that we need to prove is that $\psi$ is $k$-sharp.

    \begin{proof}[\textbf{Proof of \cref{p: small cancellation loops}}]
        Suppose for a contradiction there is $g \in G^\ast$ and $\bar{b}\in(U)^{k}$ with $\psi_g(\bar{b})=\bar{b}$. In the case that $\bar{b} \notin O^k$ then by \cref{l: cyc red fixed tuple} we may assume that $g$ has a cyclically reduced and very reduced representative $u \in \Wvr(T)$. By \cref{l:cr vr word fixing k-tuple} we have $|u| \geq |w|$, and so by assumption \ref{sc b} in \cref{p: small cancellation loops} and \cref{l:crossing 3} there exists $g_0 \in G^\ast$ with $\psi_{g_0}(\bar{b}) \in (O)^k$. Let $g_1 = g^{g_0}$. Then $\psi_{g_1}(\psi_{g_0}(\bar{b})) = \psi_{g_0}(\bar{b})$. So we may assume that $\bar{b} \in (O)^k$.

        Let $u \in \Wvr(T)$ be a very reduced representative of $g$, and let $u = v_0 \cdot w^{\varepsilon_0} \cdots v_{n-1} \cdot w^{\varepsilon_{n-1}} \cdot v_n$ be the decomposition of $u$ given by \cref{l:traverse decomp}. By \cref{l:traverse decomp}, using the facts that $\phi$ is $k$-sharp and $A, A'$ do not lie in the same $\phi$-orbit, we have $n \geq 2$. If $\varepsilon_0 = 1$ then, letting $g_0 \in G$ be the group element represented by $v_0 w$, we have $\psi_{g^{g_0}}(\bar{a}') = \bar{a}'$, and applying \cref{l:traverse decomp} to a very reduced representative of $g^{g_0}$ we have that the first instance of $w^{\pm 1}$ must be $w^{-1}$. So we may assume that $\varepsilon_0 = -1$, and thus $u$ is of the form $v_0 w^{-1} v_1 w v_2 \cdots$, where $v_1$ is non-empty (possibly $v_0$, $v_2$ are empty, but the two instances of $w^{\pm 1}$ must occur). But as the underlying sets of $\psi_{v_0 w^{-1}}(\bar{b})$ and $\psi_{v_0 w^{-1} v_1}(\bar{b})$ are both $A$, we have that $\psi_{v_1}(A) = A$, so $\phi_{v_1}(A) = A$ (by \cref{l:traverse decomp}), and thus as $A$ is a strict $\Omega$-set with respect to $\phi$ we have that $v_1$ represents an element of $\Omega$. As $v_1$ is reduced, in fact $v_1 = \omega$ for some $\omega \in \Omega$. But then $u = v_0 w^{-1} \omega w \cdots$, and as $w \in \Wr(T_\Omega)$ this contradicts the fact that $u$ is very reduced. So $\psi$ is $k$-sharp.
    \end{proof}

\section{Constructing sharply \texorpdfstring{$\Theta$}{Theta}-transitive actions} \label{s:constructing actions on sets}
    
    We assume as in \cref{s:extending sharp partial actions on sets} that we are given robust $\Theta\leq\sym_{k}$, a $(\Theta, m)$-seed action $\lambda : U \curvearrowleft H$ and a completion $(G, T)$ of $H$. We extend $\lambda$ trivially to a partial action $U \curvearrowleft (G, T)$ as in \cref{partial actions of completions}. 
    
    \textit{In this section only, we make the additional assumption that $m = 1$ and $H = \Theta$.} (From \cref{s:seed actions on structures} onwards we will not make this assumption, and we consider $m \geq 1$ and $(\Theta, m)$-seed groups in general.)

    We assume that the reader is acquainted with basic facts about Polish spaces (e.g.\ that a countable product of Polish spaces is Polish, and that a subspace of a Polish space is Polish iff it is a $\text{G}_\delta$-set); see \cite{Kec95} for further background. 

    \begin{defn} \label{d:spaces of actions on a set}
        Recall that $\sym_U$ denotes the permutation group of $U$. We equip $\sym_U$ as usual with the pointwise convergence topology: for each $n \in \N$ and $\bar{a}, \bar{b} \in U^n$, we specify a basic open set $\{f \in \sym_U \mid \bar{b} = \bar{a} \cdot f\}$. We then have that $\sym_U$ is a Polish space. As $G$ is a countable group, the product space $(\sym_U)^G$ is also Polish. Let $\Act(U, G)$ be the set of (right) actions of $G$ on $U$. We equip $\Act(U, G)$ with the subspace topology induced by $(\sym_U)^G$. As $\Act(U, G)$ is a closed subspace of $(\sym_U)^G$, we then also have that $\Act(U, G)$ is Polish.

        We write $\Act_\lambda(U, G)$ for the subspace of $\Act(U, G)$ consisting of the actions which extend $\lambda$. As $T_\mc{D}$ is finite, the topology on $\Act_\lambda(U, G)$ has a base given by the basic open sets $\mc{V}_\phi = \{\mu \in \Act_\lambda(U, G) \mid \phi \leq \mu\}$, where $\phi$ runs over all finite extensions of $\lambda$.

        We write:
        \begin{align*}
            \Ta_\lambda(U, G) &= \{\mu \in \Act_\lambda(U, G) \mid \mu \text{ is a taut extension of } \lambda \}\\
            \STTa_{\Theta, \lambda}(U, G) &= \{\mu \in \Ta_\lambda(U, G) \mid \mu \text{ is sharply } \Theta \text{-transitive}\}
        \end{align*}

        It is straightforward to check that $\Act_\lambda(U, G)$ is a closed subspace of $\Act(U, G)$ and that $\Ta_\lambda(U, G)$ is a closed subspace of $\Act_\lambda(U, G)$. So $\Act_\lambda(U, G)$ and $\Ta_\lambda(U, G)$ are Polish spaces. 

        Let $\mc{V}_\phi$ be a basic open set of $\Act_\lambda(U, G)$, and let $\mu \in \mc{V}_\phi \cap \Ta_\lambda(U, G)$. As $\mu$ is taut, we may expand $\phi$ to a taut partial action compatible with $\mu$ (here $\phi$ immediately inherits conditions \ref{c-ksharp}, \ref{c-no invariant sets} from $\mu$, and we expand $\phi$ compatibly with $\mu$ if necessary to satisfy conditions \ref{c-fullness}, \ref{c-centraliser acts trivially}). So $\Ta_\lambda(U, G)$ has a base given by $\widetilde{\mc{V}}_\phi = \mc{V}_\phi \cap \Ta_\lambda(U, G)$, where $\phi$ runs over all taut finite extensions of $\lambda$.
    \end{defn}

    In the below, recall that a non-empty perfect Polish space has cardinality $2^{\aleph_0}$ (see \cite[Corollary 6.3]{Kec95}).
    
    \begin{proposition} \label{p:generic action on sets} 
        The space $\Ta_\lambda(U, G)$ is a non-empty perfect Polish space (hence has cardinality $2^{\aleph_0}$), and the subspace $\STTa_{\Theta,\lambda}(U, G)$ is a dense $\text{G}_{\delta}$ subset of $\Ta_\lambda(U,G)$. 
    \end{proposition}
    \begin{proof}
        First, observe that by \cref{fin taut exts exist} a taut finite extension $\lambda'$ of $\lambda$ exists, and by \cref{c: extension ksharp} we have that $\lambda'$ extends to a global taut action on $U$. So the space $\Ta_\lambda(U, G)$ is non-empty. 
        
        We next show that $\Ta_\lambda(U, G)$ is perfect. Let $\mu \in \Ta_\lambda(U, G)$, and let $\widetilde{\mc{V}}_\phi$ be a basic open set of $\Ta_\lambda(U, G)$ containing $\mu$, where $\phi$ is a taut finite extension of $\lambda$. Let $\Omega \in \mc{D}(\Theta)$, and let $A \in [U]^k$ be an $\Omega$-set with $A \not\sub \supp(\phi)$. Let $t \in T_\Omega$. By \cref{free extensions exist} and \cref{l:partial k-sharp extend}, there is a taut finite extension $\psi \geq \phi$ with $\psi_t(A) \neq \mu_t(A)$. As each taut finite extension extends to a global taut action, we have $\varnothing \neq \widetilde{\mc{V}}_\psi \sub \widetilde{\mc{V}}_\phi \setminus \{\mu\}$, so $\Ta_\lambda(U, G)$ is perfect.

        We now show that $\STTa_{\Theta,\lambda}(U, G)$ is a $\text{G}_{\delta}$ subset of $\Ta_\lambda(U,G)$. Let $A_{0}\in[U]^{k}$ be a $\Theta$-invariant set for which $(\lambda : A_0 \curvearrowleft \Theta) \simeq (\pi : \mathbf{k} \curvearrowleft \Theta)$. For $C\in[U]^{k}$ let $\mc{W}_C$ be the set of actions in $\Ta_\lambda(U, G)$ with $A_0, C$ in the same orbit. It is straightforward to see that each $\mc{W}_C$ is an open subset of $\Ta_\lambda(U, G)$. By definition, each $\mu \in \STTa_{\Theta, \lambda}(U, G)$ acts transitively on $[U]^k$ and thus lies in each $\mc{W}_C$. Let $\mu \in \bigcap_{C \in [U]^k} \mc{W}_C$. As $\mu$ is $k$-sharp, the set-stabiliser $\sSt_\mu(A_0)$ is finite, and as $\sSt_\mu(A_0)$ is conjugate to a subgroup of $\Theta$ by \cref{torsion completion}, we have $|\sSt_\mu(A_0)| \leq |\Theta|$. As $\Theta \sub \sSt_\mu(A_0)$, we have $\sSt_\mu(A_0) = \Theta$. As $\mu$ acts transitively on $[U]^k$, we therefore have $\mu \in \STTa_{\Theta, \lambda}(U, G)$. So $\STTa_{\Theta, \lambda}(U, G) = \bigcap_{C \in [U]^k} \mc{W}_C$, and thus $\STTa_{\Theta, \lambda}(U, G)$ is a $\text{G}_\delta$ set.
        
        Finally, we show that $\STTa_{\Theta, \lambda}(U, G)$ is dense in $\Ta_\lambda(U, G)$. As $\Ta_\lambda(U, G)$ is a Polish space and $\STTa_{\Theta, \lambda}(U, G)$ is a countable intersection of the open sets $\mc{W}_C$, by the Baire category theorem it suffices to show that each $\mc{W}_C$ is dense in $\Ta_\lambda(U, G)$.

        Let $C \in [U]^k$. Let $\mu \in \Ta_\lambda(U, G)$, and let $\widetilde{\mc{V}}_\phi$ be a basic open set of $\Ta_\lambda(U, G)$ containing $\mu$, where $\phi$ is a taut finite extension of $\lambda$. We will show that $\mc{W}_C \cap \widetilde{\mc{V}}_\phi \neq \varnothing$. If $\Theta$ is unruly and there is a $\phi$-strict $\Theta$-set $C'$ in the $\phi$-orbit of $C$, then by \cref{l: key seed action props}\ref{seedact unruly unique}, as $A_0$ is the unique $\Theta$-set we have $C' = A_0$, and so $\widetilde{\mc{V}}_\phi \sub \mc{W}_C$.

        Otherwise, by \cref{l:stabilisers of k sets} there is $\Omega \in \mc{D}(\Theta)$ and a $\phi$-strict $\Omega$-set $C'$ in the $\phi$-orbit of $C$. It thus suffices to find a taut finite extension $\psi \geq \phi$ with $C'$, $A_0$ in the same $\psi$-orbit (note that given any taut finite extension of $\lambda$, \cref{c: extension ksharp} guarantees that the basic open set specified by this extension is non-empty).

        By \cref{l:disjointness through free extensions}, there is a non-empty positive word $u'\in\Wr(T_\Omega)$, a coherent $\Omega$-free extension $\phi^{(1)}$ of $\phi$ to $A_0 \cup C'$ by $u'$ and a generator $t \in T_\Omega$ such that $\phi^{(1)}_{u'}(A_0 \cup C') \cap \dom \phi^{(1)}_t=\NFr_{\lambda}(\Omega)$. Let $\phi^{(2)}$ be an atomic $\Omega$-free extension of $\phi^{(1)}$ to $\phi^{(1)}_{u'}(A_0 \cup C')$ by $t$. Then $\phi^{(2)}_{u't}(A_0 \cup C') \cap \dom \phi^{(2)}_s = \NFr_\lambda(\Omega)$ for each $s \in T_\Omega$. Let $\phi^{(3)}$ be an atomic $\Omega$-free extension of $\phi^{(2)}$ to $\phi^{(2)}_{u't}(C')$ by $t$. Let $A'_0 = \phi^{(3)}_{u't}(A_0)$ and $C'' = \phi^{(3)}_{u't^2}(C')$. Then $A'_0 \cap C'' = \NFr_\lambda(\Omega)$ and $(A'_0 \cup C'') \cap \dom \phi^{(3)}_s = \NFr_\lambda(\Omega)$ for each $s \in T_\Omega \setminus \{t\}$. By \cref{l:partial k-sharp extend}, the partial action $\phi^{(3)}$ is taut and $C''$ is a $\phi^{(3)}$-strict $\Omega$-set. By Lemma \ref{l:partial k-sharp extend}\ref{no more permutations}, we have that $C'', A'_0$ are not in the same $\phi^{(3)}$-orbit.  
        
        Applying \cref{l:disjointness through free extensions} again to $\phi^{(3)}$, there is a non-empty positive word $u \in \Wr(T_\Omega)$ with $\dom \phi^{(3)}_u = \NFr_\lambda(\Omega)$. Let $s \in T_\Omega \setminus \{t\}$. Apply \cref{l:existence small cancellation} with $S = T_\Omega$, the letters $s = \tld{s}$ and $t$, the positive word $u$ and $N = 6$ to obtain $w \in \Wr(T_\Omega)$ satisfying the conditions specified in the lemma. Note that $w$ begins with $s$ and ends with $s^{-1}$. Enumerate $C''$, $A'_0$ as $\bar{c}''$, $\bar{a}'_0$ so that the map $\bar{c}'' \mapsto \bar{a}'_0$ is an $\Omega$-isomorphism extending $\id_{\NFr_\lambda(\Omega)}$. Then by \cref{extensions by w-arcs exist} there exists an $\Omega$-extension $\psi$ of $\phi^{(3)}$ by $w$-arcs from $\bar{c}''$ to $\bar{a}'_0$, and taking $L_{tv} = L_{sc} = \lfloor \frac{|w|}{6} \rfloor$, the conditions of \cref{p: small cancellation loops} are satisfied, so $\psi$ is taut. So $\STTa_{\Theta, \lambda}(U, G)$ is dense in $\Ta_\lambda(U, G)$. 
    \end{proof}

We now apply Proposition \ref{p:generic action on sets} in the particular cases $\Theta = \sym_2, \sym_3$. We now assume some background knowledge of sharply $k$-transitive groups: in particular, the notions of a non-split group and the characteristic of a sharply $k$-transitive action (see for example \cite{Ten16I}). Recall from \cite{Ten16I} that:
\begin{itemize}
    \item a sharply $2$-transitive action $U \curvearrowleft G$ has \emph{characteristic 2} if all involutions have no fixed points;
    \item a sharply $2$-transitive group $G$ is \emph{non-split} if it does not have a non-trivial abelian normal subgroup,
\end{itemize}
and recall that we define the same two properties for sharply $3$-transitive $G$ if they hold for each point-stabiliser of $G$.

\begin{defn}
    Let $G$ be a group and $U$ a countably infinite set.
    \begin{itemize}
        \item We define $\ShChar_2(U, G)$ to be the space of $2$-sharp actions $U \curvearrowleft G$ such that every involution acts freely.
        \item We define $\ShChar_3(U, G)$ to be the space of $3$-sharp actions $U \curvearrowleft G$ such that each involution has exactly one fixed point and each element of $G$ of order $3$ acts freely.
    \end{itemize}
    For $k = 2, 3$, we define $\ShTChar_k(U, G)$ to be the subspace of $\ShChar_k(U, G)$ consisting of those actions which are sharply $k$-transitive. Note that $\ShTChar_k(U, G)$ is exactly the space of sharply $k$-transitive actions $U \curvearrowleft G$ of characteristic $2$.
\end{defn}

\begin{cor} \label{c: non-split fin pres from prop}
    Let $k = 2, 3$. Let $\Theta = \sym_k$ and let $G$ be a completion of $\Theta$. Let $U$ be an infinite set. Then:
    \begin{enumerate}[label=(\roman*)]
        \item \label{i: gen for set actions} the Polish space $\ShChar_k(U, G)$ has cardinality $2^{\aleph_0}$ and $\ShTChar_k(U, G)$ is a comeagre subset of $\ShChar_k(U, G)$;
        \item \label{i: non split} $G$ is non-split and finitely presented.
    \end{enumerate}
\end{cor}
\begin{proof}
    \ref{i: gen for set actions}: Let $\Se_k(U, \Theta)$ be the space of seed actions $U \curvearrowleft \Theta$. It is straightforward to check that the restriction map $\rho : \mu \mapsto \mu|_\Theta$ is a continuous open map $\rho : \ShChar_k(U, G) \to \Se_k(U, \Theta)$. By Lemma \ref{l: action of fin subgp of G} we have $\Ta_\lambda(U, G) \sub \ShChar_k(U, G)$ for each $\lambda \in \Se_k(U, \Theta)$. By Lemma \ref{l: taut transitive first three conds}, it is straightforward to see that for each $\lambda \in \Se_k(U, \Theta)$ we have $\rho^{-1}(\lambda) = \Ta_\lambda(U, G)$. By Proposition \ref{p:generic action on sets} we have that $\STTa_{\Theta, \lambda}(U, G)$ is a dense $\mathrm{G}_\delta$ set in $\Ta_\lambda(U, G)$ and $|\Ta_\lambda(U, G)| = 2^{\aleph_0}$ for each $\lambda \in \Se_k(U, \Theta)$, and so $|\ShChar_k(U, G)| = 2^{\aleph_0}$ and we have that $\ShTChar_k(U, G) \cap \rho^{-1}(\lambda)$ is comeagre in $\rho^{-1}(\lambda)$ for each $\lambda \in \Se_k(U, \Theta)$. It is straightforward to check that $\ShTChar_k(U, G)$ has the Baire Property. Thus by \cite[Theorem 1.33]{Mel26} (a generalised version of the Kuratowski-Ulam theorem) we have that $\ShTChar_k(U, G)$ is comeagre in $\ShChar_k(U, G)$ as required.

    \ref{i: non split}: It is immediate from the definition of completions of a seed group (Definition \ref{d:free completion}) that $G$ is finitely presented. It remains to show that $G$ is non-split. We imagine that it is folklore that sharply $k$-transitive virtually free groups which are not virtually cyclic are non-split, but we give a proof for the benefit of the reader.
    
    \begin{claim*}\label{normal subgroups}
        A non-abelian group $G$ that admits a sharply $2$-transitive action on an infinite set $\mu : U \curvearrowleft G$ cannot have a non-trivial abelian normal subgroup $N$ that is finite or finite-index.
    \end{claim*}
     \begin{subproof}
        First consider the case where $N$ has finite index in $G$. As $N$ is abelian we have $|G : \centr_G(N)| \leq |G : N|$. Let $g \in N$, $g \neq 1$. We have $|G : \centr_G(g)| \leq |G : \centr_G(N)|$, and so $g$ has finitely many conjugates in $G$. As $\mu$ is sharp, there is $a \in U$ with $a \neq a \cdot g$. For each $b \in U \setminus \{a\}$, by $2$-transitivity of $\mu$ there is $h \in G$ with $(a, a \cdot g) \cdot h = (a, b)$, so $a \cdot g^h = b$, and as $U$ is infinite this gives a contradiction.
     	  
        The remaining case is where $N$ is finite. As $N \trianglelefteq G$, each element of $G$ must respect the partition of $U$ into $N$-orbits, and so by sharp $2$-transitivity of $G$ we have that $\mu|_N : U \curvearrowleft N$ is transitive or trivial. As $N$ is finite, the action $\mu|_N$ cannot be transitive, so $\mu|_N$ is trivial and thus by $2$-sharpness $N = \mathtt{1}$.
     \end{subproof} 

     We now show that the completion $G$ of the seed group $\Theta$ is non-split. First consider the case $k = 2$. Note that $G$ is not virtually cyclic: for any finite-index subgroup $G' \leq G$, we have that for each $g \in G$, some non-zero power of $g$ lies in $G'$, and so such $G'$ cannot be cyclic as $G$ contains a non-abelian free group (and thus two elements $g_0,g_1$ where no power of $g_0$ commutes with a power of $g_1$). Applying the above Claim, if $G$ is split, then it must contain some abelian $N \trianglelefteq G$ which is both infinite and has infinite index. Let $F$ be a free subgroup of $G$ of finite index. Then $N \cap F$ is an infinite-index infinite abelian normal subgroup of $F$. However, any non-trivial abelian subgroup of $F$ is cyclic, and its normaliser is the maximal cyclic subgroup of $F$ containing it (see \cite[Ch.1, Thm. 2.19]{LS1979}). Therefore $F$ itself would have to be cyclic, which is impossible.
     
     In the case $k=3$, given a sharply $3$-transitive action $\mu : U \curvearrowleft G$, we need to check that the stabiliser $G_{a}=\St_{\mu}(a)$ of a point $a\in U$ admits no non-trivial abelian normal subgroup. Since $G_{a}$ inherits the property of being virtually free from $G$ and the action $U\setminus \{a\} \curvearrowleft G_{a}$ is sharply $2$-transitive, the previous argument applies (recalling the notation of Definition \ref{d:free completion} and letting $\sigma \in \sym_3$ be an involution with $\sg{\sigma} \in \mc{D}(\sym_3)$, we have that $G_{a}$ contains some conjugate of the non-abelian free group $F_{\subg{\sigma}}$ and thus is not virtually cyclic).
\end{proof}

\begin{remark}
    Another proof of Corollary \ref{c: non-split fin pres from prop}\ref{i: non split} is via the fact that $G$ admits an action on a simplicial tree that is faithful, minimal and irreducible (the last corresponds to type $V$ in \cite{chiswell2001introduction}, p.134). It is a consequence of \cite[Ch.\ IV, Thm 2.9]{chiswell2001introduction} that if a group $G$ acts minimally and irreducibly on a simplicial tree $S$, then any abelian normal subgroup $A\trianglelefteq G$ must be in the kernel of the action. This reduces the problem to showing that $G$ (resp.\ $G_{a}$) does not have any non-trivial finite abelian normal subgroup, which can be shown either from the structure of $G$ as a graph product of groups or using the Claim in the above proof.
\end{remark}

Corollary \ref{c: non-split fin pres from prop} is a strengthened form of Corollary \ref{c: non split fin pres} from the Introduction (see \cref{ex:s2 s3} for the two examples of groups $G$ stated in Corollary \ref{c: non split fin pres}), and so this completes the proof of Corollary \ref{c: non split fin pres}.
     
\section{Sharply \texorpdfstring{$k$}{k}-homogeneous actions on structures via pleasant structural seed actions} \label{s:seed actions on structures} 
 
This section begins the second half of the paper devoted to finding sharply $k$-homogeneous actions on relational \Fr structures. We will follow the general strategy of the proof of \cref{p:generic action on sets}.

\begin{notn}
    Given a structure $\mc{M}$, we will generally write $M$ for the domain of $\mc{M}$ (the underlying set of the first-order structure). For $A \sub M$ we write $\mc{M}|_A$ or $A^{\mc{M}}$ for the substructure of $\mc{M}$ induced on $A$. We write $[\mc{M}]^k$ for the set of substructures of $\mc{M}$ of size $k$.

    Let $\mc{M}$ be a structure. We call an isomorphism between finitely generated substructures of $\mc{M}$ a \emph{finite partial isomorphism} of $\mc{M}$. Sometimes to avoid excessively heavy typography we will write $A \sub \mc{M}$ when we should more properly write $\mc{A} \sub \mc{M}$: we only do this in the case where $\mc{M}$ is a previously specified ``ambient" structure (for us always a \Fr structure), and as substructures are in one-to-one correspondence with their domains in this case, this should hopefully not cause confusion. We do this most often when using SWIRs (see \cref{s:structures and independence}).
\end{notn}

\begin{defn} \label{d:partial actions structure}
    Let $\mc{M}$ be a locally finite structure, and let $G$ be a group with generating set $T$. A \emph{partial action of $(G,T)$ on $\mc{M}$} is a partial action $\phi=(\phi_g)_{g \in G}$ of $(G,T)$ on the universe $M$ (see \cref{d:general partial action}) such that each $\phi_g$ is an isomorphism of substructures of $\mc{M}$. 
\end{defn}

In the below definition, recall that $\pi : \mathbf{k} \curvearrowleft \Theta$ is the permutation action.

\begin{defn} \label{d:theta r structure}
    Let $\M$ be a strong relational \Fr structure. Let $k \geq 1$ and let $\Theta\leq\sym_{k}$ be a robust subgroup. Let $r \geq 1$. We say that $\M$ is a \emph{$(\Theta,r)$-structure} if:
    \begin{enumerate}[label=(\roman*)]
        \item \label{automorphism tuples} for each $\mc{A}\in[\mc{M}]^k$, there is $\Omega \in \mc{D}(\Theta) \cup \{\Theta\}$ such that $(A \curvearrowleft \Aut(\mc{A})) \simeq (\mathbf{k} \curvearrowleft \Omega)$; 
        \item \label{isomorphism types} there are exactly $r$ isomorphism classes of substructures $\mc{A} \in [\mc{M}]^k$ such that $(A \curvearrowleft \Aut(\mc{A})) \simeq (\mathbf{k} \curvearrowleft \Theta)$;
        \item \label{lower transitivity} if $\Theta$ is unruly and $r > 1$, then $\mc{M}$ is transitive.
    \end{enumerate}
\end{defn}
 
\begin{remark}
    The conditions in the previous definition are rather ad hoc, having been tailored to the particular classes of examples that we consider in this paper: our objective is to obtain sharply $k$-homogeneous actions for many well-known \Fr structures while keeping technicalities to a minimum (particularly the descriptions of the seed groups $H$), so as not to obscure the core ideas. We conjecture that it should be possible to construct sharply $k$-homogeneous actions for some \Fr structures which do not satisfy the above conditions. For example, one could consider the case where some $\mc{A}\in[\mc{M}]^{k}$ has an automorphism group which is an unruly proper subgroup of $\Theta$. Likewise, one could relax \ref{lower transitivity} and attempt to construct sharply $3$-homogeneous actions on structures whose elements have multiple isomorphism types (for example, the generic red-blue-coloured graph).
\end{remark}
    
\begin{definition}
    All the notions introduced in \cref{ssec:partial actions on sets} make sense for actions on structures. In particular, if $\lambda$ is a seed action of $H$ on $M$ by automorphisms (shortly, an action on $\M$), then a taut extension of $\lambda$ will simply be a finite extension of $\lambda$ satisfying the properties of \cref{d:taut partial actions} which is also a partial action on the structure $\M$. We can also speak of free extensions and extensions on $w$-arcs.  
\end{definition}

In the below, recall the definition of taut extensions from \cref{d:taut partial actions}.

\begin{defn}
    Let $\mc{M}$ be a $(\Theta, r)$-structure. Let $s \geq 1$ and let $H$ be a $(\Theta, s)$-seed group. We say that $H$ is \emph{compatible} with $\mc{M}$ if $\Theta$ is docile and $s = 1$, or if $\Theta$ is unruly and $s = r$. 
\end{defn}

\begin{defn} \label{d:arc extensiveness}
    Let $\mc{M}$ be a $(\Theta, r)$-structure. Let $H$ be a compatible seed group and let $(G,T)$ be a completion of $H$. Let $\lambda : M \curvearrowleft H$ be a seed action of sets, where each $\lambda_h$, $h \in H$, is an automorphism of $\mc{M}$, and extend $\lambda$ trivially to $(G, T)$.

    Let $\mc{F}$ be a collection of taut finite extensions $\phi \geq \lambda$, where each $\phi$ is a partial action of $(G, T)$ on $\mc{M}$. We say that $\mc{F}$ is \emph{arc-extensive} if it is non-empty and the following hold:
    \begin{enumerate}[label=(\Roman*)]
        \item \label{condI} given $\Omega \in \mc{D}(\Theta)$, $\phi\in\mc{F}$, $t\in T_{\Omega}^{\pm 1}$, an $\Omega$-substructure $\mc{A}\in[\mc{M}]^{k}$ and $B \fin M$, there exists $\phi' \in \mc{F}$ which is an atomic $\Omega$-free extension of $\phi$ by $t$ to a substructure containing $\mc{A}$, such that the $\lambda$-orbit of $\phi'_t(A_\Omega) \setminus \im(\phi_t)$ is disjoint from $B$;
        \item \label{condII} for each $\Omega \in \mc{D}(\Theta)$ and $\phi \in \mc{F}$, there is $\tld{s} \in T_\Omega$ such that, given two $\Omega$-isomorphic $\Omega$-substructures $\mc{A}, \mc{A}' \in [\mc{M}]^k$ with $\mc{A}$ a strict $\Omega$-substructure in $\phi$ and $\mc{A}$, $\mc{A}'$ not in the same $\phi$-orbit, the following holds:
   	\begin{enumerate}[label=(\alph*),ref=(\Roman{enumi}\alph*)]
   		\item\label{II sep} there is an $\Omega$-free extension $\phi'$ of $\phi$ such that there are $\Omega$-substructures $\mc{B},\mc{B}'\in[\mc{M}]^{k}$ in the $\phi'$-orbits of $\mc{A}$, $\mc{A}'$ with $\mc{B}$ a strict $\Omega$-substructure in $\phi'$ and $B\cap B'=\NFr_{\lambda}(\Omega)$, where $\mc{B}, \mc{B}'$ are not in the same $\phi'$-orbit,
   		\item\label{II arc} in addition, there is $s \in T_{\Omega}$ and $N>0$ such that for every word $w\in\Wr(T_{\Omega})$ with the properties:
        \begin{itemize}
            \item $w = v \cdot \tld{s} \cdot (v')^{-1}$, where $v, v' \in \Wr(T_\Omega)$ are positive words beginning with $s$, ending with a letter distinct from $\tld{s}$ and satisfying $\sylc(v) = \sylc(v')$,
            \item $|w| \geq N$,
            \item every subword $w'$ of $w$ with $|w'| \geq \lfloor \frac{|w|}{N} \rfloor$ has $\dom(\phi'_{w'})=\NFr_{\lambda}(\Omega)$,
            \item $w$ has cancellation $< \lfloor \frac{|w|}{N} \rfloor$, 
        \end{itemize}
          there is some $\Omega$-extension $\phi''$ of $\phi'$ by $w$-arcs from some enumeration of $B$ to some enumeration of $B'$, and furthermore an $\Omega$-free extension $\psi$ of $\phi''$ with $\psi \in \mc{F}$.
   	\end{enumerate}	 
   \end{enumerate}
\end{defn}
    
\begin{defn} \label{d:seed action on structure}
    Let $\M$ be a $(\Theta,r)$-structure. Let $H$ be a compatible seed group. Let $\lambda : M \curvearrowleft H$ be a seed action of sets. We say that $\lambda$ is a \emph{pleasant structural seed action on $\mc{M}$} if the following hold:
    \begin{enumerate}[label=(\roman*)]
        \item \label{seed-struc auto}$\lambda$ is an action by automorphisms of $\M$;
        \item \label{seed-struc reference} if $\Theta$ is unruly, then letting $A_0, \cdots, A_{r-1}$ be the family of reference sets for the $\lambda$-action (see \cref{l: key seed action props}) and $\mc{A}_i = \mc{M}|_{A_i}$ for $i < r$, we have:
        \begin{itemize}
            \item $(A_i \curvearrowleft \Aut(\mc{A}_i)) \simeq (\pi : \mathbf{k} \curvearrowleft \Theta)$ for each $i < r$,
            \item for each $\mc{A} \in [\mc{M}]^k$ with $(A \curvearrowleft \Aut(\mc{A})) \simeq (\pi : \mathbf{k} \curvearrowleft \Theta)$ we have $\mc{A} \cong \mc{A}_i$ for some $i < r$;
        \end{itemize}
        \item \label{seed-struc all permutations} for each $\mc{A} \in [\mc{M}]^k$ such that there is $\Omega \in \mc{D}(\Theta)$ with $(A \curvearrowleft \Aut(\mc{A})) \simeq (\pi : \mathbf{k} \curvearrowleft \Omega)$, there is some $\Omega$-invariant $\mc{A}' \in [\mc{M}]^k$ with $\mc{A} \cong \mc{A}'$;
        \item \label{seed-struc equivariant isomorphism condition} given $\Omega \in \mc{D}(\Theta)$ and $\Omega$-substructures $\mc{A}, \mc{A}' \in [\mc{M}]^k$ with $\mc{A} \cong \mc{A}'$, there exists $\rho \in \Norm_\Theta(\Omega)$ (the normaliser of $\Omega$ in $\Theta$) such that the structures $\lambda_{\rho}(\mc{A})$, $\mc{A}'$ are $\Omega$-isomorphic;
        \item \label{seed-struc arc extension} for each completion $(G, T)$ of $H$, there is an arc-extensive collection $\mc{F}$ of taut finite extensions of $\lambda$. 
    \end{enumerate}
    If it is possible to take $\mc{F}$ to be the collection of all finite taut extensions of $\lambda$ to $G$ (where we assume these extensions are partial actions on $\mc{M}$), then we say that $\lambda$ is \emph{generous}. 
\end{defn} 
  
Of all the conditions listed in the previous definition, \ref{seed-struc equivariant isomorphism condition} is perhaps the least transparent. Although establishing this condition in some cases will require certain care in the construction of $\lambda$, we note in the below lemma that the condition comes for free in some important cases. 

\begin{lem} \label{l:condition omega tuples comes for free} 
        Let $\Theta \in \{\sym_2, \sym_3\}$ and let $\mc{M}$ be a $(\Theta, r)$-structure for $r \geq 1$. Let $H$ be a compatible seed group and $\lambda : M \curvearrowleft H$ a seed action on sets. Suppose that $\lambda$ satisfies all conditions of \cref{d:seed action on structure} except condition \ref{seed-struc equivariant isomorphism condition}. Then $\lambda$ also satisfies \ref{seed-struc equivariant isomorphism condition}, and so is a pleasant structural seed action on $\mc{M}$.
\end{lem}
\begin{proof}
    Let $\Omega \in \mc{D}(\Theta)$ and let $\mc{A}, \mc{A}' \in [\mc{M}]^k$ be isomorphic $\Omega$-substructures. The case $\Omega = \mathtt{1}$ is trivial, so assume $\Omega \neq \mathtt{1}$. As the action $\lambda : A \curvearrowleft \Omega$ is faithful (via \cref{l: key seed action props}(c)) and $\lambda$ acts by automorphisms, we have that $\Omega$ is isomorphic to a subgroup of $\Aut(\mc{A})$. Thus we have $\Aut(\mc{A}) \neq \mathtt{1}$.
    
    As $\mc{M}$ is a $(\Theta, r)$-structure, we have $(A \curvearrowleft \Aut(\mc{A})) \simeq (\mathbf{k} \curvearrowleft \Delta)$ for some $\Delta \in \mc{D}(\Theta) \cup \{\Theta\}$. In the case $\Delta = \Theta$, then as $\Theta = \sym_k$, any bijection $A \to A'$ is an isomorphism $\mc{A} \to \mc{A}'$. By \cref{l: key seed action props}(c), we have $(A \curvearrowleft \Omega) \simeq (\mathbf{k} \curvearrowleft \Omega) \simeq (A' \curvearrowleft \Omega)$, so there is an $\Omega$-equivariant bijection $A \to A'$ which then is necessarily an $\Omega$-equivariant isomorphism $\mc{A} \to \mc{A}'$ as required. If $\Theta = \sym_2$, then as $\Delta \neq \mathtt{1}$ we have $\Delta = \Theta$, so we are done. It remains to consider the case $\Theta = \sym_3$ and $\Delta \in \mc{D}(\Theta)$.
    
    We have $\mc{D}(\Theta)=\{\mathtt{1},\sg{\sigma},\sg{\tau}\}$, where $\sigma$ has order $2$ and $\tau$ order $3$.  We have two cases:
    \begin{itemize}
        \item if $\Delta = \sg{\sigma}$, then $\Aut(\mc{A})$, $\Aut(\mc{A}')$ both have only a single nontrivial element, and so any isomorphism $\mc{A} \to \mc{A}'$ is automatically $\Omega$-equivariant;
        \item if $\Delta = \sg{\tau}$, then $\Omega \cong \sg{\tau}$. Let $\omega \in \Omega$, $\omega \neq 1$. Let $f : \mc{A} \to \mc{A}'$ be an isomorphism. If $f$ is not $\Omega$-equivariant, then we must have $f \circ \lambda_\omega|_A = \lambda_{\omega^{-1}}|_{A'} \circ f$. As $\Theta = \sym_3$, we have $\omega \sigma = \sigma \omega^{-1}$. So $\sigma \in \Norm_\Theta(\Omega)$ and $\lambda_\sigma(\mc{A})$ is an $\Omega$-substructure, and we have \[f \circ \lambda_\sigma|_{\lambda_\sigma(A)} \circ \lambda_{\omega^{-1}}|_{\lambda_\sigma(A)} = f \circ \lambda_\omega|_A \circ \lambda_\sigma|_{\lambda_\sigma(A)} = \lambda_{\omega^{-1}}|_{A'} \circ f \circ \lambda_\sigma|_{\lambda_\sigma(A)}.\] So $f \circ \lambda_\sigma|_{\lambda_\sigma(A)}$ is an $\Omega$-equivariant isomorphism $\lambda_\sigma(\mc{A}) \to \mc{A}'$. 
    \end{itemize}
\end{proof}
    
\begin{defn} \label{d:space of actions on structures}
    Let $\mc{M}$ be a $(\Theta, r)$-structure, $H$ a compatible seed group, $(G, T)$ a completion of $H$ and $\lambda$ a pleasant structural seed action on $\mc{M}$ with arc-extensive family $\mc{F}$.
    
    We write $\Act(\mc{M}, G)$ for the subspace of $\Act(M, G)$ consisting of actions by automorphisms of $\mc{M}$. It is straightforward to show that $\Act(\mc{M}, G)$ is a closed subset of $\Act(M, G)$, and thus is a Polish space. We define $\Act_\lambda(\mc{M}, G) = \Act_\lambda(M, G) \cap \Act(\mc{M}, G)$, $\Ta_\lambda(\mc{M}, G) = \Ta_\lambda(M, G) \cap \Act(\mc{M}, G)$, which are likewise Polish spaces. We also define $\SHTa_{k, \lambda}(\mc{M}, G)$ to be the subspace of $\Ta_\lambda(\mc{M}, G)$ consisting of the actions which are sharply $k$-homogeneous.

    We also write $\TaF$ for the subspace of $\Ta_\lambda(\mc{M}, G)$ consisting of actions $\mu$ such that
    for every finite extension $\phi \geq \lambda$ with $\phi<\mu$, there is $\psi \in \mc{F}$ with $\phi \leq \psi < \mu$, and we let $\SHTaF = \SHTa_{k, \lambda}(\mc{M}, G) \cap \TaF$.

    For $\phi \in \mc{F}$, let $\mc{V}'_\phi = \widetilde{\mc{V}}^{}_\phi \cap \TaF$. By the definition of $\TaF$, the $\mc{V}'_\phi$, $\phi \in \mc{F}$, give a base of $\TaF$.
\end{defn}
     
\begin{prop} \label{p:sharply k-homog actions on structures}
    The space $\TaF$ is a non-empty perfect Polish space (hence has cardinality $2^{\aleph_0}$), and the subspace $\SHTaF$ a dense $\text{G}_{\delta}$ subset of $\TaF$. 
\end{prop}

\begin{proof}
    The proof is an adaptation of that of \cref{p:generic action on sets}.
    
    As $\mc{F}$ is non-empty, by condition \ref{condI} in \cref{d:arc extensiveness} and a back-and-forth argument we have that $\TaF$ is non-empty. Using condition \ref{condI}, an argument analogous to that in the proof of \cref{p:generic action on sets} shows that $\TaF$ is perfect.

    Let $q \leq \aleph_0$ be such that there are exactly $q$ isomorphism classes of $k$-substructures of $\mc{M}$. Let $\{\mc{A}_i \mid i < q, \mc{A}_i \in [\mc{M}]^k\}$ be a system of representatives of these isomorphism classes, where if $\Theta$ is unruly we take $\mc{A}_0, \cdots, \mc{A}_{r-1}$ to be the substructures of $\mc{M}$ induced on the reference sets $A_0, \cdots, A_{r-1}$ of $\lambda$. If $\Theta$ is docile let $I = \mathbf{q}$, and if $\Theta$ is unruly let $I = \{i \in \N \mid r \leq i < q\}$. Using property \ref{seed-struc all permutations} of the definition of a pleasant structural seed action (\cref{d:seed action on structure}), we may assume for each $i \in I$ that $\mc{A}_i$ is $\Omega^{(i)}$-invariant, where $\Omega^{(i)} \in \mc{D}(\Theta)$ is the subgroup of $\Theta$ with $(A_i \curvearrowleft \Aut(\mc{A}_i)) \simeq (\mathbf{k} \curvearrowleft \Omega^{(i)})$. For each $i \in I$, as $(A_i \curvearrowleft \Omega^{(i)}) \simeq (\mathbf{k} \curvearrowleft \Omega^{(i)})$, we have $(A_i \curvearrowleft \Omega^{(i)}) \simeq (A_i \curvearrowleft \Aut(\mc{A}_i))$. Likewise for $i < q$ with $i \notin I$ we have $(A_i \curvearrowleft \Theta_i) \simeq (A_i \curvearrowleft \Aut(\mc{A}_i))$.
    
    For each $\mc{C} \in [\mc{M}]^k$ let $f(\mc{C})\in\mathbf{q}$ be the unique index such that the structures $\mc{C}, \mc{A}_{f(\mc{C})}$ are isomorphic, and let $\mc{W}_{\mc{C}}$ be the subset of $\TaF$ consisting of those actions for which $\mc{A}_{f(\mc{C})}$, $\mc{C}$ are in the same orbit. It is straightforward to see that each $\mc{W}_{\mc{C}}$ is open and $\SHTaF = \bigcap_{\mc{C}\in[\mc{M}]^{k}}\mc{W}_{\mc{C}}$. So $\SHTaF$ is a $\text{G}_\delta$ set.

    We now show that $\SHTaF$ is dense in $\TaF$. By the Baire category theorem, it suffices to show that each $\mc{W}_{\mc{C}}$ is dense in $\TaF$.

    Let $\mc{C} \in [\mc{M}]^k$. Let $\mu \in \TaF$ and let $\mc{V}'_\phi$ be an open set containing $\mu$ (here $\phi \in \mc{F}$). If $\Theta$ is unruly and there is $i < r$ and a $\phi$-strict $\Theta_i$-substructure in the $\phi$-orbit of $\mc{C}$, then this substructure must be $\mc{A}_i$ by \cref{l: key seed action props}\ref{seedact unruly unique}, so $i = f(\mc{C})$ and $\mc{V}'_\phi \sub \mc{W}_{\mc{C}}$. Otherwise, by \cref{l:stabilisers of k sets}, there is $\Omega \in \mc{D}(\Theta)$ and a $\phi$-strict $\Omega$-substructure $\mc{C}'$ in the $\phi$-orbit of $\mc{C}$. By property \ref{seed-struc equivariant isomorphism condition} of the definition of a pleasant structural seed action, there is $\rho \in \Norm_\Theta(\Omega)$ such that the $\Omega$-substructures $\lambda_\rho(\mc{C}')$, $\mc{A}_{f(\mc{C})}$ are $\Omega$-isomorphic. Let $\mc{C}'' = \lambda_\rho(\mc{C}')$. As $\mc{C}'$ is a $\phi$-strict $\Omega$-substructure and $\rho \in \Norm_\Theta(\Omega)$, we have that $\mc{C}''$ is a $\phi$-strict $\Omega$-substructure. If $\mc{C}''$, $\mc{A}_{f(\mc{C})}$ are in the same $\phi$-orbit then we are done, so assume not. Let $\tld{s}$, $\phi'$, $\mc{B}$ and $\mc{B}'$ be given by condition \ref{condII} applied to $\Omega$, $\phi$, $\mc{C}''$, $\mc{A}_{f(\mc{C})}$, and let $s$ and $N$ be given by condition \ref{II arc}. As $\phi$ is taut and $\phi'$ is an $\Omega$-free extension of $\phi$, we have that $\phi'$ is taut by \Cref{l:partial k-sharp extend}. By \Cref{l:disjointness through free extensions}\ref{positive word} there is a positive word $u \in \Wr(T_\Omega)$ with $\dom(\phi'_u) = \NFr_\lambda(\Omega)$. Let $t \in T_\Omega \setminus \{\tld{s}\}$. Apply \Cref{l:existence small cancellation} with 
    $s, \tld{s}, t$, the positive word $u$ and $N$ to obtain a word $w$ satisfying the conditions in bulletpoints in \ref{II arc}; by \ref{II arc} we then have $\psi \in \mc{F}$ with $\psi \geq \phi$ and with $\mc{C}''$, $\mc{A}_{f(\mc{C})}$ in the same $\psi$-orbit. By condition \ref{condI} in \cref{d:arc extensiveness} the extension $\psi$ extends to an element of $\TaF$, so $\mc{W}_{\mc{C}} \cap \mc{V}'_\phi \neq \varnothing$ as required.
\end{proof}
            
\section{SWIRs: stationary weak independence relations} 
\label{s:structures and independence}

\begin{notn}
    In the below, we write $AB$ to mean $A \cup B$, and $\langle A \rangle$ for the substructure generated by a set $A$.
\end{notn}

\begin{defn}
    Let $\mc{M}$ be a \Fr structure (where the language may include function and constant symbols), and let $A, B, B' \fg \mc{M}$. We write $B \equiv_A B'$ if there exists $f \in \Aut(M)$ with $f(B) = B'$ and $f|_A = \id_A$, or equivalently if (some finite generating sets of) $B, B'$ have the same quantifier-free type over (some finite generating set of) $A$ in some enumeration. (The two definitions are equivalent as $M$ is ultrahomogeneous.)

    We will be a little flexible in our notation. If a particular isomorphism is clear from context (for example when there is a given isomorphism taking $B$ to $B'$) then we use this isomorphism. For instance, in the below stationarity axiom, when we write $B \equiv_A B' \Rightarrow B \equiv_{\langle AC \rangle} B'$, both statements refer to the same isomorphism from $B$ to $B'$.
\end{defn}

\begin{defn} \label{d:swir}
    Let $\mc{M}$ be a \Fr structure (in any first-order language, which may also include constant and function symbols). A \emph{stationary weak independence relation (SWIR)} on $\mc{M}$ consists of a ternary relation $\ind$ on the set of finitely generated substructures of $\mc{M}$, which we write as $B \ind_A C$, satisfying the following properties:
    \begin{itemize}
        \item Invariance (Inv): 
        \[\text{for }g \in \Aut(\mc{M}), B \ind_A C \Rightarrow gB \ind_{gA} gC;\]
        \item Existence (Ex): for all $A, B, C \fg \mc{M}$,
        \medskip
        \begin{itemize}
            \item (LEx): there exists $B' \fg \mc{M}$ such that $B' \ind_A C$ and $B \equiv_A B'$;
            \item (REx): there exists $C' \fg \mc{M}$ such that $B \ind_A C'$ and $C \equiv_A C'$;
        \end{itemize}
        \medskip
        \item Stationarity (Sta): for all $A, B, C \fg \mc{M}$,
        \medskip
        \begin{itemize}
            \item (LSta): if $B \ind_A C \wedge B' \ind_A C$ and $B \equiv_A B'$, then $B \equiv_{\langle AC \rangle} B'$;
            \item (RSta): if $B \ind_A C \wedge B \ind_A C'$ and $C \equiv_A C'$, then $C \equiv_{\langle AB \rangle} C'$;
        \end{itemize}
        \medskip
        \item Monotonicity (Mon):
        \begin{itemize}
            \item (LMon): $\langle BD \rangle \ind_A C \Rightarrow B \ind_A C \,\wedge\, D \ind_{\langle AB \rangle} C$;
            \item (RMon): $B \ind_{A} \langle CD \rangle \Rightarrow B \ind_A C \,\wedge\, B \ind_{\langle AC \rangle} D$;
        \end{itemize}
        \item Transitivity (Tr):
        \begin{itemize}
            \item (LTr): $B \ind_A C \,\wedge\, B \ind_{\langle AC \rangle} D \Rightarrow B \ind_A D$;
            \item (RTr): $B \ind_A C \,\wedge\, D \ind_{\langle AB \rangle} C \Rightarrow D \ind_A C$.
        \end{itemize}
    \end{itemize}
\end{defn}

\begin{rem}
    This is not a minimal list of properties: 
    \begin{itemize}
        \item (REx) follows from (LEx) and (Inv) (\cite[Remark 3.1.2(iii)]{Li21});
        \item (Tr) follows from the other four properties (\cite[Remark 3.1.2(ii)]{Li21});
        \item (Mon) follows from the other four properties (see \cite[Lemma 2.10]{KSW26}, which is a straightforward adaptation of \cite[Lemma 3.2]{Bau16}).
    \end{itemize}

     We therefore see that when checking that a certain relation $\ind$ is a SWIR, we do not need to check both (Mon) and (Tr); it suffices to check one of the two.
\end{rem}

\begin{rem} \label{r: SWIR notn conv for substructures}
    In fact, we should more properly write $\mc{A}, \mc{B}, \mc{C} \fg \mc{M}$ and $\mc{B} \ind_{\mc{A}} \mc{C}$ rather than $B \ind_A C$, given that $\ind$ is defined on the set of finitely generated substructures of $\mc{M}$. We do not do this, for the sake of avoiding ponderous typography. Also note that the substructures of $M$ are in one-to-one correspondence with their domains, so this notation does not lead to any confusion.
\end{rem}

\begin{defn}
    Let $\mc{M}$ be a \Fr structure, and let $\ind$ be a SWIR on $\mc{M}$. If in addition $\ind$ is symmetric, i.e.\ for all $A, B, C \fg \mc{M}$ we have $B \ind_A C \Leftrightarrow C \ind_A B$, then we call $\ind$ a \emph{stationary independence relation} (SIR).
\end{defn}

See \Cref{intro: swir background} for the historical background here: the symmetric version SIR was first defined by Tent and Ziegler (\cite{TZ13}), and the definition of SWIR without the assumption of symmetry was later given in \cite{Li21}.

\begin{eg} \label{ex: structures with SWIRS}
We now give several key examples of \Fr structures with SWIRs. In each example, the SWIR is given by a certain canonical notion of amalgamation. In fact, the existence of a SWIR for a \Fr structure $\mc{M}$ is equivalent to the existence of a \emph{standard amalgamation operator} $\otimes$ on $\Age(\mc{M})$: see \Cref{d: standard amalg op} and \cite[Prop.\ 2.18, Prop.\ 2.19]{KSW26}.

\begin{itemize}
    \item A relational \Fr structure $\mc{M}$ with free amalgamation has a SIR: for $A, B, C \fin \mc{M}$, define $B \ind_A C$ if $BA, CA$ are freely amalgamated over $A$. Examples of such $\mc{M}$ with free amalgamation include the random graph, the random $K_n$-free graph, the random $n$-hypergraph, the random oriented graph and the Henson digraphs.
    \item The random poset has a SIR: define $B \ind_A C$ if $ABC$ is the transitive closure of $BA$ and $CA$ (that is, for $b \in B \setminus A$ and $c \in C \setminus A$ such that there does not exist $a \in A$ with $b < a < c$ or $b > a > c$, we have that $b, c$ are incomparable).
    \item $\Q$ has a SWIR: define $B \ind_A C$ if for $b \in B \setminus A$ and $c \in C \setminus A$ such that there does not exist $a \in A$ with $b < a < c$ or $b > a > c$, we have $b < c$.
    \item The random tournament has a SWIR: define $B \ind_A C$ if for $b \in B \setminus A$ and $c \in C \setminus A$ we have $b \ra c$.
    \item Let $n \geq 3$. The same definition of $\ind$ as for the random tournament also gives a SWIR for the \Fr limit of the class of finite oriented graphs omitting $n$-anticliques.
    \item The \Fr limit of the free superposition of a free amalgamation class with the class of finite linear orders has a SWIR. So, for example, the random ordered graph has a SWIR. More generally, if $\mc{M}_0$, $\mc{M}_1$ are the \Fr limits of relational strong amalgamation classes $\mc{K}_0$, $\mc{K}_1$ and $\mc{M}_0$, $\mc{M}_1$ have SWIRs, then the \Fr limit of the free superposition of $\mc{K}_0$, $\mc{K}_1$ has a SWIR. (This is straightforward to check. See \cite[Section 3.7]{Bod15} for the definition of free superposition.)
\end{itemize}    
\end{eg}

We now give a number of basic properties of SWIRs: see \cite{KSW26} for the proofs, which are not difficult.

\begin{lem}[{\cite[Lemma 2.8]{KSW26}}]
    Let $\mc{M}$ be a \Fr structure with SWIR $\ind$. Then the relation $\ind$ is trivial when one argument equals the base, i.e.\ for all $A, B \fg \mc{M}$:
    \begin{itemize}
        \item (LB): $A \ind_A B$;
        \item (RB): $A \ind_B B$.
    \end{itemize}
\end{lem}

\begin{lem}[{\cite[Lemma 2.12]{KSW26}}] \label{base_up}
    Let $\mc{M}$ be a \Fr structure with SWIR $\ind$. Then:
    \begin{align*}
        B \ind_A C &\Leftrightarrow \langle AB \rangle \ind_A C; \\
        B \ind_A C &\Leftrightarrow B \ind_A \langle AC \rangle.
    \end{align*}
\end{lem}

\begin{lem}[{\cite[Lemma 2.9]{KSW26}}] \label{revmon}
    Let $\mc{M}$ be a \Fr structure with SWIR $\ind$. Then
    \begin{align*}
             \langle BD \rangle \ind_A C &\Leftrightarrow B \ind_A C \wedge D \ind_{\langle AB \rangle} C; \\
             B \ind_A \langle CD \rangle &\Leftrightarrow B \ind_A C \wedge B \ind_{\langle AC \rangle} D.
        \end{align*}
\end{lem}

\begin{lem}[{\cite[Lemma 2.13]{KSW26}}] \label{SWIR strong amalg}
    Let $\mc{M}$ be a \Fr structure with SWIR $\ind$. Suppose that $\mc{M}$ has strong amalgamation. Then if $B \ind_A C$, we have that $B \setminus A$ and $C \setminus A$ are disjoint as sets.
\end{lem}

We will also need a notion of \emph{standard amalgamation operator} on $\Age(\mc{M})$, which will be equivalent to the existence of a SWIR on $\mc{M}$. We take this from \cite{KSW26} and give a very brief presentation -- see \cite{KSW26} for a fuller account (where the commutative diagrams in the properties below are explicitly given).

\begin{notn}
    Let $\mc{M}$ be a \Fr structure. Let $h_i : D_i \to E_i$, $i = 0, 1$, be embeddings of structures in $\Age(\mc{M})$. We write $h : (D_0, D_1) \to (E_0, E_1)$ for the pair $h = (h_0, h_1)$. If $D_0 = D_1 = D$, we write $h : D \to (E_0, E_1)$ and if $E_0 = E_1 = E$ we write $h : (D_0, D_1) \to E$. 

    Let $h : (D_0, D_1) \to (E_0, E_1)$, $h' : (D'_0, D'_1) \to (E'_0, E'_1)$ be pairs of embeddings. An embedding $i : h \to h'$ is defined to be a pair $(i_D, i_E)$ where $i_D : (D_0, D_1) \to (D'_0, D'_1)$, $i_E : (E_0, E_1) \to (E'_0, E'_1)$ and $i_E \circ h = h' \circ i_D$.
\end{notn}

\begin{defn}[{\cite[Definition 2.17]{KSW26}}] \label{d: standard amalg op}
    Let $\mc{M}$ be a \Fr structure. For each pair of embeddings $e : A \to (B_0, B_1)$ of structures in $\Age(\mc{M})$, a \emph{standard amalgamation operator} $\otimes$ gives an amalgam $B_0 \otimes_A B_1$ of $e = (e_0, e_1)$, such that in addition $\otimes$ satisfies:

    \begin{itemize}
        \item Minimality: the substructure of $B_0 \otimes_A B_1$ generated by the union of the images of $B_0$ and $B_1$ is $B_0 \otimes B_1$ itself;
        \item Invariance: given a pair of embeddings $e : A \to (B_0, B_1)$, $e' : A' \to (B'_0, B'_1)$ and an isomorphism $i : e \to e'$, we have $B^{}_0 \otimes_A B^{}_1 \cong B'_0 \otimes_A B'_1$ via an isomorphism that respects $i$;
        \item Transitivity: let $e : A \to (B, C)$ be a pair of embeddings. Then:
        \begin{itemize}
            \item for each embedding $b : B \to B'$ we have $B' \otimes_B (B \otimes_A C) \cong B' \otimes_A C$ via an embedding-respecting isomorphism;
            \item for each embedding $c : C \to C'$ we have $(B \otimes_A C) \otimes_C C' \cong B \otimes_A C'$ via an embedding-respecting isomorphism.
        \end{itemize}
    \end{itemize}
\end{defn}

\begin{defn}
    Let $\mc{M}$ be a \Fr structure and let $\otimes$ be a standard amalgamation operator on $\Age(\mc{M})$. If for each pair of embeddings $e_0 : A \to B_0$, $e_1 : A \to B_1$ we have $B_0 \otimes_A B_1 = B_1 \otimes_A B_0$ (also with the same embeddings from $B_0$ and $B_1$ in the amalgam), then we say that $\otimes$ is \emph{symmetric}. 
\end{defn}

\begin{lem}[{\cite[Lemma 2.20]{KSW26}}]
    Let $M$ be a \Fr structure with standard amalgamation operator $\otimes$. Then $\otimes$ has the following properties:
    \begin{itemize}
        \item Monotonicity: given embeddings $A \to (B, C)$, $B \to B'$, $C \to C'$, there is a unique embedding-respecting embedding $B \otimes_A C \to B' \otimes_A C'$;
        \item Associativity: given embeddings $A \to (B, C)$, $A' \to (C, D)$, there is a unique embedding-respecting isomorphism $(B \otimes_A C) \otimes_{A'} D \to B \otimes_A (C \otimes_{A'} D)$.
    \end{itemize}
\end{lem}

\begin{prop}[{\cite[Proposition 2.19]{KSW26}}] \label{p:swir sao}
    Let $\mc{M}$ be a \Fr structure with SWIR $\ind$. Then there exists a standard amalgamation operator $\otimes$ on $\Age(\mc{M})$ such that for $A, B, C \fg \mc{M}$, we have $B \ind_A C$ if and only if $\langle ABC \rangle \cong \langle BA \rangle \otimes_A \langle CA \rangle$ (via an embedding-respecting isomorphism). In addition, if $\ind$ is a SIR then $\otimes$ is symmetric.
\end{prop}

\begin{rem}
    The converse direction of the above also holds (see \cite[Proposition 2.18]{KSW26}): if $\Age(\mc{M})$ has a standard amalgamation operator $\otimes$, then $\mc{M}$ has a SWIR. We define $B \ind_A C$ if there exists an embedding-respecting isomorphism $\langle ABC \rangle \to \langle AB \rangle \otimes_A \langle AC \rangle$. We omit the details of the proof as we shall not use this direction.
\end{rem}

\section{The escape property} \label{s:making independent}

\begin{term*}
    Recall the definition of strong amalgamation from Subsection \ref{ss: intro main str results} in the Introduction. We say that a \Fr structure $\mc{M}$ is a \emph{strong \Fr structure} if $\Age(\mc{M})$ has strong amalgamation.    
\end{term*}

We begin by recalling the following basic fact.
\begin{lem}[{\cite[straightforward adaptation of 2.15]{Cam90}}]
    Let $\mc{N}$ be a locally finite \Fr structure (where the language may include constant and function symbols). For each finite subset $A \sub N$, let $\fcl(A)$ be the union of the finite orbits of the pointwise stabiliser of $A$ in $\Aut(\mc{N})$. We call $\fcl(A)$ the \emph{finite-orbit closure} of $A$.

    If $\mc{N}$ has strong amalgamation, then $\fcl(A) = \langle A \rangle$ for $A \fin M$. (That is, $\mc{N}$ has trivial finite-orbit closure.)
\end{lem}
\begin{rem}
    In fact the converse is also true: if $\mc{N}$ is a locally finite \Fr structure with $\fcl(A) = \langle A \rangle$ for $A \fin N$, then $\mc{N}$ has strong amalgamation. The proof is via Neumann's lemma -- see \cite[Section 2.7]{Cam90}.

    Also note that if $\mc{N}$ is $\omega$-categorical, then via the Ryll-Nardzewski theorem the finite-orbit closure coincides with the algebraic closure in the usual sense (the union of finite $A$-definable subsets). 
\end{rem}

We now introduce two definitions we often use later.

\begin{defn}
    Let $\mc{N}$ be a first-order structure. We say that $\mc{N}$ is \emph{quasi-relational} if it is locally finite and there exists an equivalence relation $E$ on $N$ such that, for all $A \sub N$, we have $\langle A \rangle = \langle \varnothing \rangle \cup \bigcup_{a \in A} E(a)$, where $E(a)$ denotes the $E$-equivalence class of $a$. (Note that this implies that $E$ is $\Aut(\mc{N})$-invariant.)
    
    Note that if $\mc{N}$ is quasi-relational, then for substructures $\mc{A}, \mc{B} \sub \mc{N}$, we have $\dom \langle AB \rangle = AB$ and $\dom \langle A \cap B \rangle = A \cap B$.

    (Note that relational structures are quasi-relational by our definition.)
\end{defn}

\begin{defn} \label{d:freedom}
    Let $\mc{N}$ be a quasi-relational strong \Fr structure with SWIR $\ind$. We say that $\ind$ satisfies the Freedom axiom (denoted (Free)) if:
    \[
    B \ind_A C \wedge A \cap BC \sub D \sub A \,\Rightarrow\, B \ind_D C.
    \]

    (This axiom was first named in \cite{Con17}, though it appears originally in \cite[Lemma 5.1]{TZ13}.)
\end{defn}

For the remainder of Section \ref{s:making independent}, we assume that $\mc{N}$ is a quasi-relational strong \Fr structure with a SWIR $\ind$, in any first-order language (possibly including function and constant symbols). The reason for considering first-order languages in general (not just relational languages) will become clear in \cref{s:ind use section}, and is related to the actions of the finite groups $\Omega \in \mc{D}(\Theta)$ (see \cref{d:seed action on structure}). 

\begin{defn}
    Let $\mc{N}$ be a quasi-relational strong \Fr structure with a SWIR $\ind$. We extend $\ind$ to all finite subsets $A, B, C \fin N$ by defining $B \ind_A C$ if $\sg{B} \ind_{\sg{A}} \sg{C}$. We also write $B \ind C$ instead of $B \ind_\varnothing C$.
\end{defn}

\begin{defn}
    Let $(F_S, S)$ be a non-abelian free group with finite generating set $S$. Let $\phi$ be a finite partial action of $(F_S, S)$ on $\mc{N}$, and let $\mc{A} \fin \mc{N}$ and $s \in S^{\pm 1}$. An \emph{atomic free extension $\psi$ of $\phi$ to $\mc{A}$ by $s$} is an extension obtained by extending $\phi_s$ to $\psi_s = \phi_s \cup \xi$, where $\xi : \mc{A} \to \mc{A}'$ is an isomorphism of substructures of $\mc{N}$, such that $A' \setminus \im(\phi_s)$ is disjoint from $\supp(\phi) \cup A$. A \emph{free extension $\psi$ of $\phi$} is an extension obtained from $\phi$ by a finite chain of atomic free extensions (with no restrictions on the extending set at each step).
\end{defn}

This section is primarily concerned with the following definition.

\begin{defn} \label{d:extensive}
    Let $(F_S, S)$ be a non-abelian free group with finite generating set $S$. Let $\mc{E}$ be a collection of finite partial actions of $(F_S, S)$ on $\mc{N}$. We say that $\mc{E}$ has the \emph{escape property} if for all $\phi \in \mc{E}$ and $\mc{C} \fin \mc{N}$ there is a free extension $\psi \in \mc{E}$ of $\phi$, some $s \in S$ and some $u\in\Wr(S)$ with $C \sub \dom(\psi_u)$ such that $\dom(\psi_{s})\ind\psi_{u}(C)$.
\end{defn}

We exploit this property in \cref{s:ind use section} to find sufficient conditions for the existence of arc-extensive families. Given the technical nature of the issues addressed here, some readers might want to skip this section on a first read and jump to \cref{s:ind use section} for additional motivation. For the most general case considered here (with no additional assumptions on $\mc{N}$ or $\ind$), we use a proof technique from \cite{KS19}: specifically, we adapt an idea from the construction of strongly repulsive automorphisms in the proof of \cite[Theorem 3.12]{KS19}). The arguments in subsections \cref{ss:free case} and \cref{ssection:Q} are, to the best of our knowledge, new.

\subsection{Independent extensions}

Recall that we assume in Section \ref{s:making independent} that $\mc{N}$ is a quasi-relational strong \Fr structure with a SWIR $\ind$.

\begin{defn} \label{d:forwards independent extension}
    Let $\rho : \mc{A} \to \mc{B}$ be a finite partial isomorphism of $\mc{N}$ and let $C, E \fin N$. We say that a finite partial isomorphism $\rho': \mc{A}' \to \mc{B}'$ extending $\rho$ is: 
    \begin{itemize}
        \item an \emph{$(E, +)$-independent extension to $C$} if $A'=\sg{CA}$ and $EA'\ind_{B}B'$,
        \item a \emph{$(-, E)$-independent extension over to $C$} if $B'=\sg{CB}$ and $A'\ind_{A}B'E$. 
    \end{itemize}
    We refer to either of these types of extension as an \emph{$E$-independent extension}, and call $C$ the \emph{extending set}. In the absence of any mention of $E$, we implicitly take $E=\varnothing$.

    Note that by ultrahomogeneity of $\mc{N}$ and (Ex), an $E$-independent extension of each type always exists.
\end{defn}
   
\begin{defn} \label{d:forwards independent extension word}
    Let $\phi$ be a finite partial action of $(F_S, S)$ on $\mc{N}$. Let $w = s_1 \cdots s_m \in \Wr(S)$ be positive and let $C, E \fin M$.
 	
    An \emph{$(E, +)$-independent extension of $\phi$ to $C$ by $w$} is a finite extension $\psi \geq \phi$ for which there exists a chain of extensions $\phi = \phi^0 \leq \cdots \leq \phi^m = \psi$ such that for $1 \leq j \leq m$, the extension $\phi^j$ is obtained from $\phi^{j-1}$ by taking $\phi^j_{s_j}$ to be a $(\supp(\phi^{j-1}) \cup E, +)$-independent extension of $\phi^{j-1}_{s_j}$ to $\phi^{j-1}_{s_{1} s_{2} \cdots s_{j-1}}(C)$. (Note that here when $j=1$ we extend to $C$).
    	
    A finite partial isomorphism $\rho$ of $\mc{N}$ can be regarded as a finite partial action $\phi$ of $(\Z,s)$ on $\M$ with $\phi_s = \rho$, and given an $(E, +)$-independent extension $\psi \geq \phi$ by $s^m$ to $C$ we call $\rho'=\psi_s$ an $m$-iterated $(E, +)$-independent extension of the map $\rho$.
\end{defn}

\begin{lem} \label{l:independent description} \hfill
    \begin{enumerate}[label=(\roman*)]
        \item Let $\rho: \mc{A} \to \mc{B}$ be a finite partial isomorphism of $\mc{N}$ and let $E, C \fin N$. Let $\rho': \mc{A}' \to \mc{B}'$ be an $(E, +)$-independent extension of $\rho$ to $C$ (so $A' = \sg{CA}$). Then $\rho'(A' \setminus A)$ is disjoint from $EA'$. An analogous statement holds for $(-, E)$-independent extensions.
        \item Let $\phi$ be a finite partial action of $(F_S, S)$ on $\mc{N}$. Let $w = s_1 \cdots s_m \in \Wr(S)$ be positive and $C, E \fin N$. Let $\psi$ be an $(E, +)$-independent extension of $\phi$ to $C$ by $w$. Suppose that $C \cap \dom(\phi_{s_1}) = \sg{\varnothing}$. Then $\supp(\psi)$ is the disjoint union of $\supp(\phi)$ and $m$ disjoint copies $D_1, D_2, \cdots, D_m$ of the set $C \setminus \sg{\varnothing}$, and for $1 \leq i \leq m$ the map $\psi_{s_i}$ sends $D_{i-1}$ to $D_i$.
    \end{enumerate}      
\end{lem}
\begin{proof}
    (i): As $EA' \ind_{B} B'$, by \cref{SWIR strong amalg} we have $EA' \cap B' \sub B$. As $\rho'$ is a bijection, we have $\rho'(A' \setminus A) \sub B' \setminus B$, so $\rho'(A' \setminus A) \cap EA' = \varnothing$. 
    
    (ii): immediate from (i).
\end{proof}

\begin{lem} \label{l: concatenations}
    Let $\rho: \mc{A} \to \mc{B}$ be a finite partial isomorphism of $\mc{N}$ and $E \fin N$. Let $\mc{C} \fin \mc{N}$, and let $\sigma$ be an $m$-iterated $(E, +)$-independent extension of $\rho$ to $C$. Then $E \cup \supp(\rho) \ind_B \im(\sigma)$. In particular, by (Mon) we have $A \ind_B \sigma^m(C)$.
\end{lem}
\begin{proof}
    Let $\rho = \rho_0 \sub \cdots \sub \rho_m = \sigma$ be the chain of iterated extensions. By definition, for $1 \leq j \leq m$ we have $E \cup \supp(\rho_{j-1}) \cup \dom(\rho_j) \ind_{\im(\rho_{j-1})} \im(\rho_j)$. By (Mon) we have $E \cup \supp(\rho) \ind_{\im(\rho_{j-1})} \im(\rho_j)$ for $1 \leq j \leq m$. The result follows by (Tr) (and, formally speaking, by induction).
\end{proof}
 
\subsection{The general case} \label{ssection:general case}

We now produce a collection of finite partial actions on $\mc{N}$ with the escape property in the general case, without any additional assumptions on $\mc{N}$. To do this, we avail ourselves of ideas from \cite{KS19} which we present in a somewhat different language. See Subsection \ref{ss:free case} for the case where we additionally assume that $\ind$ satisfies (Free), and see Subsection \ref{ssection:Q} for the case of the structure $(\Q, <)$. In these two latter cases, we obtain stronger results than in this subsection, enabling us to prove correspondingly stronger results (showing genericity in addition to existence) in Section \ref{s:ind use section}: see Corollaries \ref{c:extension by arcs structure} and \ref{c: pleasant ssa SWIR rigid}.
    
\begin{defn} \label{d:chain-independent} 
    Let $\rho$, $\sigma$ be finite partial isomorphisms of $\mc{N}$ with $\rho \sub \sigma$. We say that $\sigma$ is \emph{chain-independent} over $\rho$ if there is $m \geq 0$ and a chain $\rho = \rho_0 \subsetneq \cdots \subsetneq \rho_m = \sigma$ of finite partial isomorphisms such that for $1 \leq i \leq m$, we have that $\rho_i$ is a \pl-independent or \mi-independent extension of $\rho_{i-1}$. We also say that $\sigma$ is a \emph{\pl-chain-independent extension} of $\rho$ in the particular case that each $\rho_i$ is a \pl-independent extension for $1 \leq i \leq m$, and likewise for \mi-independent extensions. If $\rho=\id_{\sg{\varnothing}}$, then we say that $\sigma$ is a \emph{chain-independent partial isomorphism}.
    
    We say that $\sigma$ is a \emph{full} chain-independent extension of $\rho$ if for $1 \leq i \leq m$ we have:
    \begin{itemize}
        \item if $\rho_i$ is a \pl-independent extension of $\rho_{i-1}$, then $\im(\rho_{i-1})$ is contained in the extending set of $\rho_i$;
        \item if $\rho_i$ is a \mi-independent extension of $\rho_{i-1}$, then $\dom(\rho_{i-1})$ is contained in the extending set of $\rho_i$.
    \end{itemize}

    Note that this implies that for $\rho_i$ a \pl-independent extension we have $\supp(\rho_{i-1}) \sub \dom(\rho_i)$, and for $\rho_i$ a \mi-independent extension we have $\supp(\rho_{i-1}) \sub \im(\rho_i)$. In particular $\supp(\rho_0) \subsetneq \cdots \subsetneq \supp(\rho_m)$.
\end{defn}

\begin{notn}
    For $J \sub \Z$ and $n \in \Z$ we write $J + n = \{j+n \mid j\in J\}$.
\end{notn}
       
\begin{defn} \label{d: indexing by intervals}
    Let $\sigma$ be a full chain-independent partial isomorphism of $\mc{N}$, as witnessed by some sequence $(\rho_i)_{0 \leq i \leq m}$ (see \cref{d:chain-independent}). We define integer intervals $J_i \sub \Z$ for $0 \leq i \leq m$, a \emph{height function} $\chi : \supp(\sigma) \to \Z$ and sets $C_{J_i}$ for $0 \leq i \leq m-1$ which are domains of finite substructures of $\mc{N}$ as follows:
    \begin{itemize}
        \item $J_0=\{0\}$ and $\chi(a) = 0$ for $a \in \sg{\varnothing}$;                            
        \item for $1 \leq i \leq m$, with $J_{i-1}$ already given and $\chi$ already defined on $\supp(\rho_{i-1})$, we define:
        \begin{itemize}
            \item if $\rho_i$ is a \pl-independent extension of $\rho_{i-1}$ we let $C_{J_{i-1}}= \dom(\rho_i)$ and $J_i = (J_{i-1}+1) \cup J_{i-1}$, and we define

            \begin{tabular}{p{5cm} p{5cm}}
                $\chi(a) = 0$ & for $a \in \dom(\rho_i) \setminus \supp(\rho_{i-1})$, \\
                $\chi(\rho_i(a)) = \chi(a) + 1$ & for $a \in \dom(\rho_i) \setminus \dom(\rho_{i-1})$,
            \end{tabular}

            \item if $\rho_i$ is a \mi-independent extension of $\rho_{i-1}$ we let $C_{J_{i-1}}=\im(\rho_i)$ and $J_i = (J_{i-1}-1) \cup J_{i-1}$, and we define

            \begin{tabular}{p{5cm} p{5cm}}
                $\chi(a) = 0$ & for $a \in \im(\rho_i) \setminus \supp(\rho_{i-1})$,\\
                $\chi(\rho_i^{-1}(a)) = \chi(a) - 1$ & for $a \in \im(\rho_i) \setminus \im(\rho_{i-1})$.
            \end{tabular}

        \end{itemize}
    \end{itemize}
    Observe that the above intervals $J_i$, height function $\chi$ and sets $C_{J_i}$ depend on the particular sequence $(\rho_i)_{0 \leq i \leq m}$ we take. Note that $J_0 \sub \cdots \sub J_m$ and that $J_i \sub [-i, i]$ and $|J_i| = i + 1$ for $0 \leq i \leq m$. For $1 \leq i \leq m$, we have $J_i = \chi(\supp(\rho_i))$ and $\chi(C_{J_{i-1}}) = J_{i-1}$.
    
    For $0 \leq i \leq m-1$ and $n \in \Z$ such that $C_{J_i} \sub \dom(\sigma^n)$, we define $C_{J_i + n} = \sigma^n(C_{J_i})$. We also define $C_\varnothing = \sg{\varnothing}$.  
\end{defn}

\begin{lem}
    Let $\sigma$ be a full chain-independent finite partial isomorphism of $\mc{N}$, with the notation of the previous definition. Then:
    \begin{enumerate}[label=(\roman*)]
        \item for each proper subinterval $J \subsetneq J_m$, the set $C_J$ is defined;
        \item for $J \sub J' \subsetneq J_m$ we have $C_J \sub C_{J'}$.
    \end{enumerate}
\end{lem}
\begin{proof}
    (i): Let $J \subsetneq J_m$ be a proper subinterval. We have $J = J_{i-1} + n$ for some $1 \leq i \leq m$ and $n \in \Z$. The case $n=0$ is by definition. Suppose that $n > 0$. Then by the definition of $J_i, \cdots, J_m$ we have that at least $n$ extensions from $\rho_i, \cdots, \rho_m$ are \pl-independent extensions. Let $\rho_{i'}$ be the first \pl-independent extension in this sequence. Then as $C_{J_{i-1}} \sub \dom(\rho_{i'})$ we therefore have $C_{J_{i-1}} \sub \dom(\sigma^n)$. The argument is analogous for $n < 0$ but with \mi-independent extensions.

    (ii): Let $1 \leq i \leq m$ and let $J'' \sub J_{i-1}$ be a subinterval. Then $J'' = J_{l-1} + n$ for some $1 \leq l \leq i$. Suppose that $n > 0$. Arguing as in (i), we have that at least $n$ extensions from $\rho_l, \cdots, \rho_{i-1}$ are \pl-independent extensions. So $C_{J''} \sub \im(\rho_{i-1}) \sub C_{J_{i-1}}$. The case $n < 0$ is analogous but with \mi-independent extensions, and the case $n = 0$ follows immediately by definition. The general claim for $J \sub J'$ follows by translating by a power of $\sigma$ if necessary.
\end{proof}
  
The following is a key part of the construction of a strongly repulsive automorphism in \cite[Theorem 3.12]{KS19}. We recall it both for the reader's convenience and also to harmonise notation.

\begin{lem} \label{l:independence intervals} 
    Let $\sigma$ be a full chain-independent finite partial isomorphism of $\mc{N}$, with the notation of \cref{d: indexing by intervals}. Let $J, J' \subsetneq J_m$ with $\min J < \min J'$. Then $C_J \ind_{C_{J \cap J'}} C_{J'}$.  
\end{lem}
\begin{proof}
    We use induction on $m$. The case $m = 0$ is trivial, and the case $m = 1$ follows immediately by the definition of independent extensions. Suppose $m \geq 2$ and that the result holds for $m-1$. By (Inv) and translating by $\sigma$ it suffices to consider the case $\max J' = \max J_m$, and by (Mon) it suffices to consider the case where in addition $\min J = \min J_m$. Let $n = \min J_m$, $l = \max J_m$. We may also assume that $J' \not\sub J$ by base triviality.

    Suppose that $\rho_m$ is a \pl-independent extension of $\rho_{m-1}$ (the case for \mi-independence is similar). So $J_{m-1} = J_m \setminus \{l\}$, and we have $J \sub J_{m-1}$ and $J' \sub J_{m-1} + 1$. Let $I = J_{m-1} \setminus \{n\}$.

    If $\rho_{m-1}$ is a \pl-independent extension of $\rho_{m-2}$, then $C_I = \sigma(C_{J_{m-2}}) = \im(\rho_{m-1})$, and if $\rho_{m-1}$ is a \mi-independent extension of $\rho_{m-2}$ then $C_I = C_{J_{m-2}} = \im(\rho_{m-1})$. As $\rho_m$ is a \pl-independent extension we have $C_{J_{m-1}} \ind_{\im(\rho_{m-1})} C_{J_{m-1} + 1}$, and hence $C_{J_{m-1}} \ind_{C_I} C_{J_{m-1} + 1}$, which implies $C_J \ind_{C_I} C_{J_{m-1} + 1}$ by (Mon). 
    
    By the induction hypothesis for $J_{m-1}$ we have $C_J \ind_{C_{J \cap I}} C_I$, so applying (Tr) to $C_J \ind_{C_{J \cap I}} C_I$ and $C_J \ind_{C_I} C_{J_{m-1} + 1}$ we have $C_J \ind_{C_{J \cap I}} C_{J_{m-1} + 1}$, and hence $C_J \ind_{C_{J \cap I}} C_{J'}$ by (Mon).
    
    Also, as $J \cap J' \sub I$ we have $J \cap J' = J \cap J' \cap I$, and by the induction hypothesis for $J_{m-1}$, applying $\sigma$ and using (Inv) we have $C_{J \cap I} \ind_{C_{J \cap J'}} C_{J'}$. As $J \cap J' \sub J \cap I$ we have $C_{J \cap I}C_{J \cap J'} = C_{J \cap I}$, and so applying (Tr) to $C_J \ind_{C_{J \cap I}} C_{J'}$ and $C_{J \cap I} \ind_{C_{J \cap J'}} C_{J'}$ we have $C_J \ind_{C_{J \cap J'}} C_{J'}$ as required.
\end{proof}
   
\begin{cor} \label{c:escape property chain-independent}
    Let $\mc{N}$ be a quasi-relational strong \Fr structure with a SWIR. Let $(F_S, S)$ be a non-abelian free group with finite generating set $S$. Let $s \in S$. Then the collection $\mc{E}$ of all finite partial actions $\phi$ such that $\phi_s$ is full chain-independent has the escape property. 
\end{cor}
\begin{proof}
    Let $\mc{C} \fin \mc{N}$ and $\phi \in \mc{E}$. Let $t \in S \setminus \{s\}$. Let $(\rho_i)_{0 \leq i \leq l}$ be a sequence of extensions witnessing the full chain-independence of $\phi_s$. Extend this sequence to define a full chain-independent partial isomorphism $\phi'_s$, witnessed by $(\rho_i)_{0 \leq i \leq m}$ with $m \geq l$, by extending $\phi_s$ so that $C \cup \dom(\phi_t) \sub \dom(\phi'_s)$ and $\phi'_s$ is a \pl-independent extension. We use the notation of \cref{d: indexing by intervals}. We have $C \cup \dom(\phi_t) \sub C_{J_{m - 1}}$. Extend $\phi'_s$ by $m$-many \pl-independent extensions to a full chain-independent extension $\psi_s$, and let $\psi \in \mc{E}$ be the partial action given by taking $\phi$ and extending $\phi_s$ to $\psi_s$. In particular we have $\psi_t = \phi_t$. As $J_{m - 1} \cap (J_{m - 1} + m) = \varnothing$, by \cref{l:independence intervals} we have $C_{J_{m-1}} \ind C_{J_{m - 1} + m}$, and so by (Mon) we have $\dom(\psi_t) \ind \psi_{s^m}(C)$.
\end{proof}

\begin{lem} \label{l:fullifying}
    Let $\rho$ be a full chain-independent finite partial isomorphism of $\mc{N}$. Then for each \pl-chain-independent extension $\rho'$ of $\rho$, there is a \pl-chain-independent extension $\sigma$ of $\rho'$ which is a full chain-independent finite partial isomorphism of $\mc{N}$.
\end{lem}


\begin{proof}
    For each \pl-chain-independent extension $\rho'$ of $\rho$, define $m_{\rho'}$ to be the least $m$ such that there is a chain $\rho = \rho_0 \subsetneq \cdots \subsetneq \rho_m = \rho'$ witnessing the chain-independence of $\rho'$. We show that for all $m \in \N$ the lemma holds for $\rho'$ with $m_{\rho'} = m$, and we do this via induction on $m$. The case $m=0$ is trivial. We now do the induction step. Let $\rho'$ be a \pl-chain-independent extension of $\rho$ with a chain $\rho = \rho_0 \subsetneq \cdots \subsetneq \rho_m = \rho'$ witnessing the chain-independence. For $1 \leq j \leq m$, let $C_j$ be the extending set of $\rho_j$ and let $C'_j = \rho_j(C_j)$, and we also let $C_{1:j} = \bigcup_{l=1}^j C_l$ and define $C'_{1:j}$ similarly.
    
    Let $A = \dom(\rho_0)$, $B = \im(\rho_0)$. By the definition of \pl-independent extensions, for $1 \leq j \leq m$ we have $AC^{}_{1:j} \ind_{BC'_{1:j-1}} BC'_{1:j}$, and hence by base triviality and (Mon) we have \[\tag{$\ast$} ABC^{}_{1:j} \ind_{BC'_{1:j-1}} BC'_{1:j} \quad\text{ for } 1 \leq j \leq m.\] Let $\rho''$ be a \pl-independent extension of $\rho'$ to $B$. Again by definition and applying base triviality and (Mon) we have \[\tag{$\ast \ast$} ABC^{}_{1:m} \ind_{BC'_{1:m}} B\rho''(B)C'_{1:m}.\]

    Let $1 \leq l \leq m$. Then from $(\ast)$ applying (Mon) we have $ABC^{}_{1:l} \ind_{BC'_{1:j-1}} BC'_{1:j}$ for $l \leq j \leq m$, and from $(\ast \ast)$ applying (Mon) we have $ABC^{}_{1:l} \ind_{BC'_{1:m}} B\rho''(B)C'_{1:m}$. So by (Tr) we have \[\tag{$\dagger$} ABC^{}_{1:l} \ind_{BC'_{1:l-1}} B\rho''(B)C'_{1:m},\] and hence by (Mon) we have $ABC^{}_{1:l} \ind_{B\rho''(B)C'_{1:l-1}} B\rho''(B)C'_{1:l}$. Thus $\rho_l \cup \rho''|_B$ is a \pl-independent extension of $\rho_{l-1} \cup \rho''|_B$.

    By $(\dagger)$ with $l = 1$ we have $ABC^{}_1 \ind_B B\rho''(B)C'_1$, so $\rho_1 \cup \rho''|_B$ is a \pl-independent extension of $\rho_0$ to $BC_1 \supseteq B$, and thus $\rho_1 \cup \rho''|_B$ is a full chain-independent finite partial isomorphism of $\mc{N}$. We have a chain of length $m-1$ of \pl-independent extensions $\rho_1 \cup \rho''|_B \subsetneq \rho_2 \cup \rho''|_B \subsetneq \cdots \subsetneq \rho_m \cup \rho''|_B = \rho''$, and so applying the induction hypothesis to the \pl-chain-independent extension $\rho''$ of $\rho_1 \cup \rho''|_B$, we obtain a \pl-chain-independent extension $\sigma$ of $\rho''$ which is a full chain-independent finite partial isomorphism of $\mc{N}$. As $\rho''$ is a \pl-independent extension of $\rho'$, we have that $\sigma$ is a \pl-chain-independent extension of $\rho'$, and thus $\sigma$ is as required.
\end{proof}
      
\subsection{Free case} \label{ss:free case}

\begin{notation}
     Let $\rho$ be a finite partial isomorphism of $\mc{N}$ and $V \sub N$. In this section we will often consider ``the image of $V$ under $\rho$" given by $\rho(V \cap \dom(\rho))$. We write $\rho(\hat{V}) = \rho(V \cap \dom(\rho))$, and avoid writing $\rho(V)$ to prevent confusion as to where $\rho$ is defined.
\end{notation}

\begin{lem} \label{l: free escape 1}
    Let $\mc{N}$ be a quasi-relational strong \Fr structure with SWIR $\ind$, and suppose that $\ind$ satisfies (Free). Let $\rho$ be a finite partial isomorphism of $\mc{N}$. Let $\mc{C}, \mc{D} \fin \mc{N}$ with $\dom(\rho) \sub D$ and $C \cap \dom(\rho) = \sg{\varnothing}$. Let $m \geq 1$, and let $\sigma$ be an $m$-iterated $(D, +)$-independent extension of $\rho$ to $C$. Then $D \ind_{\im(\rho^m)} \sigma^m(C)$.
\end{lem}
\begin{proof}
    We use induction on $m$. The case $m=1$ follows from the definition of a $(D, +)$-independent extension: we have $D \dom(\sigma) \ind_{\im(\rho)} \im(\sigma)$ and hence $D \ind_{\im(\rho)} \sigma(C)$ by (Mon). Suppose $m > 1$. Let $\sigma'$ be a \pl-independent extension of $\sigma$ to $D$. By the induction hypothesis we have $D \ind_{\im(\rho^{m-1})} \sigma^{m-1}(C)$, and so applying $\sigma'$ and using (Inv) (and the ultrahomogeneity of $\mc{N}$ to extend $\sigma'$ to an element of $\Aut(\mc{N})$) we have $\sigma'(D) \ind_{\sigma'(\im(\rho^{m-1}))} \sigma^m(C)$. As $\dom(\rho) \sub D$, by (Mon) we thus have $\im(\rho) \ind_{\sigma'(\im(\rho^{m-1}))} \sigma^m(C)$. 
    
    We now use (Free) to reduce the base of the independence relation. We have $\sigma'(\im(\rho^{m-1})) \cap \im(\rho) = \im(\rho^m)$, and as $C \cap \dom(\rho) = \sg{\varnothing}$ we have $\sigma'(\im(\rho^{m-1})) \cap \sigma^m(C) = \sg{\varnothing}$. So by (Free) we have $\im(\rho) \ind_{\im(\rho^m)} \sigma^m(C)$. By \cref{l: concatenations} we have $D \ind_{\im(\rho)} \sigma^m(C)$, so applying (Tr) we obtain $D \ind_{\im(\rho^m)} \sigma^m(C)$.
\end{proof}

\begin{lem} \label{l:freeness and independent extensions}
    Suppose that $\ind$ satisfies (Free). Let $\rho$ be a finite partial isomorphism of $\mc{N}$. Let $E \fin N$. Let $U, V \sub \supp(\rho)$ with $\dom(\rho) \ind_U V$. Let $m \geq 1$ and let $\sigma$ be an $m$-iterated $(E, +)$-independent extension of $\rho$ to $V$. Let $\hat{U} = U \cap \dom(\sigma^m)$. Then $E \cup \supp(\rho) \ind_{\sigma^m(\hat{U})} \sigma^m(V)$.
\end{lem}
\begin{proof}
    The proof is almost identical to that of the previous lemma. We use induction on $m$. Let $\sigma'$ be a \pl-independent extension of $\sigma$ to $\sigma^{m-1}(\hat{U})$ (that is, to $U$ in the case $m = 1$). We have $\dom(\rho) \ind_{\sigma^{m-1}(\hat{U})} \sigma^{m-1}(V)$: this follows for $m = 1$ by the assumption of the lemma, and for $m > 1$ by the induction assumption and (Mon). So applying $\sigma'$ and using (Inv) we have $\im(\rho) \ind_{\sigma'(\sigma^{m-1}(\hat{U}))} \sigma^m(V)$. As $\sigma'(\sigma^{m-1}(\hat{U})) \cap \im(\rho) \sub \sigma^m(\hat{U})$ and $\sigma'(\sigma^{m-1}(\hat{U})) \cap \sigma^m(V) \sub \sigma^m(\hat{U})$, by (Free) we have $\im(\rho) \ind_{\sigma^m(\hat{U})} \sigma^m(V)$. By \cref{l: concatenations} we have $E \cup \supp(\rho) \ind_{\im(\rho)} \im(\sigma)$, so by (Mon) we have $E \cup \supp(\rho) \ind_{\im(\rho) \cup \sigma^m(\hat{U})} \sigma^m(V)$, and the result follows by (Tr) applied to $\im(\rho) \ind_{\sigma^m(\hat{U})} \sigma^m(V)$ and $E \cup \supp(\rho) \ind_{\im(\rho) \cup \sigma^m(\hat{U})} \sigma^m(V)$.
\end{proof}

\begin{lem} \label{l:escape property free} 
    Let $\mc{N}$ be a quasi-relational strong \Fr structure with a SWIR $\ind$ satisfying (Free). Let $(F_S, S)$ be a non-abelian free group with finite generating set $S$. Let $\mc{E}$ be the collection of all finite partial actions $\phi$ of $(F_S, S)$ on $\mc{N}$ such that there is no $\mc{A} \fin \mc{N}$ with $\mc{A} \neq \sg{\varnothing}$ and $\phi_s(\mc{A})=\mc{A}$ for all $s \in S$. Then $\mc{E}$ has the escape property.  
\end{lem}
\begin{proof}
    First observe that for $\phi \in \mc{E}$, given a \pl-independent extension $\psi$ of $\phi$ to a finite set by a word $w \in \Wr(S)$, we have $\psi \in \mc{E}$.

    Let $\phi \in \mc{E}$ and $\mc{C} \fin \mc{N}$. Let $\mc{V}$ be the collection of all quadruples $(\mc{A}_0, u, \psi, t)$, where $\mc{A}_0 \fin \mc{N}$, $u \in \Wr(S)$ is a positive word, $\psi$ is a \pl-independent extension of $\phi$ to $C$ by $u$ and $t \in S$, satisfying the following:
    \begin{enumerate}[label=(\roman*)]
        \item \label{indep cond} $\dom(\psi_t) \ind_{A_0} \psi_u(C)$;
        \item \label{disjoint cond} $\dom(\psi_t) \cap \psi_u(C) = \sg{\varnothing}$;
        \item \label{no mapping into cond} for each $m \in \N$ and $a \in (A_0 \setminus \sg{\varnothing}) \cap \dom(\psi_{t^m})$ we have $\psi_{t^m}(a) \notin \psi_u(C)$.
    \end{enumerate}

    In the remainder of the proof, we show in a series of steps that there exists $(\mc{A}_0, u, \psi, t) \in \mc{V}$ with $\mc{A}_0 = \sg{\varnothing}$, which implies by condition \ref{indep cond} that $\mc{E}$ has the escape property.

    \textbf{Claim 1:}  $\mc{V}$ is non-empty.
    
    \begin{subproof}
        Observe that there is $s \in S$, a positive word $v \in \Wr(S)$ and a \pl-independent extension $\phi'$ of $\phi$ to $C$ by $v$ such that $\phi'_v(C) \cap \dom(\phi'_s) = \sg{\varnothing}$: this follows by an entirely analogous argument to that in the proof of \cref{l:disjointness through free extensions}, using \pl-independent extensions instead of atomic $\Omega$-free extensions. Let $C' = \phi'_v(C)$. Let $A_0 = \bigcap_{n \geq 1} \im(\phi'_{s^n})$. As $\phi'$ is finite we have $\phi'_s(A_0) = A_0$, and also we have that there exists $m > 0$ such that $A_0 = \im(\phi'_{s^m})$. Let $\psi$ be a \pl-independent extension of $\phi'$ to $C'$ by $s^m$. By \cref{l: free escape 1} we have $\supp(\phi') \ind_{A_0} \psi_{s^m}(C')$. Let $t \in S \setminus \{s\}$. As $\dom(\psi_t) = \dom(\phi'_t)$, by (Mon) we have $\dom(\psi_t) \ind_{A_0} \psi_{s^m}(C')$, so condition \ref{indep cond} holds. As $C' \cap \dom(\phi'_s) = \sg{\varnothing}$, by the definition of $\psi$ we have that conditions \ref{disjoint cond}, \ref{no mapping into cond} hold. So $\mc{V}$ is non-empty.    
    \end{subproof}

    \textbf{Claim 2:} for each $(\mc{A}_0, u, \psi, t) \in \mc{V}$ and for all $n > 0$, given a \pl-independent extension $\psi'$ of $\psi$ to $\psi_u(C)$ by $t^n$ and given $s \in S \setminus \{t\}$, we have $(\psi_{t^n}(\hat{\mc{A}}_0), ut^n, \psi', s) \in \mc{V}$.
    \begin{subproof}
        As $\dom(\psi_t) \ind_{A_0} \psi_u(C)$ (this is condition \ref{indep cond}), by \cref{l:freeness and independent extensions} we have $\supp(\psi) \ind_{\psi'_{t^n}(\hat{A}_0} \psi'_{ut^n}(C)$, and as $\dom(\psi'_s) = \dom(\psi_s)$ we therefore have $\dom(\psi'_s) \ind_{\psi'_{t^n}(\hat{A}_0)} \psi'_{ut^n}(C)$. As $(\mc{A}_0, u, \psi, t)$ satisfies condition \ref{no mapping into cond}, we have $\psi'_{t^n}(\hat{A}_0) = \psi_{t^n}(\hat{A}_0)$ (note that here we really do need the full strength of condition \ref{no mapping into cond} for all $m \in \N$), and so $\dom(\psi'_s) \ind_{\psi_{t^n}(\hat{A}_0)} \psi'_{ut^n}(C)$. Thus $(\psi_{t^n}(\hat{\mc{A}}_0), ut^n, \psi', s)$ satisfies condition \ref{indep cond}. As $(\mc{A}_0, u, \psi, t)$ satisfies conditions \ref{disjoint cond}, \ref{no mapping into cond}, it is straightforward to see that $(\psi_{t^n}(\hat{\mc{A}}_0), ut^n, \psi', s)$ also satisfies these two conditions.
    \end{subproof}

    \textbf{Claim 3:} for each quadruple $(\mc{A}_0, u, \psi, t) \in \mc{V}$ with $\mc{A}_0 \neq \sg{\varnothing}$, there is a positive word $w \in \Wr(S)$ with $A_0 \nsubseteq \dom(\psi_w)$.
    \begin{subproof}
        Suppose not. Then the set $\bigcup \{\psi_v(A_0) \mid v \in \Wr(S) \text{ positive}\}$ is $\psi_s$-invariant for all $s \in S$ and also not equal to $\sg{\varnothing}$, contradicting that $\psi \in \mc{E}$.
    \end{subproof}

    Suppose for a contradiction that for each $(\mc{A}_0, u, \psi, t) \in \mc{V}$ we have $\mc{A}_0 \neq \sg{\varnothing}$. Let $\mc{V}'$ be the subset of $\mc{V}$ consisting of the quadruples $(\mc{A}_0, u, \psi, t) \in \mc{V}$ with $|A_0|$ minimal. Let $\mc{X}$ be the set consisting of the tuples $(\mc{A}_0, u, \psi, t, w)$ where $(\mc{A}_0, u, \psi, t) \in \mc{V}'$ and where $w \in \Wr(S)$ is a positive word with $A_0 \nsubseteq \dom(\psi_w)$ (note that necessarily $w$ is non-empty). Take $(\mc{A}_0, u, \psi, t, w) \in \mc{X}$ such that $|\syl(w)|$ is minimal (over all tuples in $\mc{X}$). (See \cref{d: syllable sequence} for the definition of $\syl(w)$.)
    
    By Claim 2, as $|A_0|$ is minimal amongst tuples in $\mc{V}$ we have $A_0 \sub \dom(\psi_{t^n})$ for each $n > 0$. Thus, as $\psi$ is finite, there is $m > 0$ with $\psi_{t^m}(A_0) = A_0$.

    Let $t_0^n$ be the first syllable of $w$ (so $n > 0$). Write $w = t_0^n v$, where necessarily the first syllable of $v$ is not $t_0$ (the subword $v$ may be empty). We have two cases.

    \textbf{Case 1:} $t_0 = t$. We have $w = t^n v$. Let $\psi'$ be a \pl-independent extension of $\psi$ to $\psi_u(C)$ by $t^n$, and let $s \in S \setminus \{t\}$. Then by Claim 2 we have $(\psi_{t^n}(\hat{\mc{A}}_0), ut^n, \psi', s) \in \mc{V}$. As $|A_0|$ is minimal we have $(\psi_{t^n}(\mc{A}_0), ut^n, \psi', s) \in \mc{V}'$. By the minimality of $|\syl(w)|$ we have $\psi_{t^n}(A_0) \sub \dom(\psi'_v)$. As $\psi_{t^n}(A_0) \nsubseteq \dom(\psi_v)$ and $\psi'$ is a \pl-independent extension by $t^n$, we have that there is $a \in A_0 \setminus \sg{\varnothing}$ and $i < |v|$ such that $\psi_{t^n v[1:i]}(a) \in \psi_u(C)$ and $v[j] = t$ for all $i < j \leq |v|$. But then $\psi_{t^n}(a) \notin \dom(\psi'_{vt^n})$, and $|\syl(vt^n)| < |\syl(w)|$, contradiction.

    \textbf{Case 2:} $t_0 \neq t$. Let $\psi'$ be a \pl-independent extension of $\psi$ to $\psi_u(C)$ by $t^m$. Recall that $m > 0$ and $\psi_{t^m}(A_0) = A_0$. Using this together with Claim 2 we have $(\mc{A}_0, ut^m, \psi', t_0) \in \mc{V}'$. If $A_0 \nsubseteq \dom(\psi'_w)$, then we are again in the situation of Case 1 (with the tuple $(\mc{A}_0, ut^m, \psi', t_0)$ and the same word $w$), and so we have a contradiction. If $A_0 \sub \dom(\psi'_w)$, then by similar reasoning to the previous case we have that the final letter of $w$ is $t$ and that $A_0 \nsubseteq \dom(\psi'_{wt^m})$. As $|\syl(wt^m)| = |\syl(w)|$, we are once more in Case 1 with the tuple $(\mc{A}_0, ut^m, \psi', t_0)$ and the word $wt^m$, and so again we get a contradiction.
\end{proof}
      
\subsection{The case \texorpdfstring{$\M=(\mathbb{Q},<)$}{M=(Q, <)}} \label{ssection:Q}
    
Recall from \cref{ex: structures with SWIRS} that $(\Q,<)$ admits a SWIR: we define $B \ind_A C$ if for $b \in B \setminus A$ and $c \in C \setminus A$ such that there does not exist $a \in A$ with $b < a < c$ or $b > a > c$, we have $b < c$.

We now define some notation that we will only use in this section.

\begin{defn} \label{d:inc dec}
    Let $\rho$ be a finite partial isomorphism of $(\Q, <)$. Let $\{\Inc, \Dec, \Fixed, \Undef\}$ be a set of symbols. For $q \in \Q$, we define the \emph{$\rho$-character} of $q$, written $\chi(q, \rho)$, as follows:
    \[
        \chi(q, \rho) =
        \begin{cases}
            \Inc &\quad q \in [a, \rho(a)] \text{ for some } a \in \dom(\rho) \text{ with } a < \rho(a)\\
            \Dec &\quad q \in [\rho(a), a] \text{ for some } a \in \dom(\rho) \text{ with } a > \rho(a)\\
            \Fixed &\quad \rho(q) = q\\
            \Undef &\quad \text{otherwise},
        \end{cases}
    \]
    and for $C \sub \Q$ we define $\Inc(C, \rho), \Dec(C, \rho), \Fixed(C, \rho), \Undef(C, \rho)$ to consist of the elements of $C$ with $\rho$-character $\Inc, \Dec, \Fixed, \Undef$ respectively. 
    
    In the particular case $C = \supp(\rho)$, we write $\Inc(\rho)$ instead of $\Inc(\supp(\rho), \rho)$, and define $\Dec(\rho), \Fixed(\rho)$ similarly. It is straightforward to see that $\Inc(\rho), \Dec(\rho), \Fixed(\rho)$ partition $\supp(\rho)$.

    We define $\SInc(\rho)$ to be the collection of maximal intervals $[a, b] \sub \Q$ with $a, b \in \Inc(\rho)$ and $[a, b] \cap (\Dec(\rho) \cup \Fixed(\rho)) = \varnothing$, and similarly define $\SDec(\rho)$ to be the collection of maximal intervals $[a, b] \sub \Q$ with $a, b \in \Dec(\rho)$ and $[a,b] \cap (\Inc(\rho) \cup \Fixed(\rho)) = \varnothing$.
\end{defn}

\begin{defn}
    Let $\rho$ be a finite partial isomorphism of $(\Q, <)$.

    Let $q \in \Q$. We define a word $w$ in the alphabet $\{\Inc, \Dec, \Fixed\}$ as follows. First, note that $\SInc(\rho) \cup \SDec(\rho) \cup \Fixed(\rho)$ is linearly ordered by the $\Q$-order. The word $w$ contains exactly:
    \begin{itemize}
        \item one occurrence of $\Inc$ for each interval in $\SInc(\rho)$ which intersects $[q, \infty)$,
        \item one occurrence of $\Dec$ for each interval in $\SDec(\rho)$ which intersects $[q, \infty)$,
        \item one occurrence of $\Fixed$ for each point in $\Fixed(\rho)$ contained in $[q, \infty)$,
    \end{itemize}
    and we order these occurrences in the word via the linear order on $\SInc(\rho) \cup \SDec(\rho) \cup \Fixed(\rho)$.
     
    Let $\pi(q, \rho)$ be the word obtained from $w$ by:
    \begin{itemize}
        \item inserting the letter $\Inc$ between each pair of consecutive letters of $w$ which are both not equal to $\Inc$;
        \item if the last letter of $w$ is not equal to $\Inc$, appending the letter $\Inc$ to the end of the word;
    \end{itemize}
    and let $\Pi(q, \rho)=|\pi(q, \rho)|$.
\end{defn}

\begin{defn} \label{d: t-advancing ext}
    Let $\phi$ be a finite partial action of $(F_S, S)$ on $\Q$. Let $t \in S$. Let $C \fin \Q$ be non-empty, and let $c = \min C$. Suppose that $[c, \infty) \cap \supp(\phi_t) \neq \varnothing$ and that there is $\eps = \pm 1$ with $\min([c, \infty) \cap \supp(\phi_t)) \in \Inc(\phi_{t^\eps})$. We call $\eps$ the \emph{sign of $(C, \phi_t)$}.
    
    We call an extension $\psi \geq \phi$ a \emph{$t$-advancing extension to $C$} if there exists a chain of extensions $\phi \leq \phi' \leq \phi'' \leq \psi$ satisfying the following:      
    \begin{enumerate}[label=(\roman*)]
        \item $\phi'$ is a \pl-independent extension of $\phi$ to $\{c\}$ by $t^\eps$;
        \item $\phi''$ is an extension of $\phi'$ by $t^\eps$ to $\Undef(C, \phi'_t)$ such that, for each $q \in \Undef(C, \phi'_t)$:
        \begin{itemize}
            \item if $q \in [a, b]$ for some $[a, b] \in \SInc(\phi'_t)$, then:
            \begin{itemize}
                \item in the case $\eps = +1$, we have $q < \phi''_t(q)$ and $(q, \phi''_t(q)) \cap \supp(\phi'') = \varnothing$,
                \item in the case $\eps = -1$, we have $\phi''_{t^{-1}}(q) < q$ and $(\phi''_{t^{-1}}(q), q) \cap \supp(\phi'') = \varnothing$,
            \end{itemize}
            \item if $q \in [a, b]$ for some $[a, b] \in \SDec(\phi'_t)$, then:
            \begin{itemize}
                \item in the case $\eps = +1$, we have $\phi''_t(q) < q$ and $(\phi''_t(q), q) \cap \supp(\phi'') = \varnothing$,
                \item in the case $\eps = -1$, we have $q < \phi''_{t^{-1}}(q)$ and $(q, \phi''_{t^{-1}}(q)) \cap \supp(\phi'') = \varnothing$,
            \end{itemize}
            \item otherwise:
            \begin{itemize}
                \item in the case $\eps = +1$, we have $q < \phi''_t(q)$ and $(q, \phi''_t(q)) \cap \supp(\phi'') = \varnothing$,
                \item in the case $\eps = -1$, we have $\phi''_{t^{-1}}(q) < q$ and $(\phi''_{t^{-1}}(q), q) \cap \supp(\phi'') = \varnothing$;
            \end{itemize}
        \end{itemize}
        \item $\psi$ is a \pl-independent extension of $\phi''$ by $t^\eps$ to $C \setminus (\Undef(C, \phi'_t) \cup \{c\})$. 
    \end{enumerate}

    Note that for each $t$-advancing extension $\psi \geq \phi$ to $C$, we have $c \leq \psi_{t^\eps}(c)$, and if $c \notin \dom(\phi_{t^\eps})$ then $c < \psi_{t^\eps}(c)$. Also note that if $C$ does not contain any element of $\phi_t$-character $\Undef$, then a $t$-advancing extension of $\phi$ to $C$ is just a \pl-independent extension of $\phi$ to $C$ by $t^\eps$. 
    
    For $m\geq 1$ we define a \emph{$m$-iterated $t$-advancing extension of $\phi$ to $C$} to be an extension $\psi$ for which there is a chain $\phi = \phi^0 \leq \phi^1 \leq \cdots \phi^m = \psi$ such that $\phi^l$ is a $t$-advancing extension of $\phi^{l-1}_{\vphantom{t^{\eps(l-1)}}}$ to $\phi^{l-1}_{t^{\eps(l-1)}}(C)$ for $1 \leq l \leq m$.
    
    (Here $\eps$ is the sign of $(C, \phi_t)$ -- note that $(\phi^l_{t^{\eps l}}(C), \phi^l_{t^{\vphantom{\eps l}}})$ has sign $\eps$ for $l < m$. Also note that, for $2 \leq l \leq m$, we have that $\phi^{l-1}_{t^{\eps(l-1)}}(C)$ does not contain any element of $\phi^{l-1}_t$-character $\Undef$, and so $\phi^l$ is a \pl-independent extension of $\phi^{l-1}$ to $\phi^{l-1}_{t^{\eps(l-1)}}(C)$ by $t^\eps$.)
\end{defn}

\begin{lem}
    We use the notation of \Cref{d: t-advancing ext} above. Let $\psi$ be a $t$-advancing extension of $\phi$ to $C$. Then $\pi(c, \phi_t) = \pi(\psi_{t^\eps}(c), \psi_t)$.
\end{lem}
\begin{proof}
    Let $x = \min([c, \infty) \cap \supp(\phi_t))$. Let $\phi \leq \phi' \leq \phi'' \leq \psi$ be a chain of extensions witnessing that $\psi$ is a $t$-advancing extension of $\phi$ to $C$. We check the claim of the lemma step by step in the chain.
    \begin{enumerate}[label=(\roman*)]
        \item It is straightforward to check that $x, c, \phi'_{t^\eps}(c)$ have the same $\phi'_t$-character and that, if there exist elements of $[c, \infty) \cap \supp(\phi_t)$ with different $\phi_t$-character to $x$, then $\phi'_{t^\eps}(c)$ is less than the least such. Thus $\pi(\phi'_{t^\eps}(c), \phi'_t) = \pi(c, \phi_t)$.
        \item Let $q \in \Undef(C, \phi'_{t^\eps})$. If $q$ lies in some maximal interval $[a, b] \in \SInc(\phi'_t) \cup \SDec(\phi'_t)$, then by definition $\{q, \phi''_{t^\eps}(q)\} \sub [a, b]$ and $\chi(q, \phi''_t) = \chi(a, \phi''_t)$. Otherwise, by definition $\chi(q, \phi''_t) = \Inc$. So $\pi(\phi''_{t^\eps}(c), \phi''_t) = \pi(\phi'_{t^\eps}(c), \phi'_t)$.
        \item Let $q \in C \setminus (\Undef(C, \phi'_{t^\eps}) \cup \{c\})$. Then $\chi(q, \phi''_t) \in \{\Inc, \Dec, \Fixed\}$, and so it is straightforward to see that $\pi(\psi^{}_{t^\eps}(c), \psi^{}_t) = \pi(\phi''_{t^\eps}(c), \phi''_t)$. \qedhere
    \end{enumerate}
\end{proof}

\begin{defn}
    Let $\phi$ be a finite partial action of $(F_S, S)$ on $\Q$. For $t \in S$, $q \in \Q$, we define $\Nx(q, \phi_t)$ to be the least $a \in [q, \infty) \cap \supp(\phi_t)$ with $\chi(a, \phi_t) \neq \chi(q, \phi_t)$. If there is no such $a$, we define $\Nx(q, \phi_t) = \infty$. (Informally, $\Nx(q, \phi_t)$ is the ``next element" above $q$ of different $\phi_t$-character.)
\end{defn}

\begin{lem} \label{l:bypassing increasing zone}
    Let $\phi$, $t$, $C$, $c = \min C$ be as in \Cref{d: t-advancing ext}. Then there exists $m \geq 0$ such that any $m$-iterated $t$-advancing extension $\psi$ of $\phi$ to $C$ satisfies the property that $\psi_{t^{\eps m}}(c) > a$ for all $a \in (-\infty, \Nx(\psi_{t^{\eps m}}(c), \psi_t)) \cap \bigcup_{s \in S \setminus \{t\}} \supp(\phi_s)$. 
\end{lem}
\begin{proof}
    It suffices to prove the lemma for $(C, \phi)$ with $C \sub \supp(\phi_t)$ (and thus with $C$ having no elements of $\phi_t$-character $\Undef$), as we may take a single $t$-advancing extension to $C$ if this is not the case. We therefore assume that each pair $(C, \phi)$ in this proof satisfies $C \sub \supp(\phi_t)$.

    For each $(C, \phi)$, letting $c = \min C$, we define $A_{(C, \phi)} = [c, \Nx(c, \phi_t)) \cap \bigcup_{s \in S \setminus \{t\}} \supp(\phi_s)$. We define $e_{(C, \phi)}$ to be the greatest element of $[c, \Nx(c, \phi_t)) \cap \supp(\phi_t)$ such that there is $a \in A_{(C, \phi)}$ with $e_{(C, \phi)} \leq a$, letting $e_{(C, \phi)} = -\infty$ if no such element exists. Let $\nu_{(C, \phi)} = |C \cap [c, e_{(C, \phi)}]|$. (Here, if $e_{(C, \phi)} = -\infty$, we have $[c, -\infty] = \varnothing$ and $\nu_{(C, \phi)} = 0$.)

    We prove the lemma for all $(C, \phi)$ satisfying the above conditions, via induction on $\nu_{(C, \phi)}$. In the base case $\nu_{(C, \phi)} = 0$, the claim is immediate, with $m = 0$. Suppose $\nu_{(C, \phi)} > 0$. Let $c' = \max((-\infty, e_{(C, \phi)}] \cap C)$, and let $j = |[c', e] \cap \supp(\phi_t)|$. Then it is straightforward to see that any $j$-iterated $t$-advancing extension $\psi \geq \phi$ to $C$ satisfies $\nu_{(\psi_{t^{\eps j}}(C), \psi)} < \nu_{(C, \phi)}$, and so by applying the induction assumption to $(\psi_{t^{\eps j}}(C), \psi)$ to obtain $m$, we then have that $m + j$ will satisfy the statement of the lemma for $(C, \phi)$.
\end{proof}
      
\begin{prop} \label{p:making independent Q}
    Let $(F_S,S)$ be a non-abelian free group with finite generating set $S$. Let $\mc{E}$ be the class of finite partial actions $\phi$ of $(F_S, S)$ on $(\Q, <)$ with the property that there is no $a \in \Q$ with $\phi_s(a)=a$ for all $s \in S$. Then $\mc{E}$ has the escape property.
\end{prop}
\begin{proof}
    Let $s \in S$. We show that, for each $j \in \N$, for all $\phi \in \mc{E}$ and $\mc{C} \fin (\Q, <)$ with $\Pi(\min C, \phi_s) \leq j$, we have the escape property for $(\phi, \mc{C})$ witnessed by some extension $\psi \geq \phi$ in $\mc{E}$ and some $u \in \Wr(S)$ with $\dom(\psi_s) \ind \psi_u(C)$. We do this by induction on $j$. In the base case $j = 0$, we have $\supp(\phi_s) < \min C$ by definition of $\Pi(\min C, \phi_s)$, and therefore immediately $\dom(\phi_s) \ind C$. Suppose $\Pi(\min C, \phi_s) = j > 0$. It suffices to find an extension $\tld\phi \geq \phi$ and a word $u \in \Wr(S)$ with $C \sub \dom(\tld\phi_u)$ and $\Pi(\min \tld\phi_u(C), \tld\phi_s) < j$. Let $c = \min C$ and let $x = \min([c, \infty) \cap \supp(\phi_s))$.

    In the case that $x$ is fixed by $\phi_s$, by the definition of the class $\mc{E}$ there is $t \in S \setminus \{s\}$ with $\chi(x, \phi_t) \neq \Fixed$. If $\chi(x, \phi_t) \in \{\Inc, \Undef\}$ define $\delta = +1$, and if $\chi(x, \phi_t) = \Dec$ define $\delta = -1$. If $x = c$, then take $\tld\phi$ to be a \pl-independent extension of $\phi$ to $C$ by $t^\delta$: we have $\Pi(\min \tld\phi_{t^\delta}(C), \tld\phi_s) < j$. If $x \neq c$, then $\chi(c, \phi_s) = \Undef$, and taking $\tld\phi$ to be a \pl-independent extension of $\phi$ to $C$ by $st^\delta$ we have $\Pi(\min \tld\phi_{st^\delta}(C), \tld\phi_s) < j$.

    Now consider the case where $x \in \Inc(\phi_{s^\eps})$ for some $\eps = \pm 1$. By \Cref{l:bypassing increasing zone}, there is $m \geq 0$ and some $m$-iterated $s$-advancing extension $\phi'$ of $\phi$ to $C$ with $\phi'_{s^{\eps m}}(c) > a$ for all $a \in (-\infty, \Nx(\phi'_{s^{\eps m}}(c), \phi'_s)) \cap \bigcup_{t \in S \setminus \{s\}} \supp(\phi_t)$. Let $t \in S \setminus \{s\}$. Let $\tld\phi$ be a \pl-independent extension of $\phi'$ to $\phi'_{s^{\eps m}}(C)$ by $t$. Then $\Pi(\min \tld\phi_{s^{\eps m}t}(C), \tld\phi_s) < j$ as required.
\end{proof}
      
\section{Independence and extensions by arcs} \label{s:ind use section}
       
In this section, we fix a robust subgroup $\Theta \leq \sym_k$, $r \geq 1$, a $(\Theta, r)$-structure $\mc{M}$ with universe $M$, a compatible seed group $H$, and a seed action of sets $\lambda : M \curvearrowleft H$ where $\lambda$ is an action by automorphisms of $\mc{M}$. The main result of this section is \cref{p:extension by arcs structure}, which provides conditions on $\mc{M}$ and $\lambda$ that ensure that $\lambda$ is a pleasant structural seed action on $\mc{M}$. The focus here will be the existence of an arc extensive family of finite taut extensions of $\lambda$ to some completion $(G,T)$ of $H$.

\subsection{Independent extensions and extensions by arcs} \label{ssec:alternation}

\begin{lem} \label{l:freely independent extension to an independent tuple}
    Let $\mc{N}$ be a quasi-relational strong \Fr structure with SWIR $\ind$. Let $\rho$ be a finite partial isomorphism of $\mc{N}$. Let $\mc{A}, \mc{C} \fin \mc{N}$ with $C \cup \supp(\rho) \sub A$ and $\dom(\rho) \ind C$. Let $m \geq 1$ and let $\sigma$ be an $m$-iterated $(A, +)$-independent extension of $\rho$ to $C$. Then for each $0 \leq j < m$ we have $A \cup \bigcup_{l = 0}^j \sigma^l(C) \ind \bigcup_{l = j + 1}^m \sigma^l(C)$.
\end{lem}
\begin{proof}
    Let $\rho = \rho_0 \sub \cdots \sub \rho_m = \sigma$ be the chain of $(A, +)$-independent extensions witnessing the $m$-iterated independence of the extension $\sigma$, and let $C_l = \sigma^l(C)$ for $0 \leq l \leq m$. We use induction on $m$.

    For the base case $m = 1$: as $\rho_1$ is an $(A, +)$-independent extension of $\rho_0$, we have \[A \cup \dom(\rho_1) \ind_{\im(\rho)} \im(\rho_1),\] so by (Mon) we have $A \ind_{\im(\rho)} C_1$. By the assumption of the lemma we have $\dom(\rho) \ind C$, so applying $\rho$ and using (Inv) we have $\im(\rho) \ind C_1$. Applying (Tr) to $\im(\rho) \ind C_1$ and $A \ind_{\im(\rho)} C_1$ we obtain $A \ind C_1$ as required.

    For the induction step: as $\rho_m$ is an $(A, +)$-independent extension of $\rho_{m-1}$, we have \[A \cup \dom(\rho_m) \ind_{\im(\rho_{m-1})} \im(\rho_m),\] hence by (Mon) we have $A \cup \bigcup_{l=0}^{m-1} C_l \ind_{\im(\rho_{m-1})} C_m$. The induction hypothesis for $m - 1$ with $j = m - 2$ gives $A \cup \bigcup_{l=0}^{m-2} C_l \ind C_{m-1}$, so by (Mon) we have $\dom(\rho_{m-1}) \ind C_{m-1}$, and applying $\sigma$ and using (Inv) we have $\im(\rho_{m-1}) \ind C_m$. Applying (Tr) to $\im(\rho_{m-1}) \ind C_m$ and $A \cup \bigcup_{l=0}^{m-1} C_l \ind_{\im(\rho_{m-1})} C_m$ we obtain $A \cup \bigcup_{l=0}^{m-1} C_l \ind C_m$. This gives the statement of the lemma for $j = m-1$.

    Now consider the case $0 \leq j \leq m-2$. Applying (Mon) to $A \cup \bigcup_{l=0}^{m-1} C_l \ind C_m$ we have $A \cup \bigcup_{l=0}^j C_l \ind_{\bigcup_{l = j+1}^{m-1} C_l} C_m$, so by base triviality we have $A \cup \bigcup_{l=0}^j C_l \ind_{\bigcup_{l = j+1}^{m-1} C_l} \bigcup_{l = j+1}^m C_l$. By the induction hypothesis we have $A \cup \bigcup_{l=0}^j C_l \ind \bigcup_{l=j+1}^{m-1} C_l$, so by (Tr) applied to the two previous statements we have $A \cup \bigcup_{l = 0}^j C_l \ind \bigcup_{l = j + 1}^m C_l$ as required.
\end{proof}

\begin{lem} \label{l:extending by power of one letter}
    Let $\mc{N}$ be a quasi-relational strong \Fr structure with SWIR $\ind$. Let $\rho$ be a finite partial isomorphism of $\mc{N}$. Let $\mc{A}, \mc{C}, \mc{C}' \fin \mc{N}$ with $C \cup C' \cup \supp(\rho) \sub A$, and also with $\dom(\rho) \ind C$ and $\dom(\rho) \ind C'$. Let $m, m' \geq 1$. Let $\sigma$ be an $m$-iterated $(A, +)$-independent extension of $\rho$ to $C$, and let $\tau$ be an $m'$-iterated $(A, +)$-independent extension of $\sigma$ to $C'$. Then $A \ind \tau^m(C)$ and $A \ind \tau^{m'}(C')$.
\end{lem}
\begin{proof}
    By \Cref{l:freely independent extension to an independent tuple} applied to $\sigma$ and the fact that $\tau^m(C) = \sigma^m(C)$, we immediately have $A \ind \tau^m(C)$.
    
    For $0 \leq i \leq j \leq m$ let $C_i = \tau^i(C)$ and $C_{i:j} = \bigcup_{i \leq l \leq j} C_l$, and for $0 \leq i \leq j \leq m'$ let $C'_i = \tau^i(C')$ and $C'_{i:j} = \bigcup_{i \leq l \leq j} C'_l$. By the definition of an $(A, +)$-independent extension and by (Mon) we have $A \ind_{\im(\rho) C_{1:i}} C_{i+1}$ for $0 \leq i < m$ (where $C_{1:0} = \varnothing$), and also $A \ind_{\im(\rho) C^{\vphantom{\prime}}_{1:m \vphantom{1:i}} C'_{1:i}} C'_{i+1}$ for $0 \leq i < m'$. So by (Tr) we have $A \ind_{\im(\rho)} C^{\vphantom{\prime}}_{1:m^{\vphantom{\prime}}} C'_{1:m'}$, and by (Mon) we have $A \ind_{\im(\rho) C'_{1:i}} C'_{i+1}$ for $0 \leq i < m'$. Again by (Mon) we have $AC'_{1:i} \ind_{\im(\rho) C'_{1:i}} C'_{i+1}$ for $0 \leq i < m'$, so the restriction of $\tau$ to $\dom(\rho) C'_{0:m'-1} \to \im(\rho) C'_{1:m'}$ is an $m'$-iterated $(A, +)$-independent extension of $\rho$ to $C'$. Thus by \Cref{l:freely independent extension to an independent tuple} we have $A \ind \tau^{m'}(C')$.
\end{proof}

\begin{defn} \label{d:F_Omega arc extension}
    Let $(F_S, S)$ be a non-abelian free group with finite generating set $S$. Let $\mc{N}$ be a quasi-relational strong \Fr structure, and let $\phi$ be a finite partial action of $(F_S, S)$ on $\mc{N}$. Let $\mc{A}, \mc{A}' \in [\mc{N}]^k$ be isomorphic via $\bar{a} \mapsto \bar{a}'$, where $\bar{a}$, $\bar{a}'$ are enumerations of $A, A'$. Let $w = t_1 \cdots t_m \in \mc{W}(S)$, $w \neq \varnothing$, be such that $A \cap \dom(\phi_{t_1}) = A' \cap \im(\phi_{t_m}) = A \cap A' = \sg{\varnothing}$.

    An \emph{extension of $\phi$ by $w$-arcs from $\bar{a}$ to $\bar{a}'$} is a finite extension $\psi \geq \phi$ such that there exists a sequence $\bar{a} = \bar{a}_0, \cdots, \bar{a}_m = \bar{a}'$ of enumerations of substructures $\mc{A}_0, \cdots, \mc{A}_n$ of $\mc{N}$ satisfying:
    \begin{enumerate}[label=(\roman*)]
        \item $\psi$ can be obtained from $\phi$ via a chain of extensions $\phi = \phi^0 \leq \cdots \leq \phi^m = \psi$, where for $1 \leq i \leq m$ the extension $\phi^i$ is constructed from $\phi^{i-1}$ by extending $\phi^{i-1}_{t^i}$ by the isomorphism $\bar{a}_{i-1} \mapsto \bar{a}_i$;
        \item for $0 \leq i < j \leq m$ we have $A_i \cap A_j = \sg{\varnothing}$;
        \item for $1 \leq i \leq m-1$ we have $A_i \cap \supp(\phi) = \sg{\varnothing}$.
    \end{enumerate}
\end{defn}

\begin{lem} \label{l:arc extension from independent}
    Let $\mc{N}$ be a quasi-relational strong \Fr structure with SWIR $\ind$. Let $(F_S, S)$ be a non-abelian free group with finite generating set $S$. Let $\phi$ be a finite partial action of $(F_S, S)$ on $\mc{N}$. Let $s, \tld{s}, t \in S$ with $\tld{s} \neq t$. Let $\mc{C}, \mc{C}' \fin \mc{N}$ be finite substructures of $\mc{N}$ isomorphic via $\bar{c} \mapsto \bar{c}'$ (where $\bar{c}, \bar{c}'$ are enumerations of $\mc{C}, \mc{C}'$) with $\dom(\phi_s) \ind C$, $\dom(\phi_s) \ind C'$ and $C \cap C' = \sg{\varnothing}$.

    Let $v, v' \in \Wr(S)$ be positive words beginning with $s$ and ending with $t$ such that $\sylc(v) = \sylc(v')$. Let $w = v \cdot \tld{s} \cdot (v')^{-1}$.

    Then there exists a finite extension $\psi \geq \phi$ such that:
    \begin{itemize}
        \item $\psi$ is an extension of $\phi$ by $w$-arcs from $\bar{c}$ to $\bar{c}'$;
        \item there is a chain of extensions $\phi = \phi^0 \leq \cdots \leq \phi^n = \phi' \leq \psi$ such that:
        \begin{itemize}
            \item for $1 \leq i \leq n$, the extension $\phi^i$ is a \pl-independent extension of $\phi^{i-1}$ by a (positive) letter in $T_\Omega$ (here we do not assume coherence of the extending sets),
            \item $\phi'_v(C)$, $\phi'_{v'}(C')$ are defined, with $\supp(\phi') = \supp(\phi) \cup \mc{P}_v^{\phi'}(C) \cup \mc{P}_{v'}^{\phi'}(C')$, and $\phi'_v(C) \cap \phi'_{v'}(C') = \phi'_v(C) \cap \supp(\phi'_{\tld{s}}) = \phi'_{v'}(C') \cap \supp(\phi'_{\tld{s}}) = \sg{\varnothing}$,
            \item $\psi$ is an extension of $\phi'$ by $\tld{s}$, where $\psi^{}_{\tld{s}}$ extends $\phi'_{\tld{s}}$ by the map $\phi'_v(\bar{c}) \mapsto \phi'_{v'}(\bar{c}')$.
        \end{itemize}
    \end{itemize}
\end{lem}
\begin{proof}
    Write $v = s_1^{m_1 \vphantom{m'_r}} \cdots s_r^{m_r \vphantom{m'_r}}$, $v' = s_1^{m'_1} \cdots s_r^{m'_r}$, where $s_1 = s$, $s_r = t$. Define $v_i = s_1^{m_1 \vphantom{m'_r}} \cdots s_i^{m_i \vphantom{m'_r}}$, $v'_i = s_1^{m'_1} \cdots s_i^{m'_i}$ for $1 \leq i \leq r$, and let $v^{}_0, v'_0$ be the empty word.

    We inductively define a chain of extensions $\phi = \psi^0 \leq \chi^1 \leq \psi^1 \leq \cdots \leq \chi^r \leq \psi^r$ as follows. Let $\psi^0 = \phi$. For $1 \leq i \leq r$, assuming the chain has been defined up to $\psi^{i-1}$, define $\chi^i$ to be a $(\psi^{i-1}_{v'_{i-1}}(C'), +)$-independent extension of $\psi^{i-1}$ by $s_i^{m_i}$ to $\psi^{i-1}_{v_{i-1}}(C)$, and define $\psi^i$ to be a \pl-independent extension of $\chi^i$ by $s_i^{m'_i}$ to $\chi^i_{v'_{i-1}}(C')$.

    We claim that for $1 \leq i \leq r$ we have $\supp(\psi^{i-1}) \ind \psi^i_{v_i}(C)$ and $\supp(\psi^{i-1}) \ind \psi^i_{v'_i}(C')$. We prove this by induction on $i$. In the case $i = 1$, as $\dom(\phi_s) \ind C$ and $\dom(\phi_s) \ind C'$, \Cref{l:extending by power of one letter} immediately gives $\supp(\psi^0) \ind \psi^1_{v_1}(C)$ and $\supp(\psi^0) \ind \psi^1_{v'_1}(C')$. In the case $i > 1$, as $s_i \neq s_{i-1}$ we have $\dom(\psi^{i-1}_{s_i}) = \dom(\psi^{i-2}_{s_i})$. By the induction assumption $\supp(\psi^{i-2}) \ind \psi^{i-1}_{v_{i-1}}(C)$ and $\supp(\psi^{i-2}) \ind \psi^{i-1}_{v'_{i-1}}(C')$, so $\dom(\psi^{i-1}_{s_i}) \ind \psi^{i-1}_{v_{i-1}}(C)$ and $\dom(\psi^{i-1}_{s_i}) \ind \psi^{i-1}_{v'_{i-1}}(C')$, and by applying \Cref{l:extending by power of one letter} we have $\supp(\psi^{i-1}) \ind \psi^i_{v_i}(C)$ and $\supp(\psi^{i-1}) \ind \psi^i_{v'_i}(C')$. This completes the proof of the claim.

    Applying the claim in the case $i = r$, we have $\supp(\psi^{r-1}) \ind \psi^r_v(C)$ and $\supp(\psi^{r-1}) \ind \psi^r_{v'}(C')$. As $s_r = t \neq \tld{s}$, we have $\supp(\psi^r_{\tld{s}}) \ind \psi^r_v(C)$ and $\supp(\psi^r_{\tld{s}}) \ind \psi^r_{v'}(C')$, so $\dom(\psi^r_{\tld{s}}) \ind \psi^r_v(C)$ and $\im(\psi^r_{\tld{s}}) \ind \psi^r_{v'}(C')$. By ultrahomogeneity, (Inv) and (Sta) we can thus extend $\psi^r$ to a new extension $\psi$ by defining $\psi_{\tld{s}}$ to extend $\psi^r_{\tld{s}}$ by $\psi^r_v(\bar{c}) \mapsto \psi^r_{v'}(\bar{c}')$, and then it is immediate that $\psi$ satisfies the required conditions of the lemma.
\end{proof}
    
\subsection{Well-centralised actions.} 

\cref{l:arc extension from independent} motivates Definitions \ref{d:well centralised} and \ref{d:orbit-rich} below. In the below definition, recall the setup at the start of \Cref{s:ind use section}.
      
\begin{defn} \label{d:extended language}
    Let $\mc{L}$ denote the language of $\mc{M}$. Let $\Omega \leq \Theta$ and $\mc{A} \fin \mc{M}$. We let $\mc{L}(\mc{A}, \Omega)$ be the language obtained by expanding $\mc{L}$ by constant symbols $c_a$ for each $a \in \mc{A}$ and function symbols $f_{\omega}$ for each $\omega \in \Omega$. Let $\mu : \mc{M} \curvearrowleft \Omega$ be an action. We let $\mc{M}(\mc{A}, \mu)$ be the $\mc{L}(\mc{A}, \Omega)$-expansion of $\mc{M}$ given by interpreting each $f_{\omega}$, $\omega \in \Omega$, as the $\mu$-action of $\omega$ on $\mc{M}$, and interpreting each $c_a$, $a \in \mc{A}$, as $a$.
\end{defn}
      
\begin{defn} \label{d:well centralised}
    Let $\Omega \leq \Theta$. Let $\mu : \mc{M} \curvearrowleft \Omega$ be an action and let $\mc{A} \fin \mc{M}$ be $\Omega$-invariant in $\mu$. We say that $\mu$ is \emph{$k$-flexible over $\mc{A}$} if $\mc{M}(A, \mu)$ admits a quasi-relational strong \Fr expansion $\mc{N}$ such that:
    \begin{enumerate}[label=(\alph*)]
        \item \label{univ subst} for all $C \sub N$, we have $\sg{C}_{\mc{N}} = \sg{C}_{\mc{M}(A, \mu)}$; 
        \item \label{univ k sets} any two substructures $B,B' \in [\mc{M}(A, \mu)]^k$ which are in the same $\Aut(\mc{M}(A, \mu))$-orbit are also in the same $\Aut(\mc{N})$-orbit.
      \end{enumerate}
      If in addition $\mc{N}$ admits a SWIR, then we say that $\mu$ is \emph{$k$-well-centralised over $\mc{A}$}.

      If we can take $\mc{N}=\mc{M}(A, \mu)$ itself, then we say that $\mu$ is \emph{strictly flexible (resp. well-centralised)}. (Note that in this case the second item is automatically verified, and thus $k$ plays no role in the condition.)
\end{defn}
  
When extending finite extensions of the seed action  $\lambda$ it is always necessary that there is ``sufficient space'' among the different orbits of the action $\lambda$ for the extension to be as free as possible (see conditions \ref{arcs different orbits} and \ref{arcs orbit outside support} in \cref{d:arc extension}). In the presence of structure the assumption that $\lambda$ has infinitely many orbits needs to be strengthened according to the following definition. 
     
\begin{defn} \label{d:orbit-rich}
    Let $\mc{N}$ be a quasi-relational strong \Fr structure with domain $N$, and let $\mu : N \curvearrowleft H$ be an action. For $A \sub N$ denote by $[A]_\mu$ the $\mu$-orbit-closure of $A$.  We say that $\mu$ is \emph{orbit-rich for $\mc{N}$} if for any $\mc{A}, \mc{B} \fin \mc{N}$ and $C \fin N$ there is $g \in \Aut(\mc{N})$ fixing $\mc{B}$ with $[gC]_\mu \cap [AB]_\mu = [B \cap gC]_\mu$.
\end{defn}
     
\begin{lem} \label{l:keeping different orbits}
    Let $\mc{N}$ be a quasi-relational strong \Fr structure. For $A \sub N$, we write $[A]$ to denote the union of all $E$-equivalence classes of points in $A$, where $E$ is the equivalence relation given in the definition of a quasi-relational structure. Suppose we have:
    \begin{itemize}
        \item some action $\mu : N \curvearrowleft H$ orbit-rich for $\mc{N}$;
        \item substructures $\mc{A}, \mc{B} \fin \mc{N}$ and a subset $C \fin N$ such that $B \cap [C] = \varnothing$.
    \end{itemize}
         
    Then there is some $g \in \Aut(\mc{N})$ fixing $B$ such that:
    \begin{enumerate}[label=(\roman*)]
        \item $[gC]_\mu \cap AB = \varnothing$, 
        \item for all $c, c'\in C$, if $[gc]_\mu=[gc']_\mu$ then $[gc]=[gc']$.
    \end{enumerate} 
\end{lem}
\begin{proof}
    Let $m$ be the number of $E$-equivalence classes having non-trivial intersection with $C$. We use induction on $m$. The case $m = 0$ is trivial. Suppose $m > 0$. As $B \cap [C] = \varnothing$ and $B$ is the domain of a substructure of $\mc{N}$, we have $\sg{\varnothing} \cap [C] = \varnothing$. So there exist $C_0, C' \sub C$ with $C = C_0 \cup C'$ such that $C_0$ is contained in $m-1$ $E$-equivalence classes, $C'$ is contained in a single $E$-equivalence class and $\sg{C_0} \cap [C'] = \varnothing$. By the induction assumption, there is $g \in \Aut(\mc{N})$ fixing $B$ with $[gC_0]_\mu \cap AB = \varnothing$ and with (ii) holding for all elements of $C_0$. As $g$ fixes $B$ and $B \cap [C] = \varnothing$, we have $B \cap [gC] = \varnothing$. Also $\sg{gC_0} \cap [gC'] = \varnothing$, so $B\sg{gC_0} \cap [gC'] = \varnothing$. As $\mu$ is orbit-rich for $\mc{N}$, there is $g' \in \Aut(\mc{N})$ fixing $B\sg{gC_0}$ with $[g'gC']_\mu \cap [AB\sg{gC_0}]_\mu = [B\sg{gC_0} \cap g'gC']_\mu$, and as $g'$ fixes $B\sg{gC_0}$ we have $B\sg{gC_0} \cap g'gC' = B\sg{gC_0} \cap gC' = \varnothing$. So $[g'gC']_\mu \cap [AB\sg{g'gC_0}]_\mu = \varnothing$, and as $[g'gC_0]_\mu \cap AB = [gC_0]_\mu \cap AB = \varnothing$, we are done.
\end{proof}
    
\subsection{Finding extensions by arcs.}

\begin{prop} \label{p:extension by arcs structure}
    Let $\Theta \leq \sym_k$ be robust and let $\mc{M}$ be a $(\Theta, r)$-structure. Let $H$ be a compatible seed group and $(G, T)$ a completion of $H$. Let $\lambda : M \curvearrowleft H$ be a seed action of sets, acting by automorphisms of $\mc{M}$, such that:
    \begin{enumerate}[label=(\alph*)]
        \item \label{flexible} for each $\Omega \in \mc{D}(\Theta)$, the restriction $\lambda|_\Omega$ is $k$-well-centralised, as witnessed by some expansion $\mathcal{N}_{\Omega}$ of $\M(\NFr_{\lambda}(\Omega), \lambda|_\Omega)$ and some SWIR $\ind^{\Omega}$ on $\mathcal{N}_{\Omega}$,
        \item \label{orbit-richness}for every $\Omega \in \mc{D}(\Theta)$ the action $\lambda$ is orbit-rich for $\mathcal{N}_{\Omega}$. 
    \end{enumerate}
     	
     Suppose that for all $\Omega \in \mc{D}(\Theta)$ there is a collection $\mc{E}_{\Omega}$ of finite partial actions of the free group $(F_{\Omega},T_{\Omega})$ on $\mc{N}_{\Omega}$ with the following properties:
     \begin{enumerate}[label=(\roman*)]
        \item \label{Eomega esc} $\mc{E}_\Omega$ has the escape property;
        \item \label{Eomega non-empty} $\mc{E}_\Omega$ contains the partial action $\phi$ defined by $\phi_t = \id_{\NFr_\lambda(\Omega)}$, $t \in T_\Omega$;
        \item \label{Eomega conj} for all $\phi \in \mc{E}_\Omega$, $g \in \Aut(\mc{N}_\Omega)$, letting $\psi$ be the partial action defined by $\psi_t = g^{-1} \circ \phi_t \circ g$, $t \in T_\Omega$, we have $\psi \in \mc{E}_\Omega$;
        \item \label{Eomega atomic} for each $\phi \in \mc{E}_\Omega$, $\mc{A} \fin \mc{N}_\Omega$, $t \in T_\Omega^{\pm 1}$, there is an atomic free extension $\psi \geq \phi$ by $t$ to a substructure containing $A$ such that $\psi \in \mc{E}_\Omega$;
        \item \label{Eomega chains} for each $\phi \in \mc{E}_\Omega$ and each chain $\phi = \phi^0 \leq \cdots \leq \phi^n$ where for all $1 \leq i \leq n$ we have that $\phi^i$ is a \pl-independent extension (via $\ind^\Omega$) of $\phi^{i-1}$ by some $t \in T_\Omega$, there is a free extension $\psi$ of $\phi^n$ with $\psi \in \mc{E}_\Omega$;
        \item \label{Eomega arcs} there is $\tld{s}_\Omega \in T_\Omega$ such that for each $\phi \in \mc{E}_\Omega$ and each chain $\phi = \phi^0 \leq \cdots \leq \phi^n$ where for all $1 \leq i \leq n$ we have that $\phi^i$ is a \pl-independent extension (via $\ind^\Omega$) of $\phi^{i-1}$ by some $t \in T_\Omega$, the following holds: for each extension $\psi$ of $\phi^n$ by $\tld{s}_\Omega$, where $\psi^{}_{\tld{s}_\Omega}$ extends $\phi^n_{\tld{s}_\Omega}$ by $\bar{c} \mapsto \bar{c}'$ and where $\bar{c}$, $\bar{c}'$ are enumerations of substructures $\mc{C}, \mc{C}' \fin \mc{N}_\Omega$ with $C \cap C' = C \cap \supp(\phi^n_{\tld{s}_\Omega}) = C' \cap \supp(\phi^n_{\tld{s}_\Omega}) = \sg{\varnothing}$, there is a free extension $\psi'$ of $\psi$ with $\psi' \in \mc{E}_\Omega$.
    \end{enumerate}
    Let $\mc{F}$ be the collection of partial actions $\phi$ of $(G,T)$ on $\mc{M}$ with $\phi$ a taut finite extension of $\lambda$ and such that, for each $\Omega \in \mc{D}(\Theta)$, the restriction of $\phi$ to $(F_\Omega, T_\Omega)$ lies in $\mc{E}_\Omega$. Then $\mc{F}$ is arc-extensive. 
\end{prop}
\begin{proof}
    It is straightforward to see from condition \ref{Eomega non-empty} that $\mc{F}$ is non-empty.

    We first show condition \ref{condI} in the definition of an arc-extensive family (\Cref{d:arc extensiveness}). We are given $\Omega \in \mc{D}(\Theta)$, $\phi \in \mc{F}$, $t \in T_\Omega^{\pm 1}$, an $\Omega$-set $A \in [\mc{M}]^k$ and $B \fin M$. As $A$ is an $\Omega$-set and $\lambda$ is a seed action of sets, by \Cref{l: key seed action props}\ref{seedact docile} we have $\NFr_\lambda(\Omega) \sub A$, and as $\lambda$ is $k$-well-centralised, by \Cref{d:well centralised}\ref{univ subst} we have that $A$ is the domain of a substructure of $\mc{N}_\Omega$. Let $\phi^\Omega$ denote the restriction of $\phi$ to $(F_\Omega, T_\Omega)$. By condition \ref{Eomega atomic} for $\mc{E}_\Omega$, there is an atomic free extension $\chi^\Omega \in \mc{E}_\Omega$ of $\phi^\Omega$ by $t$ to a substructure $\mc{C} \fin \mc{N}_\Omega$ with $A \sub C$. By the definition of an atomic free extension, we have that $\chi^\Omega_t(C) \setminus \im(\phi^\Omega_t)$ is disjoint from $C \cup \supp(\phi^\Omega)$.
    
    By taking a superset of $B$ if necessary, we may assume that $B$ contains $\supp(\phi) \cup C$ and the paradigm set of $\lambda$, and we may also assume that $B$ is the domain of a substructure of $\mc{N}_\Omega$. There exists $g \in \Aut(\mc{N}_\Omega)$ fixing $C \cup \supp(\phi^\Omega)$ with $g(\chi^\Omega_t(C) \setminus \im(\phi^\Omega_t)) \cap B = \varnothing$. Let $\tld{\chi}^\Omega = g^{-1} \cdot \chi^\Omega \cdot g$. Then $\tld{\chi}^\Omega \in \mc{E}_\Omega$ by condition \ref{Eomega conj}, and $\tld{\chi}^\Omega$ is an extension of $\phi^\Omega$. We have $(\tld{\chi}^\Omega_t(C) \setminus \im(\phi^\Omega_t)) \cap B = \varnothing$. Note that $\tld{\chi}^\Omega_t(C) \setminus \im(\phi^\Omega_t)$ is $\Omega$-invariant, as $C$ is the domain of a substructure of $\mc{N}_\Omega$ and $\phi$ is taut. So by \Cref{l:keeping different orbits}, there is $g' \in \Aut(\mc{N}_\Omega)$ fixing $B$ with $[g'(\tld{\chi}^\Omega_t(C) \setminus \im(\phi^\Omega_t))]_\lambda \cap B = \varnothing$, and such that two elements of $g'(\tld{\chi}^\Omega_t(C) \setminus \im(\phi^\Omega_t))$ lie in the same $\lambda$-orbit iff they lie in the same $\Omega$-orbit. Let $\psi^\Omega = (\tld{\chi}^\Omega)^{g'}$. Then $\psi^\Omega \in \mc{E}_\Omega$ by condition \ref{Eomega conj}, and defining a family of partial bijections $\psi$ by extending $\phi$ on $T_\Omega$ to $\psi^\Omega$, we have that $\psi$ is a partial action by \Cref{normal forms action} (note that the conditions of \Cref{normal forms action} are satisfied as we extend by partial actions of $\mc{N}_\Omega$). By construction $\psi$ is an atomic $\Omega$-free extension of $\phi$ and hence is taut by \Cref{l:partial k-sharp extend}, so $\psi \in \mc{F}$.

    We now show condition \ref{condII} in the definition of an arc-extensive family. Let $\Omega \in \mc{D}(\Theta)$, and take $\tld{s}_\Omega$ from condition \ref{Eomega arcs}. Let $\phi \in \mc{F}$. Let $\mc{A}, \mc{A}' \in [\mc{M}]^k$ be $\Omega$-isomorphic $\Omega$-substructures with $A$ $\Omega$-strict and $A, A'$ not in the same $\phi$-orbit. Similarly to the proof of condition \ref{condI}, we have that $A, A'$ are the domains of substructures of $\mc{N}_\Omega$, and additionally by \Cref{d:well centralised}\ref{univ k sets} we have that these substructures of $\mc{N}_\Omega$ are isomorphic. Let $\chi = \phi^\Omega$ be the restriction of $\phi$ to $(F_\Omega, T_\Omega)$. As $\mc{E}_\Omega$ has the escape property (condition \ref{Eomega esc}), there is a free extension $\chi^1 \in \mc{E}_\Omega$ of $\chi$, a letter $s_0 \in T_\Omega$ and a word $u \in \Wr(T_\Omega)$ with $AA' \sub \dom(\chi^1_u)$ such that $\dom(\chi^1_{s_0}) \ind^\Omega \chi^1_u(AA')$. Let $A^{}_1 = \chi^1_u(A)$, $A'_1 = \chi^1_u(A')$. We have $\dom(\chi^1_{s_0}) \cap (A^{}_1 A'_1) = \sg{\varnothing}$. Let $\chi^2$ be a \pl-independent extension of $\chi^1$ by $s_0^2$ to $A^{}_1 A'_1$, and let $A^{}_2 = \chi^2_{s_0}(A^{}_1)$, $A'_2 = \chi^2_{s_0^2}(A'_1)$. Then $A^{}_2 \cap A'_2 = \sg{\varnothing}$, and by \Cref{l:freely independent extension to an independent tuple} we have $\dom(\chi^2_s) \ind^\Omega A^{}_2 A'_2$ for all $s \in T_\Omega \setminus \{s_0\}$. Let $s \in T_\Omega \setminus \{s_0\}$, let $N = 6$, and let $w = v \cdot \tld{s}_{\Omega} \cdot (v')^{-1} \in \Wr(T_\Omega)$ be a word satisfying the four bulletpointed conditions in \ref{II arc}, where $v, v' \in \Wr(T_\Omega)$ are positive words beginning with $s$ and ending with a letter distinct from $\tld{s}_\Omega$ with $\sylc(v) = \sylc(v')$. By \Cref{l:arc extension from independent}, there is an extension $\chi^3$ of $\chi^2$ by $w$-arcs from (an enumeration of) $A^{}_2$ to (an enumeration of) $A'_2$ satisfying the bulletpointed conditions in \Cref{l:arc extension from independent}, and by condition \ref{Eomega arcs} there is a free extension $\psi^\Omega$ of $\chi^3$ with $\psi^\Omega \in \mc{E}_\Omega$. As in the proof of condition \ref{condI}, by applying automorphisms of $\Aut(\mc{N}_\Omega)$ we may assume that when we extend $\phi$ by extending $\phi^\Omega$ to $\psi^\Omega$ to give a partial action $\psi$, we have that $\psi$ can be obtained by a chain of extensions as follows: an $\Omega$-free extension $\phi'$ where the images $B, B'$ of $A, A'$ satisfy the conditions of \ref{II sep} (here we use \Cref{l:partial k-sharp extend}), then an $\Omega$-extension $\phi''$ by $w$-arcs from an enumeration of $B$ to an enumeration of $B'$, then an $\Omega$-free extension $\psi$, exactly as specified in parts \ref{II sep}, \ref{II arc} of condition \ref{condII} (we leave the easy verifications to the reader). By \Cref{l:partial k-sharp extend} and \Cref{p: small cancellation loops} we have that $\psi$ is taut. Hence $\psi \in \mc{F}$.
\end{proof}

\begin{lem} \label{l:Eomega examples}
    Let $\Theta$, $\mc{M}$, $H$, $(G, T)$, $\lambda$ and $(\mc{N}_\Omega)_{\Omega \in \mc{D}(\Theta)}$ be as in \Cref{p:extension by arcs structure}.

    Consider the following two cases, where in each case we specify a family $(\mc{E}_\Omega)_{\Omega \in \mc{D}(\Theta)}$.
    \begin{itemize}
        \item \textbf{General case:} for each $\Omega \in \mc{D}(\Theta)$, fix some $t_\Omega \in T_\Omega$, and let $\mc{E}_\Omega$ be the collection of finite partial actions $\phi$ of $(F_\Omega, T_\Omega)$ on $\mc{N}_\Omega$ where $\phi_{t_\Omega}$ is full chain-independent.
        \item \textbf{Free case:} in the case where for each $\Omega \in \mc{D}(\Theta)$ we have that $\ind^\Omega$ satisfies (Free), let $\mc{E}_\Omega$ be the collection of finite partial actions $\phi$ of $(F_\Omega, T_\Omega)$ on $\mc{N}_\Omega$ such that there is no $\mc{A} \fin \mc{N}_\Omega$ with $\mc{A} \neq \sg{\varnothing}$ and with $\phi_t(\mc{A}) = \mc{A}$ for all $t \in T_\Omega$.
    \end{itemize}
    In each case above, for each $\Omega \in \mc{D}(\Theta)$ we have that $\mc{E}_\Omega$ satisfies properties \ref{Eomega esc}-\ref{Eomega arcs} in \Cref{p:extension by arcs structure}.
\end{lem}
\begin{proof}
    \textbf{General case:} \ref{Eomega esc} is given by \Cref{c:escape property chain-independent}. \ref{Eomega non-empty}: $\phi$ defined by $\phi_t = \id_{\NFr_\lambda(\Omega)}$ for all $t \in T_\Omega$ has $\phi_{t_\Omega}$ full chain-independent by definition. \ref{Eomega conj}: follows by (Inv) of $\ind^\Omega$. \ref{Eomega atomic}: let $\phi \in \mc{E}_\Omega$, $\mc{A} \fin \mc{N}_\Omega$, $t \in T_\Omega$ (note that we assume $t$ is positive). If $t \neq t_\Omega$, then for $t^{\pm 1}$ the property holds immediately by taking an atomic free extension of $\phi$ by $t^{\pm 1}$ to $\mc{A}$, so assume $t = t_\Omega$. Take a \pl-independent extension of $\phi_t$ to $\mc{A}$ and use \Cref{l:fullifying} to obtain a full chain-independent extension $\psi_t$ of $\phi_t$; the extension $\psi$ of $\phi$ obtained by extending $\phi_t$ to $\psi_t$ then lies in $\mc{E}_\Omega$. It remains to show \ref{Eomega atomic} for $t^{-1} = t_\Omega^{-1}$: take a \mi-independent extension of $\phi_t$ over to $\mc{A}$ and use \Cref{l:fullifying}. \ref{Eomega chains} and \ref{Eomega arcs}: these follow similarly to \ref{Eomega atomic}, taking $\tld{s}_\Omega \in T_\Omega \setminus \{t_\Omega\}$ and using \Cref{l:fullifying}.

    \textbf{Free case:} \ref{Eomega esc} is given by \cref{l:escape property free}. \ref{Eomega non-empty}, \ref{Eomega conj} are straightforward, and \ref{Eomega atomic}, \ref{Eomega chains} follow from the definition of atomic free extension. \ref{Eomega arcs}: taking $\tld{s}_\Omega \in T_\Omega$, this property follows directly from the conditions $C \cap C' = C \cap \supp(\phi^n_{\tld{s}_\Omega}) = C' \cap \supp(\phi^n_{\tld{s}_\Omega}) = \sg{\varnothing}$.
\end{proof}

The below corollary follows immediately from \Cref{p:extension by arcs structure} and \Cref{l:Eomega examples} (in the Free case, recall property \ref{c-no invariant sets} of taut extensions):

\begin{cor} \label{c:extension by arcs structure}
    Let $\Theta$, $\mc{M}$, $H$, $(G, T)$, $\lambda$ and $(\mc{N}_\Omega)_{\Omega \in \mc{D}(\Theta)}$ be as in \Cref{p:extension by arcs structure}. Suppose that in addition $\lambda$ satisfies conditions \ref{seed-struc reference}, \ref{seed-struc all permutations}, \ref{seed-struc equivariant isomorphism condition} in the definition of a pleasant structural seed action (\Cref{d:seed action on structure}). 

    \begin{enumerate}[label=(\roman*)]
        \item For each $\Omega \in \mc{D}(\Theta)$, fix $t_\Omega \in T_{\Omega}$. Let $\mc{F}$ be the collection of partial actions $\phi$ of $(G, T)$ on $\mc{M}$ with $\phi$ a finite taut extension of $\lambda$ such that, for each $\Omega \in \mc{D}(\Theta)$, the restriction of $\phi$ to $(F_{\Omega}, T_\Omega)$ is a partial action on $\mc{N}_\Omega$ with $\phi_{t_{\Omega}}$ full chain-independent. Then $\lambda$ is a pleasant structural seed action with arc-extensive family $\mc{F}$.
        \item If for all $\Omega \in \mc{D}(\Theta)$ we have that $\lambda|_\Omega$ is strictly well-centralised and $\ind^{\Omega}$ on $\mc{M}(\NFr_{\lambda}(\Omega), \lambda|_\Omega)$ satisfies (Free), then letting $\mc{F}$ be the collection of all finite taut extensions of $\lambda$, we have that $\lambda$ is a pleasant structural seed action with arc-extensive family $\mc{F}$. (That is, $\lambda$ is generous.)
    \end{enumerate}
\end{cor}

\begin{cor} \label{c: pleasant ssa SWIR rigid}
    Let $\mc{M}$ be a relational \Fr structure with strong amalgamation. Let $k \geq 1$. Suppose $\mc{M}$ has a SWIR and $k$-substructures of $\mc{M}$ are rigid (i.e.\ have trivial automorphism group). Let $H = \Theta = \mathtt{1} \leq \sym_k$, and let $\lambda : \mc{M} \curvearrowleft H$ be the trivial action. Then $\lambda$ is a pleasant structural seed action on $\mc{M}$. If $\ind$ satisfies (Free), then $\lambda$ is generous. Also, if $\mc{M} = (\Q, <)$ then $\lambda$ is generous.
\end{cor}
\begin{proof}
    We have that $\mc{M}$ is a $(\Theta, r)$-structure for some $r \geq 1$: condition \ref{automorphism tuples} in Definition \ref{d:theta r structure} follows by rigidity of $k$-substructures, \ref{isomorphism types} and \ref{lower transitivity} are trivial. Conditions \ref{seed-struc reference}, \ref{seed-struc all permutations}, \ref{seed-struc equivariant isomorphism condition} in the definition of a pleasant structural seed action are immediate. As $\mc{M}(\NFr_\lambda(\Theta), \lambda)$ is just $\mc{M}$ equipped with the identity, we may work directly with $\mc{M}$: we have that $\lambda$ is well-centralised as $\mc{M}$ has a SWIR, and as $\mc{M}$ has strong amalgamation it is immediate that $\lambda$ is orbit-rich. By Corollary \ref{c:extension by arcs structure} we have that $\lambda$ is a pleasant structural seed action and, if $\ind$ satisfies (Free), then $\lambda$ is generous.
    
    Now suppose $\mc{M} = (\Q, <)$. Let $\mc{E}$ be the class of finite partial actions $\phi$ of $(F_\Theta, T_\Theta)$ on $(\Q, <)$ with the property that there is no $a \in \Q$ with $\phi_t(a)=a$ for all $t \in T_\Theta$. Then $\mc{E}$ has the escape property by \Cref{p:making independent Q}, and properties \ref{Eomega non-empty}-\ref{Eomega arcs} in \Cref{p:extension by arcs structure} are straightforward to check. By \Cref{p:extension by arcs structure}, the collection $\mc{F}$ of all finite taut extensions of $\lambda$ is arc-extensive, so $\lambda$ is a generous pleasant structural seed action.
\end{proof}

\subsection{The reducts of \texorpdfstring{$(\Q, <)$}{(Q, <)}}

\begin{cor} \label{c:reducts Q}
The following holds for the non-trivial reducts of $(\mathbb{Q},<)$. 
    \begin{itemize}
        \item for all $k\geq 2$ the betweenness structure $(\mathbb{Q},B^{(3)})$ admits a pleasant structural seed action by a copy of $\cyc_{2}$ inside $\sym_{k}$ with maximum support (that is, with a fixed point if $k$ is odd and no fixed points if $k$ is even),
        \item for all $k\geq 2$ the cyclic ordering $(\mathbb{Q},C^{(3)})$ admits a pleasant structural seed action of $\cyc_{k}\leq\sym_{k}$,
        \item for all odd $k \geq 3$ the dihedral group $\dih_k$ admits a pleasant structural seed action on the separation structure $(\mathbb{Q},S^{(4)})$. 
    \end{itemize} 
\end{cor}
\begin{proof}
    	Consider, for instance, the case of $\M=(\mathbb{Q},C^{(3)})$. Fix $k\geq 2$ and notice that $\cyc_{k}\leq\sym_{k}$ and all its subgroups are docile with no fixed points. Up to equivariant isomorphism, there is a unique free action $\mu$ of $\Omega=\cyc_k$ on $(\mathbb{Q},C^{(3)})$. It is not hard to see that the structure $\M(\varnothing, \mu)$ admits an expansion interdefinable with the structure $\mathcal{N}=(\mathbb{Q}\times\mathbf{k},E^{(2)},<, P_{0}^{1},\dots P_{k-1}^{1})$, where:
    	\begin{itemize}
    		\item $P_{i}^{\mathcal{N}}=\mathbb{Q}\times\{i\}$,
    		\item $E$ is an equivalence class in which each class consists of exactly one point in each $P_{i}$, 
    		\item $<$ is a dense linear order on the domain of $P_{0}$. 
    		\item Each $P_{i}$ corresponds to some interval of $(\mathbb{Q},C^{(3)})$ which is a translate under the action $\mu$ of some fundamental domain. 
    	\end{itemize} 
      It is easy to see how this structure inherits a SWIR from $\mathbb{Q}$. It is also easy to see that for any two $\Omega$-sets $A,A'\subseteq [\mathbb{Q}]^{k}$, which are always isomorphic in the sense of $\mathcal{M}(\varnothing, \mu)$, their images in $\mathbb{Q}\times\mathbf{k}$ are also isomorphic in $\mathcal{N}$, even though some of the isomorphisms between the two have been lost when passing from $\M(\varnothing, \mu)$. Property \ref{seed-struc equivariant isomorphism condition} of \cref{d:seed action on structure} is also immediate. The remaining cases are not meaningfully more complex and are left to the reader. 
    \end{proof}
    
\section{Constructing flexible seed actions on structures} \label{s:construction good finite actions}

Fix a robust subgroup $\Theta \leq \sym_k$, a $(\Theta, r)$ structure $\mc{M}$ with age $\mc{K}$, a compatible seed group $H$ and a completion $(G, T)$ of $H$.

In \Cref{s:construction good finite actions}, under certain additional assumptions on $\Theta$ and $\mc{M}$, we construct actions $\lambda : \mc{M} \curvearrowleft H$ satisfying the requirements of \Cref{c:extension by arcs structure}, namely:

\begin{itemize}
    \item $\lambda$ will be a seed action of sets $\lambda : M \curvearrowleft H$ acting by automorphisms of $\mc{M}$ and satisfying conditions \ref{seed-struc reference}, \ref{seed-struc all permutations}, \ref{seed-struc equivariant isomorphism condition} in the definition of a pleasant structural seed action (\Cref{d:seed action on structure});
    \item $\lambda$ will satisfy conditions \ref{flexible}, \ref{orbit-richness} in \Cref{p:extension by arcs structure}.
\end{itemize}
 
We will then be able to apply \Cref{c:extension by arcs structure} to conclude that $\lambda$ is a pleasant structural seed action on $\mc{M}$, and thus we will have by \Cref{p:sharply k-homog actions on structures} that $(G, T)$ admits a sharply $k$-homogeneous action on $\mc{M}$. (In some cases, we will obtain more: namely, genericity results for such actions.) We do this final step in \Cref{s:main theorem}, which collects our results together.
 
In the below we also fix an action $\lambda^0$ on a finite substructure $\mc{A}_0 \fin \mc{M}$ as follows:
\begin{itemize}
    \item if $\Theta$ is docile, we take $\mc{A}_0$ and $\lambda^0$ to be empty;
    \item if $\Theta$ is unruly, we take some $\mc{A}_0 \in [\mc{M}]^k$ with $(A_0 \curvearrowleft \Aut(\mc{A}_0)) \simeq (\pi : \mathbf{k} \curvearrowleft \Theta)$, and we take $\lambda^0 : \mc{A}_0 \curvearrowleft \Theta$ isomorphic to the permutation action $\pi$.
\end{itemize}

\begin{rem}
    Though we restrict ourselves to robust $\Theta$, subgroups $\Omega \in \mc{D}(\Theta) \cup \{\Theta\}$ and actions $\lambda^0$ as above, the constructions in \Cref{s:construction good finite actions} work more generally for arbitrary finite groups and finite actions. We work in this specific context so that the reader can immediately see how the results of this section fit into the general framework of the paper.
\end{rem}

\begin{defn} \label{d:theta-class}
    Recall \cref{d:extended language} and the extended language introduced therein. Let $\Omega \in \mc{D}(\Theta) \cup \{\Theta\}$, let $\mc{A} \fin \mc{M}$ and let $\mu^{0} : \mc{A} \curvearrowleft \Omega$ be a structural action. We write $\mc{K}(\mu^0, \Omega)$ for the class of all $\mc{L}(A, \Omega)$-structures $\mc{B}$ satisfying:
    \begin{itemize}
        \item $\mc{B}|_{\mc{L}} \in \mc{K}$;
        \item the interpretation of the constant symbols $(c_a)_{a \in A}$ of $\mc{L}(A, \Omega)$ gives an embedding $\mc{A} \to \mc{B}|_{\mc{L}}$ (henceforth taken to be an inclusion); 
        \item the interpretation of the function symbols $(f_{\omega})_{\omega \in \Omega}$ of $\mc{L}(A, \Omega)$ gives an action by automorphisms $\mc{B}|_{
        \mc{L}} \curvearrowleft \Omega$ which is free on $B \setminus A$ and which restricts to $\mu^0 : \mc{A} \curvearrowleft \Omega$.
    \end{itemize}

    As the action $\mu^0$ specifies the group involved in the action, we usually suppress $\Omega$ from the notation and just write $\mc{K}(\mu^0)$. We also consider the case where $\mc{A} = \varnothing$ and we extend an empty action $\mu^0$: we write $\mc{K}(\varnothing)$ in this case, and when the group is not clear from context we write $\mc{K}(\varnothing, \Omega)$.
\end{defn} 
    
\begin{defn} \label{d:universal actions}
    Let $\Omega \in \mc{D}(\Theta) \cup \{\Theta\}$, let $\mc{A} \fin \mc{M}$ and let $\mu^{0} : \mc{A} \curvearrowleft \Omega$ be a structural action. We say that an action $\mu : \mc{M} \curvearrowleft \Omega$ is \emph{universal over $\mu^0$} if $\mc{K}(\mu^0)$ is a strong \Fr class and $\mu$ is $\Omega$-isomorphic to the action given by the interpretation of the function symbols $(f_\omega)_{\omega \in \Omega}$ in $\FrLim(\mc{K}(\mu^0))$.
\end{defn}

\begin{defn} \label{d:strongly universal actions}
    Recall that we fixed $\lambda^0 : \mc{A}_0 \curvearrowleft \Theta$ at the start of this section. Let $\lambda : \mc{M} \curvearrowleft \Theta$ be a structural action extending $\lambda^0$. We say that $\lambda$ is \emph{strongly universal over $\lambda^0$} if $\lambda$ is universal over $\lambda^0$ and in addition:
    \begin{enumerate}[label=(\roman*)]
        \item\label{univ over fx} for every $\Omega \in \mc{D}(\Theta)$, the action $\lambda|_\Omega$ is universal over its restricted action on $\NFr_\lambda(\Omega)$; 
        \item\label{univ act orbit-rich} for every $\Omega \in \mc{D}(\Theta)$, the action $\lambda$ is orbit-rich for $\mc{M}(\NFr_\lambda(\Omega), \lambda|_\Omega)$. 
    \end{enumerate}
         
    If for each $\Omega \in \mc{D}(\Theta)$ we additionally have that $\lambda|_\Omega$ is strictly well-centralised over $\NFr_\lambda(\Omega)$, then we say that $\lambda$ is \emph{strictly strongly universal over $\lambda^0$}. (See \Cref{d:well centralised} for the definition of strictly well-centralised actions.)
\end{defn}

The proof of the following lemma is immediate.
\begin{lem} \label{l:strong univ rephrasing}
    An action $\lambda : \mc{M} \curvearrowleft \Theta$ universal over $\lambda^0$ and satisfying condition \ref{univ over fx} in \Cref{d:strongly universal actions} exists if and only if $\mc{K}(\lambda^0)$ is a strong \Fr class and for each $\Omega \in \mc{D}(\Theta)$ we have that $\mc{K}(\lambda^0|_{\NFr_{\lambda^0}(\Omega) \curvearrowleft \Omega})$ is a strong \Fr class with $\FrLim(\mc{K}(\lambda^0))|_{\mc{L}(\NFr_{\lambda^0}(\Omega), \Omega)} = \FrLim(\mc{K}(\lambda^0|_{\NFr_{\lambda^0}(\Omega) \curvearrowleft \Omega}))$.
\end{lem}

\begin{defn} \label{d: Se_k M Theta}
    Let $\Se_k(\mc{M}, \Theta)$ be the space of structural actions $\mc{M} \curvearrowleft \Theta$ with a $\Theta$-invariant set on which the action is $\Theta$-isomorphic to $\pi : \mathbf{k} \curvearrowleft \Theta$ and with all other orbits free. (We write $\Se$ as a mnemonic for ``structural action which is a \emph{seed} action of \emph{sets}" -- see Observation \ref{o: Theta seed action}.)
 
    We note the following particular examples. It is straightforward to see the following:
    \begin{itemize}
        \item $\Se_2(\mc{M}, \sym_2)$ consists of the free structural actions $\mc{M} \curvearrowleft \sym_2$;
        \item $\Se_3(\mc{M}, \sym_3)$ consists of the structural actions $\mc{M} \curvearrowleft \sym_3$ where each involution of $\sym_3$ has exactly one fixed point and where each element of $\sym_3$ of order $3$ acts freely.
    \end{itemize}
\end{defn}

The following is a straightforward adaptation to our context of the folkloric fact that the set of structures isomorphic to the \Fr limit of a locally finite amalgamation class $\mc{C}$ is comeagre in the space of structures with age contained in $\mc{C}$, where we fix a countably infinite domain (see for example \cite{KT17}):

\begin{fact} \label{f:genericity of universal actions}
    Suppose that $\mc{M}$ admits an action $\lambda : \mc{M} \curvearrowleft \Theta$ universal over $\lambda^0$. Then the set of actions $\mu : \mc{M} \curvearrowleft \Theta$ which are universal over an action $\Theta$-isomorphic to $\lambda^0$ is a comeagre subset of $\Se_k(\mc{M}, \Theta)$.
\end{fact}
  
We also use the following folkloric fact, the proof of which is straightforward (via verification of the extension property):
\begin{fact} \label{f:reducts}
    Let $\mc{L}, \mc{L}'$ be first-order languages with $\mc{L} \sub \mc{L}'$. Let $\mc{K}$, $\mc{K}'$ be \Fr classes of $\mc{L}$- and $\mc{L}'$-structures respectively, where for each $\mc{C}' \in \mc{K}'$ we have $\mc{C}'|_{\mc{L}} \in \mc{K}$, and suppose that:
    \begin{enumerate}[label=(\Alph*)] 
        \item \label{reduct cond1} each $\mc{C} \in \mc{K}$ embeds into some $\mc{D} \in \mc{K}$ such that $\mc{D}$ admits an expansion $\mc{D}' \in \mc{K}'$; 
        \item \label{reduct cond2} for each $\mc{C}' \in \mc{K}'$ and each embedding $f_0 : \mc{C}'|_{\mc{L}} \to \mc{D}$ with $\mc{D} \in \mc{K}$, there exists an embedding $f_1 : \mc{D} \to \mc{E}$ with $\mc{E} \in \mc{K}$ such that $\mc{E}$ expands to some $\mc{E}' \in \mc{K}'$ and $f_1 \circ f_0$ is an embedding $\mc{C}' \to \mc{E}'$.
    \end{enumerate}
    Then $\FrLim(\mc{K}')|_{\mc{L}} = \FrLim(\mc{K})$.
\end{fact}
    
Now we move on to condition \ref{univ act orbit-rich} in \cref{d:strongly universal actions}. For this we need some stronger assumptions.  

\begin{defn} \label{d:free orbit property}
    Let $\Omega \in \mc{D}(\Theta)$, and let $\mu^{0}$ be the restriction of $\lambda^0$ to $\NFr_{\lambda^0}(\Omega) \curvearrowleft \Omega$. Let $\mc{P} \sub \mc{K}(\mu^0)$, $\mc{P}' \sub \mc{K}(\lambda^0)$ be strong \Fr classes. We say that the pair $(\mc{P},\mc{P}')$ has the \emph{free orbit property} if $\mc{P}$, $\mc{P}'$ satisfy the assumptions of \cref{f:reducts} and the following conditions hold:
    \begin{enumerate}[label=(\alph*)]
        \item \label{fo cond1} in condition \ref{reduct cond1}, the given embedding $\mc{C} \to \mc{D}$ gives an injection from the set of free orbits of $\mc{C}$ to the set of free orbits of $\mc{D}'$ disjoint from $A_0$;
        \item \label{fo cond2} in condition \ref{reduct cond2}, the embedding $f_1 : \mc{D} \to \mc{E}$ gives an injection from the set of free orbits of $\mc{D}$ disjoint from $\im(f_0)$ to the set of free orbits of $\mc{E}'$.  
    \end{enumerate} 
\end{defn}
    
The above definition is motivated by the following lemma.  
\begin{lem} \label{l:orbit-rich}
    Let $\Omega \in \mc{D}(\Theta)$, and let $\mu^{0}$ be the restriction of $\lambda^0$ to $\NFr_{\lambda^0}(\Omega) \curvearrowleft \Omega$. Let $\mc{P} \sub \mc{K}(\mu^0)$, $\mc{P}' \sub \mc{K}(\lambda^0)$ be strong \Fr classes with $\FrLim(\mc{P})|_{\mc{L}} = \mc{M}$ such that $(\mc{P}, \mc{P}')$ has the free orbit property. Let $\lambda$ be the action on $\mc{M}$ given by $\FrLim(\mc{P}')$. (Note that, as $\mc{P}'$ is a strong \Fr class, we have that $\lambda$ is orbit-rich with respect to $\FrLim(\mc{P}')$.) Then $\lambda$ is orbit-rich with respect to $\mc{M}(\NFr_\lambda(\Omega), \lambda|_\Omega) = \FrLim(\mc{P})$. 
\end{lem}

\begin{proof}
    Throughout this proof, given a structure $\mc{S}$ we denote its domain by $S$ (as usual in this paper). Write $\mc{N} = \FrLim(\mc{P})$, $\mc{N}' = \FrLim(\mc{P}')$. Write $\mc{L}_{\mc{N}}$ for the language of $\mc{N}$. Let $\mc{A}, \mc{B} \fin \mc{N}$ and $C \fin N$. Let $\mc{B}'_1$ be the substructure of $\mc{N}'$ generated by $B$, and let $\mc{B}_1$ be the $\mc{L}_{\mc{N}}$-reduct of $\mc{B}'_1$. We have $\mc{B} \sub \mc{B}_1$, and as $\mc{P}$ has strong amalgamation, by applying an automorphism of $\mc{N}$ if necessary, we may assume $(C \setminus B) \cap B_1 = \varnothing$. As $\mc{B}_1$ is a substructure of $\mc{N}$, this implies $(C \setminus B) \cap [B_1]_{\lambda|_\Omega} = \varnothing$. Let $\mc{D} = \sg{B_1 C}_{\mc{N}}$. Considering the inclusion $\mc{B}_1 \hookrightarrow \mc{D}$, as $(\mc{P}, \mc{P}')$ has the free orbit property, there is $\mc{E}' \in \mc{P}'$ with $\mc{L}_{\mc{N}}$-reduct $\mc{E}$ and an embedding $f : \mc{D} \to \mc{E}$ such that $f$ restricts to an embedding $\mc{B}'_1 \to \mc{E}'$ and such that $f$ gives an injection from the set of free orbits of $\mc{D}$ disjoint from $B_1$ into the set of free orbits of $\mc{E}'$. Let $X = E \setminus f(B_1)$. As $\mc{P}'$ has strong amalgamation, we may assume that $\mc{E}' \sub \mc{N}'$, that $f$ is the identity on $\mc{B}'_1$ and that $X \cap [A]_\lambda = \varnothing$. Extend $f : \mc{D} \to f(\mc{D})$ by ultrahomogeneity of $\mc{N}$ to $g \in \Aut(\mc{N})$. Then, by the fact that $f$ was given by the free orbit property, it is straightforward to see that $[gC]_\lambda \cap [AB]_\lambda = [(gC) \cap B]_\lambda$.
\end{proof}

\subsection{Strictly strongly universal actions.} \label{ss:strictly strongly universal actions}

\subsubsection{Mutually independent families}
\begin{defn} \label{d:mutually indep}
    Suppose $\mc{M}$ has a SIR $\ind$. Let $\mc{C} \fin \mc{M}$, let $l \geq 1$, and for each $i < l$ let $\mc{D}_i \fin \mc{M}$ with $\mc{C} \sub \mc{D}_i$. For each $I \sub \mathbf{l}$, define $D_I = \bigcup_{i \in I} D_i$. We say that the family $(D_i)_{i < l}$ is \emph{mutually independent over $C$} if, for all disjoint non-empty $I, J \sub \mathbf{l}$, we have $D_I \ind_C D_J$.
    
    Note that the above condition implies that $D_i \cap D_j = C$ for all distinct $i, j < l$, using (Mon) and Lemma \ref{SWIR strong amalg}.
\end{defn}

\begin{lem} \label{l: finding mutually indep}
    Suppose $\mc{M}$ has a SIR $\ind$. Let $\mc{C} \fin \mc{M}$, let $l \geq 1$ and for $i < l$ let $\mc{D}_i \fin \mc{M}$ with $\mc{C} \sub \mc{D}_i$. Suppose that $D_0 \cdots D_{i-1} \ind_C D_i$ for all $i < l$. Then $(D_i)_{i < l}$ is mutually independent over $C$.
\end{lem}
\begin{proof}
    We use induction on $l$. The case $l = 1$ is trivial. Suppose $l > 1$. Let $I, J \sub \mathbf{l}$ be non-empty and disjoint. By the induction hypothesis for $l-1$, if $I, J \sub \mathbf{l-1}$ then $D_I \ind_C D_J$. Thus we may assume $l-1 \in J$ (without loss of generality). By the induction hypothesis we have $D_I \ind_C D_{J \setminus \{l-1\}}$, and by assumption and (Mon) we have $D_I D_{J \setminus \{l-1\}} \ind_C D_{l-1}$; hence $D_I \ind_{C D_{J \setminus \{l-1\}}} D_{l-1}$ by (Mon). Applying (Tr) with $D_I \ind_C D_{J \setminus \{l-1\}}$ and $D_I \ind_{C D_{J \setminus \{l-1\}}} D_{l-1}$ we have $D_I \ind_C D_J$ as required.
\end{proof}

\begin{lem} \label{l:gluing}
    Suppose that $\mc{M}$ is equipped with a SIR $\ind$. Let $\mc{C} \fin \mc{M}$, let $l \geq 1$ and let $\mc{D}_i \fin \mc{M}$ for $i < l$, with $(D_i)_{i < l}$ mutually independent over $C$. Let $\mc{D}$ be the substructure of $\mc{M}$ with domain $D = \bigcup_{i < l} D_i$, and let $f : D \to D$ be a bijection fixing $C$ setwise such that there exists $\sigma \in \sym_l$ with $f|_{\mc{D}_i} : \mc{D}_i \to \mc{D}_{\sigma(i)}$ an isomorphism for all $i < l$. Then $f \in \Aut(\mc{D})$.
\end{lem}
\begin{proof}
    For $j < l$, let $\mc{D}_{\leq j}$ be the substructure of $\mc{M}$ with domain $D_{\leq j} = \bigcup_{i \leq j} D_i$. We show by induction on $j$ that $f|_{\mc{D}_{\leq j}}$ is an isomorphism for $j < l$. In the case $j = 0$, this is by assumption. Let $j > 0$, and suppose that $f|_{\mc{D}_{\leq j}}$ is an isomorphism. By assumption we have $D_{\leq j} \ind_C D_{j+1}$, and as $\sigma(j+1) \notin \sigma(\{0, \cdots, j\})$, by assumption we have $\bigcup_{i < j} D_{\sigma(i)} \ind_C D_{\sigma(j+1)}$. Equivalently $f(D_{\leq j}) \ind_{f(C)} f(D_{j+1})$ (here we use the fact that $f(C) = C$). As $f|_{\mc{D}_{\leq j}}$ is an isomorphism, by ultrahomogeneity it extends to some $g \in \Aut(\mc{M})$. By (Inv), with $g^{-1}$ applied to $f(D_{\leq j}) \ind_{f(C)} f(D_{j+1})$, we have $D_{\leq j} \ind_C (g^{-1} \circ f)(D_{j+1})$, so as $D_{\leq j} \ind_C D_{j+1}$ we have $(g^{-1} \circ f)(D_{j+1}) \equiv_{CD_{\leq j}} D_{j+1}$ by (Sta). As $g$ is an automorphism of $\mc{M}$, we therefore have that $f|_{D_{\leq j+1}}$ is an isomorphism.
\end{proof}

\begin{lem} \label{l: S2 S3 str}
    Let $\mc{M}$ be a relational \Fr structure with strong amalgamation. Suppose that $\mc{M}$ has a SIR. Then:
    \begin{enumerate}[label=(\roman*)]
        \item \label{i: S2 str} $\mc{M}$ is a $(\sym_2, r)$-structure for some $r \geq 1$;
        \item \label{i: S3, 1 str} if $\mc{M}$ has a single isomorphism type of $3$-substructure, then $\mc{M}$ is a transitive $(\sym_3, 1)$-structure;
        \item \label{i: S3 str} if $\mc{M}$ is transitive, then $\mc{M}$ is a $(\sym_3, r)$-structure for some $r \geq 1$.
    \end{enumerate}
\end{lem}
\begin{proof}
    We first claim that for $c \in \mc{M}$ and $k = 2, 3$, there is $\mc{A} \in [\mc{M}]^k$ with $\Aut(\mc{A}) \cong \sym_k$ and $\qftp(a) = \qftp(c)$ for all $a \in A$: as $\mc{M}$ has strong amalgamation, by repeatedly applying (Ex), there are distinct $a_i \in M$, $i < k$, with each $a_i$ having the same quantifier-free type as $c$ and with $a_0 \cdots a_{i-1} \ind a_i$ for all $i < k$. Taking $\mc{A} = \mc{M}|_{\{a_0, \cdots, a_{k-1}\}}$, the claim then follows by Lemma \ref{l: finding mutually indep} and Lemma \ref{l:gluing}.

    \ref{i: S2 str}: We have $\mc{D}(\sym_2) = \{\mathtt{1}, \sym_2\}$, so condition \ref{automorphism tuples} in Definition \ref{d:theta r structure} is immediate. Condition \ref{lower transitivity} is trivial. For condition \ref{isomorphism types} it suffices to find $\mc{A} \in [\mc{M}]^2$ with $\Aut(\mc{A}) \neq \mathtt{1}$: such $\mc{A}$ exists by the above claim.

    \ref{i: S3, 1 str}: Condition \ref{automorphism tuples} in \Cref{d:theta r structure} is easily checked; condition \ref{isomorphism types} and the transitivity of $\mc{M}$ follow by the above claim.

    \ref{i: S3 str}: Conditions \ref{automorphism tuples}, \ref{isomorphism types} are verified as in the previous case. As $\mc{M}$ is transitive, we have condition \ref{lower transitivity}.
\end{proof}

\subsubsection{Neat extensions}

\begin{defn} \label{d:neat extension}
    Suppose $\mc{M}$ is equipped with a SIR $\ind$. Let $\Omega \in \mc{D}(\Theta) \cup \{\Theta\}$. Let $\mc{C} \sub \mc{D} \fin \mc{M}$, let $\nu^0 : \mc{C} \curvearrowleft \Theta$ be an action, and let $\mu : \mc{D} \curvearrowleft \Omega$ be an action extending $\nu^0|_\Omega$ such that $\mu$ acts freely on $D \setminus C$. 

    Let $\mc{E} \fin \mc{M}$ with $\mc{D} \sub \mc{E}$. We say that an action $\nu : \mc{E} \curvearrowleft \Theta$ extending $\nu^0$ and $\mu$ is a \emph{neat extension of $(\nu^{0}, \mu)$} if 
    \begin{itemize}
        \item $E$ is the $\nu$-orbit-closure of $D$;
        \item every $\nu$-orbit disjoint from $C$ is free and contains exactly one of the free $\mu$-orbits within $D \setminus C$;
        \item for any system of representatives $(\theta_i)_{i \in \mathbf{q}}$ of the right cosets in $\Omega \backslash \Theta$, the family $(\nu_{\theta_{j}}(D))_{j\in\mathbf{q}}$ is mutually independent over $C$.
    \end{itemize}
\end{defn}
    
\begin{lem} \label{l:neat extension}
    In the situation of \cref{d:neat extension} there is, up to $\Theta$-isomorphism, a unique neat extension of $(\nu^{0},\mu)$ to an action $\nu$.
\end{lem}
\begin{proof}
    We show existence, with uniqueness being straightforward. First we construct a set action. As $\mu$ acts freely on $D \setminus C$, the $\mu$-action on each orbit in $D \setminus C$ is $\Omega$-isomorphic to the right multiplication action $\Omega \curvearrowleft \Omega$, which extends to the right multiplication action $\Theta \curvearrowleft \Theta$ for which $\Theta$ is a free orbit. Thus there is a set $S \supseteq D$ and a set action $\nu^S : S \curvearrowleft \Theta$ extending $\nu^0$, $\mu$ such that $S \setminus C$ consists of free $\Theta$-orbits, where each free $\Theta$-orbit in $S \setminus C$ contains exactly one free $\mu$-orbit of $D \setminus C$. For each $\theta \in \Theta$, let $D_\theta = \nu^S_\theta(D)$.

    \begin{claim*}
        For all $\theta, \theta' \in \Theta$, if $D_\theta \cap D_{\theta'} \neq C$, then $D_\theta = D_{\theta'}$ and $\theta' \in \Omega \theta$.
    \end{claim*}
    \begin{subproof}
        As $D_\theta \cap D_{\theta'} \neq C$, there are $d, d' \in D \setminus C$ with $d \cdot \theta = d' \cdot \theta'$, so $d = d' \cdot \theta' \theta^{-1}$, and so as $\nu^S$ is free on $S \setminus C$ and each $\Omega$-orbit of $D \setminus C$ is contained in exactly one $\Theta$-orbit of $S \setminus C$, there is $\omega \in \Omega$ with $\theta' \theta^{-1} = \omega$. As $D$ is $\Omega$-invariant, we have $D_\theta = D_{\theta'}$.
    \end{subproof}
    Let $\Omega \theta_0, \cdots, \Omega \theta_{l-1}$ be a system of right coset representatives in $\Omega \backslash \Theta$, with $\theta_0 = 1$. For each $i < l$, define a structure $\mc{D}_{\theta_i}$ by specifying that $\nu^S_{\theta_i}|_D : D \to D_{\theta_i}$ is an isomorphism (extending $\nu^S_{\theta_i}|_C : C \to C$). Note that as $D$ is $\Omega$-invariant, for each $\theta \in \Theta$, writing $\theta = \omega \theta_i$ for some $\omega \in \Omega$ and $i < l$, we have $D_\theta = D_{\theta_i}$ and $\nu^S_\theta|_D : D \to D_\theta$ is an isomorphism. Hence, for all $\theta \in \Theta$ and $i < l$, we have that $\nu^S_\theta|_{D_{\theta_i}}$ is an isomorphism $\mc{D}_{\theta_i} \to \mc{D}_{\theta_j}$ for some $j < l$.

    Let $\mc{D}_0 = \mc{D}$. Using ($\ind$-Ex) and induction, embed $\mc{D}_{\theta_1}, \cdots, \mc{D}_{\theta_{l-1}}$ as some $\mc{D}_1, \cdots, \mc{D}_{l-1} \sub \mc{M}$ such that $D_0 \cdots D_{i-1} \ind_C D_i$ for all $i < l$. By Lemma \ref{l: finding mutually indep}, we have that $(D_i)_{i \in \mathbf{l}}$ is mutually independent over $C$. Let $\mc{E}$ be the structure with domain $\bigcup_{i \in \mathbf{l}} D_i$, and let $\nu : E \curvearrowleft \Theta$ be the set action induced by $\nu^S$ (via the embeddings $D_{\theta_i} \to D_i$). By \Cref{l:gluing} we have that $\nu$ is a structural action $\mc{E} \curvearrowleft \Theta$, and it is clear that $\nu$ is a neat extension.
\end{proof}

\subsubsection{Pleasant structural seed actions from strictly strongly universal actions}

\begin{lem} \label{l:SWIR gives action amalgam}
    Recall that SWIRs induce standard amalgamation operators (see \Cref{p:swir sao}).
    
    Suppose that $\mc{M}$ is equipped with a SWIR $\ind$, inducing a standard amalgamation operator $\otimes$. Let $\Omega \in \mc{D}(\Theta) \cup \Theta$ and let $\mu^0 : \mc{A} \curvearrowleft \Omega$ be an action on some $\mc{A} \fin \mc{M}$. Then $\mc{K}(\mu^0)$ is a strong \Fr class, and has a standard amalgamation operator $\otimes'$ as follows: given $\mc{C}, \mc{D} \in \mc{K}(\mu^0)$ with intersection $\mc{B} \in \mc{K}(\mu^0)$, take the $\mc{L}$-structure on $C \cup D$ to be $C|_{\mc{L}} \otimes_{B|_{\mc{L}}} D|_{\mc{L}}$, and expand to an $\mc{L}(A, \Omega)$-structure by interpreting the constant symbols as $A \sub B$ and by interpreting each function symbol $f_\omega$, $\omega \in \Omega$, as $f_\omega^{\mc{C}} \cup f_\omega^{\mc{D}}$ (note that $f_\omega^{\mc{C}}$, $f_\omega^{\mc{D}}$ agree on $f_\omega^{\mc{B}}$). If $\ind$ satisfies (Free), then so does the SWIR on $\FrLim(\mc{K}(\mu^0))$ induced by $\otimes'$.
\end{lem}
\begin{proof}
    We only need to check that $f := f_\omega^{\mc{C}} \cup f_\omega^{\mc{D}}$ is an automorphism of $C|_{\mc{L}} \otimes_{B|_{\mc{L}}} D|_{\mc{L}}$, and this follows immediately from ($\otimes$-Inv) (which in turn follows from ($\ind$-Inv) and ($\ind$-Sta)).
\end{proof}

\begin{rem}
    Note that it is not necessarily the case in the above lemma that $\FrLim(\mc{K}(\mu^0))|_{\mc{L}} = \mc{M}$, as $(\mc{K}, \mc{K}(\mu^0))$ may not satisfy the conditions of \Cref{f:reducts}; we only proved \Cref{l:neat extension} in the case where $\ind$ is symmetric.  
\end{rem}

\begin{prop} \label{p:str strongly univ}
    Suppose that $\mc{M}$ is equipped with a SIR. Then there exists an action $\lambda : \mc{M} \curvearrowleft \Theta$ which is strictly strongly universal over $\lambda^0$. In fact, the set of actions which are strictly strongly universal over an action $\Theta$-isomorphic to $\lambda^0$ is comeagre in $\Se_k(\mc{M}, \Theta)$.
\end{prop}
\begin{proof}
    For each $\Omega \in \mc{D}(\Theta)$, let $\mu^{0, \Omega} = \lambda^0|_{\NFr_{\lambda^0}(\Omega) \curvearrowleft \Omega}$. By \Cref{l:strong univ rephrasing}, it suffices to show the following:
    \begin{enumerate}[label=(\roman*)]
        \item \label{Kl0 Fr class} $\mc{K}(\lambda^0)$ is a strong \Fr class, as is $\mc{K}(\mu^{0, \Omega})$ for each $\Omega \in \mc{D}(\Theta)$;
        \item \label{Kl0Omega reduct} for each $\Omega \in \mc{D}(\Theta)$, the \Fr limit $\FrLim(\mc{K}(\mu^{0, \Omega}))$ is the reduct of $\FrLim(\mc{K}(\lambda^0))$, and we also have $\FrLim(\mc{K}(\mu^{0, \Omega}))|_\mc{L} = \mc{M}$;
        \item \label{orb rich and str well cent} writing $\lambda$ for the action induced by $\FrLim(\mc{K}(\lambda^0))$, we have that $\lambda$ is orbit-rich for $\mc{M}(\NFr_\lambda(\Omega), \lambda|_\Omega)$ and $\lambda|_\Omega$ is strictly well-centralised over $\NFr_\lambda(\Omega)$.
    \end{enumerate}

    We have \ref{Kl0 Fr class} by \Cref{l:SWIR gives action amalgam}. \ref{Kl0Omega reduct} follows from \Cref{l:neat extension} and \Cref{f:reducts}. For \ref{orb rich and str well cent}, orbit-richness follows from \Cref{l:neat extension} and \Cref{l:orbit-rich}, and \Cref{l:SWIR gives action amalgam} gives that $\lambda|_\Omega$ is strictly well-centralised over $\NFr_\lambda(\Omega)$. The genericity statement follows from \Cref{f:genericity of universal actions}.
\end{proof}

\begin{cor} \label{c:seed actions SIR}
    Let $\mc{M}$ be a strong relational \Fr structure with a SIR. Then the following hold:
    
    \begin{enumerate}[label=(\roman*)]
        \item \label{SIR seed action S2} There is a pleasant structural seed action $\mc{M} \curvearrowleft \sym_2$. In fact, there is a comeagre subset of $\Se_2(\mc{M}, \sym_2)$ consisting of pleasant structural seed actions.
        \item \label{SIR seed action (S3, 1)} If there is only one isomorphism type of $\mc{A} \in [\mc{M}]^3$ with $\Aut(\mc{A}) \cong \sym_3$, then $\sym_{3}$ admits a pleasant structural seed action on $\M$. In fact, there is a comeagre subset of $\Se_3(\mc{M}, \sym_3)$ consisting of pleasant structural seed actions. 
    \end{enumerate}
    If moreover the SIR $\ind$ satisfies (Free), then the pleasant structural seed actions in question are generous. 
\end{cor}
\begin{proof}
    \ref{SIR seed action S2}: Lemma \ref{l: S2 S3 str} gives that $\mc{M}$ is a $(\sym_2, r)$-structure for some $r \geq 1$. By \Cref{p:str strongly univ}, there is a strictly strongly universal action $\lambda : \mc{M} \curvearrowleft \sym_2$ (in fact a comeagre set of such actions). By the definition of strictly strongly universal actions, $\lambda$ satisfies conditions \ref{flexible}, \ref{orbit-richness} in \Cref{p:extension by arcs structure}. It is immediate that $\lambda$ is a seed action of sets. We now check conditions \ref{seed-struc reference}, \ref{seed-struc all permutations}, \ref{seed-struc equivariant isomorphism condition} in \Cref{d:seed action on structure} (the definition of a pleasant structural seed action). Condition \ref{seed-struc reference} is trivial as $\sym_2$ is docile. Condition \ref{seed-struc all permutations} follows from the fact that $\mc{K}(\varnothing)$ is a \Fr class, and condition \ref{seed-struc equivariant isomorphism condition} follows by \cref{l:condition omega tuples comes for free}. By \Cref{c:extension by arcs structure} we have that $\lambda$ is a pleasant structural seed action on $\mc{M}$. Each strictly strongly universal action is thus a pleasant structural seed action, giving the additional genericity result.

    \ref{SIR seed action (S3, 1)}: Note that $\mc{M}$ is vertex-transitive: suppose, for a contradiction, that there are vertices $a_0$ and $b_0$ with different isomorphism types. Then using (Ex) we may produce $a_0, a_1, a_2$ which are mutually independent and $b_0, b_1, b_2$ which are mutually independent. But then $\Aut(a_0 a_1 a_2) \cong \sym_3$, $\Aut(b_0 b_1 b_2) \cong \sym_3$, contradicting the assumption that $r = 1$. We then have via Lemma \ref{l: S2 S3 str} that $\mc{M}$ is a $(\sym_3, 1)$-structure. The remainder of the argument is similar to the case of $\sym_2$, and we leave the checking to the reader.
\end{proof}

\subsection{Seed actions on \texorpdfstring{$k$}{k}-tournaments}
     
In this section, we construct a seed action of $\cyc_3 \leq \sym_3$ on $\mb{T}_2$ (recall that we have already constructed a seed action $\mb{T}_2 \curvearrowleft \sym_2$), and we construct a seed action of $\alt_k$ on the generic $k$-tournament $\mb{T}_k$ for $3 \leq k \leq 5$. 

Unlike in the previous subsections, here in order to satisfy condition \ref{seed-struc equivariant isomorphism condition} of \cref{d:seed action on structure}, we need to restrict to a subclass of $\mc{K}(\lambda^0)$. We first see this in the below Lemma \ref{l:3 seed action 2tournament}.

\subsubsection{The random tournament}

\begin{lem} \label{l:3 seed action 2tournament}
    There exists a pleasant structural seed action of $\Theta = \cyc_3 \leq \sym_3$ on the random tournament $\mb{T}_2$.
\end{lem}
\begin{proof}
    It is straightforward to check that $\mb{T}_2$ is a $(\cyc_3, 1)$-structure, using the fact that there are only two isomorphism classes of tournaments on $3$ vertices (namely, a $3$-cycle and a linear order): we leave the verification to the reader. Let $R$ denote the tournament relation of $\mb{T}_2$, let $\mc{L}$ denote the tournament language, and let $\mc{K} = \Age(\mb{T}_2)$. As noted in Example \ref{ex: structures with SWIRS}, the random tournament $\mb{T}_2$ admits a SWIR, or equivalently a standard amalgamation operator. Let $\tau$ be a generator of $\cyc_3$. Let $\mc{P}'$ be the subclass of $\mc{K}(\emp, \cyc_3)$ consisting of all $\mc{C}' \in \mc{K}(\emp, \cyc_3)$ such that the action on $\mc{C}'|_{\mc{L}}$ specified by $\mc{C}'$ satisfies the condition that each orbit consists of three points $a_0, a_1, a_2$ with $R(a_0, a_1) \wedge R(a_1, a_2) \wedge R(a_2, a_0)$ and $\tau$ acting as the cycle $(a_0 a_1 a_2)$. The argument in the proof of \cref{l:SWIR gives action amalgam} straightforwardly adapts to give that $\mc{P}'$ is a strong \Fr class with a standard amalgamation operator.

    We now check that $(\mc{K}, \mc{P}')$ has the free orbit property (considering $\mc{K}$ as the class acted on by the trivial group). By the definition of $\mc{P}'$ we immediately have $\mc{C}'|_{\mc{L}} \in \mc{K}$ for each $\mc{C}' \in \mc{P}'$. Let $\mc{C}' \in \mc{P}'$ (where potentially $\mc{C}' = \emp$), let $\mc{C} = \mc{C}'|_{\mc{L}}$, and let $\mc{D}_0 \in \mc{K}$ with $\mc{C} \sub \mc{D}_0$. Let $n = |\mc{D}_0 \setminus \mc{C}|$ and enumerate the vertices of $\mc{D}_0 \setminus \mc{C}$ as $\bar{v}_0 = (v_{0, i})_{i < n}$. Define a tournament $\mc{E}$ containing $\mc{D}_0$ by adding vertices $v_{1, i}, v_{2, i}$ for $i < n$, letting $\bar{v}_j = (v_{j, i})_{i < n}$ for $j = 1, 2$, with $\qftp_{\mc{E}}(\bar{v}_j/C \cdot \tau^j) = \qftp_{\mc{D}_0}(\bar{v}_0/C)$ for $j = 1, 2$ and $R(v_{0, i}, v_{1, i'}) \wedge R(v_{1, i'}, v_{2, i''}) \wedge R(v_{2, i''}, v_{0, i})$ for all $i, i', i'' < n$. Extend the action $\mc{C} \curvearrowleft \cyc_3$ specified by $\mc{C}'$ to $\mc{E}$ by defining $v_{j, i} \cdot \tau = v_{j + 1 \pmod{3}, i}$ for $j < 3$; it is straightforward to check that this action is by tournament automorphisms, that each element of $\mc{D}_0 \setminus \mc{C}$ lies in a different free orbit, and that $\mc{E}$ equipped with this action gives an element of $\mc{P}'$. We have thus shown the free orbit property for $(\mc{K}, \mc{P}')$.
    
    By Fact \ref{f:reducts} we have $\FrLim(\mc{P}')|_{\mc{L}} = \mb{T}_2$, and letting $\lambda : \mb{T}_2 \curvearrowleft \cyc_3$ be the action specified by $\FrLim(\mc{P}')$, we have by Lemma \ref{l:orbit-rich} that $\lambda$ is orbit-rich with respect to $\mb{T}_2$. As $\mc{P}'$ is a strong \Fr class, we have that $\lambda$ is orbit-rich with respect to $\FrLim(\mc{P}') = \mb{T}_2(\emp, \lambda)$. As $\mb{T}_2$ and $\FrLim(\mc{P}')$ both have SWIRs, we therefore have conditions \ref{flexible} and \ref{orbit-richness} of Proposition \ref{p:extension by arcs structure}.

    To apply Corollary \ref{c:extension by arcs structure}, it remains to check conditions \ref{seed-struc reference}, \ref{seed-struc all permutations}, \ref{seed-struc equivariant isomorphism condition} of Definition \ref{d:seed action on structure}. Condition \ref{seed-struc reference} is vacuously true as $\cyc_3$ is docile, condition \ref{seed-struc all permutations} follows by the definition of $\mc{P}'$ (in the case $\Omega = \cyc_3$ we take a single free orbit giving an element of $\mc{P}'$), and condition \ref{seed-struc equivariant isomorphism condition} also follows from the definition of $\mc{P}'$: in the case $\Omega = \cyc_3$, any $\Omega$-substructure consists of three points $a_0, a_1, a_2$ with $R(a_0, a_1) \wedge R(a_1, a_2) \wedge R(a_2, a_0)$ and $\tau$ acting as the cycle $(a_0 a_1 a_2)$, and so any two isomorphic $\Omega$-substructures are $\Omega$-isomorphic. We now apply Corollary \ref{c:extension by arcs structure} to give the result.
\end{proof}

\subsubsection{The random \texorpdfstring{$k$}{k}-tournament for \texorpdfstring{$3 \leq k \leq 5$}{3 <= k <= 5}} \hfill

In this subsection, we construct a generous pleasant structural seed action on the generic $k$-tournament $\mb{T}_k$. See Definition \ref{d: ht def} for the definition of a $k$-tournament.

\begin{rem}
    We warn the reader that there is another notion of $k$-hypertournament in the literature (the earliest occurrence of which known to the authors is \cite{Ass86}), where the $k$-ary relation holds for exactly one enumeration of each $k$-set. 
\end{rem}

Fix $k \in \{3, 4, 5\}$ and take $\Theta = \alt_k \leq \sym_k$. Note that $\Theta$ is robust (see Lemma \ref{l:examples}), and note that for $k = 3$ the group $\Theta$ is docile and for $k \in \{4, 5\}$ the group $\Theta$ is unruly. Let $\mc{M} = \mb{T}_k$. It is straightforward to check that $\mc{M}$ is a $(\Theta, 1)$-structure (see Definition \ref{d:theta r structure}). Let $\mc{K} = \Age(\mc{M})$. Let $H = \Theta$ (so $H$ is a compatible seed group with $s = r = 1$) and let $G = (H, T)$ be a completion of $H$. Fix an action $\lambda^0$ as at the start of Section \ref{s:construction good finite actions}.

For $k \geq 3$, there is no SWIR on $\mb{T}_k$: suppose for a contradiction that $\mb{T}_k$ has a SWIR $\ind$, and let $\mc{A}, \mc{B} \sub \mb{T}_k$ with $|A| = k-1$, $|B| = 1$ and $A \ind B$. Then the automorphism of $\mc{A}$ given by the transposition of two points (note that $\mc{A}$ has no $k$-tournament edges) extends to an automorphism of the $k$-substructure of $\mb{T}_k$ induced on $AB$ -- but this automorphism is an odd permutation, giving a contradiction. Given the absence of a SWIR, our approach in this subsection is necessarily more ad hoc.

\begin{defn} \label{d:restricted class}
    We define an action $\tld{\lambda}^0$ as follows. If $k \in \{4, 5\}$ (note that in this case $\Theta$ is unruly) we take $\tld{\lambda}^0 = \lambda^0$, and if $k = 3$ (in this case $\Theta$ is docile) we take $\tld{\lambda}^0$ to be a free action of $\Theta = \alt_3$ on a single $3$-tournament edge.
    
    Recall Definition \ref{d:theta-class}. We write $\tld{\mc{K}}(\tld{\lambda}^0)$ for the subclass of $\mc{K}(\lambda^0)$ consisting of those $\mc{B} \in \mc{K}(\lambda^0)$ such that, for all $\Omega \in \mc{D}(\Theta)$, for each $\Omega$-substructure $\mc{A} \in [\mc{B}|_{\mc{L}}]^k$, the action $\mc{A} \curvearrowleft \Omega$ specified by $\mc{B}$ is isomorphic via a $k$-tournament isomorphism to $\tld{\lambda}^0|_{\Omega}$. 
\end{defn}

Throughout this section, when we say that a $k$-tuple $\bar{e}'$ is an even/odd permutation of a $k$-tuple $\bar{e}$, this means that $\bar{e}' = (e_{\sigma(0)}, \cdots, e_{\sigma(k-1)})$ for some even/odd $\sigma \in \sym_k$.

\begin{lem} \label{l:extending sets to tournaments equivariantly}
    Let $k \geq 2$. Let $A$ be a set, let $G$ be a group with generating set $T$, and let $\phi$ be a partial action of $(G,T)$ on $A$. Let $P \sub (A)^k$ satisfy the following:
    \begin{enumerate}[label=(\roman*)]
        \item \label{tour1} for $\bar{e} \in P$, each even permutation of $\bar{e}$ lies in $P$ and no odd permutation of $\bar{e}$ lies in $P$;
        \item \label{tour2} for $g \in G$, $\bar{e}, \bar{e}' \in P$ with $\bar{e} \sub \dom(\phi_g)$, if $\phi_g(\bar{e})$, $\bar{e}'$ have the same underlying set then $\phi_g(\bar{e})$ is an even permutation of $\bar{e}'$;
        \item \label{tour3} for $g \in G$, $\bar{e} \in (A)^k$ with $\bar{e} \sub \dom(\phi_g)$, if $\phi_g(\bar{e})$, $\bar{e}$ have the same underlying set then $\phi_g(\bar{e})$ is an even permutation of $\bar{e}$. 
    \end{enumerate}
    Then there is a $k$-tournament $\mc{A}$ on the set $A$ with edge set containing $P$ and such that each $\phi_g$, $g \in G$, is a partial isomorphism of $\mc{A}$.
\end{lem}
\begin{proof}
    Let $X \sub [A]^k$ consist of the elements of $[A]^k$ which do not lie in the $\phi$-orbit of the underlying set of any element of $P$. Let $Y \sub X$ be a set containing exactly one representative for each $\phi$-orbit contained in $X$, let $Z \sub (A)^k$ be obtained by imposing an arbitrary ordering on each element of $Y$ and then closing the resulting set of $k$-tuples under even permutations. Let $E$ be the union of the $\phi$-orbits of the elements of $P \cup Z$. It is straightforward to verify that $(A, E)$ is a $k$-tournament for which each $\phi_g$ is a partial isomorphism.
\end{proof}

\begin{lem} \label{l:structure on omega tuples}
    Let $\lambda : U \curvearrowleft \Theta$ be a seed action of sets extending $\lambda^0$, and let $\phi'$ be a taut finite extension of $\lambda$. Let $C \fin U$ be a $G$-invariant set in the partial action $\phi'$, with $A_0 \sub C$. Let $\phi = \phi'|_C$. Let $m < \omega$. For each $i < m$, let $\mc{B}_i$ be a $k$-tournament with domain $B_i$ such that $\mc{A}_0 \sub \mc{B}_i$ and $B_i \sub C$. Suppose that the following conditions hold:
    \begin{enumerate}[label=(\alph*)]
        \item \label{no incompatible permutations condition} for all $g \in G$, $i, i' < m$, $\bar{e} \in R^{\mc{B}_i} \cap (\dom(\phi_g))^k$, $\bar{e}' \in R^{\mc{B}_{i'}}$, if $\phi_g(\bar{e})$, $\bar{e}'$ have the same underlying set, then $\phi_g(\bar{e})$ is an even permutation of $\bar{e}'$;
        \item \label{action by permutations} for all $i < m$, $\Omega \in \mc{D}(\Theta)$, for each $\Omega$-substructure $\mc{A} \in [\mc{B}_i]^k$ in the action $\lambda$, the action $\lambda|_\Omega : \mc{A} \curvearrowleft \Omega$ is isomorphic via a $k$-tournament isomorphism to $\tld{\lambda}^0|_\Omega$.
    \end{enumerate}
    Then there is a $k$-tournament $\mc{C}$ with domain $C$ such that $\mc{B}_i \sub \mc{C}$ for all $i < m$ and such that the following hold:
    \begin{enumerate}[label=(\roman*)]
        \item \label{restriction of lambda} $\mc{C}$ equipped with the $\Theta$-action induced by $\lambda$ lies in $\tld{\mc{K}}(\tld{\lambda}^0)$;
        \item \label{phit condition} $\phi_g$ is a $k$-tournament isomorphism for all $g \in G$. 
    \end{enumerate}
\end{lem}
\begin{proof}
    For each $\Omega \in \mc{D}(\Theta)$, take a maximal collection $\mc{S}_\Omega$ of strict $\Omega$-sets in $C$ in distinct $\phi$-orbits not intersecting $\bigcup_{i < m} [B_i]^k$. Note that, for $\Omega, \Omega' \in \mc{D}(\Theta)$, given a strict $\Omega$-set $D$ and a strict $\Omega'$-set $D'$, if $\phi_g(D) = D'$ for some $g \in G$, then $\Omega'^g = \Omega$ by strictness; by Corollary \ref{c: conj of subset of H in H} we have that $\Omega'$ is conjugate to $\Omega$ in $\Theta$ and as $\mc{D}(\Theta)$ contains exactly one representative of each conjugacy class of docile subgroups we have $\Omega = \Omega'$. Thus, for distinct $\Omega$, $\Omega'$, the set of orbits of the elements of $\mc{S}_\Omega$ and the set of orbits of the elements of $\mc{S}_{\Omega'}$ are disjoint. Let $\mc{S} = \bigcup_{\Omega \in \mc{D}(\Theta)} \mc{S}_\Omega$. By maximality of each $\mc{S}_\Omega$, the fact that $\mc{A}_0 \sub \bigcap_{i < m} \mc{B}_i$, Lemma \ref{l:stabilisers of k sets} and Lemma \ref{l: key seed action props}\ref{seedact unruly unique}, we have that each element of $[C]^k$ is in the $\phi$-orbit of some element of $\mc{S} \cup \bigcup_{i < m} [B_i]^k$.

    For each $\Omega \in \mc{D}(\Theta)$, $D \in \mc{S}_\Omega$, as $\Omega \leq \Theta = \alt_k$ and $(\lambda|_\Omega : D \curvearrowleft \Omega) \simeq (\mathbf{k} \curvearrowleft \Omega)$ by Lemma \ref{l: key seed action props}\ref{seedact docile}, there is a $k$-tournament relation $R_D$ on $D$ such that $\lambda|_\Omega$ acts by automorphisms of $(D, R_D)$ and $(\lambda|_\Omega : (D, R_D) \curvearrowleft \Omega) \simeq \tld{\lambda}^0|_\Omega$ via a $k$-tournament isomorphism. Let $\mc{R} = (\bigcup_{D \in \mc{S}} R_D) \cup (\bigcup_{i < m} R^{\mc{B}_i})$; by the statement at the end of the preceding paragraph, each element of $[C]^k$ lies in the $\phi$-orbit of the underlying set of some element of $\mc{R}$.

    We claim that $\mc{R}$ satisfies the conditions \ref{tour1}--\ref{tour3} of Lemma \ref{l:extending sets to tournaments equivariantly}. Condition \ref{tour1} follows as the $\mc{B}_i$ and $(D, R_D)$ are $k$-tournaments. Condition \ref{tour2} follows from condition \ref{no incompatible permutations condition} in the statement of the lemma and the fact that, for each $D \in \mc{S}$, we have that $(D, R_D)$ is a $k$-tournament and the $\phi$-orbit of $D$ does not intersect $\bigcup_{i < m} [B_i]^k$. Condition \ref{tour3} follows from the following facts: each $D \in \mc{S}_\Omega$ is $\Omega$-strict, the set $A_0$ is $\Theta$-strict (in the case that $\Theta$ is unruly) by Lemma \ref{l:stabilisers of k sets} and Lemma \ref{l: key seed action props}\ref{seedact unruly unique}, each element of $[C]^k$ is in the $\phi$-orbit of some element of $\mc{S} \cup \bigcup_{i < m} [B_i]^k$, and the setwise-stabiliser for each $D \in \mc{S}_\Omega$ and for $A_0$ acts on the corresponding set with an action isomorphic to the corresponding standard permutation action (respectively: $\pi : \mathbf{k} \curvearrowleft \Omega$ and $\pi : \mathbf{k} \curvearrowleft \Theta$). We may therefore apply Lemma \ref{l:extending sets to tournaments equivariantly} to $\mc{R}$ to obtain a $k$-tournament $\mc{C}$ with domain $C$ such that $\mc{B}_i \sub \mc{C}$ for all $i < m$ and $\phi_g$ is a $k$-tournament isomorphism for all $g \in G$. As $\lambda$ is a seed action of sets, we have that the expansion of $\mc{C}$ by $\lambda$ lies in $\mc{K}(\lambda^0)$. It remains to check that this expansion lies in $\tld{\mc{K}}(\tld{\lambda}^0)$.

    Let $\Omega \in \mc{D}(\Theta)$ and let $\mc{D} \in [\mc{C}]^k$ be an $\Omega$-substructure. Let $\mu^0 : \mc{D} \curvearrowleft \Omega$ be the restriction of $\lambda|_\Omega$ to $\mc{D}$. We need to show that $\mu^0$ is isomorphic to $\tld{\lambda}^0|_\Omega$ via a $k$-tournament isomorphism. By the definition of $\mc{S}$, Lemma \ref{l:stabilisers of k sets} and Lemma \ref{l: key seed action props}\ref{seedact unruly unique} (also recalling that $\phi_g$ is a $k$-tournament isomorphism for all $g \in G$), there is $\Omega' \in \mc{D}(\Theta) \cup \{\Theta\}$ and $g \in G$ such that $\mc{D}' := \phi_g(\mc{D})$ is a strict $\Omega'$-substructure, and where one of the following two cases holds: $\Omega' \in \mc{D}(\Theta)$ and $D' \in \mc{S}_{\Omega'} \cup \bigcup_{i < m} [B_i]^k$, or $\Omega' = \Theta$ and $D' = A_0$. In each case, the restriction $\nu^1 := \lambda|_{\Omega'} : \mc{D}' \curvearrowleft \Omega'$ is isomorphic to $\tld{\lambda}^0|_{\Omega'}$ via a $k$-tournament isomorphism (either by construction, by condition \ref{action by permutations} or by the fact that $\lambda$ extends $\lambda^0$).

    Define an action $\mu^1 : \mc{D}' \curvearrowleft \Omega$ by $\mu^1_\omega = \phi_g \circ \phi_\omega \circ \phi_{g^{-1}}|_{D'}$ for each $\omega \in \Omega$. Note that $\mu^1 \simeq \mu^0$ via a $k$-tournament isomorphism (as $\phi_g$ is a $k$-tournament isomorphism). As $D'$ is $\Omega'$-strict in $\phi$, we have $\Omega^g \sub \Omega'$, and so by Corollary \ref{c: conj of subset of H in H} there is $\theta \in \Theta$ with $\omega^g = \omega^\theta$ for all $\omega \in \Omega$. So $\mu^1_\omega = \lambda_\theta \circ \lambda_\omega \circ \lambda_{\theta^{-1}}|_{D'}$ for each $\omega \in \Omega$, and thus $(\mu^1 : \mc{D}' \curvearrowleft \Omega) \simeq (\lambda : \mc{D}' \curvearrowleft \Omega^\theta) \simeq \tld{\lambda}^0|_{\Omega^\theta}$ by a $k$-tournament isomorphism, and $\tld{\lambda}^0|_{\Omega^\theta} \simeq \tld{\lambda^0}|_\Omega$ via a $k$-tournament isomorphism (as $\tld{\lambda}^0_\theta$ is a $k$-tournament isomorphism), so we are done.
\end{proof}
    
\begin{lem} \label{l:tournament action}
    $\tld{\mc{K}}(\tld{\lambda}^0)$ is a \Fr class and the pair $(\mc{K},\tld{\mc{K}}(\tld{\lambda}^0))$ has the free orbit property.
\end{lem}
\begin{proof}
    To see that $\tld{\mc{K}}(\tld{\lambda}^0)$ has the strong amalgamation property: take $\mc{A}, \mc{B}_0, \mc{B}_1 \in \tld{\mc{K}}(\tld{\lambda}^0)$ with $\mc{A} \sub \mc{B}_0, \mc{B}_1$ and $B_0 \cap B_1 = A$. Let $C$ be the disjoint union of $B_0, B_1$ over $A$, take an action $C \curvearrowleft \Theta$ given by the union of the actions specified by $\mc{B}_0, \mc{B}_1$, and take a seed action of sets $\lambda$ extending this action on $C$. Take $\lambda'$ from Lemma \ref{fin taut exts exist}. Then Lemma \ref{l:structure on omega tuples} gives an amalgam $(C, \lambda) \in \tld{\mc{K}}(\tld{\lambda}^0)$. As $(\mc{A}_0, \lambda^0) \in \tld{\mc{K}}(\tld{\lambda}^0)$ embeds into each element of $\tld{\mc{K}}(\tld{\lambda}^0)$, the amalgamation property for $\tld{\mc{K}}(\tld{\lambda}^0)$ implies the joint embedding property, and it is immediate that $\tld{\mc{K}}(\tld{\lambda}^0)$ is hereditary and has countably many isomorphism types. So $\tld{\mc{K}}(\tld{\lambda}^0)$ is a \Fr class.

    To check the free orbit property for $(\mc{K}, \tld{\mc{K}}(\tld{\lambda}^0))$, it suffices to check condition \ref{fo cond2} in Definition \ref{d:free orbit property}, as condition \ref{fo cond1} follows from \ref{fo cond2}: given $C \in \mc{K}$, take a strong amalgam of $C$ with $A_0$ over $\emp$ and apply \ref{fo cond2} (taking $\mc{C}' = (A_0, \lambda^0)$ and $\mc{D}$ to be the strong amalgam just mentioned). We now check \ref{fo cond2}. Suppose we are given $(\mc{C}, \mu) \in \tld{\mc{K}}(\tld{\lambda}^0)$ (where $\mu$ is the action on $\mc{C}$), and $\mc{D} \in \mc{K}$ with  $\mc{C} \sub \mc{D}$. Extend $\mu$ to a seed action on sets $\lambda$ where each element of $D \setminus C$ lies in a different free orbit, and let $E$ be the union of the $\lambda$-orbits of the elements of $D$. Then apply Lemma \ref{l:structure on omega tuples} to $E$ with $\mc{B}_0 = D$ and $\phi' = \lambda'$.
\end{proof}

\begin{lem} \label{l:seed actions on tournaments}
    Let $\lambda$ be the action on $\mc{M} = \mb{T}_k$ specified by the \Fr limit of $\tld{\mc{K}}(\tld{\lambda}^0)$. Then $\lambda$ is a generous pleasant structural seed action.
\end{lem}
\begin{proof}
    Property \ref{seed-struc auto} in Definition \ref{d:seed action on structure} is immediate. \ref{seed-struc reference} follows from the definition of $\lambda^0$, the fact that $\tld{\mc{K}}(\tld{\lambda}^0) \sub \mc{K}(\lambda^0)$ and the fact that all $k$-tournament edges are isomorphic. \ref{seed-struc all permutations} follows from the following: all $k$-tournament edges are isomorphic; in the case that $\Theta$ is unruly, we have that $\lambda$ extends $\lambda^0$; in the case that $\Theta$ is docile, we have that $\lambda$ is specified by $\FrLim(\tld{\mc{K}}(\tld{\lambda}^0))$. \ref{seed-struc equivariant isomorphism condition} follows from the definition of $\tld{\mc{K}}(\tld{\lambda}^0)$.

    We now check \ref{seed-struc arc extension}. Let $(G, T)$ be a completion of $\Theta$, and let $\mc{F}$ be the class of finite taut extensions $\phi \geq \lambda$ (where each $\phi_g$, $g \in G$, is a partial isomorphism of $\mc{M}$). We check condition \ref{condII} in Definition \ref{d:arc extensiveness} (the definition of an arc-extensive family), leaving \ref{condI} to the reader as the argument is analogous. Let $\phi \in \mc{F}$ and $\Omega \in \mc{D}(\Theta)$. We first construct an extension $\psi$ of $\phi$ as a partial action of sets, following the final two paragraphs of the proof of Proposition \ref{p:generic action on sets}, and then we add $k$-tournament structure using Lemma \ref{l:structure on omega tuples}.

    Suppose we are given $\Omega$-isomorphic $\Omega$-substructures $\mc{A}, \mc{A}' \in [\mc{M}^k]$, enumerated as $\bar{a}$, $\bar{a}'$ so that $\bar{a} \mapsto \bar{a}'$ is an $\Omega$-isomorphism, with $A$ $\Omega$-strict for $\phi$ and $A, A'$ not in the same $\phi$-orbit. Let $\mc{O}$ be the substructure of $\mc{M}$ whose domain $O$ is the $\lambda$-orbit closure of $\supp(\phi) \cup A \cup A' \cup A_0$. By following the argument in the final two paragraphs of the proof of Proposition \ref{p:generic action on sets}, regarding $\phi$ as a partial action on the set $M$, we obtain $\tld{s}$ satisfying the requirements of \ref{condII} (not depending on $A, A'$, as Lemma \ref{l:disjointness through free extensions} only gives dependence on $\phi$) and extensions $\phi \leq \phi' \leq \psi$ such that:
    \begin{itemize}
        \item $\phi'$ is an $\Omega$-free extension of $\phi$ satisfying condition \ref{II sep}, where $B, B'$ are not in the same $\phi'$-orbit (by Lemma \ref{l:partial k-sharp extend}\ref{no more permutations}), and where $\bar{b}$, $\bar{b}$' are enumerations of $B, B'$ giving the images of $\bar{a}$, $\bar{a}'$;
        \item $\psi$ is an $\Omega$-extension by $w$-arcs from $\bar{b}$ to $\bar{b}'$ (where $w$ is any word satisfying the conditions in \ref{II arc} and where we take $s = \tld{s}$ and $N = 6$ as in the argument in the proof of Proposition \ref{p:generic action on sets}).
    \end{itemize}
    We have that $\psi$ is taut by Proposition \ref{p: small cancellation loops}.

    Let $C$ be the $\lambda$-orbit closure of $\supp(\psi) \cup A_0$. We have $O \sub C$. As $\lambda$ is given by $\FrLim(\tld{\mc{K}}(\tld{\lambda}^0))$, to produce the required extension of $\phi$ acting by partial automorphisms of $\mc{M}$, it suffices to show that there is a $k$-tournament structure $\mc{C}$ on $C$ such that equipping $\mc{C}$ with the $\lambda$-action gives an element of $\tld{\mc{K}}(\tld{\lambda}^0)$ and such that $\psi$ is a partial action by partial automorphisms of $\mc{C}$. We obtain this $k$-tournament structure $\mc{C}$ by applying Lemma \ref{l:structure on omega tuples} to $\psi$, $C$ and $\mc{O}$: condition \ref{no incompatible permutations condition} follows by Lemma \ref{l:traverse decomp} and condition \ref{action by permutations} is immediate.
\end{proof}

\subsection{Constructing pleasant structural seed actions for \texorpdfstring{$(\sym_{3},r)$}{(S3, r)}-structures with \texorpdfstring{$r>1$}{r > 1}.} \label{subs:larger seeds}
 
Recall the presentation of the standard $(\sym_{3},r)$-seed group given in \cref{d:standard s3 seed}:
\[
    H_r=\sg{\sigma,\tau,z_{1},\dots, z_{r-1}\,|\,\sigma^{2}=1,\tau^{3}=1,\tau^{-1}=\tau^{\sigma}=\tau^{\sigma^{z_{1}}}=\cdots = \tau^{\sigma^{z_{r-1}}}}.
\]

\subsubsection{Complexes and groups}
\renewcommand{\P}[0]{\mathcal{P}}
\newcommand{\C}[0]{\mathcal{C}}
We will assume the reader is familiar with the basic theory of covering spaces, as presented in \cite{Hat02} or \cite{munkres2025elements}. A succinct exposition of this theory in its combinatorial incarnation can be found in Chapter III of \cite{LS1979} or Chapter II, Sections 23, 24 of \cite{Bog08}.

\begin{definition}
	Let $H$ be a group admitting the presentation $\mathcal{P}:\sg{T\,|\, \mathcal{R}}$, meaning that 	
	$H\cong \mathbb{F}(T)/\sg{\sg{\mathcal{R}}}$, for some set $T$ of generators of $H$ and $\mathcal{R}$ some subset of the free group $\mathbb{F}(T)$, where the isomorphism maps $t\in T$ to its copy in $\mathbb{F}(T)$.
	By a $(T,\mathcal{R})$-complex we mean a particular class of  combinatorial $2$-complexes 
	$\mathcal{C}$, where each orientation of an edge in $\mathcal{C}$ is given a label from $T^{\pm 1}$, and a labelling of each $2$-cell with some relator from $\mathcal{R}$. 
    
	Note that this allows one to assign a well defined word $w(\gamma)$ in the alphabet $T^{\pm1}$ to each combinatorial path $\gamma$ in the $1$-skeleton of $\mathcal{C}$. 
	
	We require that at each vertex $v\in\mathcal{C}^{0}$:
	\begin{itemize}
		\item if one orientation of $e$ is labelled by $t$, then the opposite orientation is labelled by $t^{-1}$, 
		\item for any $s\in T^{\pm1}$ there is at most one oriented edge with label $s$ departing from $v$, 
		\item  every combinatorial path $p$ for which $w(p)$ is the reduced word representing one of the elements in $\mathcal{R}$ is a loop, and there is exactly one $2$-cell in $\mathcal{C}$ attached along $p$.
	\end{itemize}      
  If, moreover, for every vertex $v\in\mathcal{C}^{0}$ there is  exactly one outgoing edge leaving $v$ with label $t$ for each $t\in T^{\pm 1}$, then we say that $\mathcal{C}$ is \emph{full}. We will refer to the labelled skeleton of a $\mathcal{P}$-complex as a \emph{$\mathcal{P}$-graph}.
  
  Note that given a $\P$-graph $\Gamma$ there is a unique $\P$-complex of which $\Gamma$ is the $1$-skeleton, which we denote by $\overline{\Gamma}$. 
 
  One instance of particular importance is the complex $\mathcal{C}_{\mathcal{P}}$  whose $1$-skeleton consists of a single vertex $v_{0}$ with one edge with orientations labelled by $t$ and $t^{-1}$ for each $t\in T$ and exactly one $2$-cell for each relator in $\mathcal{R}$. We refer to $\mathcal{C}_{\mathcal{P}}$ as the \emph{presentation complex} of $\mathcal{P}$. 
\end{definition}

For every combinatorial complex $\C$ and $v\in \C^{0}$ there is a notion of a fundamental group of $\C$ based at $v$ defined in purely combinatorial terms, but naturally isomorphic to the fundamental group of the underlying topological space over that point. The following is equivalent to Ch. III, Proposition 2.3 in \cite{LS1979} and motivates the use of the term ``presentation complex".
\begin{fact}\label{presentation complex}
	The map $\phi:\mathbb{F}(T)\to \pi_{1}(|\mathcal{C}_{\mathcal{P}}|)$ mapping each $t\in T$ to the loop in $\mathcal{C}_{\mathcal{P}}$ labelled by $t$ factors through an isomorphism between $H$ and $\pi_{1}(\mathcal{C}_{\mathcal{P}},v_{0})$. 
\end{fact}

\newcommand{\onto}[0]{\twoheadrightarrow}
\newcommand{\CH}[0]{\mathcal{C}_{\mathcal{P},H}} 
\newcommand{\CK}[0]{\mathcal{C}_{\mathcal{P},K}}

From now on, we shall identify $H$ and $\pi_{1}(\C_{\P},v_{0})$.
The theory of covering maps can be developed in the combinatorial context, as in Chapter III of \cite{LS1979} or Chapter II in \cite{Bog08}.
A combinatorial covering map induces a covering map in the topological sense between the geometric realizations, and it is easy to see from the lifting Theorem that if the target is a $\P$-complex, then there is a unique $\P$-complex structure on the covering complex compatible with the covering map.   

The following statement gathers a few standard facts that can be found in Ch. III of \cite{LS1979}, including Proposition 3.1, and Proposition 3.4, as well as the theory of deck transformations as in Section 1.3 in \cite{Hat02}, which is also implicitly developed in the results of Ch. III, Section 4 of \cite{LS1979}.
\begin{fact}
	\label{basic covering}The labelling of a full $\P$-complex $\C$ determines a covering map from $\C$ onto $\C_{\P}$.
    For every subgroup $K\leq H$ there is a covering complex  $\C_{\P,K}$ of $\C_{\P}$, unique up to isomorphism of covering spaces, whose fundamental group with respect to some distinguished lift $v_{K}$ of $v_{0}$ is mapped isomorphically onto $K$ by the homomorphism between fundamental groups induced by the covering map $p_{K}:\C_{\P,K}\onto \C_{\P}$. 

    Write $\tilde{\C}_{\P}=\C_{\P,\{1\}}$ and $p=p_{\{1\}}$. Fix some vertex $\tilde{v}_{0}\in\tilde{\C}_{\P}^{0}$ lifting the unique vertex $v_{0}\in\C_{\P}$. This vertex determines an identification between elements in $H$ and automorphisms of $\tilde{\C}_{\P}$ (i.e., a label preserving automorphism of the graph) as follows.  

    Any given $h_{0}\in H$  acts on $\tilde{\C}_{\P}$ by sending $\tilde{v}_{0}$ to the vertex $h_{0}\cdot \tilde{v}_{0}$ which is the endpoint of the unique lift of a path $\gamma$ in $\C_{\P}$ representing $\gamma$ and starting at the vertex  $\tilde{v}_{0}$. Any vertex of $\tilde{\C}_{\P}$ is of the form $h\cdot \tilde{v}_{0}$ for some $h\in H$ and $h_{0}$ maps $h\cdot \tilde{v}_{0}$ to $(h_{0}h)\cdot\tilde{v}_{0}$. This map extends to $1$ and $2$-cells in a combinatorially unique way.
    
    The map $p$ factors as $p_{K}\circ q_{K}$ for some covering map $q_{K}:\tilde{\C}_{\P}\onto\C_{\P,K}$, where $q_{K}$ is isomorphic to the quotient of $\tilde{\C}_{\P}$ by the action of $K$ on $\tilde{\C}_{\P}$ described above.
\end{fact}
  
From \cref{presentation complex} and \cref{basic covering} we can deduce that for any $K\leq H$ the action of the elements of $T$ on the right on the coset space $K\backslash H$ can be directly read from the labelling of the oriented edges of $\C_{\P,K}$. This is described in the Corollary below, whose content corresponds to that of Propositions 4.1, 4.2, and 4.3 in \cite{LS1979}.

\begin{corollary} \label{description actions} 
	   Given $K\leq H$, there is a bijection $\psi_{K}: \CK^{0}\cong K\backslash H$ such that for every oriented edge labelled by $t\in T^{\pm 1}$ from a vertex $v$ to a vertex $v'$ we have $\psi_{K}(v')=\psi_{K}(v)\cdot t$, where the second expression refers to the obvious action on the right of $H$ on $K\backslash H$. 
\end{corollary}

We will now focus on the group $H=H_{r}$, for $r\geq 2$. Recall that $H_{r}$ admits a presentation $\P$ in the generators $\{\sigma,\tau,z_{i}\}_{i=1}^{r-1}$ with set of relators $\mathcal{R}=\{\sigma^{2},\tau^{3},\tau\tau^{\sigma^{z_{i}}}\}_{i=0}^{r-1}$, where for convenience we take $z_{0}=1$. 
 
Our goal is to use the criterion in \cref{description actions} to understand the action of the generators in $T$ on the right on $G$ (the so-called Cayley graph of $H$ with respect to $T$) as well as their action on $\sg{\sigma}\backslash H$.   

We will do this using as building blocks the $\mathcal{P}$-graphs $\Gamma_{0}$, depicted in \cref{fig:diagram GS0} and $\Delta_{0}$, depicted in \cref{fig:diagram DS0}. For $1\leq i\leq r-1$ we also consider $\Gamma_{i}$, depicted in \cref{fig:diagram GSi}, and $\Delta_{i},$ depicted in \cref{fig:diagram DSi}.

\begin{figure}
\centering
\begin{subfigure}[b]{0.45\textwidth}
\centering
\makebox[\linewidth][c]{
\begin{tikzpicture}[dot/.style={circle, fill=black, minimum size=2pt,inner sep=2pt}, >=Stealth]

\useasboundingbox (-2.4,-0.8) rectangle (2.4,3.5);

\node[dot] (T')  at ( 0.0, 3.2) {};
\node[dot] (L')  at (-2, 0) {};
\node[dot] (R')  at ( 2, 0) {};

\node[dot] (T)  at ( 0.0, 1.7) {};
\node[dot] (L)  at (-0.7, 0.7) {};
\node[dot] (R)  at ( 0.7, 0.7) {};

\draw[->] (T)  to[bend left=25]  node[left] {$\sigma$} (T');
\draw[->] (T') to[bend left=25]  node[right] {$\sigma$} (T);
\draw[->] (R)  to[bend right=25] node[left] {$\sigma$} (R');
\draw[->] (R') to[bend right=25] node[right] {$\sigma$} (R);
\draw[->] (L)  to[bend left=25]  node[right] {$\sigma$} (L');
\draw[->] (L') to[bend left=25]  node[left] {$\sigma$} (L);

\draw[->] (T)  to[bend left=15]  node[below, xshift=-0.5mm] {$\tau$} (R);
\draw[->] (R)  to[bend left=15]  node[above, yshift=-0.5mm] {$\tau$} (L);
\draw[->] (L)  to[bend left=15]  node[below, xshift=0.5mm] {$\tau$} (T);

\draw[->] (L') to[bend right=20] node[below] {$\tau$} (R');
\draw[->] (R') to[bend right=20] node[above] {$\tau$} (T');
\draw[->] (T') to[bend right=20] node[above] {$\tau$} (L');
\end{tikzpicture}
}
\caption{The $\mathcal{P}$-graph $\Gamma_{0}$.}
\label{fig:diagram GS0}
\end{subfigure}
\hfill 
\begin{subfigure}[b]{0.45\textwidth}
\centering
\makebox[\linewidth][c]{
\begin{tikzpicture}[dot/.style={circle, fill=black, minimum size=2pt,inner sep=2pt}, >=Stealth, loop/.style={looseness=20,in=-45,out=45}]

\useasboundingbox (-0.3,-0.8) rectangle (2.4,3.5);

\node[dot] (A)  at (2.25, 1.5) {};
\node[dot] (B)  at (0, 3.2) {};
\node[dot] (C)  at (0, 0) {};

\draw[->] (A)  to[loop] node[right, yshift=1.2mm] {$\sigma$} (A);

\draw[->] (A)  to[bend right=0] node[above] {$\tau$} (B);
\draw[->] (B)  to[bend right=0] node[left] {$\tau$} (C);
\draw[->] (C)  to[bend right=0] node[below] {$\tau$} (A);

\draw[->] (B)  to[bend left=35] node[right] {$\sigma$} (C);
\draw[->] (C)  to[bend left=35] node[left] {$\sigma$} (B);
\end{tikzpicture}
}
\caption{The $\mathcal{P}$-graph $\Delta_{0}$.}
\label{fig:diagram DS0}
\end{subfigure}

\vspace{2em} 

\begin{subfigure}[b]{0.45\textwidth}
\centering
\makebox[\linewidth][c]{%
\begin{tikzpicture}[dot/.style={circle, fill=black, minimum size=2pt,inner sep=2pt}, >=Stealth]

\useasboundingbox (-2.3, -0.75) rectangle (9.5, 3.5);

\node[dot] (T)  at (2, 1.5) {};
\node[dot] (L)  at (0, 3) {};
\node[dot] (R)  at (0, 0) {};

\node[dot] (T') at (9, 1.5) {};
\node[dot] (L') at (7, 3) {};
\node[dot] (R') at (7, 0) {};

\node[dot] (T1) at (4.25, 1.5) {};
\node[dot] (L1) at (2, 3) {};
\node[dot] (R1) at (2, 0) {};

\node[dot] (T2) at (6.25, 1.5) {};
\node[dot] (L2) at (4, 3) {};
\node[dot] (R2) at (4, 0) {};

\draw[->] (T)  to[bend left=0] node[above] {$\tau$} (R);
\draw[->] (R)  to[bend left=0] node[above, xshift=2mm] {$\tau$} (L);
\draw[->] (L)  to[bend left=0] node[below] {$\tau$} (T);

\draw[->] (T') to[bend right=0] node[above] {$\tau$} (L');
\draw[->] (L') to[bend right=0] node[above, xshift=-2mm] {$\tau$} (R');
\draw[->] (R') to[bend right=0] node[below] {$\tau$} (T');

\draw[->] (T1) to node[above, yshift=-0.5mm] {$z_{i}$} (T);
\draw[->] (T2) to node[below] {$z_{i}$} (T');
\draw[->] (T1) to[bend left=25] node[above] {$\sigma$} (T2);
\draw[->] (T2) to[bend left=25] node[below] {$\sigma$} (T1);

\draw[->] (L1) to node[above, yshift=-0.5mm] {$z_{i}$} (L);
\draw[->] (L2) to node[below] {$z_{i}$} (L');
\draw[->] (L1) to[bend left=25] node[above] {$\sigma$} (L2);
\draw[->] (L2) to[bend left=25] node[below] {$\sigma$} (L1);

\draw[->] (R1) to node[above, yshift=-0.5mm] {$z_{i}$} (R);
\draw[->] (R2) to node[below] {$z_{i}$} (R');
\draw[->] (R1) to[bend left=25] node[above] {$\sigma$} (R2);
\draw[->] (R2) to[bend left=25] node[below] {$\sigma$} (R1);

\end{tikzpicture}%
}
\caption{The $\mathcal{P}$-graph $\Gamma_{i}$, $1\leq i\leq r-1$.}
\label{fig:diagram GSi}
\end{subfigure}
\hfill
\begin{subfigure}[b]{0.45\textwidth}
\centering
\makebox[\linewidth][c]{%
\begin{tikzpicture}[dot/.style={circle, fill=black, minimum size=2pt,inner sep=2pt},
>=Stealth, loop/.style={looseness=15,in=120,out=240}] 

\useasboundingbox (-2.5, -0.75) rectangle (6.5, 3.5);

\node[dot] (A)  at (0, 1.5) {}; 
\node[dot] (A') at (1, 1.5) {};
\node[dot] (B') at (2, 2.5) {};
\node[dot] (C') at (2, 0.5) {};
\node[dot] (X)  at (4, 0.7) {};
\node[dot] (Y)  at (4, 2.3) {};

\draw[->] (A)  to[loop] node[left] {$\sigma$} (A);

\draw[->] (A)  to[bend left=0] node[below] {$z_{i}$} (A');
\draw[->] (C') to[bend right=0] node[right] {$\tau$} (B');
\draw[->] (B') to[bend right=0] node[left] {$\tau$} (A');
\draw[->] (A') to[bend right=0] node[below] {$\tau$} (C');
\draw[->] (Y)  to[bend right=0] node[above] {$z_{i}$} (B');
\draw[->] (X)  to[bend left=0] node[above] {$z_{i}$} (C');

\draw[->] (X)  to[bend left=25] node[left] {$\sigma$} (Y);
\draw[->] (Y)  to[bend left=25] node[right] {$\sigma$} (X);
\end{tikzpicture}%
}
\caption{The $\mathcal{P}$-graph $\Delta_{i}$, $1\leq i\leq r-1$.} 
\label{fig:diagram DSi}
\end{subfigure}
\caption{}
\end{figure}

The following is a consequence of Seifert-Van Kampen theorem. 
\begin{observation} 
	The complexes $\overline{\Gamma_{0}}$, $\overline{\Gamma_{i}}$ for $1\leq i\leq r-1$ are simply connected. The complexes $\overline{\Delta_{0}}$ and $\overline{\Delta_{i}}$ contain a distinguished vertex $u_{0}$ such that the fundamental group based at $u_{0}$ is generated by an element represented by a self-edge labelled by $\sigma$ at $u_{0}$. 
\end{observation}

\begin{definition}\label{d:trees}
	By a \emph{finite $\{1\}$-tree} we mean a finite tree $X=(V,E)$ whose vertex and edge sets are partitioned as 
	$V=\coprod_{i=0}^{r-1}V_{i}$ and $E=\coprod_{i=1}^{r-1}(E_{\tau}^{i}\coprod E_{\sigma}^{i})$ and such that the following properties are satisfied: 
	\begin{itemize}
		\item all the leaves of $X$ are in $V_{0}$, 
		\item every interior node $v\in V_{0}$ is connected for every $1\leq i\leq r-1$  to at most $3$ other vertices in $V_{i}$ by edges of type $E_{\sigma}^{i}$ and to at most $2$ other vertices of type $V_{i}$ by edges in $E_{\tau}^{i}$,
		\item every interior node $v\in V_{i}$, $1\leq i\leq r-1$ is connected to at most $2$ other vertices in $V_{0}$ by edges of type $E_{\sigma}^{i}$ and to at most $3$ other vertices of type $V_{0}$ by edges in $E_{\tau}^{i}$.
	\end{itemize}
  By the \emph{infinite $\{1\}$-tree} we mean an infinite tree with no leaves and satisfying the last two conditions above, and where the order of every vertex is maximal. 
  
  The notion of a \emph{(finite or complete infinite) $\sg{\sigma}$-tree} is almost entirely analogous, and differs only in the following two points:
  \begin{itemize}
  	\item there is a distinguished vertex $w_{0}\in V_{0}$, which for every $1\leq i\leq r-1$ is connected to at most one vertex of type $V_{i}$ by an edge of type $E_{\tau}$ and to at most two of them by an edge of type $E_{\sigma}$;  	
  	\item every vertex incident to $w_{0}$ is incident to at most one other vertex through an edge of type $E_{\sigma}$ and to at most one other vertex in $V_{0}$ through an edge of type $E_{\tau}$. 
  \end{itemize}
   
\end{definition}

\begin{definition}\label{nice p graphs}
	By a \emph{finite $(\P,\{1\})$-graph} $\Gamma$ we mean a $\P$-graph that can be constructed as a tree-like amalgam of $\P$-graphs of the form $\Gamma_{i}$ in the following sense. 
	We use the term \emph{$h$-cycle} for $h\in H$ to refer to a subgraph of a $\mathcal{P}$-graph spanned by the orbit under $\sg{h}$ of a single vertex. 
	There is some $\{1\}$-tree $X=(V,E)$ such that $\Gamma$ is a quotient of the disjoint union of a collection 
	\begin{equation*}
		\{\Lambda_{v}\,|\,0\leq i\leq r-1,\,v\in V_{i}\},
	\end{equation*}
	where $\Lambda_{v}$ is a copy of $\Gamma_{i}$ that obeys the following rules:
	\begin{enumerate}
		\item Given an edge $e$ between $v\in V_{0}$ and $v'\in V_{i}$, $1\leq i\leq r-1$, we identify:
		\begin{itemize}
			\item  a $\tau$-cycle in $\Lambda_{v}$ with a $\tau$-cycle in  $\Lambda_{v'}$ if $E\in E_{\tau}$, 
			\item a  $\sigma$-cycle in $\Lambda_{v}$ with a $\sigma$-cycle in $\Lambda_{v'}$ if $E\in E_{\sigma}$.
		\end{itemize}
		\item For every $v\in V_{0}$, every $1\leq i\leq r-1$ and every $\tau$-cycle or $\sigma$-cycle in $\Lambda_{v}$ there is at most one edge in $E$ for which the corresponding gluing involves this cycle. 
		\item For every $1\leq i\leq r-1$  and $v\in V_{i}$, and every $\tau$-cycle or $\sigma$-cycle in $\Lambda_{v}$ there is at most one edge in $E$ corresponding to a gluing involving this cycle. 
	\end{enumerate}
  The complete $(\P,\{1\})$-graph is the (clearly unique) result of gluing copies of $\Gamma_{i}$ in a similar manner according to the pattern determined by the infinite $\{1\}$-tree $X$.
  
  A \emph{finite (resp. full infinite) $(\P,\sg{\sigma})$-graph} is defined by gluing complexes over $\tau$ and $\sigma$-cycles in a similar manner, as constrained by some finite (resp. complete infinite) $\sg{\sigma}$-tree $X$. The only differences are as follows:
  \begin{enumerate}
  	\item $\Lambda_{w_{0}}=\Delta_{0}$ (see \cref{fig:diagram DS0}).
  	\item For any $v\in V_{i}$, $1\leq i\leq r-1$ adjacent to $w_{0}$ we either have $\Lambda_{v}\cong\Delta_{i}$ 
  	or $\Lambda_{v}\cong\Gamma_{i}$; more precisely:
  	\begin{itemize}
  		\item the unique possible $v$ connected to $w_{0}$ by a $E_{\tau}$ satisfies $\Lambda_{v}\cong\Gamma_{i}$ and is glued to $\Lambda_{w_{0}}$ through a $\tau$-cycle; 
  		\item each of the two potential vertices $v$ connected to $w_{0}$ by an edge of type $E_{\sigma}$ is amalgamated to $\Lambda_{w_{0}}$ through a distinct $\sigma$-cycle and we have: 
  		\begin{itemize}
  			\item if the gluing takes place along the $\sigma$-cycle which is a loop, then $\Lambda_{v}\cong\Delta_{i}$;  
  			\item if the gluing takes place along the $\sigma$-cycle of order $2$ in $\Lambda_{w_{0}}$, then $\Lambda_{v}\cong\Gamma_{i}$.  
  		\end{itemize}
  	\end{itemize}
  \end{enumerate}
  The complex $\Lambda_{v}$ for any vertex at distance at least $2$ from $w_{0}$ is given in the exact same way as for $(\P,\sg{\sigma})$-graphs. And the gluing between two vertices $v,v'\in V\setminus \{w_{0}\}$ takes place according to the same rules as for a $(\P,\{1\})$-graph. 
  
  A \emph{finite nice $\P$-graph} is a disjoint union of one $(\P,\sg{\sigma})$-graph and a finite union of finite $(\P,\{1\})$-graphs. 
\end{definition}

\begin{observation}\label{P-graphs and actions}
	If $\Gamma$ is a finite $(\P,\{1\})$-graph or the full $(\P,\{1\})$-graph, then $\overline{\Gamma}$ is simply connected. 
	In particular, the labels on the complete $(\P,\{1\})$-graph describe an action of $H$ on its vertex set isomorphic to the action by multiplication on the right of $H$ on itself. 
	
	If $\Gamma$ is a finite $(\P,\sg{\sigma})$-graph or the complete $(\P,\sg{\sigma})$-graph, then $\overline{\Gamma}$ has fundamental group generated by a loop labelled by $\sigma$. 
	In particular, the labels on the complete $(\P,\sg{\sigma})$-graph describe the action on the right of the generators of $T$ on $\sg{\sigma}\backslash H$ (and thus implicitly the entire action). 
\end{observation}
\begin{proof}
	Consider first the case in which $\Gamma$ is either a finite $(\P,\{1\})$-graph or the complete $(\P,\{1\})$-graph. It is not hard to see that $\overline{\Gamma}$ is obtained by gluing the corresponding complexes $\overline{\Gamma_{i}^{v}}$, which are simply connected, over complexes of the form $\overline{\gamma}$, where $\gamma$ is a $\tau$ or $\sigma$ cycle, which are also simply connected. The conclusion is a simple application of Seifert-Van Kampen theorem. 
	
	The other case is entirely analogous. The only difference is that exactly one of the complexes being glued has a non-trivial fundamental group of order two generated by a single edge loop labelled by $\sigma$. In both cases the second part of the conclusion follows from \cref{description actions}.
\end{proof}

\begin{prop} \label{p:general s3}
        Let $\mathcal{M}$ be a $(\sym_{3},r)$-structure with an SIR. Then a pleasant structural seed action $\lambda$ of $H_r$ on $\M$ exists. 
\end{prop}
\begin{proof}
    We prove the result in the case $r=2$, since the general case poses no additional conceptual difficulty. We write $z = z_1$. Let $S = \sg{\sigma, \tau} \leq H$, $S' = \sg{\sigma^z, \tau} \leq H$. We assume $S = \sym_3$, and we have an isomorphism $S \to S'$ given by $\sigma \mapsto \sigma^z$, $\tau \mapsto \tau$. Take $\mc{A}_0 \in [\mc{M}]^3$ with $\Aut(\mc{A}_0) \cong \sym_3$. Let $\lambda^0 : \mc{A}_0 \curvearrowleft S$ be a structural action with $(\lambda^0 : A_0 \curvearrowleft S) \simeq (\pi : \mathbf{3} \curvearrowleft S)$.
    
    We construct an action $\lambda : \mc{M} \curvearrowleft H$ satisfying the following conditions:
    \begin{enumerate}[label=(\Roman*)]
        \item \label{universal copies} $\lambda|_S$ is strictly strongly universal over $\lambda^0$;
        \item \label{A1 copy} there is $\mc{A}_1 \in [\mc{M}]^3$ with $\Aut(\mc{A}_1) \cong \sym_3$ and $\mc{A}_1 \not\cong \mc{A}_0$ such that $A_1$ is $S'$-invariant;
        \item \label{free action} $\lambda$ has infinitely many orbits, with one $\lambda$-orbit isomorphic to $\sg{\sigma}\backslash H \curvearrowleft H$ and the other orbits free;
        \item \label{many orbits} $\lambda$ is orbit-rich for $\M(A_0, \lambda|_S)$.
    \end{enumerate}
    Given $\lambda$ satisfying the above conditions, we show that $\lambda$ is a pleasant structural seed action as follows: by condition \ref{universal copies} we have that $\lambda|_\Omega$ is strictly well-centralised for all $\Omega \in \mc{D}(\Theta)$ (condition \ref{p:extension by arcs structure}\ref{flexible}). Also by condition \ref{universal copies}, the action $\lambda|_S$ is orbit-rich for $\mc{M}(\NFr_\lambda(\Omega), \lambda|_\Omega)$ for each $\Omega \in \mc{D}(\Theta)$, and combining this with condition \ref{many orbits}, it is straightforward to see that $\lambda$ is orbit-rich for $\mc{M}(\NFr_\lambda(\Omega), \lambda|_\Omega)$ for each $\Omega \in \mc{D}(\Theta)$. Thus $\lambda$ also satisfies condition \ref{p:extension by arcs structure}\ref{orbit-richness}. We now check conditions \ref{d:seed action on structure}\ref{seed-struc reference}, \ref{seed-struc all permutations}, \ref{seed-struc equivariant isomorphism condition} in the definition of a pleasant structural seed action (\Cref{d:seed action on structure}): for \ref{d:seed action on structure}\ref{seed-struc reference}, note that $A_0$ is $S$-invariant, and use conditions \ref{free action}, \ref{A1 copy} and \Cref{l: key seed action props}\ref{seedact unruly unique}. \ref{d:seed action on structure}\ref{seed-struc all permutations} follows by \ref{universal copies}, and \ref{d:seed action on structure}\ref{seed-struc equivariant isomorphism condition} follows by \Cref{l:condition omega tuples comes for free}. By \Cref{c:extension by arcs structure}, we then have that $\lambda : \mc{M} \curvearrowleft H$ is a pleasant structural seed action.

   The next step is to construct $\lambda : \mc{M} \curvearrowleft H$ satisfying \ref{universal copies}--\ref{many orbits}.
    
    By \cref{p:str strongly univ}, there is a strictly strongly universal action $\mu : \mc{M} \curvearrowleft S$ over $\lambda^0$. Let $\mc{A}_1 \in \mc{K} = \Age(\mc{M})$ with $(A_1 \curvearrowleft \Aut(\mc{A}_1)) \simeq (\mathbf{3} \curvearrowleft \sym_3)$ and $\mc{A}_1 \not\cong \mc{A}_0$. By \Cref{l:neat extension} (which gives the existence of neat extensions), there is $\mc{A}_1 \sub \mc{A}'_1 \in \mc{K}$ and a free action $\nu^1 : \mc{A}'_1 \curvearrowleft S$ extending the action $\mc{A}_1 \curvearrowleft \sg{\tau}$ where $\tau$ acts by cyclically permuting the vertices of $\mc{A}_1$. The argument in \Cref{l:SWIR gives action amalgam} shows that the class of structures in $\mc{K}$ equipped with arbitrary actions of $S$ is a strong \Fr class with a canonical amalgam extending the canonical amalgam of $\mc{K}$ induced by the SIR on $\mc{M}$. As $(\mc{A}_0, \lambda^0)$, $(\mc{A}'_1, \nu^1)$ lie in this class, by taking the canonical amalgam of $(\mc{A}_0, \lambda^0)$, $(\mc{A}'_1, \nu^1)$ over $\varnothing$, noting that the result lies in $\mc{K}(\lambda^0)$ and using the extension property for $(\mc{M}, \mu)$ over $(\mc{A}_0, \lambda^0)$, we may assume $A_0 \ind \bigcup_{s \in S} \mu_s(A_1)$.
    
    For simplicity, write $\mathcal{N}=\M(\mu,A_{0})$ and $\mathcal{H}=\mathcal{K}(\lambda^{0})$. Then $\mathcal{H}$ is a strong \Fr class and $\mathcal{N}$ is isomorphic to its \Fr limit. We also write $\mathcal{K}_{\Omega}$, for $\Omega\in\{\sg{\sigma},\sg{\tau}\}$ for the class $\mathcal{K}(\mu_{\restriction A_{0}\times \Omega})$ of actions of $\Omega$ on $\mathcal{K}$-structures, and $\mathcal{N}_{\Omega}$ for the $\mathcal{N}_{\restriction \mathcal{L}(A_{0},\Omega)}$. 
     
    Consider the collection $\mc{X}$ of triples $\phi=(\phi_{\sigma},\phi_{\tau},\phi_{z})$ satisfying the following properties, where  $\supp^{*}(\phi)$ is the union of the $S$-orbits of the elements of $\supp(\phi_{z})$.
    \begin{enumerate}[label=(\roman*)]
    	  \item $\phi_{\sigma}=(\mu_{\sigma})_{\restriction \supp^{*}(\phi)}$, $\phi_{\tau}=(\mu_{\tau})_{\restriction \supp^{*}(\phi)}$ 
        \item the labelled graph representing the restrictions of the maps $\mu_{\sigma},\mu_{\tau},\phi_{z}$ to $\supp^{*}(\phi)$ is a nice finite $\P$-graph in the sense of \cref{nice p graphs};
        \item there is a unique non-free orbit of $S'$ under $\phi$ and it is equal to $A_{1}$. 
    \end{enumerate}
   
   Note that all the definition does is to describe a particular type of partial actions of $H$ on $\M$. Since $H$ is not a completion of a given seed group, but rather a prospective group itself, we are not able to use the already established toolkit for dealing with taut extensions, although most of the intuition gained in our treatment of the concept carries through. In particular, we can define a partial isomorphism $\phi_{g}$ between substructures of $\supp^{*}(\phi)$ for every $g\in H$, as well as the notion of extension $\leq$ in an entirely analogous way.  
   
   Every $\phi\in\mc{X}$ determines some substructure $\phi_{\restriction \mathcal{H}}\in\mathcal{H}$ of $\mathcal{N}$ in an obvious way, by forgetting the map $\phi_{z}$ (the underlying set is identical to a subset of $M$ containing $A_{0}\cup A_{1}$). 
   

    \begin{lem}
    	To produce $\lambda:\mc{M} \curvearrowleft H$ satisfying \ref{universal copies}--\ref{many orbits}, it suffices to show that $\mc{X} \neq \varnothing$ and:      
    \begin{enumerate}
        \item[$(\star)$] \label{tp} for every $\phi \in \mc{X}$, $B \fin M$, $\eps = \pm 1$ and $a \in M \setminus \dom(\phi_{z^\eps})$,  there is some extension $\psi \in \mc{X}$ of $\phi$ with $a \in \dom(\psi_{z^{\eps}})$ such that $\supp^*(\psi) \cap B\subseteq \supp^*(\phi)$  such that for all $c, c' \in \supp^*(\phi)\cup\{a\}$, if $c$, $c'$ belong to different $\phi$-orbits, then they belong to different $\psi$-orbits. 
    \end{enumerate}   
    \end{lem}
    \begin{subproof}
        Let $\mc{N} = \mc{M}(A_0, \lambda|_S)$, and let $\mc{K}' = \Age(\mc{N})$. Let $a_0, a_1, \cdots$ be an enumeration of the vertices of $M$. 
        Our goal is to construct our action as the union $\lambda$ of an infinite chain of finite extensions $\phi^{0}\leq \phi^{1}\leq\phi^{2}\dots$ in $\mc{X}$. To ensure that the union of such a chain is an action it suffices to use a standard back-and-forth argument consisting of enumerating $M$ as $\{c_{l}\}_{l\geq 0}$ and to use \ref{tp} to ensure that the following condition is satisfied:
        \begin{enumerate}[label=(\alph*)]
        	\item\label{first condition X} $c_{l}\in\dom(\phi^{2l+1}_{z})$ and $c_{l}\in\im(\phi^{2l+2}_{z})$ for every $l\geq 0$.
        \end{enumerate}

        Such an action $\lambda$ extends $\mu$ by construction, so the satisfaction of condition \ref{universal copies} is automatic. Condition \ref{A1 copy} is likewise automatic from the definition of $\mc{X}$. 
        It also follows from the construction, together with \cref{P-graphs and actions} that one of the orbits of $\lambda$ is isomorphic to the action of $H$ on $\sg{\sigma}\backslash H$ on the right, while all the other orbits are free. What is still missing is a way to ensure that there are infinitely many free orbits, as well as to ensure that condition \ref{many orbits} is satisfied. This requires that we refine the back-and-forth procedure as follows. 
        
        Together with our increasing sequence $(\phi^{n})_{n\geq 0}\subseteq\mc{X}$, we will also construct an increasing sequence of symmetric binary relations $(R_{n})_{n\geq 0}$.
        
        The relations $R_{n}$ will be a record of pairs of elements in $M$ which we want to end up in different orbits under the global action $\lambda$. We will use this to deduce that properties \ref{free action} and \ref{many orbits} hold.        
        
        Let $((B^{}_{k},B'_{k}))_{k\geq 1}$ be a list that enumerates all the pairs of finite substructures $(B,B')$ of $\mathcal{N}$ such that $B\cap B'=\sg{\varnothing}$.
        We ensure the validity of \ref{many orbits} by guaranteeing that at every stage $n\geq 0$ the following properties are satisfied:       
       	\begin{enumerate}[label=(\alph*)] 
       		\setcounter{enumi}{1}
      		\item \label{second condition X}For all $a,b\in M$ such that $\phi^{n}_{h}(a)=b$ for some $h\in H$, we have $\{(a,b),(b,a)\}\cap R_{n}=\varnothing$.
      		\item \label{third condition X} For all $a,b,c\in M$ such that $\phi^{n}_{h}(a)=b$ for some $h\in H$ and $(a,c)\in R_{n}$ we have $(b,c)\in R_{n}$.
      		\item \label{fourth condition X} For every $n\geq 1$ there is some copy $C'_{n}$ of $B'_{n}$ over $B_{n}$ as a structure in $\mathcal{H}$ such that $(C'_{n}\setminus \sg{\emp})\times B_{n}\subseteq R_{n+1}$. 
      	\end{enumerate}
      	 Conditions \ref{second condition X} and \ref{third condition X} will ensure that $R=\bigcup_{n\geq 0}R_{n}$ is in the complement of the orbit equivalence relation by the action of $\lambda$. Together with \ref{fourth condition X} this will imply \ref{many orbits}, which together with what we already know will also imply \ref{free action} in full. 
        
        For our induction step, assume that $\phi^{n},R_{n}$ are given. 
        First take some extension $\psi\in\mc{X}$ of $\phi^{n}$ such that condition \ref{first condition X} is satisfied. Let $C$ be the union of the $\mu$-orbit of $c_{\lfloor \frac{n}{2}\rfloor}$ and $\supp^{*}(\phi^{n})$. Since $\mathcal{H}$ is a strong \Fr class, whose limit is $\mathcal{N}$, 
        we can find some element $g\in\Aut(\mathcal{N})$ fixing $C$ such that the partial action $\phi^{n+1}=g(\psi)$ obtained from conjugating $\psi$ by $g$ is such that no two elements in the same $\phi^{n+1}$ orbit are related by $R_{n}$ (note that here we use the fact that distinct $\phi^{n}$-orbits are mapped to distinct $\psi$-orbits). Note that condition \ref{second condition X} is still satisfied by $R_{n}$ with respect to the extended action $\phi^{n+1}$.  
                
        We then move on to extending $R_{n}$ to $R_{n+1}$. Write
        \begin{equation*}
        	D=B_{n}\cup\supp^{*}(\phi)\cup\{a\,|\,\exists b \,\,(a,b)\in R_{n}\}.
        \end{equation*} 
        Once again, the fact that $\mathcal{H}$ is a strong \Fr class and $\mathcal{N}$ its limit implies that we can choose some isomorphic (as a structure in $\mathcal{H}$) copy $C'_{n}$ of $B'_{n}$ over $B_{n}$ with $C'_{n}\cap D=\sg{\varnothing}$. 
        To conclude, we let $R'_{n+1}$ be the result of first adding to $R_{n}$ all pairs of the form $(d,c)$ and $(c,d)$ for $c\in C'_{n}\setminus \sg{\varnothing}$ and $d\in D$. It is clear that $R'_{n+1}$ still satisfies condition \ref{second condition X} with respect to $\phi^{n+1}$. To conclude we let
        $R_{n+1}\supseteq R'_{n+1}$ be the closure of $R_{n+1}'$ under the action of $\phi^{n+1}$, i.e., the result of adding to $R'_{n+1}$ all pairs of the form $(\phi^{n+1}_{h}(a),\phi^{n+1}_{h'}(b))$, where $h,h'\in H$ and $(a,b)\in R'_{n+1}$, $a\in\dom(\phi^{n+1}_{h})$, $b\in\dom(\phi^{n+1}_{h'})$. It is easy to check that now $(\phi^{n+1},R_{n+1})$ satisfies \ref{second condition X}, \ref{third condition X}, while \ref{fourth condition X} follows from the construction of $R'_{n+1}$.
    \end{subproof}
  
\begin{figure} 
\centering
\begin{tikzpicture}[dot/.style={circle, minimum size=2pt,inner sep=0pt, }, >=Stealth,
loop/.style={looseness=15,in=120,out=60}]

\node[dot] (A)  at ( 1.25,  1.5) {$a_{0,0}$};
\node[dot] (B)  at (-1,  3) {$a_{0,1}$};
\node[dot] (C)  at (-1,  0) {$a_{0,2}$};

\node[dot] (A')  at (3.5,  1.5) {$a_{1,0}$};
\node[dot] (B')  at (5.75,  3) {$a_{1,2}$};
\node[dot] (C')  at (5.75,  0) {$a_{1,1}$};
\node[dot] (X)  at (7.75,  0.4) {$b_{1}$};
\node[dot] (Y)  at (7.75,  2.6) {$b_{2}$};

\draw[->] (A)  to[loop]  node[above] {$\sigma$} (A);

\draw[->] (A)  to[bend right=0]  node[above] {$\tau$} (B);
\draw[->] (B)  to[bend right=0]  node[left] {$\tau$} (C);
\draw[->] (C)  to[bend right=0]  node[below] {$\tau$} (A);

\draw[->] (B)  to[bend left=35]  node[right] {$\sigma$} (C);
\draw[->] (C)  to[bend left=35]  node[left] {$\sigma$} (B);

\draw[->] (A)  to[bend left=0]  node[below] {$z$} (A');
\draw[->] (C') to[bend right=0]  node[right] {$\tau$} (B');
\draw[->] (B') to[bend right=0]  node[left, xshift=2mm, yshift=1.8mm] {$\tau$} (A');
\draw[->] (A') to[bend right=0]  node[below] {$\tau$} (C');
\draw[->] (Y)  to[bend right=0]  node[below] {$z$} (B');
\draw[->] (X)  to[bend left=0]  node[below] {$z$} (C');

\draw[->] (X)  to[bend left=25]  node[left] {$\sigma$} (Y);
\draw[->] (Y)  to[bend left=25]  node[right] {$\sigma$} (X);
\end{tikzpicture}
\caption{An illustration of \cref{l:X existence}. }
\label{fig:basic X partial action}
\end{figure}

    \begin{lem}\label{l:X existence}
    	 $\mc{X} \neq \varnothing$.
    \end{lem}
    \begin{subproof}
        We construct $\phi\in\mathcal{X}$ as follows. For $i=0,1$ we enumerate $A_{i}=\{a_{i,l}\}_{l=0}^{2}$, where $a_{0,0}$ is the point fixed by $\lambda^{0}_{\sigma}=\mu_{\sigma}$ and $a_{1,0}\in A_{1}$ an arbitrary point in $A_{1}$ to be fixed by $z^{-1}\sigma z$.   
     	  
     	 Let $C=A_{0}\cup\bigcup_{s\in S}\mu_{s}(A_{1})$. By existence, there are $b_{1},b_{2}\in M$ such that
     	\begin{itemize}
     		\item $a_{0,0}b_{1}b_{2}\cong a_{1,0}a_{1,1}a_{1,2} $,
     		\item and $C\ind_{a_{0,0}} b_{1}b_{2}$.
     	\end{itemize}
       It follows from (Sta) that $(\mu_{\sigma})_{\restriction C}\cup\{(b_{1},b_{1}),(b_{2},b_{1})\}$ is a partial isomorphism. 
       
       The universality of $\mu_{\restriction M\times \sg{\sigma}}$ over the fixed point $A_{0}$, together with the fact that $\mu$ is orbit-rich with respect to $\mathcal{N}_{\sg{\sigma}}$ allow us to choose $b_{1}$ and $b_{2}$ in such a way that:
       \begin{itemize}
       	\item $\mu_{\sigma}$ exchanges $b_{1}$ and $b_{2}$;
       	\item the orbit of $b_{i}$ under $\mu$ is disjoint from $C$. 
       \end{itemize}
       We conclude that the partial map $\phi_{z}=\{(a_{0,0},a_{1,0}),(b_{1},a_{1,1}),(b_{2},a_{1,2})\}$, determines a member $\phi$ of $\mathcal{X}$.
     \end{subproof}
     
     We divide the proof of property \ref{tp} into two cases, dealt with in the sublemmas \ref{l:case 1} and \ref{l:case -1} below. For that purpose, note that as long as one is able to find an extension $\psi\in\mathcal{X}$ of $\phi$ containing $a$ in the domain of $\psi_{z}$ (resp. its codomain) and mapping distinct orbits of $\supp^{*}(\phi)\cup\{a\}$ under the partial action $\phi$ to distinct orbits under the action of $\psi$ we are done, since the condition $\supp^{*}(\psi)\cap B\subseteq\supp^{*}(\phi)$ can be always ensured a posteriori by conjugating $\psi$ by an element of the group fixing $\supp^{*}(\phi)\cup a\cdot_{\mu}S$, using the strong universality of $\mu$.

     \begin{lemma}
     	\label{l:case 1} If $\epsilon=1$, then the extension in \ref{tp} exists. 
     \end{lemma} 
     \begin{subproof}
     	We follow the diagram in \cref{fig:diagram epsilon=1}.	Write $a_{0}=a$. Notice that our assumption $\phi\in\mc{X}$ implies that $a'_{0}=\phi_{\sigma}(a_{0})\notin\dom(\phi_{z})$. Let $D=\supp^{*}(\phi)\cup a\cdot_{\mu}S$.
     	
     	We first claim that we can find new elements $a_{i},a'_{i}$, $i=1,2$
     	such that the following holds:
     	\begin{enumerate}[label=(\roman*)]
     		\item\label{condition 1} $\mu_{\sigma}(a_{i})=a'_{i}$ for $i=1,2$; 
     		\item\label{condition 2} the partial map  
     		\begin{equation*}
     			\phi_{\tau^{z^{-1}}}\cup\{(a_{0},a_{1}),(a_{1},a_{2}),(a_{2},a_{0}),(a'_{1},a'_{0}),(a'_{2},a'_{1}),(a'_{0},a'_{2})\}
     		\end{equation*}
     		is a partial isomorphism;
     		\item \label{condition 3} $a_{1}$ and $a_{2}$ are in distinct $\mu$-orbits.
     	\end{enumerate}
     
       To do this, we proceed as follows. Let $D_{0}=\dom(\phi_{z})$, and let 
       $S^{*}=(S')^{z^{-1}}:= \sg{\sigma, \tau^{z^{-1}}}$ on $\dom(\phi_{z})$ and note that the fact that $\phi$ is in $\mathcal{X}$ implies that $\phi$ restricts to a full action of $(S')^{z^{-1}}:= \sg{\sigma, \tau^{z^{-1}}}$ on $D_{0}$.
      
      Then take a neat extension of the action of $S^{*}$ on $D_{0}$ and 
      $\mu_{\restriction (D_{0}\cup\{a_{0},a'_{0})\times\sg{\sigma}\}}$. 
      This yields $a_{i},a'_{i}$, $i=1,2$ in $M$, as well as isomorphisms $\theta_{\tau^{z^{-1}}}$ and $\theta_{\sigma}$ of $D'_{0}=D_{0}\cup\{a_{i},a'_{i}\}_{i=1,2}$ inducing an action of $S^{*}$ on $D'_{0}$. At this point condition \ref{condition 2} is verified by any tuple in 
      the orbit of $a_{1}a_{2}a'_{1}a'_{2}$ over $D_{0}$. The other two will follow from a careful choice of a tuple within that orbit. 
      
      Note that we may require 
      $D\ind_{D_{0}}\{a_{i},a'_{i}\}_{i=1,2}$, which ensures by (Sta) that 
      $\rho_{\sigma}=\mu_{\sigma}\cup\theta_{\sigma}$ is a partial isomorphism of order two. To conclude, we note that the  universality of $\mu_{\restriction \sg{\sigma}}$ and the fact that $\mathcal{N}$ is orbit-rich with respect to $\mathcal{N}_{\sg{\sigma}}$, respectively, can be used to find a conjugate of $(a_{i}a_{i}')_{i=1,2}$ over $D$ so that:
      \begin{itemize}
      	\item $\rho_{\sigma}\subseteq\mu_{\sigma}$, which implies condition \ref{condition 1};
      	\item $a_{1}$ and $a_{2}$ lie in different $\mu$-orbits, i.e. condition \ref{condition 3}.
      \end{itemize}
      
      Let $D'=D\cup\{a_{i},a'_{i}\}_{i=1,2}$.     
      Finally, to conclude we need to find new elements $b_{i},b'_{i}$, $i\in\mathbf{3}$ such that: 
      \begin{enumerate}[label=(\alph*)]
      \item \label{t iso} $\psi_{z}=\phi_{z}\cup\{(a_{i},b_i),(a'_{i},b'_{i})\}$ is a partial isomorphism,
      \item \label{b rot} the map $\zeta=\{(b_{i},b_{i+1}),(b'_{i},b'_{i-1})\}_{i\in \mathbf{3}}$, where the indices $i$ are taken modulo $3$, is contained in $\mu_{\tau}$. 
      \item \label{distinct orbits} $b_{i}$ and $b'_{i}$ are in distinct $\mu$-orbits.
      \end{enumerate}
      Notice that we may satisfy condition \ref{t iso} by taking $\psi_{z}$ to be a $(+)$-independent extension of 
      $\phi_{z}$ from $D'$.\footnote{Recall that this means that $D\ind_{\im(\phi_{z})}\{b_{i},b'_{i}\}_{i\in\mathbf{3}}$.}
      This implies by (Sta) that the map $\zeta\cup(\mu_{\tau})_{\restriction D'}$ is a partial isomorphism. Another use of the universality of $\mu_{\restriction \sg{\tau}}$ allows us to assume that $\zeta\subseteq\mu_{\tau}$, so that \ref{b rot} holds, while orbit-richness of $\mu$ with respect to $\mathcal{N}_{\sg{\tau}}$ allows us to further assume that \ref{distinct orbits} is satisfied.
       
      It is easy to check from the diagram that the resulting $\psi_{z}$ determines a member of $\mathcal{X}$.
     \end{subproof}  
     
\begin{figure}[t]
\begin{tikzpicture}[
dot/.style={circle, minimum size=2pt,inner sep=0pt},
>=Stealth
]
\node[dot] (T)  at ( 2.25,  1.5) {$b^{\vphantom{\prime}}_{1}$};
\node[dot] (L)  at (0,  3) {$b^{\vphantom{\prime}}_{0}$};
\node[dot] (R)  at (0,  0) {$b^{\vphantom{\prime}}_{2}$};

\node[dot] (T')  at ( 9.75,  1.5) {$b'_{1}$};
\node[dot] (L')  at (7.5,  3) {$b'_{0}$};
\node[dot] (R')  at (7.5,  0) {$b'_{2}$};

\node[dot] (T1)  at ( 4.25,  1.5) {$a^{\vphantom{\prime}}_{1}$};
\node[dot] (L1)  at (2,  3) {$a^{\vphantom{\prime}}_{0}$};
\node[dot] (R1)  at (2,  0) {$a^{\vphantom{\prime}}_{2}$};

\node[dot] (T2)  at ( 6.25,  1.5) {$a'_{1}$};
\node[dot] (L2)  at (4,  3) {$a'_{0}$};
\node[dot] (R2)  at (4,  0) {$a'_{2}$};

\draw[->] (T)  to[bend left=0]  node[above] {$\tau$} (R);
\draw[->] (R)  to[bend left=0]  node[above, xshift=2mm] {$\tau$} (L);
\draw[->] (L)  to[bend left=0]  node[below] {$\tau$} (T);

\draw[->] (T')  to[bend right=0]  node[above] {$\tau$} (L');
\draw[->] (L')  to[bend right=0]  node[above, xshift=-2mm] {$\tau$} (R');
\draw[->] (R')  to[bend right=0]  node[below] {$\tau$} (T');

\draw[->] (T1)  to  node[below] {$z$} (T);
\draw[->] (T2)  to  node[below] {$z$} (T');
\draw[->] (T1)  to[bend left=25]  node[above] {$\sigma$} (T2);
\draw[->] (T2)  to[bend left=25]  node[below] {$\sigma$} (T1);

\draw[->] (L1)  to  node[below] {$z$} (L);
\draw[->] (L2)  to  node[below] {$z$} (L');
\draw[->] (L1)  to[bend left=25]  node[above] {$\sigma$} (L2);
\draw[->] (L2)  to[bend left=25]  node[below] {$\sigma$} (L1);

\draw[->] (R1)  to  node[below] {$z$} (R);
\draw[->] (R2)  to  node[below] {$z$} (R');
\draw[->] (R1)  to[bend left=25]  node[above] {$\sigma$} (R2);
\draw[->] (R2)  to[bend left=25]  node[below] {$\sigma$} (R1);

\end{tikzpicture}

\caption{Illustration of \cref{l:case 1}. It serves also as an illustration of \cref{l:case -1} after exchanging the role of the letters $a$ and $b$.}
\label{fig:diagram epsilon=1}
\end{figure}
     
     \begin{lemma}
     	\label{l:case -1} If $\epsilon=-1$, then the extension in \ref{tp} exists. 
     \end{lemma} 
     \begin{subproof}
      
  	  Let $a=a_{0}$, $a_{1}=\mu_{\tau}(a_{0})$, $a_{2}=\mu_{\tau^{2}}(a_{0})$ and $A=\{a_{i}\}_{i\in\mathbf{3}}$. Write $D=\supp^{*}(\phi)\cup a\cdot_{\mu}S$. 
    	First, claim that we can find a triple of points $A'=\{a'_{i}\}_{i\in\mathbf{3}}$ such that:
    	\begin{itemize}
    		\item  $\phi_{\sigma^{z^{-1}}}\cup\{(a_{i},a'_{i})\}_{i\in\mathbf{3}}$ is a partial isomorphism,
    		\item $\{(a'_{1},a'_{0}),(a'_{2},a'_{1}),(a'_{0},a'_{2})\}\subseteq\mu_{\tau}$. 
    	\end{itemize}
    	 For this we may write $D_{0}=\im(z)$ and consider a neat extension of 
    	 $\mu_{\restriction  D_{0}\cup\{a_{i}\}_{i\in\mathbf{3}}}$ and the action of $S'$ on $D_{0}$ given by $\phi$. This is an action of $S'$ on $D_{0}a_{0}'a_{1}'a_{2}'$ which ensures that the first condition above is satisfied up to replacing $a'_{0}a'_{1}a'_{2}$ by any conjugate over $D_{0}$. 
    	 
    	 By first taking $D\ind_{D_{0}}a'_{0}a'_{1}a'_{2}$ and then replacing $a'_{0}a'_{1}a'_{2}$ by a suitable conjugate over $D$ using the universality of $\mu_{\restriction \sg{\tau}}$, as in the proof of \cref{l:case 1}, we can ensure that the second condition above is also satisfied. 
    	
    	Let $D'=D\cup\{a'_{i}\}_{i\in\mathbf{3}}$. To conclude, we only need to find $b_{i},b'_{i}, i\in\mathbf{3}$ in such a way that 
    	\begin{enumerate}[label=(\roman*)]
    		\item \label{first} $\mu_{\sigma}(b_{i})=b'_{i}$ for all $i\in\mathbf{3}$,
    		\item \label{second} the map 	$\psi_{z}=\phi_{z}\cup\{(b_{i},a_{i}),(b'_{i},a'_{i})\}_{i\in\mathbf{3}}$ is an isomorphism.
    		\item \label{third} $b_{i}$ and $b_{j}$ belong to different $\mu$-orbits for distinct $i,j\in\mathbf{3}$. 
    	\end{enumerate}
     
     The argument is again very similar to the one in \cref{l:case 1}. 
     To begin with, note that we may first choose the new points so that $\psi_{z}$ is a $(+)$-independent extension of $\phi_{z}$ from $D'$. Then not only we have 
     \ref{second}, but also by (Sta) that $(\mu_{\sigma})_{\restriction D'}\cup\{(b_{i},b'_{i}),(b'_{i},b_{i})\}_{i\in\mathbf{3}}$ is a partial isomorphism.
     Conditions \ref{first} and \ref{third} can be then ensured using the universality of $\mu_{\restriction \sg{\sigma}}$ and the fact that $\mu$ is orbit-rich with respect to $\mathcal{N}_{\sg{\sigma}}$, respectively.
     \end{subproof}
      This concludes the proof of \cref{p:general s3}.   
    \end{proof}

\section{The main theorem for sharply \texorpdfstring{$k$}{k}-homogeneous actions} \label{s:main theorem}
 
\begin{defn}
    Let $\mc{M}$ be a $(\Theta, r)$-structure, $H$ a compatible seed group and $(G, T)$ a completion of $H$. We define the following spaces of structural actions:
    \begin{align*}
        \Sh_k(\mc{M}, G) &= \{\mu \in \Act(\mc{M}, G) \mid \mu \text{ is } k\text{-sharp}\},\\
        \ShHomog_k(\mc{M}, G) &= \{\mu \in \Act(\mc{M}, G) \mid \mu \text{ is sharply } k\text{-homogeneous}\}.
    \end{align*}
    In the below Theorem \ref{t:main}, we consider the following subspaces of the above:
    \begin{itemize}
        \item in the case $\Theta = \sym_2$: let $\ShChar_2(\mc{M}, G) \sub \Sh_2(\mc{M}, G)$ consist of the actions in $\Sh_2(\mc{M}, G)$ where each involution of $G$ has no fixed points, and let $\ShHChar_2(\mc{M}, G) = \ShChar_2(\mc{M}, G) \cap \ShHomog_2(\mc{M}, G)$;
        \item in the case $\Theta = \sym_3$: let $\ShChar_3(\mc{M}, G) \sub \Sh_3(\mc{M}, G)$ consist of the actions in $\Sh_3(\mc{M}, G)$ where each involution has exactly one fixed point and where each element of $G$ of order $3$ has no fixed points, and let $\ShHChar_3(\mc{M}, G) = \ShChar_3(\mc{M}, G) \cap \ShHomog_3(\mc{M}, G)$.
    \end{itemize}
\end{defn}
 
In the below Theorem \ref{t:main}, recall the definition of a $(\Theta, r)$-structure (see \cref{d:theta r structure}), and recall the definition of freeness of a SWIR (\cref{d:freedom}). Theorem \ref{t:main} is a stronger, more precise version of Theorem \ref{t:main_intro}: it gives additional genericity results, and includes the additional case \ref{sir 3, 1}.

\begin{thm} \label{t:main}
    Let $\mc{M}$ be a relational \Fr structure with strong amalgamation. 
    \begin{enumerate}[label=(\Roman*), ref=(\Roman*)]
        \item \label{case swir} Let $k \geq 1$. Suppose $\M$ has a SWIR $\ind$ and substructures of $\M$ of size $k$ are rigid. Let $G$ be a finitely generated non-abelian free group. Then the following hold.
        \begin{enumerate}[label=(\roman*), ref=(\Roman{enumi}.\roman*)]
            \item \label{swir general} $G$ admits a sharply $k$-homogeneous action on $\M$.
            \item \label{swir freeness} If $\ind$ satisfies (Free), then $\ShHomog_k(\mc{M}, G)$ is a comeagre subset of $\Sh_k(\mc{M}, G)$. 
        \end{enumerate} 	
        \item \label{case sir} Suppose $\mc{M}$ has a SIR. Then the following hold. 
        \begin{enumerate}[label=(\roman*), ref=(\Roman{enumi}.\roman*)]
            \item \label{sir 2} For any non-abelian finitely generated free groups $F$, $F'$, the group $G=(\cyc_{2} \times F) \ast F'$ admits a sharply $2$-homogeneous action on $\mc{M}$. If additionally $\ind$ satisfies (Free), then $\ShHChar_2(\mc{M}, G)$ is a comeagre subset of $\ShChar_2(\mc{M}, G)$.
            \item \label{sir 3, 1} If $\mc{M}$ is a transitive $(\sym_3, 1)$--structure, then any completion $G$ of $H = \sym_3$ (see Example \ref{ex:s2 s3}) admits a sharply $3$-homogeneous action on $\mc{M}$. If additionally $\ind$ satisfies (Free), then $\ShHChar_3(\mc{M}, G)$ is a comeagre subset of $\ShChar_3(\mc{M}, G)$.
            \item \label{sir 3} If $\mc{M}$ is a transitive $(\sym_{3},r)$-structure for some $r > 1$, then any completion $G$ of the standard $(\sym_{3},r)$-seed group $H_r$ (see \cref{d:standard s3 seed}) admits a sharply $3$-homogeneous action on $\M$.
        \end{enumerate}
        \item \label{case 2 tournament} Suppose $\M$ is the random tournament. Then for any non-abelian finitely generated free groups $F$, $F'$, the group $G=(\cyc_{3}\times F)\ast F'$ admits a sharply $3$-homogeneous action on $\M$.
        \item \label{case reducts of Q} Let $k \geq 1$. The following holds for the reducts of $(\mathbb{Q},<)$: 
        \begin{enumerate}[label=(\roman*), ref=(\Roman{enumi}.\roman*)]
            \item \label{case Q} Let $G$ be a finitely generated non-abelian free group. Then $G$ admits a sharply $k$-homogeneous action on $(\Q, <)$ and $\ShHomog_k((\Q, <), G)$ is a comeagre subset of $\Sh_k((\Q, <), G)$.
            \item \label{case betweenness} Let $\M=(\mb{Q},B^{(3)})$ be the generic betweenness structure. Then:
            \begin{enumerate}[label=(\alph*), ref=(\Roman{enumi}.\roman{enumii}.\alph*)]
                \item \label{case betw 1} any finitely generated non-abelian free group admits a sharply $1$-homogeneous action on $\M$;
                \item \label{case betw 2} given $k \geq 2$ and $G$ a completion of $\cyc_{2}$, where $\cyc_2 \leq \sym_k$ has maximum support in the permutation action, the group $G$ admits a sharply $k$-homogeneous action on $\M$.
            \end{enumerate} 
            \item \label{case cyc} Let $\M=(\mb{Q},C^{(3)})$ be the generic cyclic order. Then:
            \begin{enumerate}[label=(\alph*), ref=(\Roman{enumi}.\roman{enumii}.\alph*)]
                \item \label{case cyc 1} any finitely generated non-abelian free group admits a sharply $1$-homogeneous action on $\M$;
                \item \label{case cyc 2} given $k \geq 2$ and $G$ a completion of $\Theta=\cyc_{k}$, the group $G$ admits a sharply $k$-homogeneous action on $\M$.
            \end{enumerate} 
            \item \label{case separation} Let $\M=(\mb{Q},S^{(4)})$ be the generic separation structure. Then:
            \begin{enumerate}[label=(\alph*), ref=(\Roman{enumi}.\roman{enumii}.\alph*)]
                \item \label{case sep 1} any finitely generated non-abelian free group admits a sharply $1$-homogeneous action on $\mc{M}$;
                \item \label{case sep 2} any completion of $\Theta = \cyc_2$ admits a sharply $2$-homogeneous action on $\mc{M}$;
                \item \label{case sep 3} for odd $k \geq 3$, any completion of $\dih_k$ admits a sharply $k$-homogeneous action on $\M$;
                \item \label{case sep 4} for even $k \geq 4$, there is no sharply $k$-homogeneous action of a group on $\mc{M}$.
            \end{enumerate}
        \end{enumerate}
        \item \label{case k-tournament} Let $3 \leq k \leq 5$, and let $\M$ be the generic $k$-tournament. Then any completion $G$ of $H=\alt_{k}$ admits a sharply $k$-homogeneous action on $\M$.  
    \end{enumerate}
\end{thm}

\begin{proof}
    We first observe that in each of the cases above (except the negative result \ref{case sep 4}), we have that $\mc{M}$ is a $(\Theta, r)$-structure for some $r \geq 1$ for the relevant $\Theta$: in \ref{case swir} with $\Theta = \mathtt{1}$ this is straightforward (see the comments in the proof of Corollary \ref{c: pleasant ssa SWIR rigid}); in \ref{case sir} this follows by Lemma \ref{l: S2 S3 str} (with $\Theta = \sym_2$ in \ref{sir 2}) or by assumption; case \ref{case 2 tournament} is by the same argument as \ref{case swir}; and cases \ref{case reducts of Q}, \ref{case k-tournament} are easily checked. In all cases except \ref{sir 3}, we have $H = \Theta$; in \ref{sir 3} we have $H = H_r$. So in all cases $H$ is a compatible seed group for the $(\Theta, r)$-structure $\mc{M}$. In all cases, the groups $G$ specified in Theorem \ref{t:main} are completions of the corresponding $H$.
    
    We now consider cases separately.

    \ref{case swir} and \ref{case Q}: By Corollary \ref{c: pleasant ssa SWIR rigid}, the trivial action $\lambda : \mc{M} \curvearrowleft \mathtt{1}$ is a pleasant structural seed action on $\mc{M}$, and also generous if $\ind$ satisfies (Free) or $\mc{M} = (\Q, <)$. Let $\mc{F}$ be the corresponding arc-extensive family for $\lambda$, where if $\lambda$ is generous the family $\mc{F}$ consists of all finite taut extensions. By Proposition \ref{p:sharply k-homog actions on structures} we have $|\TaF| = 2^{\aleph_0}$ and $\SHTaF$ is dense $\mathrm{G}_\delta$ in $\TaF$, so $\SHTaF \neq \emp$, giving \ref{swir general}. For \ref{swir freeness} and \ref{case Q}: we have that $\mc{F}$ consists of all finite taut extensions, so $\SHTa_{k, \lambda}(\mc{M}, G)$ is dense $\mathrm{G}_\delta$ in $\Ta_\lambda(\mc{M}, G)$. Considering the definition of a taut extension (Definition \ref{d:taut partial actions}), we have that any $k$-sharp action $\mu : \mc{M} \curvearrowleft G$ is taut: \ref{c-fullness}, \ref{c-centraliser acts trivially} are immediate, and \ref{c-no invariant sets} follows by rigidity of $k$-substructures and $k$-sharpness. So $\Ta_\lambda(\mc{M}, G) = \Sh_k(\mc{M}, G)$ and $\SHTa_{k, \lambda}(\mc{M}, G) = \ShHomog_k(\mc{M}, G)$, and we are done.

    \ref{sir 2}, \ref{sir 3, 1}: The existence of the corresponding sharply $k$-homogeneous action follows from Corollary \ref{c:seed actions SIR} and Proposition \ref{p:sharply k-homog actions on structures}. Suppose that $\ind$ additionally satisfies (Free), and let $k \in \{2, 3\}$ and $H = \Theta = \sym_k$ (the argument is the same in both cases). By Corollary \ref{c:seed actions SIR}, there is a comeagre subset $\Lambda \sub \Se_k(\mc{M}, \Theta)$ consisting of generous pleasant structural seed actions. It is straightforward to check that the restriction map $\rho : \mu \mapsto \mu|_\Theta$ is a continuous open map $\rho : \ShChar_k(\mc{M}, G) \to \Se_k(\mc{M}, \Theta)$. By Lemma \ref{l: action of fin subgp of G} we have $\Ta_\lambda(\mc{M}, G) \sub \ShChar_k(\mc{M}, G)$ for each $\lambda \in \Lambda$. By Lemma \ref{l: taut transitive first three conds}, it is straightforward to see that for each $\lambda \in \Lambda$ we have $\rho^{-1}(\lambda) = \Ta_\lambda(\mc{M}, G)$. By Proposition \ref{p:sharply k-homog actions on structures} we have that $\SHTa_{k, \lambda}(\mc{M}, G)$ is a dense $\mathrm{G}_\delta$ set in $\Ta_\lambda(\mc{M}, G)$ for each $\lambda \in \Lambda$, and so $\ShHChar_k(\mc{M}, G) \cap \rho^{-1}(\lambda)$ is comeagre in $\rho^{-1}(\lambda)$ for each $\lambda \in \Lambda$. It is straightforward to check that $\ShHChar_k(\mc{M}, G)$ has the Baire Property. Thus by \cite[Theorem 1.33]{Mel26} (a generalised version of the Kuratowski-Ulam theorem) we have that $\ShHChar_k(\mc{M}, G)$ is comeagre in $\ShChar_k(\mc{M}, G)$ as required.

    \ref{case betw 1}, \ref{case cyc 1}, \ref{case sep 1}: these follow from \ref{case Q} via the following claim $(\ast$), whose proof is immediate.
    \begin{enumerate}
        \item[$(\ast)$] Let $\mu : \mc{M} \curvearrowleft G$ be sharply $k$-homogeneous and let $\mc{M}'$ be a reduct of $\mc{M}$ with the property that, for each bijection $f$ of $k$-sets of $M$, the map $f$ is a partial isomorphism of $\mc{M}$ if and only if it is a partial isomorphism of $\mc{M}'$. Then $\mu$ is a sharply $k$-homogeneous action on $\mc{M}'$.
    \end{enumerate}

    \ref{case betw 2}, \ref{case cyc 2}, \ref{case sep 3}: these follow from Proposition \ref{p:sharply k-homog actions on structures} together with \cref{c:reducts Q}. \ref{case sep 2} follows from \ref{case cyc 2} with $k = 2$ and claim $(\ast)$. \ref{case sep 4} follows from \cref{p: Theta not robust}\ref{i: no sh k-homog} and \cref{l:examples}\ref{robust Dk}. 

    The remaining cases \ref{sir 3}, \ref{case 2 tournament}, \ref{case k-tournament} follow by Proposition \ref{p:sharply k-homog actions on structures} together with: \cref{p:general s3} in case \ref{sir 3}, Lemma \ref{l:3 seed action 2tournament} in case \ref{case 2 tournament} and \cref{l:seed actions on tournaments} in case \ref{case k-tournament}.
\end{proof}

\begin{eg}
    See \cref{ex: examples for main str thm} for examples of \Fr structures to which Theorem \ref{t:main} can be applied. The additional case \ref{sir 3, 1} of Theorem \ref{t:main} (not included in Theorem \ref{t:main_intro} in the Introduction, which is a summarised version of Theorem \ref{t:main}) applies in particular to \Fr structures with free amalgamation where all relations are of arity $\geq 4$: for example, the random $4$-hypergraph.
\end{eg}

\section{Further questions}
\label{s:questions}

We have the following further questions regarding actions on sets:

\begin{question}
    Let $k = 2, 3$. Do there exist sharply $k$-transitive actions of non-split finitely presented groups on an infinite set in characteristics other than $2$? Do there exist sharply $k$-transitive actions of non-split simple finitely presented groups on an infinite set? 
\end{question}

\begin{definition} \label{d:tau numbers} 
    Let $U$ be a countably infinite set. For $k\geq 2$ write:
	\begin{itemize}
		\item $\tau(k)$ for the minimum number of orbits of the diagonal action on $(U)^{k}$ induced by a $k$-sharp action of a group on $U$,
		\item $\tau'(k)$ for the minimum number of orbits of the diagonal action on $(U)^{k}$ induced by a $k$-sharp action $\lambda$  of a group on $U$, under the condition that $\lambda$ induces a transitive action on $[U]^{k}$ (in other words, this is asking for the maximum size of the stabiliser $\Theta$ of $k$-tuples for such an action). 
	\end{itemize}
\end{definition}

\begin{question}
	What is the asymptotic behavior of $\tau(k)$ and $\tau'(k)$ as $k$ goes to infinity? Is there any infinite set $X\subseteq\mathbb{N}$ on which $\frac{k!}{\tau'(k)}$ is bounded on $X$ from below by some positive multiple of the function $k^{c}$ for $c>1$?
\end{question}

We believe there is a considerable amount of scope for further investigation of sharply $k$-homogeneous actions on structures.

\begin{question}
    Is there a relational \Fr structure that does \emph{not} admit a sharply $2$-homogeneous action? Is there a relational \Fr structure with rigid $k$-subsets which does not admit a sharply $k$-homogeneous action?
\end{question}

\begin{question}
    Is it possible to extend Theorem \ref{t:main_intro} to $(\Theta, r)$-structures $\mc{M}$ for robust $\Theta$ in general, under suitable assumptions on $\mc{M}$ (e.g.\ the presence of a SIR)? 
\end{question}

A key class of \Fr structures which we did not investigate are those that only admit \emph{local} SWIRs: SWIRs where $B \ind_A C$ is defined only when the base $A$ is non-empty. The central example here is the rational Urysohn space, but there are many other well-known structures of this kind (see \cite{KSW26}).

\begin{question}
    Does there exist a sharply $k$-homogeneous action on the rational Urysohn space for $k\in\{2,3\}$? 
\end{question}

Also note that, for $n \geq 3$, the generic $n$-anticlique-free oriented graph has a SWIR but has some non-rigid $2$-substructures and $3$-substructures, and so is not covered by Theorem \ref{t:main}\ref{case swir}:

\begin{question}
    Let $n \geq 3$. Does the generic $n$-anticlique-free oriented graph admit a sharply $k$-homogeneous action for $k = 2, 3$?
\end{question}

It may be that the approach we took for $k$-tournaments could be applied to this structure.

We restricted attention in Definition \ref{d:theta r structure} and Subsection \ref{subs:larger seeds} to transitive structures in the case $\Theta = \sym_3$. We conjecture that this restriction is not necessary:

\begin{conj}
    Let $\mc{M}$ be a relational \Fr structure with strong amalgamation and a SIR, where $\mc{M}$ is not transitive. We conjecture that there exists a sharply $3$-homogeneous action on $\mc{M}$ by a finitely generated non-abelian virtually free group.  
\end{conj}

Note that even for the generic $m$-tournament we do not have a full answer to the following:

\begin{question}
    For what values of $m\geq 2$ and $k\geq 2$ does the generic $m$-tournament admit a sharply $k$-homogeneous action? 
\end{question}

One could also investigate further which groups admit sharply $k$-homogeneous actions on certain \Fr structures:

\begin{question}
	Which finitely generated (resp. countable) orderable groups admit a sharply $k$-homogeneous action on $(\Q, <)$?
	Does the property hold for all non-cyclic right angled Artin groups which do not decompose as a direct product of two infinite groups? Does the property hold for every torsion-free one-relator group?
\end{question} 

\begin{question}
    Let $\mc{M}$ be a strong relational \Fr structure and $k\geq 1$. What group operations is the collection of all finitely generated groups $G$ that admit a sharply $k$-homogeneous action on $\mc{M}$ closed under? For instance, under which kinds of amalgamated products is the said class closed?
\end{question}

\begin{question}
    Can a simple finitely generated (or finitely presented) group act sharply $k$-homogeneously on some relational \Fr structure which is not a pure set? Can it act on $(\mathbb{Q},<)$ or any of its reducts?  
\end{question}

One could also generalise from \Fr structures to $k$-homogeneous structures:

\begin{question}
    When do sharply $k$-homogeneous actions exist for $k$-homogeneous $\mc{M}$? What about the case where $G$ and $\mc{M}$ are finite?
\end{question}
 
\bibliographystyle{alpha}
\bibliography{references}

@article{Ame25,
  title={Non-split sharply 2-transitive groups of bounded exponent},
  author={Amelio, Marco},
  journal={arXiv preprint arXiv:2509.11958},
  year={2025}
}

@article{AA26,
  title={Non-split sharply 2- and 3-transitive groups in {$\mathrm{SL}_n(\mathbb Z)$}},
  author={Amelio, Marco and Andr{\'e}, Simon},
  journal={Proceedings of the American Mathematical Society},
  note={In press. arXiv preprint arXiv:2604.07001},
  year={2026}
}

@article{AAT23,
  title={Non-split sharply 2-transitive groups of odd positive characteristic},
  author={Amelio, Marco and Andr{\'e}, Simon and Tent, Katrin},
  journal={International Mathematics Research Notices},
  volume={2025},
  number={19},
  pages={rnaf294},
  year={2025},
  publisher={Oxford University Press}
}

@article{ABW18,
    title={Sharply 2-transitive groups of finite {M}orley rank},
    author={Altinel, Tuna and Berkman, Ayse and Wagner, Frank Olaf},
    journal={arXiv preprint arXiv:1811.10854},
    year={2018}
}

@article{AG24,
    title={Finitely generated simple sharply 2-transitive groups},
    author={Andr{\'e}, Simon and Guirardel, Vincent},
    journal={Compositio Mathematica},
    volume={160},
    number={8},
    pages={1941--1957},
    year={2024},
    publisher={London Mathematical Society}
}

@article{AT23,
  title={Simple sharply 2-transitive groups},
  author={Andr{\'e}, Simon and Tent, Katrin},
  journal={Transactions of the American Mathematical Society},
  volume={376},
  number={06},
  pages={3965--3993},
  year={2023}
}

@article{Ass86,
  title={Encha{\^{i}}nabilit{\'{e}} et seuil de monomorphie des tournois n-aires},
  author={Assous, Roland},
  journal={Discrete Mathematics},
  volume={62},
  number={2},
  pages={119--125},
  year={1986},
  publisher={Elsevier}
}

@article{Bau16,
    title={Free amalgamation and automorphism groups},
    author={Baudisch, Andreas},
    journal={The Journal of Symbolic Logic},
    volume={81},
    number={3},
    pages={936--947},
    year={2016},
    publisher={Cambridge University Press}
}

@article{BF92,
    title={A combination theorem for negatively curved groups},
    author={Bestvina, Mladen and Feighn, Mark},
    journal={Journal of Differential Geometry},
    volume={35},
    number={1},
    pages={85--101},
    year={1992},
    publisher={Lehigh University}
}

@book{BN94,
    title={Groups of finite Morley rank},
    author={Borovik, Alexandre and Nesin, Ali},
    publisher={Oxford University Press},
    year={1994}
}

@article{Bod15,
    title={Ramsey classes: examples and constructions.},
    author={Bodirsky, Manuel},
    journal={Surveys in Combinatorics},
    volume={424},
    year={2015}
}

@book{Bog08,
  title={Introduction to group theory},
  author={Bogopol'skij, Oleg Vladimirovi{\v{c}}},
  volume={6},
  year={2008},
  publisher={European Mathematical Society Z{\"u}rich}
}

@article{BSWY26,
  title={Determining the normal subgroups of the automorphism groups of ultrahomogeneous structures via stabilisers},
  author={Bernert, Thomas and Sullivan, Rob and Winkel, Jeroen and Yang, Shujie},
  journal={arXiv preprint arXiv:2603.27890},
  year={2026}
}

@article{BT81,
  title={Graphs which contain all small graphs},
  author={Bollob{\'a}s, B{\'e}la and Thomason, Andrew},
  journal={European Journal of Combinatorics},
  volume={2},
  number={1},
  pages={13--15},
  year={1981},
  publisher={Academic Press}
}

@article{Cam76,
  title={Transitivity of permutation groups on unordered sets},
  author={Cameron, Peter J.},
  journal={Mathematische Zeitschrift},
  volume={148},
  pages={127--139},
  year={1976},
  publisher={Springer}
}

@book{Cam90, 
    series={London Mathematical Society Lecture Note Series}, 
    title={Oligomorphic Permutation Groups}, 
    publisher={Cambridge University Press}, 
    author={Cameron, Peter J.}, 
    year={1990}, 
    collection={London Mathematical Society Lecture Note Series}
}

@article{Cam00,
  title={Homogeneous {C}ayley objects},
  author={Cameron, Peter J.},
  journal={European Journal of Combinatorics},
  volume={21},
  number={6},
  pages={745--760},
  year={2000},
  publisher={Elsevier}
}

@article{Che15,
    title={{H}enson graphs and {U}rysohn-{H}enson graphs as {C}ayley graphs},
    author={Cherlin, Gregory},
    journal={Functional Analysis and Its Applications},
    volume={3},
    number={49},
    pages={189--200},
    year={2015},
    publisher={Springer-Verlag GmbH}
}

@book{chiswell2001introduction,
  title={Introduction to {$\lambda$}-trees},
  author={Chiswell, Ian},
  year={2001},
  publisher={World Scientific}
}

@misc{Con09,
    title = {Dihedral groups {II}},
    author = {Conrad, Keith},
    year = {2009},
    howpublished = {\url{https://kconrad.math.uconn.edu/blurbs/grouptheory/dihedral2.pdf}}
}

@article{Con17,
    title={An axiomatic approach to free amalgamation},
    author={Conant, Gabriel},
    journal={The Journal of Symbolic Logic},
    volume={82},
    number={2},
    pages={648--671},
    year={2017},
    publisher={Cambridge University Press}
}

@article{CT23,
  title={Mock hyperbolic reflection spaces and {F}robenius groups of finite {M}orley rank},
  author={Clausen, Tim and Tent, Katrin},
  journal={Model Theory},
  volume={2},
  number={2},
  pages={137--175},
  year={2023},
  publisher={Mathematical Sciences Publishers}
}

@article{CT20,
  title={On the geometry of sharply 2-transitive groups},
  author={Clausen, Tim and Tent, Katrin},
  journal={arXiv preprint arXiv:2002.05187},
  year={2020}
}

@book{DM96,
  title={Permutation groups},
  author={Dixon, John D and Mortimer, Brian},
  volume={163},
  year={1996},
  publisher={Springer Science \& Business Media}
}

@article{Fro02,
    title={{\"{U}}ber primitive {G}ruppen des {G}rades $n$ und der {K}lasse $n - 1$},
    author={G. Frobenius},
    year={1902},
    journal={S. B. Akad. Berlin},
    pages={455--459}
}

@article{GG14,
  title={Sharply 2-transitive linear groups},
  author={Glasner, Yair and Gulko, Dennis D},
  journal={International Mathematics Research Notices},
  volume={2014},
  number={10},
  pages={2691--2701},
  year={2014},
  publisher={OUP}
}

@article{GG21,
  title={Non-split linear sharply 2-transitive groups},
  author={Glasner, Yair and Gulko, Dennis},
  journal={Proceedings of the American Mathematical Society},
  volume={149},
  number={6},
  pages={2305--2317},
  year={2021}
}

@article{GS71,
  title={A constructive solution to a tournament problem},
  author={Graham, Ronald L and Spencer, Joel H},
  journal={Canadian Mathematical Bulletin},
  volume={14},
  number={1},
  pages={45--48},
  year={1971},
  publisher={Cambridge University Press}
}

@article{Hal54,
    title={On a theorem of {J}ordan},
    author={Hall Jr., Marshall},
    journal={Pacific J. Math},
    volume={4},
    pages={219--226},
    year={1954}
}

@book{Hat02,
    title = {Algebraic topology},
    author = {Hatcher, Allen},
    publisher = {Cambridge University Press},
    year = {2002}
}

@article{Hen71,
    title={A family of countable homogeneous graphs},
    author={Henson, C Ward},
    journal={Pacific J. Math},
    volume={38},
    number={1},
    pages={69--83},
    year={1971},
    publisher={Mathematical Sciences Publishers}
}

@book{Hod93,
    AUTHOR = {Hodges, Wilfrid},
    TITLE = {Model theory},
    SERIES = {Encyclopedia of Mathematics and its Applications},
    VOLUME = {42},
    PUBLISHER = {Cambridge University Press, Cambridge},
    YEAR = {1993},
    PAGES = {xiv+772},
    ISBN = {0-521-30442-3},
}

@article{JK04,
    title={The random tournament as a {C}ayley tournament},
    author={Jaligot, Eric and Khelif, Anatole},
    journal={Aequationes mathematicae},
    volume={67},
    pages={73--79},
    year={2004},
    publisher={Springer}
}

@article{Jor72,
    title={Recherches sur les substitutions},
    author={Jordan, Camille},
    journal={Journal de Math{\'e}matiques Pures et Appliqu{\'e}es},
    volume={17},
    pages={351--367},
    year={1872}
}

@article{KT17,
  title={Universal-homogeneous structures are generic},
  author={Kabluchko, Zakhar and Tent, Katrin},
  journal={arXiv preprint arXiv:1710.06137},
  year={2017}
}

@book{Kec95,
    title={Classical descriptive set theory},
    author={Alexander S. Kechris},
    year={1995},
    publisher={Springer New York, NY}
}

@article{KS19,
    title={Automorphism groups of finite topological rank},
    author={Kaplan, Itay and Simon, Pierre},
    journal={Transactions of the American Mathematical Society},
    volume={372},
    number={3},
    pages={2011--2043},
    year={2019}
}

@article{karrass1973finite,
	title={Finite and infinite cyclic extensions of free groups},
	author={Karrass, Abraham and Pietrowski, Alfred and Solitar, Donald},
	journal={Journal of the Australian Mathematical Society},
	volume={16},
	number={4},
	pages={458--466},
	year={1973},
	publisher={Cambridge University Press}
}

@article{Ker74,
    title={On infinite sharply multiply transitive groups},
    author={Kerby, William},
    journal={Hamburger mathematische Einzelschriften – Heft 6},
    publisher={Vandenhoeck \& Ruprecht},
    year={1974}
}

@article{KSW26,
  title={Group-extensive embeddings into {F}ra{\"{i}}ss{\'{e}} structures and stationary weak independence relations},
  author={Kwiatkowska, Aleksandra and Sullivan, Rob and Winkel, Jeroen},
  journal={Journal of Algebra},
  year={2026},
  note={In press}
}

@article{Li19,
    title={Automorphism groups of homogeneous structures with stationary weak independence relations},
    author={Li, Yibei},
    journal={arXiv preprint arXiv:1911.08540},
    year={2019}
}

@article{Li20,
    title={Automorphism groups of linearly ordered homogeneous structures},
    author={Li, Yibei},
    journal={arXiv preprint arXiv:2009.02475},
    year={2020}
}

@phdthesis{Li21,
    author = {Li, Yibei},
    title = {Normal subgroups of the automorphism groups of some homogeneous structures},
    year = {2021},
    school = {Imperial College London},
    url = {http://hdl.handle.net/10044/1/89167}
}

@book{LS1979,
	title={Combinatorial group theory},
	author={Roger C. Lyndon and Paul E. Schupp},
	year={1979},
        publisher={Springer Berlin, Heidelberg}
}

@article{Mac96, 
  title={Sharply multiply homogeneous permutation groups, and rational scale types},
  author={Macpherson, Dugald},
  year={1996},
  journal={Forum Mathematicum},
  publisher={de Gruyter},
  volume={8},
  number={4},
  pages={501--507}
}

@article{Mac11,
    title={A survey of homogeneous structures},
    author={Macpherson, Dugald},
    journal={Discrete mathematics},
    volume={311},
    number={15},
    pages={1599--1634},
    year={2011},
    publisher={Elsevier}
}

@article{May06,
  title={Sharply 2-transitive groups with point stabilizer of exponent 3 or 6},
  author={Mayr, Peter},
  journal={Proceedings of the American Mathematical Society},
  volume={134},
  number={1},
  pages={9--13},
  year={2006}
}

@book{Mel26,
    title = {Some aspects of topological dynamics of Polish groups (with an introduction to descriptive set theory)},
    author = {Melleray, Julien},
    year = {2026},
    publisher = {Cours Sp{\'{e}}cialis{\'{e}}s de la Soci{\'{e}}t{\'{e}} Math{\'{e}}matique de France},
    volume = {34},
    note = {To appear. See arXiv:2602.23799.}
}

@book{munkres2025elements,
  title={Elements of algebraic topology},
  author={Munkres, James R and Krantz, Steven G and Parks, Harold R},
  year={2025},
  publisher={Chapman and Hall/CRC}
}

@article{MT11,
    title={Simplicity of some automorphism groups},
    author={Macpherson, Dugald and Tent, Katrin},
    journal={Journal of Algebra},
    volume={342},
    number={1},
    pages={40--52},
    year={2011},
    publisher={Elsevier}
}

@article{Neu40,
  title={On the commutativity of addition},
  author={Neumann, Bernhard Hermann},
  journal={Journal of the London Mathematical Society},
  volume={1},
  number={3},
  pages={203--208},
  year={1940},
  publisher={Oxford University Press}
}

@book{Pas68,
    title={Permutation groups},
    author={Donald S. Passman},
    year={1968},
    publisher = {Benjamin, New York}
}

@article{RST17,
    title={A sharply 2-transitive group without a non-trivial abelian normal subgroup},
    author={Rips, Eliyahu and Segev, Yoav and Tent, Katrin},
    journal={J. Eur. Math. Soc.},
    volume={19},
    number={10},
    pages={2895--2910},
    year={2017}
}

@article{RT19,
    title={Sharply 2-transitive groups of characteristic 0},
    author={Rips, Eliyahu and Tent, Katrin},
    journal={Journal f{\"u}r die reine und angewandte Mathematik (Crelles Journal)},
    volume={2019},
    number={750},
    pages={227--238},
    year={2019},
    publisher={De Gruyter}
}

@book{serre2002trees,
  title={Trees},
  author={Serre, Jean-Pierre},
  year={2002},
  publisher={Springer Science \& Business Media}
}

@article{Ten00,
  title={Sharply $n$-transitive groups in o-minimal structures},
  author={Tent, Katrin},
  journal={Forum Mathematicum},
  volume={12},
  number={1},
  pages={65--76},
  year={2000}
}

@article{Ten16,
    title={Sharply 3-transitive groups},
    author={Tent, Katrin},
    journal={Advances in Mathematics},
    volume={286},
    pages={722--728},
    year={2016},
    publisher={Elsevier}
}

@article{Ten16I,
    title={Infinite sharply multiply transitive groups},
    author={Tent, Katrin},
    journal={Jahresbericht der Deutschen Mathematiker-Vereinigung},
    volume={118},
    number={2},
    pages={75--85},
    year={2016},
    publisher={Springer}
}

@inproceedings{Tit49,
    title={Groupes triplement transitifs et generalisations},
    author={Tits, Jacques},
    booktitle={Alg{\`e}bre et th{\'e}orie des nombres},
    pages={207--208},
    year={1949},
    organization={CNRS}
}

@book{Tit52,
    title={G\'{e}n\'{e}ralisations des groupes projectifs bas\'{e}es sur leurs propri\'{e}t\'{e}s de transitivit\'{e}},
    author={Tits, Jacques},
    publisher={Palais des Acad\'{e}mies, Bruxelles},
    year={1952}
}

@article{Tit52b,
  title={Sur les groupes doublement transitifs continus},
  author={Tits, Jacques},
  journal={Commentarii Mathematici Helvetici},
  volume={26},
  number={1},
  pages={203--224},
  year={1952},
  publisher={Springer}
}

@article{Tit56,
  title={Sur les groupes doublement transitifs continus: correction et compl{\'e}ments},
  author={Tits, Jacques},
  journal={Comment. Math. Helv},
  volume={30},
  pages={234--240},
  year={1956}
}

@article{Tur04,
  title={Splitting of sharply 2-transitive groups of characteristic 3},
  author={T{\"u}rkelli, Seyfi},
  journal={Turkish Journal of Mathematics},
  volume={28},
  number={3},
  pages={295--298},
  year={2004}
}

@article{TZ13,
    title={On the isometry group of the {U}rysohn space},
    author={Tent, Katrin and Ziegler, Martin},
    journal={Journal of the London Mathematical Society},
    volume={87},
    number={1},
    pages={289--303},
    year={2013},
    publisher={Oxford University Press}
}

@article{TZ16,
    title={Sharply 2-transitive groups},
    author={Tent, Katrin and Ziegler, Martin},
    journal={Advances in Geometry},
    volume={16},
    number={1},
    pages={131--134},
    year={2016},
    publisher={De Gruyter}
}

@article{Wah86,
  title={Lokal endliche, scharf zweifach transitive {P}ermntationsgruppen},
  author={von W{\"a}hling, H.},
  journal={Abhandlungen aus dem Mathematischen Seminar der Universit{\"a}t Hamburg},
  volume={56},
  number={1},
  pages={107--113},
  year={1986},
  publisher={Springer}
}

@article{Yos79,
  title={On infinite four-transitive permutation groups},
  author={Yoshizawa, Mitsuo},
  journal={Journal of the London Mathematical Society},
  volume={2},
  number={3},
  pages={437--438},
  year={1979},
  publisher={Oxford University Press}
}

@inproceedings{Zas35K,
    title={Kennzeichnung endlicher linearer {G}ruppen als {P}ermutationsgruppen},
    author={Zassenhaus, Hans},
    booktitle={Abhandlungen aus dem Mathematischen Seminar der Universit{\"a}t Hamburg},
    volume={11},
    pages={17--40},
    year={1935},
    organization={Springer}
}

@inproceedings{Zas35U,
    title={{\"U}ber endliche {F}astk{\"o}rper},
    author={Zassenhaus, Hans},
    booktitle={Abhandlungen aus dem mathematischen Seminar der Universit{\"a}t Hamburg},
    volume={11},
    pages={187--220},
    year={1935},
    organization={Springer}
}

\end{document}